\documentclass{amsart}
\usepackage[T1]{fontenc}

\usepackage{amsmath}
\usepackage{amssymb}
\usepackage{mathrsfs}

\usepackage{float}
\usepackage[dvips]{graphicx,color}
 
\usepackage{ascmac}

\usepackage{tikz}
\usetikzlibrary{calc,decorations.markings}
\usepackage{amsthm}
\theoremstyle{definition}
\newtheorem{THM}{Theorem}
\newtheorem{LEM}[THM]{Lemma}
\newtheorem{PROP}[THM]{Proposition}

\newtheorem{DEF}[THM]{Definition}
\newtheorem{RMK}[THM]{Remark}

\newtheorem*{THM*}{Theorem}
\newtheorem*{LEM*}{Lemma}
\newtheorem*{PROP*}{Proposition}
\newtheorem*{COR*}{Corollary}
\newtheorem*{DEF*}{Definition}
\newtheorem*{RMK*}{Remark}
\newtheorem*{EX*}{Example}

\numberwithin{figure}{section}
\numberwithin{equation}{section}
\numberwithin{THM}{section}

\title{The $\mathfrak{sl}_3$ colored Jones polynomials for $2$-bridge links}
\author{Wataru Yuasa}
\address{Department of Mathematics\\
  Tokyo Institute of Technology\\
  2-12-1 Ookayama, Meguro-ku, Tokyo 152-8551, Japan}
\email[]{yuasa.w.aa@m.titech.ac.jp}
\date{\today}

\begin{document}
\begin{abstract}
Kuperberg introduced web spaces for some Lie algebras which are generalizations of the Kauffman bracket skein module on a disk. 
We derive some formulas for $A_1$ and $A_2$ clasped web spaces by graphical calculus using skein theory. 
These formulas are colored version of skein relations, twist formulas and bubble skein expansion formulas.
We calculate the $\mathfrak{sl}_2$ and $\mathfrak{sl}_3$ colored Jones polynomials of $2$-bridge knots and links explicitly using twist formulas.
\end{abstract}
\maketitle
\tikzset{->-/.style={decoration={
  markings,
  mark=at position #1 with {\arrow[black,thin]{>}}},postaction={decorate}}}
\tikzset{-<-/.style={decoration={
  markings,
  mark=at position #1 with {\arrow[black,thin]{<}}},postaction={decorate}}}
\tikzset{
    triple/.style args={[#1] in [#2] in [#3]}{
        #1,preaction={preaction={draw,#3},draw,#2}
    }
}

\section{Introduction}
After Kauffman~\cite{Kauffman87} gave a reformulation of the Jones polynomial~\cite{Jones85} of a link using the Kauffman bracket,
the linear skein theory has developed in various directions.
This theory is related to quantum representations and treats modules whose elements are diagrams in a disk such as graphs, tangles, webs and so on. 
Lickorish~\cite{Lickorish91, Lickorish93, Lickorish92} introduced the linear skein theory based on quantum $\mathfrak{sl}_2$ representations. 
He constructed the quantum $SU(2)$ invariants of closed $3$-manifolds suggested by Witten~\cite{Witten89}, 
rigorously defined by Reshetikhin and Turaev~\cite{ReshetikhinTuraev91}. 
This linear skein theory is developed based on the Kauffman bracket~\cite{Kauffman87} and the Kauffman bracket skein module. 
The Kauffman bracket also gives polynomial invariants of knots and links called the colored Jones polynomials. 
(We call it the $\mathfrak{sl}_2$ colored Jones polynomials in this paper.)
Through the linear skein theory for the Kauffman bracket, 
we can define and calculate quantum $SU(2)$ invariants of closed $3$-manifolds and links graphically (see details in \cite{KauffmanLins94}, Chapter~$13,14$ of \cite{Lickorish97}).

Many quantum invariants and corresponding skein theory have been constructed for other quantum groups.
Kuperberg~\cite{Kuperberg94} defined brackets for Lie algebras $A_2$, $G_2$ and $C_2$ and corresponding quantum invariants of regular isotopy classes of link diagrams.
He also introduced web spaces for simple Lie algebra of rank $2$ in \cite{Kuperberg96}.
Web spaces are generalizations of the Kauffman bracket skein module. 
The Kauffman bracket skein module correspond to the $A_1$ web space.
The linear skein theory associated with $A_n$ was also introduced by Murakami, Ohtsuki and Yamada~\cite{MurakamiOhtsukiYamada98} and Sikora~\cite{Sikora05}. 
It gives a reformulation of the HOMFLY polynomial~\cite{HOMFLY85, PT88}. 
The quantum invariant of $3$-manifolds was also given by using these linear skein theories in a similar approach to Lickorish~\cite{Lickorish91, Lickorish93} (see Ohtsuki and Yamada~\cite{OhtsukiYamada97} for $SU(3)$ and Yokota~\cite{Yokota97} for $SU(n)$).

We will treat the skein theory for $A_1$ and $A_2$.
$A_1$ web spaces have particular elements called the Jones-Wenzl idempotents defined in \cite{Jones83, Wenzl87}. 
The Jones-Wenzl idempotents play an important role to construct the quantum $SU(2)$ invariants of $3$-manifolds and the $\mathfrak{sl}_2$ colored Jones polynomials of knots and links.
These idempotents are generalized to the $A_2$ case (see \cite{Kuperberg96} and \cite{OhtsukiYamada97}). 
We call them the $A_2$ clasps.
The $A_2$ clasps also play an important role to construct quantum $SU(3)$ invariants of $3$-manifolds and the $\mathfrak{sl}_3$ colored Jones polynomials of knots and links. 
Many people have calculated clasped $A_1$ web spaces and gave explicit formulas of $\mathfrak{sl}_2$ colored Jones polynomials for some knots and links by graphical calculus. 
These explicit formulas are useful for case studies of conjectures related to quantum invariants of knots and links: 
the volume conjecture~\cite{Kashaev97, MurakamiMurakami01}, the AJ conjecture~\cite{Garoufalidis04}, the slope conjecture~\cite{Garoufalidis11} and so on.

On the other hand, 
there are few examples, see Kim~\cite{Kim06, Kim07}, of graphical calculus for clasped $A_2$ web spaces. 
Only two explicit formulas of the $\mathfrak{sl}_3$ colored Jones polynomial were obtained using the representation theory of quantum groups. 
The formula for trefoil knot was given by Lawrence~\cite{Lawrence03}, more generally for the torus knot $T(2,k)$ was given by Garoufalidis, Morton and Vuong~\cite{GaroufalidisMortonVuong13}.
As far as the author knows, 
there is no example of graphical calculus of the $\mathfrak{sl}_3$ colored Jones polynomial for a non-trivial link. 

In this paper, 
we will give some formulas for clasped $A_1$ and $A_2$ web spaces. 
These formulas explicitly give the $\mathfrak{sl}_2$ colored Jones polynomials for a $2$-bridge link and the $\mathfrak{sl}_3$ colored Jones polynomials of type $(n,0)$ for it. 
We remark that the $\mathfrak{sl}_3$ colored Jones polynomial treated in \cite{Lawrence03} and \cite{GaroufalidisMortonVuong13} is type $(n,m)$.

The paper is organized as follows. 
We firstly introduce the $A_1$ and $A_2$ web spaces and clasps based on Kuperberg~\cite{Kuperberg96} in section~$2$. 
The colored skein relations and twist formulas for $A_1$ and $A_2$ web spaces are given in section~$3$. 
The $A_1$ web space is the Kauffman bracket skein module. 
In this case, 
Hajij gave the colored Kauffman bracket skein relation in \cite{Hajij14.2}. 
We will show it in another method using a lemma from the theory of integer partitions.
This method is used to show the colored $A_2$ bracket skein relations and twist formulas.
In section~$4$, 
we will give the $A_2$ bracket bubble skein expansion formula.
This formula is an $A_2$ version of the Kauffman bracket bubble skein expansion formula in \cite{Hajij14}.
In section~$5$, 
we give an explicit formulas for the $\mathfrak{sl}_2$ and $\mathfrak{sl}_3$ colored Jones polynomials for $2$-bridge knots and links. 
In this paper, 
we only treat the $\mathfrak{sl}_3$ colored Jones polynomials of type $(n,0)$. 

\section{Preliminaries}
In this section, 
we review definitions of two vector spaces, 
$A_1$ web spaces and $A_2$ web spaces, 
introduced by Kuperberg~\cite{Kuperberg96}. 
For each web space, 
we inductively define a particular element called a clasp according to~\cite{Kuperberg96, OhtsukiYamada97}.
\subsection{Quantum integers and $q$-Pochhammer symbol}
First, 
we organize notations of quantum integers and the $q$-Pochhammer symbol which we use in this paper. 

Let $k$ be an integer and $q$ an indeterminate. 
We denote $q^{\frac{k}{2}}-q^{-\frac{k}{2}}$ by $\{k\}$.
A quantum integer is defined by 
\[
 \left[k\right]=\frac{\{k\}}{\{1\}}.
\]
Let us denote $\{k\}!=\prod_{l=1}^{k}\{l\}$ and $\left[k\right]!=\prod_{l=1}^k\left[l\right]$. 
Let $n$ be a non-negative integer. 
Then, a version of $q$-binomial coefficient is defined by
\[
 {n \brack k}=\frac{\left[n\right]!}{\left[k\right]!\left[n-k\right]!}=\frac{\{n\}!}{\{k\}!\{n-k\}!}
\]
for $k\leq n$.
If $k>n$,
we define it by $0$.

A $q$-Pochhammer symbol is defined by 
\[
 (q;q)_k=\prod_{l=1}^{k}(1-q^l).
\]
We sometimes abbreviate it as $(q)_k$.
Another version of $q$-binomial coefficient is defined by
\[
 {n\choose k}_q=\frac{(q;q)_n}{(q;q)_k(q;q)_{n-k}}
\]
for $k\leq n$.
If $k>n$,
we define it by $0$.
We also define a $q$-multinomial coefficient as
\[
 {n\choose n_1,n_2,\dots,n_m}_q=\frac{(q)_n}{(q)_{n_1}(q)_{n_2}\cdots(q)_{n_m}},
\]
where $n_1, n_2,\dots, n_m$ are non-negative integers such that $n_1+n_2+\dots+n_m=n$.

It is easy to show the following transformation formulas:
\begin{itemize}
\item $\{k\}!=(-1)^kq^{-k(k+1)/4}(q;q)_k$,
\item ${n\brack k}=q^{-(n-k)k/2}{n\choose k}_q$,
\item $(q;q)_k=(-1)^kq'^{-\frac{1}{2}k(k+1)}(q';q')_k$,
\item ${n\choose k}_q=q'^{k^2-nk}{n\choose k}_{q'}$,
\end{itemize}
where $q'=q^{-1}$.
We use the following formulas for quantum integers in this paper.
\begin{LEM}\label{qinteger}
For any integers $a,b$ and $c$,
\begin{enumerate}
\item 
$\left[a\right]\left[b\right]
=\sum_{i=1}^{a}\left[a+b-(2i-1)\right]
=\sum_{i=1}^{b}\left[a+b-(2i-1)\right]$,
\item
$\left[a\right]\left[b\right]-\left[a-c\right]\left[b-c\right]
=\left[a+b-c\right]\left[c\right]$,
\item 
$\left[a\right]\left[b-c\right]+\left[c\right]\left[a-b\right]=\left[b\right]\left[a-c\right]$
\end{enumerate}
\end{LEM}
\begin{proof}
By easy calculation using the definition of quantum integers.
\end{proof}
\subsection{The $A_1$ web spaces}
Let $D_m$ denote the disk $\{\, z\in\mathbb{C} \mid \left|z\right|\leq 1\, \}$ with the set $E_m=\{\, \exp(2\pi\sqrt{-1}/m)^{j-1}\mid j=1,2,\dots,m\,\}$ of marked points on the boundary. 
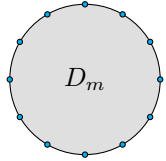
\begin{center}
\begin{tikzpicture}
\draw [thin, fill=lightgray!50] (0,0) circle [radius=1];
\foreach \i in {0,1,...,11} \draw[fill=cyan] ($(0,0) !1! \i*30:(1,0)$) circle [radius=1pt];
\node (0,0) {$D_m$};
\end{tikzpicture}
\end{center}
An {\em $A_1$ basis web} on $D_m$ is the boundary-fixing isotopy class of a proper embedding of arcs into $D_m$ with no intersection points such that each endpoint lies in $E_m$. 
For any point $p$ on the arcs, 
its neighborhood is either
$\,\tikz[baseline=-.6ex]{
\draw [thin, dashed, fill=lightgray!50] (0,0) circle [radius=.5];
\draw (-.5,0)--(.5,0);
\node (p) at (0,0) [above]{$p$};
\fill (0,0) circle [radius=1pt];
}\,$
or
$\,\tikz[baseline=-.6ex]{
\draw [thin, dashed] (0,0) circle [radius=.5];
\clip (0,0) circle [radius=.5];
\draw [thin, fill=lightgray!50] (0,-.5) rectangle (.5,.5);
\draw (0,0)--(.5,0);
\node (p) at (0,0) [left]{$p$};
\draw [fill=cyan] (0,0) circle [radius=1pt];
}\,$.

Let $B_m$ be the set of $A_1$ basis webs on $D_m$. 
We consider $B_0$ has a single element $\emptyset$ called the empty disk. 
For example,
$B_6$ consists of the following $A_1$ basis webs:
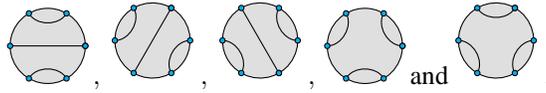
\[
\,\begin{tikzpicture}
\draw [thin, fill=lightgray!50] (0,0) circle [radius=.5];
\draw (0:.5) -- (180:.5);
\draw (60:.5) to[out=-120, in=-60] (120:.5);
\draw (240:.5) to[out=60, in=120] (300:.5);
\foreach \i in {0,1,...,6} \draw[fill=cyan] ($(0,0) !1! \i*60:(.5,0)$) circle [radius=1pt];
\end{tikzpicture}\,,
\,\begin{tikzpicture}[rotate=60]
\draw [thin, fill=lightgray!50] (0,0) circle [radius=.5];
\draw (0:.5) -- (180:.5);
\draw (60:.5) to[out=-120, in=-60] (120:.5);
\draw (240:.5) to[out=60, in=120] (300:.5);
\foreach \i in {0,1,...,6} \draw[fill=cyan] ($(0,0) !1! \i*60:(.5,0)$) circle [radius=1pt];
\end{tikzpicture}\,,
\,\begin{tikzpicture}[rotate=-60]
\draw [thin, fill=lightgray!50] (0,0) circle [radius=.5];
\draw (0:.5) -- (180:.5);
\draw (60:.5) to[out=-120, in=-60] (120:.5);
\draw (240:.5) to[out=60, in=120] (300:.5);
\foreach \i in {0,1,...,6} \draw[fill=cyan] ($(0,0) !1! \i*60:(.5,0)$) circle [radius=1pt];
\end{tikzpicture}\,,
\,\begin{tikzpicture}
\draw [thin, fill=lightgray!50] (0,0) circle [radius=.5];
\draw (0:.5) to[out=180, in=-120] (60:.5);
\draw[rotate=120] (0:.5) to[out=180, in=-120] (60:.5);
\draw[rotate=240] (0:.5) to[out=180, in=-120] (60:.5);
\foreach \i in {0,1,...,6} \draw[fill=cyan] ($(0,0) !1! \i*60:(.5,0)$) circle [radius=1pt];
\end{tikzpicture}\,
\text{and}
\,\begin{tikzpicture}[rotate=60]
\draw [thin, fill=lightgray!50] (0,0) circle [radius=.5];
\draw (0:.5) to[out=180, in=-120] (60:.5);
\draw[rotate=120] (0:.5) to[out=180, in=-120] (60:.5);
\draw[rotate=240] (0:.5) to[out=180, in=-120] (60:.5);
\foreach \i in {0,1,...,6} \draw[fill=cyan] ($(0,0) !1! \i*60:(.5,0)$) circle [radius=1pt];
\end{tikzpicture}\,.
\]
The {\em $A_1$ web space $W_m$} is the $\mathbb{Q}(q^{\frac{1}{4}})$-vector space spanned by $B_m$, 
where $\mathbb{Q}(q^{\frac{1}{4}})$ is the field of rational functions in one variable $q^{\frac{1}{4}}$. 
We next define tangles in $D_m$ and the $A_1$ bracket, also known as the Kauffman bracket.
A {\em tangle diagram} in $D_m$ is a proper immersion of $1$-manifolds into $D_m$ such that any intersection point is transverse double points with crossing data. 
For any point $p$ on a tangle diagram, 
it has one of the following neighborhoods:
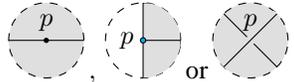
\[
\,\begin{tikzpicture}
\draw [thin, dashed, fill=lightgray!50] (0,0) circle [radius=.5];
\draw (-.5,0)--(.5,0);
\node (p) at (0,0) [above]{$p$};
\fill (0,0) circle [radius=1pt];
\end{tikzpicture}\,
,
\,\begin{tikzpicture}
\draw [thin, dashed] (0,0) circle [radius=.5];
\clip (0,0) circle [radius=.5];
\draw [thin, fill=lightgray!50] (0,-.5) rectangle (.5,.5);
\draw (0,0)--(.5,0);
\node (p) at (0,0) [left]{$p$};
\draw [fill=cyan] (0,0) circle [radius=1pt];
\end{tikzpicture}\,
\text{or}
\,\begin{tikzpicture}
\draw [thin, dashed, fill=lightgray!50] (0,0) circle [radius=.5];
\draw (135:.5) -- (-45:.5);
\draw [lightgray!50, double=black, double distance=0.4pt, ultra thick] (-135:.5) -- (45:.5);
\node (p) at (0,0) [above]{$p$};
\end{tikzpicture}\,
.
\]
In particular, 
a tangle diagram in $D_0$ is a link diagram in the disk.
A tangle diagram $G$ is {\em regularly isotopic} to $G'$ if $G$ is obtained from $G'$ by a finite sequence of boundary-fixing isotopies of $D_m$ and the following moves:

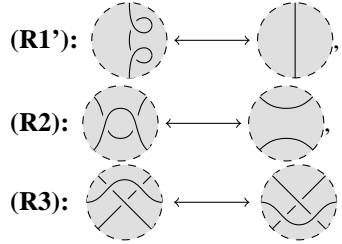
\begin{description}
\item[(R1')]
\tikz[baseline=-.6ex]{
\draw [thin, dashed, fill=lightgray!50] (0,0) circle [radius=.5];
\draw (.3,-.2)
to[out=south, in=east] (.2,-.3)
to[out=west, in=south] (.0,.0)
to[out=north, in=west] (.2,.3)
to[out=east, in=north] (.3,.2);
\draw[lightgray!50, double=black, double distance=0.4pt, ultra thick] (0,-.5) 
to[out=north, in=west] (.2,-.1)
to[out=east, in=north] (.3,-.2);
\draw[lightgray!50, double=black, double distance=0.4pt, ultra thick] (.3,.2)
to[out=south, in=east] (.2,.1)
to[out=west, in=south] (0,.5);
}
\tikz[baseline=-.6ex]{
\draw [<->, xshift=1.5cm] (1,0)--(2,0);
}
\tikz[baseline=-.6ex]{
\draw[xshift=3cm, thin, dashed, fill=lightgray!50] (0,0) circle [radius=.5];
\draw[xshift=3cm] (90:.5) to (-90:.5);
},
\item[(R2)]
\tikz[baseline=-.6ex]{
\draw [thin, dashed, fill=lightgray!50] (0,0) circle [radius=.5];
\draw (135:.5) to [out=south east, in=west](0,-.2) to [out=east, in=south west](45:.5);
\draw [lightgray!50, double=black, double distance=0.4pt, ultra thick](-135:.5) to [out=north east, in=left](0,.2) to [out=right, in=north west] (-45:.5);
}
\tikz[baseline=-.6ex]{
\draw [<->, xshift=1.5cm] (1,0)--(2,0);
}
\tikz[baseline=-.6ex]{
\draw[xshift=3cm, thin, dashed, fill=lightgray!50] (0,0) circle [radius=.5];
\draw[xshift=3cm] (135:.5) to [out=south east, in=west](0,.2) to [out=east, in=south west](45:.5);
\draw[xshift=3cm] (-135:.5) to [out=north east, in=left](0,-.2) to [out=right, in=north west] (-45:.5);
},
\item[(R3)]
\tikz[baseline=-.6ex]{
\draw [thin, dashed, fill=lightgray!50] (0,0) circle [radius=.5];
\draw (-135:.5) -- (45:.5);
\draw [lightgray!50, double=black, double distance=0.4pt, ultra thick] (135:.5) -- (-45:.5);
\draw[lightgray!50, double=black, double distance=0.4pt, ultra thick](180:.5) to [out=right, in=left](0,.3) to [out=right, in=left] (-0:.5);
}
\tikz[baseline=-.6ex]{
\draw[<->, xshift=1.5cm] (1,0)--(2,0);
}
\tikz[baseline=-.6ex]{
\draw [xshift=3cm, thin, dashed, fill=lightgray!50] (0,0) circle [radius=.5];
\draw [xshift=3cm] (-135:.5) -- (45:.5);
\draw [xshift=3cm, lightgray!50, double=black, double distance=0.4pt, ultra thick] (135:.5) -- (-45:.5);
\draw[xshift=3cm, lightgray!50, double=black, double distance=0.4pt, ultra thick](180:.5) to [out=right, in=left](0,-.3) to [out=right, in=left] (-0:.5);
}.
\end{description}

In the above pictures, 
the outside of the left tangle diagram is coincide with the outside of the right tangle diagram.
A {\em tangle} is the regular isotopy class of a tangle diagram and $T_m$ denotes the set of tangles in $D_m$.
The diagram below is an example of a tangle diagram in $D_6$:
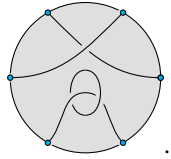
\begin{center}
\begin{tikzpicture}
\draw [thin, fill=lightgray!50] (0,0) circle [radius=1];
\draw (0:1) to[out=west, in=south east] (120:1);
\draw [lightgray!50, double=black, double distance=0.4pt, ultra thick]
(60:1) to[out=south west, in=east] (180:1);
\draw (.0,.1) to[out=west, in=north] (-.2,-.2) to[out=south, in=west] (0,-.5);
\draw[lightgray!50, double=black, double distance=0.4pt, ultra thick] 
(-120:1) to[out=north east, in=west] 
(0,-.2) to[out=east, in=north west] 
(-60:1);
\draw[lightgray!50, double=black, double distance=0.4pt, ultra thick] (.0,.1) to[out=east, in=north] (.2,-.2) to[out=south, in=east] (0,-.5);
\foreach \i in {0,1,...,6} \draw[fill=cyan] ($(0,0) !1! \i*60:(1,0)$) circle [radius=1pt];
\end{tikzpicture}.
\end{center}

\begin{DEF}[The Kauffman bracket]
We define a $\mathbb{Q}(q^{\frac{1}{4}})$-linear map $\langle\,\cdot\,\rangle_2\colon \mathbb{Q}(q^{\frac{1}{4}})T_m\to W_m$ by the following.
\begin{itemize}
\item 
$
\Big\langle\,\tikz[baseline=-.6ex]{
\draw [thin, dashed, fill=lightgray!50] (0,0) circle [radius=.5];
\draw[lightgray!50, double=black, double distance=0.4pt, ultra thick] 
(135:.5) -- (-45:.5);
\draw[lightgray!50, double=black, double distance=0.4pt, ultra thick] 
(-135:.5) -- (45:.5);
}\,\Big\rangle_{\! 2}
=q^{\frac{1}{4}}
\Big\langle\,\tikz[rotate=90, baseline=-.6ex]{
\draw[thin, dashed, fill=lightgray!50] (0,0) circle [radius=.5];
\draw (135:.5) to [out=south east, in=west](0,.2) to [out=east, in=south west](45:.5);
\draw (-135:.5) to [out=north east, in=left](0,-.2) to [out=right, in=north west] (-45:.5);
}\,\Big\rangle_{\! 2}
+q^{-\frac{1}{4}}
\Big\langle\,\tikz[baseline=-.6ex]{
\draw[thin, dashed, fill=lightgray!50] (0,0) circle [radius=.5];
\draw (135:.5) to [out=south east, in=west](0,.2) to [out=east, in=south west](45:.5);
\draw (-135:.5) to [out=north east, in=left](0,-.2) to [out=right, in=north west] (-45:.5);
}\,\Big\rangle_{\! 2}
$,
\item 
$\Big\langle G\sqcup
\,\tikz[baseline=-.6ex]{
\draw[thin, dashed, fill=lightgray!50] (0,0) circle [radius=.5];
\draw (0,0) circle [radius=.3];
}\,\Big\rangle_{\! 2}
=-\left[2\right] \langle G\rangle_{\! 2}$,
\end{itemize}
where $G$ is any tangle in $T_m$.
We call this linear map the $A_1$ bracket or the Kauffman bracket.
\end{DEF}
For any tangle diagram, 
we can obtain a sum of tangles with no crossings by using the Kauffman bracket relation and eliminate trivial loops from these tangles.
Therefore, 
we can define the map from the set of tangle diagrams to $W_m$.
We can easily confirm that this map doesn't change under the moves of regular isotopy.

We next define $A_1$ clasps which also called {\em Jones-Wenzl projectors}, {\em magic elements} etc.
We consider an $A_1$ web space $W_{n+n}=W_{2n}$. 
We use a tangle diagram whose components are decorated with non-negative integers.
A components decorated by $n$ means $n$ parallelization of it.
A tangle diagram with decoration is defined by the following local pictures:
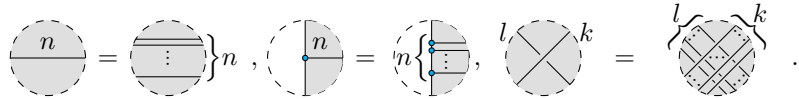
\[
\,\tikz[baseline=-.6ex]{
\draw [thin, dashed, fill=lightgray!50] (0,0) circle [radius=.5];
\draw (-.5,0)--(.5,0) node at (0,0) [above]{$n$};
}\,
=
\,\tikz[baseline=-.6ex]{
\draw [thin, dashed, fill=lightgray!50] (0,0) circle [radius=.5];
\draw (150:.5) -- (30:.5) (160:.5) -- (20:.5);
\node[rotate=90] at (0,0) {${\scriptstyle \cdots}$};
\draw (210:.5) -- (-30:.5);
\node at (.7,.0) {$\Big\}n$};
}\,
,
\,\tikz[baseline=-.6ex]{
\draw [thin, dashed] (0,0) circle [radius=.5];
\clip (0,0) circle [radius=.5];
\draw [thin, fill=lightgray!50] (0,-.5) rectangle (.5,.5);
\draw (0,0)--(.5,0);
\node at (.2,0) [above]{$n$};
\draw [fill=cyan] (0,0) circle [radius=1pt];
}\,
=
\,\tikz[baseline=-.6ex]{
\begin{scope}
\draw [thin, dashed] (0,0) circle [radius=.5];
\clip (0,0) circle [radius=.5];
\draw [thin, fill=lightgray!50] (0,-.5) rectangle (.5,.5);
\draw (0,.2) -- (.5,.2) (0,.1) -- (.5,.1) (0,-.2) -- (.5,-.2);
\end{scope}
\draw [fill=cyan] (0,.2) circle [radius=1pt] 
(0,.1) circle [radius=1pt] 
(0,-.2) circle [radius=1pt];
\node at (.1,0) [left]{$n\Big\{$};
\node[rotate=90] at (.2,-.05) {${\scriptstyle \cdots}$};
}\,
,
\,\tikz[baseline=-.6ex]{
\draw [thin, dashed, fill=lightgray!50] (0,0) circle [radius=.5];
\draw (135:.5) -- (-45:.5);
\draw [lightgray!50, double=black, double distance=0.4pt, ultra thick] (-135:.5) -- (45:.5);
\node at (45:.5) [right]{$k$} node at (135:.5) [left]{$l$};
}\,
=
\,\tikz[baseline=-.6ex]{
\draw [thin, dashed, fill=lightgray!50] (0,0) circle [radius=.5];
\draw (165:.5) -- (-75:.5) (150:.5) -- (-60:.5) (120:.5) -- (-30:.5);
\node[rotate=45] at (-.3,.3) {${\scriptstyle{\cdots}}$};
\node[rotate=45] at (.3,-.3) {${\scriptstyle{\cdots}}$};
\draw [lightgray!50, double=black, double distance=0.4pt, thick] (-150:.5) -- (60:.5) (-120:.5) -- (30:.5) (-105:.5) -- (15:.5);
\node[rotate=-45] at (.3,.3) {${\scriptstyle{\cdots}}$};
\node[rotate=-45] at (-.3,-.3) {${\scriptstyle{\cdots}}$};
\node at (0,0) {${\scriptstyle{\cdots}}$};
\node at (45:.5) [above right]{$k$} node at (135:.5) [above left]{$l$};
\node[rotate=-135] at (40:.6) {$\Big\{$};
\node[rotate=-45] at (140:.6) {$\Big\{$};
}\, .
\]

We omit the boundary of the disk $D_m$ from tangle diagrams in $D_m$ later in the paper. 
For example, 
a tangle $\tikz[baseline=-.6ex]{
\draw [thin,fill=lightgray!50] (0,0) circle [radius=.5];
\draw[lightgray!50, double=black, double distance=0.4pt, ultra thick] 
(135:.5) -- (-45:.5);
\draw[lightgray!50, double=black, double distance=0.4pt, ultra thick] 
(-135:.5) -- (45:.5);
\draw[fill=cyan] (135:.5) circle[radius=1pt]
(-135:.5) circle[radius=1pt]
(45:.5) circle[radius=1pt]
(-45:.5) circle[radius=1pt];
}
$
is denoted by
$
\tikz[baseline=-.6ex]{
\draw[white, double=black, double distance=0.4pt, ultra thick] 
(135:.5) -- (-45:.5);
\draw[white, double=black, double distance=0.4pt, ultra thick] 
(-135:.5) -- (45:.5);
}
$.

We define an $A_1$ clasp
$\tikz[baseline=-.6ex]{
\draw (-.5,0) -- (.5,0);
\draw[fill=white] (-.1,-.3) rectangle (.1,.3);
\node at (.1,0) [above right]{${\scriptstyle n}$};
}\,\in W_{n+n}
$
inductively by the following.
\begin{DEF}[The $A_1$ clasps, The Jones-Wenzl idempotents etc.]
\begin{align}
\tikz[baseline=-.6ex]{
\draw (-.5,0) -- (.5,0);
\draw[fill=white] (-.1,-.3) rectangle (.1,.3);
\node at (.1,0) [above right]{${\scriptstyle 1}$};
}\,
&= 
\,\tikz[baseline=-.6ex]{
\draw (-.5,0) -- (.5,0) node at (0,0) [above]{${\scriptstyle 1}$};
}\,\in W_{1+1}\notag\\
\tikz[baseline=-.6ex]{
\draw (-.5,0) -- (.5,0);
\draw[fill=white] (-.1,-.3) rectangle (.1,.3);
\node at (.1,0) [above right]{${\scriptstyle n}$};
}\,
&=
\bigg\langle\,\tikz[baseline=-.6ex]{
\draw (-.5,.1) -- (.5,.1) (-.5,-.4) -- (.5,-.4);
\draw[fill=white] (-.1,-.2) rectangle (.1,.4);
\node at (0,0) [above right]{${\scriptstyle n-1}$};
\node at (.2,-.4) [above right]{${\scriptstyle 1}$};
}\,\bigg\rangle_{\! 2}
+\frac{\left[n-1\right]}{\left[n\right]}
\bigg\langle\,\tikz[baseline=-.6ex]{
\draw (-.9,.1) -- (-.4,.1) (-.3,.2) -- (.3,.2) (.4,.1) -- (.9,.1);
\draw (-.9,-.4) 
to[out=east, in=west] (-.3,-.4) 
to[out=east, in=south] (-.1,-.2) 
to[out=north, in=east] (-.3,0);
\draw (.9,-.4) 
to[out=west, in=east] (.3,-.4)
to[out=west, in=south] (.1,-.2)
to[out=north, in=west] (.3,0);
\draw[fill=white] (-.4,-.2) rectangle (-.3,.4);
\draw[fill=white] (.3,-.2) rectangle (.4,.4);
\node at (-.3,0) [above left]{${\scriptscriptstyle n-1}$};
\node at (.3,0) [above right]{${\scriptscriptstyle n-1}$};
\node at (0,.1) [above]{${\scriptscriptstyle n-2}$};
\node at (-.6,-.3) {${\scriptscriptstyle 1}$};
\node at (.6,-.3) {${\scriptscriptstyle 1}$};
}\,\bigg\rangle_{\! 2} \in W_{n+n}
\end{align}
\end{DEF}
The $A_1$ clasp has the following properties.
\begin{LEM}[Kauffman-Lins~\cite{KauffmanLins94} etc.]
For any positive integer $n$,
\begin{itemize} 
\item $\Big\langle\,\tikz[baseline=-.6ex]{
\draw (-.6,0) -- (.6,0);
\draw[fill=white] (-.4,-.3) rectangle (-.2,.3);
\draw[fill=white] (.2,-.3) rectangle (.4,.3);
\node at (-.3,0) [above right]{${\scriptstyle n}$};
\node at (.3,0) [above right]{${\scriptstyle n}$};
}\,\Big\rangle_{\! 2}
=
\,\tikz[baseline=-.6ex]{
\draw (-.5,0) -- (.5,0);
\draw[fill=white] (-.1,-.3) rectangle (.1,.3);
\node at (0,0) [above right] {${\scriptstyle n}$};
}\,
$,
\item $\Big\langle\,\tikz[baseline=-.6ex]{
\draw (-.5,0) -- (-.1,0) (0,.2) -- (.5,.2) (0,-.2) -- (.5,-.2);
\draw (0,-.1) arc (-90:90:.1);
\draw[fill=white] (-.2,-.3) rectangle (0,.3);
\node at (0,0) [right] {${\scriptstyle 1}$};
\node at (.4,.2) [right] {${\scriptstyle n-k-2}$};
\node at (.4,-.2) [right] {${\scriptstyle k}$};
}\,\Big\rangle_{\! 2}=0$\quad ($k=0,1,\dots,n-2$).
\end{itemize}
\end{LEM}
It is easy to calculate the following.
\begin{LEM}\label{A1clasplem}\ 
For $k=0,1,\dots,n$,
\begin{itemize}
\item $\Big\langle\,\tikz[baseline=-.6ex]{
\draw (-.5,0) -- (-.1,0);
\draw[white, double=black, double distance=0.4pt, ultra thick] 
(0,-.2) to[out=right, in=left] (.5,.2);
\draw[white, double=black, double distance=0.4pt, ultra thick] 
(0,.2) to[out=right, in=left] (.5,-.2);
\draw[fill=white] (-.2,-.3) rectangle (0,.3);
\node at (-.1,.0) [above left]{${\scriptstyle n}$};
\node at (.4,.2) [right] {${\scriptstyle k}$};
\node at (.4,-.2) [right] {${\scriptstyle n-k}$};
}\,\Big\rangle_{\! 2}
=q^{\frac{k(n-k)}{4}}
\Big\langle\,\tikz[baseline=-.6ex]{
\draw (-.5,0) -- (.5,0);
\draw[fill=white] (-.1,-.3) rectangle (.1,.3);
\node at (.1,0) [above right]{${\scriptstyle n}$};
}\,\Big\rangle_{\! 2}$,  
$\Big\langle\,\tikz[baseline=-.6ex]{
\draw (-.5,0) -- (-.1,0);
\draw[white, double=black, double distance=0.4pt, ultra thick] 
(0,.2) to[out=right, in=left] (.5,-.2);
\draw[white, double=black, double distance=0.4pt, ultra thick] 
(0,-.2) to[out=right, in=left] (.5,.2);
\draw[fill=white] (-.2,-.3) rectangle (0,.3);
\node at (-.1,.0) [above left]{${\scriptstyle n}$};
\node at (.4,.2) [right] {${\scriptstyle k}$};
\node at (.4,-.2) [right] {${\scriptstyle n-k}$};
}\,\Big\rangle_{\! 2}
=q^{-\frac{k(n-k)}{4}}
\Big\langle\,\tikz[baseline=-.6ex]{
\draw (-.5,0) -- (.5,0);
\draw[fill=white] (-.1,-.3) rectangle (.1,.3);
\node at (.1,0) [above right]{${\scriptstyle n}$};
}\,\Big\rangle_{\! 2},$
\item 
$\Big\langle\,\tikz[baseline=-.6ex]{
\draw (-.4,.2) -- (.4,.2);
\draw[rounded corners=.1cm] (-.2,-.2) rectangle (.2,.1);
\draw[fill=white] (-.05,.0) rectangle (.05,.3);
\node at (.0,.3)[above]{$\scriptstyle{n}$};
\node at (.4,.2)[right]{$\scriptstyle{n-k}$};
\node at (.0,-.2)[below]{$\scriptstyle{k}$};
}\,\Big\rangle_{\! 2}
=(-1)^k\frac{\left[n+1\right]}{\left[n-k+1\right]}
\Big\langle\,\tikz[baseline=-.6ex]{
\draw (-.4,.0) -- (.4,.0);
\draw[fill=white] (-.05,-.2) rectangle (.05,.2);
\node at (.0,.2)[above]{$\scriptstyle{n-k}$};
}\,\Big\rangle_{\! 2},$
\item 
$\Big\langle\,\tikz[baseline=-.6ex]{
\draw (-.4,.1) -- (.0,.1);
\draw[white, double=black, double distance=0.4pt, ultra thick] 
(.2,-.2) to[out=west, in=west]
(.2,.1) to[out=east, in=west]
(.5,.1);
\draw[white, double=black, double distance=0.4pt, ultra thick] 
(.0,.1) to[out=east, in=west]
(.2,.1) to[out=east, in=east]
(.2,-.2);
\draw[fill=white] (-.1,-.1) rectangle (.0,.3);
\node at (-.1,-.1)[below]{$\scriptstyle{n}$};
}\,\Big\rangle_{\! 2}
=(-1)^nq^{\frac{n^2+2n}{4}}
\Big\langle\,\tikz[baseline=-.6ex]{
\draw (-.4,.0) -- (.4,.0);
\draw[fill=white] (-.05,-.2) rectangle (.05,.2);
\node at (.0,.2)[above]{$\scriptstyle{n}$};
}\,\Big\rangle_{\! 2}$, 
$\Big\langle\,\tikz[baseline=-.6ex]{
\draw (-.4,.1) -- (.0,.1);
\draw[white, double=black, double distance=0.4pt, ultra thick] 
(.0,.1) to[out=east, in=west]
(.2,.1) to[out=east, in=east]
(.2,-.2);
\draw[white, double=black, double distance=0.4pt, ultra thick] 
(.2,-.2) to[out=west, in=west]
(.2,.1) to[out=east, in=west]
(.5,.1);
\draw[fill=white] (-.1,-.1) rectangle (.0,.3);
\node at (-.1,-.1)[below]{$\scriptstyle{n}$};
}\,\Big\rangle_{\! 2}
=(-1)^nq^{-\frac{n^2+2n}{4}}
\Big\langle\,\tikz[baseline=-.6ex]{
\draw (-.4,.0) -- (.4,.0);
\draw[fill=white] (-.05,-.2) rectangle (.05,.2);
\node at (.0,.2)[above]{$\scriptstyle{n}$};
}\,\Big\rangle_{\! 2}.$
\end{itemize}
\end{LEM}

Let $N=(n_1,n_2,\dots,n_k)$ be a $k$-tuple of positive integers.
\begin{DEF}[Clasped $A_1$ web spaces]
We define a subspace $W_N$ of $W_{n_1+n_2+\dots+n_k}$ called a clasped $A_1$ web space as the following:
\[
 W_N=\Bigg\{\,
\bigg\langle\tikz[baseline=-.6ex]{
\draw [thin, fill=lightgray!50] (0,0) circle [radius=1];
\draw (0,0) -- (-30:1);
\draw (0,0) -- (0:1);
\draw (0,0) -- (30:1);
\draw (0,0) -- (60:1);
\coordinate (a) at ($(0,0)!.8!-30:(-12:1)$);
\coordinate (b) at ($(0,0)!.9!-30:(12:1)$);
\coordinate (a0) at ($(0,0)!.8!0:(-12:1)$);
\coordinate (b0) at ($(0,0)!.9!0:(12:1)$);
\coordinate (a1) at ($(0,0)!.8!30:(-12:1)$);
\coordinate (b1) at ($(0,0)!.9!30:(12:1)$);
\coordinate (a2) at ($(0,0)!.8!60:(-12:1)$);
\coordinate (b2) at ($(0,0)!.9!60:(12:1)$);
\draw[fill=white,rotate=-30] (a) rectangle (b);
\draw[fill=white,rotate=0] (a0) rectangle (b0);
\draw[fill=white,rotate=30] (a1) rectangle (b1);
\draw[fill=white,rotate=60] (a2) rectangle (b2);
\foreach \i in {0,1,...,11} \draw[fill=cyan] ($(0,0) !1! \i*30:(1,0)$) circle [radius=1pt];
\draw[fill=cyan] (-30:1) circle [radius=1pt] node [right]{${\scriptstyle n_k}$};
\draw[fill=cyan] (0:1) circle [radius=1pt] node [right]{${\scriptstyle n_1}$};
\draw[fill=cyan] (30:1) circle [radius=1pt] node [right]{${\scriptstyle n_2}$};
\draw[fill=cyan] (60:1) circle [radius=1pt] node [right]{${\scriptstyle n_3}$};
\node[rotate=10] at (100:.8){$\cdots$} node[rotate=25] at (-65:.8){$\cdots$} node[rotate=-60] at (-150:.8){$\cdots$};
\node[circle, draw, fill=lightgray!30] (0,0) {\quad$w$\quad\quad};
}\bigg\rangle_{\! 2}
\;{\Bigg\vert}\; w\in W_{n_1+n_2+\dots+n_k}\,\Bigg\}.
\]
\end{DEF}
For example, 
if $k=3, n_1=1, n_2=2, n_3=3$ and 
$w=
\Big\langle\,\tikz[baseline=-.6ex]{
\draw [thin, fill=lightgray!30] (0,0) circle [radius=.5];
\draw [lightgray!30, double=black, double distance=0.4pt, ultra thick]
(140:.5) to[out=south east, in=north east] (-140:.5);
\draw [lightgray!30, double=black, double distance=0.4pt, ultra thick]
(100:.5) to[out=south east, in=north east] (-100:.5);
\draw[lightgray!30, double=black, double distance=0.4pt, ultra thick] (0:.5) to[out=west, in=north east] (-120:.5);
\draw[fill=cyan] 
(0:.5) circle [radius=1pt]
(100:.5) circle [radius=1pt]
(140:.5) circle [radius=1pt]
(-100:.5) circle [radius=1pt]
(-120:.5) circle [radius=1pt]
(-140:.5) circle [radius=1pt];
}\,\Big\rangle_{\! 2}
=q^{\frac{1}{4}}\,\tikz[baseline=-.6ex]{
\draw [thin, fill=lightgray!30] (0,0) circle [radius=.5];
\draw (0:.5) to[out=west, in=north] (-100:.5);
\draw (100:.5) to[out=south, in=north east] (-120:.5);
\draw (140:.5) to[out=south east, in=north east] (-140:.5);
\draw[fill=cyan] 
(0:.5) circle [radius=1pt]
(100:.5) circle [radius=1pt]
(140:.5) circle [radius=1pt]
(-100:.5) circle [radius=1pt]
(-120:.5) circle [radius=1pt]
(-140:.5) circle [radius=1pt];
}\,
+q^{-\frac{1}{4}}\,\tikz[baseline=-.6ex]{
\draw [thin, fill=lightgray!30] (0,0) circle [radius=.5];
\draw (0:.5) to[out=west, in=south] (100:.5);
\draw (-120:.5) 
to[out=north east, in=west] (-110:.3) 
to[out=east, in=north east]
(-100:.5);
\draw (140:.5) to[out=south east, in=north east] (-140:.5);
\draw[fill=cyan] 
(0:.5) circle [radius=1pt]
(100:.5) circle [radius=1pt]
(140:.5) circle [radius=1pt]
(-100:.5) circle [radius=1pt]
(-120:.5) circle [radius=1pt]
(-140:.5) circle [radius=1pt];
}\,
$,
then the above diagram means
\[
\Bigg\langle\,\tikz[baseline=-.6ex]{
\draw [thin, fill=lightgray!50] (0,0) circle [radius=1];
\coordinate (a0) at ($(0,0)!.7!0:(-30:1)$);
\coordinate (b0) at ($(0,0)!.8!0:(30:1)$);
\coordinate (a1) at ($(0,0)!.7!120:(-30:1)$);
\coordinate (b1) at ($(0,0)!.8!120:(30:1)$);
\coordinate (a2) at ($(0,0)!.7!-120:(-30:1)$);
\coordinate (b2) at ($(0,0)!.8!-120:(30:1)$);
\fill [lightgray!50] (0,0) circle [radius=.5];
\draw [lightgray!50, double=black, double distance=0.4pt, ultra thick]
(130:1) to[out=south east, in=north east] (-130:1);
\draw [lightgray!50, double=black, double distance=0.4pt, ultra thick]
(110:1) to[out=south east, in=north east] (-110:1);
\draw[lightgray!50, double=black, double distance=0.4pt, ultra thick] (0:1) to[out=west, in=north east] (-120:1);
\draw[fill=white,rotate=0] (a0) rectangle (b0);
\draw[fill=white,rotate=30] (a1) rectangle (b1);
\draw[fill=white,rotate=-30] (a2) rectangle (b2);
\draw[fill=cyan]
(0:1) circle [radius=1pt]
(110:1) circle [radius=1pt]
(130:1) circle [radius=1pt]
(-110:1) circle [radius=1pt]
(-120:1) circle [radius=1pt]
(-130:1) circle [radius=1pt];
}\,\Bigg\rangle_{\! 2}
=
q^{\frac{1}{4}}\Bigg\langle\,\tikz[baseline=-.6ex]{
\draw [thin, fill=lightgray!50] (0,0) circle [radius=1];
\coordinate (a0) at ($(0,0)!.7!0:(-30:1)$);
\coordinate (b0) at ($(0,0)!.8!0:(30:1)$);
\coordinate (a1) at ($(0,0)!.7!120:(-30:1)$);
\coordinate (b1) at ($(0,0)!.8!120:(30:1)$);
\coordinate (a2) at ($(0,0)!.7!-120:(-30:1)$);
\coordinate (b2) at ($(0,0)!.8!-120:(30:1)$);
\fill [lightgray!50] (0,0) circle [radius=.5];
\draw (130:1) to[out=south east, in=north east] (-130:1);
\draw (110:1) to[out=south east, in=north east] (-120:1);
\draw (0:1) to[out=west, in=north east] (-110:1);
\draw[fill=white,rotate=0] (a0) rectangle (b0);
\draw[fill=white,rotate=30] (a1) rectangle (b1);
\draw[fill=white,rotate=-30] (a2) rectangle (b2);
\draw[fill=cyan]
(0:1) circle [radius=1pt]
(110:1) circle [radius=1pt]
(130:1) circle [radius=1pt]
(-110:1) circle [radius=1pt]
(-120:1) circle [radius=1pt]
(-130:1) circle [radius=1pt];
}\,\Bigg\rangle_{\! 2}
+q^{-\frac{1}{4}}\Bigg\langle\,\tikz[baseline=-.6ex]{
\draw [thin, fill=lightgray!50] (0,0) circle [radius=1];
\coordinate (a0) at ($(0,0)!.7!0:(-30:1)$);
\coordinate (b0) at ($(0,0)!.8!0:(30:1)$);
\coordinate (a1) at ($(0,0)!.7!120:(-30:1)$);
\coordinate (b1) at ($(0,0)!.8!120:(30:1)$);
\coordinate (a2) at ($(0,0)!.7!-120:(-30:1)$);
\coordinate (b2) at ($(0,0)!.8!-120:(30:1)$);
\fill [lightgray!50] (0,0) circle [radius=.5];
\draw (130:1) to[out=south east, in=north east] (-130:1);
\draw (110:1) to[out=south east, in=west] (0:1);
\draw (-120:1) 
to[out=north east, in=south west] (-115:.4)
to[out=north east, in=north east] (-110:1);
\draw[fill=white,rotate=0] (a0) rectangle (b0);
\draw[fill=white,rotate=30] (a1) rectangle (b1);
\draw[fill=white,rotate=-30] (a2) rectangle (b2);
\draw[fill=cyan]
(0:1) circle [radius=1pt]
(110:1) circle [radius=1pt]
(130:1) circle [radius=1pt]
(-110:1) circle [radius=1pt]
(-120:1) circle [radius=1pt]
(-130:1) circle [radius=1pt];
}\,\Bigg\rangle_{\! 2}
=q^{\frac{1}{4}}\,\tikz[baseline=-.6ex]{
\draw [thin, fill=lightgray!50] (0,0) circle [radius=1];
\coordinate (a0) at ($(0,0)!.7!0:(-30:1)$);
\coordinate (b0) at ($(0,0)!.8!0:(30:1)$);
\coordinate (a1) at ($(0,0)!.7!120:(-30:1)$);
\coordinate (b1) at ($(0,0)!.8!120:(30:1)$);
\coordinate (a2) at ($(0,0)!.7!-120:(-30:1)$);
\coordinate (b2) at ($(0,0)!.8!-120:(30:1)$);
\fill [lightgray!50] (0,0) circle [radius=.5];
\draw (130:1) to[out=south east, in=north east] (-130:1);
\draw (110:1) to[out=south east, in=north east] (-120:1);
\draw (0:1) to[out=west, in=north east] (-110:1);
\draw[fill=white,rotate=0] (a0) rectangle (b0);
\draw[fill=white,rotate=30] (a1) rectangle (b1);
\draw[fill=white,rotate=-30] (a2) rectangle (b2);
\draw[fill=cyan]
(0:1) circle [radius=1pt]
(110:1) circle [radius=1pt]
(130:1) circle [radius=1pt]
(-110:1) circle [radius=1pt]
(-120:1) circle [radius=1pt]
(-130:1) circle [radius=1pt];
}\,.
\]
\subsection{The $A_2$ web spaces}
We define the $A_2$ web spaces.
Let $\varepsilon=(\varepsilon_1,\varepsilon_2,\dots,\varepsilon_m)$ be a $m$-tuple of signs ${+}$ or ${-}$. 
Let $D_\varepsilon$ denote $D_m$ whose marked point $\exp(2\pi\sqrt{-1}/m)^{j-1}$ is decorated by $\varepsilon_j$ for $j=1,2,\dots,m$.
A {\em bipartite uni-trivalent graph} $G$ is a directed graph such that every vertex is either trivalent or univalent and these vertices are divided into sinks or sources. 
A sink (resp. source) is a vertex such that all edges adjoining to the vertex point into (resp. away from) it.
A {\em bipartite trivalent graph} $G$ in $D_{\varepsilon}$ is an embedding of a uni-trivalent graph into $D_\varepsilon$ such that for any vertex $v$ has the following neighborhoods:
\begin{itemize}
\item \,\tikz[baseline=-.6ex]{
\draw [thin, dashed, fill=lightgray!50] (0,0) circle [radius=.5];
\draw[-<-=.5] (0:0) -- (90:.5); 
\draw[-<-=.5] (0:0) -- (210:.5); 
\draw[-<-=.5] (0:0) -- (-30:.5);
\node (v) at (0,0) [above right]{$v$};
\fill (0,0) circle [radius=1pt];
}\, 
or 
\,\tikz[baseline=-.6ex]{
\draw [thin, dashed] (0,0) circle [radius=.5];
\clip (0,0) circle [radius=.5];
\draw [thin, fill=lightgray!50] (0,-.5) rectangle (.5,.5);
\draw[-<-=.5] (0,0)--(.5,0);
\node (p) at (0,0) [left]{${+}$};
\node at (0,0) [above right]{$v$};
\draw [fill=cyan] (0,0) circle [radius=1pt];
}\, if $v$ is a sink,
\item \tikz[baseline=-.6ex]{
\draw [thin, dashed, fill=lightgray!50] (0,0) circle [radius=.5];
\draw[->-=.5] (0:0) -- (90:.5); 
\draw[->-=.5] (0:0) -- (210:.5); 
\draw[->-=.5] (0:0) -- (-30:.5);
\node (v) at (0,0) [above right]{$v$};
\fill (0,0) circle [radius=1pt];
}\,
or \,\tikz[baseline=-.6ex]{
\draw [thin, dashed] (0,0) circle [radius=.5];
\clip (0,0) circle [radius=.5];
\draw [thin, fill=lightgray!50] (0,-.5) rectangle (.5,.5);
\draw[->-=.5] (0,0)--(.5,0);
\node (p) at (0,0) [left]{${-}$};
\node at (0,0) [above right]{$v$};
\draw [fill=cyan] (0,0) circle [radius=1pt];
}\, if $v$ is a source.
\end{itemize}
An {\em $A_2$ basis web} is the boundary-fixing isotopy class of a bipartite trivalent graph $G$ in $D_{\varepsilon}$, 
where any internal face of $D\setminus G$ has at least six sides. 
Let us denote $B_\varepsilon$ is the set of $A_2$ basis webs in $D_{\varepsilon}$.
For example,
$B_{(+,-,+,-,+,-)}$ has the following $A_2$ basis webs:
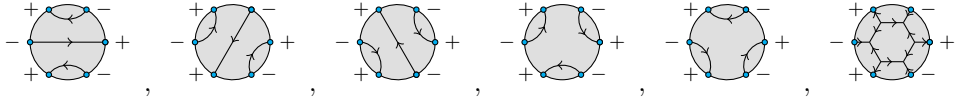
\[
\,\begin{tikzpicture}
\draw [thin, fill=lightgray!50] (0,0) circle [radius=.5];
\draw[-<-=.5] (0:.5) -- (180:.5);
\draw[->-=.5] (60:.5) to[out=-120, in=-60] (120:.5);
\draw[-<-=.5] (240:.5) to[out=60, in=120] (300:.5);
\foreach \i in {0,1,...,6} \draw[fill=cyan] ($(0,0) !1! \i*60:(.5,0)$) circle [radius=1pt];
\node at (0:.5) [right]{$\scriptstyle{+}$};
\node at (60:.5) [right]{$\scriptstyle{-}$};
\node at (120:.5) [left]{$\scriptstyle{+}$};
\node at (180:.5) [left]{$\scriptstyle{-}$};
\node at (240:.5) [left]{$\scriptstyle{+}$};
\node at (300:.5) [right]{$\scriptstyle{-}$};
\end{tikzpicture}\,,
\,\begin{tikzpicture}
\begin{scope}[rotate=60]
\draw [thin, fill=lightgray!50] (0,0) circle [radius=.5];
\draw[->-=.5] (0:.5) -- (180:.5);
\draw[-<-=.5] (60:.5) to[out=-120, in=-60] (120:.5);
\draw[->-=.5] (240:.5) to[out=60, in=120] (300:.5);
\foreach \i in {0,1,...,6} \draw[fill=cyan] ($(0,0) !1! \i*60:(.5,0)$) circle [radius=1pt];
\end{scope}
\node at (0:.5) [right]{$\scriptstyle{+}$};
\node at (60:.5) [right]{$\scriptstyle{-}$};
\node at (120:.5) [left]{$\scriptstyle{+}$};
\node at (180:.5) [left]{$\scriptstyle{-}$};
\node at (240:.5) [left]{$\scriptstyle{+}$};
\node at (300:.5) [right]{$\scriptstyle{-}$};
\end{tikzpicture}\,,
\,\begin{tikzpicture}
\begin{scope}[rotate=-60]
\draw [thin, fill=lightgray!50] (0,0) circle [radius=.5];
\draw[->-=.5] (0:.5) -- (180:.5);
\draw[-<-=.5] (60:.5) to[out=-120, in=-60] (120:.5);
\draw[->-=.5] (240:.5) to[out=60, in=120] (300:.5);
\foreach \i in {0,1,...,6} \draw[fill=cyan] ($(0,0) !1! \i*60:(.5,0)$) circle [radius=1pt];
\end{scope}
\node at (0:.5) [right]{$\scriptstyle{+}$};
\node at (60:.5) [right]{$\scriptstyle{-}$};
\node at (120:.5) [left]{$\scriptstyle{+}$};
\node at (180:.5) [left]{$\scriptstyle{-}$};
\node at (240:.5) [left]{$\scriptstyle{+}$};
\node at (300:.5) [right]{$\scriptstyle{-}$};
\end{tikzpicture}\,,
\,\begin{tikzpicture}
\draw [thin, fill=lightgray!50] (0,0) circle [radius=.5];
\draw[-<-=.5] (0:.5) to[out=180, in=-120] (60:.5);
\draw[rotate=120, -<-=.5] (0:.5) to[out=180, in=-120] (60:.5);
\draw[rotate=240, -<-=.5] (0:.5) to[out=180, in=-120] (60:.5);
\foreach \i in {0,1,...,6} \draw[fill=cyan] ($(0,0) !1! \i*60:(.5,0)$) circle [radius=1pt];
\node at (0:.5) [right]{$\scriptstyle{+}$};
\node at (60:.5) [right]{$\scriptstyle{-}$};
\node at (120:.5) [left]{$\scriptstyle{+}$};
\node at (180:.5) [left]{$\scriptstyle{-}$};
\node at (240:.5) [left]{$\scriptstyle{+}$};
\node at (300:.5) [right]{$\scriptstyle{-}$};
\end{tikzpicture}\,,
\,\begin{tikzpicture}
\begin{scope}[rotate=60]
\draw [thin, fill=lightgray!50] (0,0) circle [radius=.5];
\draw[->-=.5] (0:.5) to[out=180, in=-120] (60:.5);
\draw[rotate=120, ->-=.5] (0:.5) to[out=180, in=-120] (60:.5);
\draw[rotate=240, ->-=.5] (0:.5) to[out=180, in=-120] (60:.5);
\foreach \i in {0,1,...,6} \draw[fill=cyan] ($(0,0) !1! \i*60:(.5,0)$) circle [radius=1pt];
\end{scope}
\node at (0:.5) [right]{$\scriptstyle{+}$};
\node at (60:.5) [right]{$\scriptstyle{-}$};
\node at (120:.5) [left]{$\scriptstyle{+}$};
\node at (180:.5) [left]{$\scriptstyle{-}$};
\node at (240:.5) [left]{$\scriptstyle{+}$};
\node at (300:.5) [right]{$\scriptstyle{-}$};
\end{tikzpicture}\,,
\,\begin{tikzpicture}
\draw [thin, fill=lightgray!50] (0,0) circle [radius=.5];
\draw[-<-=.5] (0:.5) -- (0:.3);
\draw[rotate=60, ->-=.5] (0:.5) -- (0:.3);
\draw[rotate=120, -<-=.5] (0:.5) -- (0:.3);
\draw[rotate=180, ->-=.5] (0:.5) -- (0:.3);
\draw[rotate=240, -<-=.5] (0:.5) -- (0:.3);
\draw[rotate=300, ->-=.5] (0:.5) -- (0:.3);
\draw[->-=.5] (0:.3) -- (60:.3);
\draw[rotate=60, -<-=.5] (0:.3) -- (60:.3);
\draw[rotate=120, ->-=.5] (0:.3) -- (60:.3);
\draw[rotate=180, -<-=.5] (0:.3) -- (60:.3);
\draw[rotate=240, ->-=.5] (0:.3) -- (60:.3);
\draw[rotate=300, -<-=.5] (0:.3) -- (60:.3);
\foreach \i in {0,1,...,6} \draw[fill=cyan] ($(0,0) !1! \i*60:(.5,0)$) circle [radius=1pt];
\node at (0:.5) [right]{$\scriptstyle{+}$};
\node at (60:.5) [right]{$\scriptstyle{-}$};
\node at (120:.5) [left]{$\scriptstyle{+}$};
\node at (180:.5) [left]{$\scriptstyle{-}$};
\node at (240:.5) [left]{$\scriptstyle{+}$};
\node at (300:.5) [right]{$\scriptstyle{-}$};
\end{tikzpicture}\,.
\]
The {\em $A_2$ web space $W_\varepsilon$} is the $\mathbb{Q}(q^{\frac{1}{6}})$-vector space spanned by $B_\varepsilon$. 
A {\em tangled trivalent graph diagram} in $D_\varepsilon$ is an immersed bipartite uni-trivalent graph in $D_\varepsilon$ whose intersection points are only transverse double points of edges with crossing data 
\,\tikz[baseline=-.6ex, scale=.8]{
\draw [thin, dashed, fill=lightgray!50] (0,0) circle [radius=.5];
\draw[->-=.8] (-45:.5) -- (135:.5);
\draw[->-=.8, lightgray!50, double=black, double distance=0.4pt, ultra thick] (-135:.5) -- (45:.5);
}\, or 
\,\tikz[baseline=-.6ex, scale=.8]{
\draw [thin, dashed, fill=lightgray!50] (0,0) circle [radius=.5];
\draw[->-=.8] (-135:.5) -- (45:.5);
\draw[->-=.8, lightgray!50, double=black, double distance=0.4pt, ultra thick] (-45:.5) -- (135:.5);
}\, .
Tangled trivalent graph diagrams $G$ and $G'$ are regularly isotopic if $G$ is obtained from $G'$ by a finite sequence of boundary-fixing isotopies and (R1'), (R2), (R3) and (R4) moves with some direction of edges.
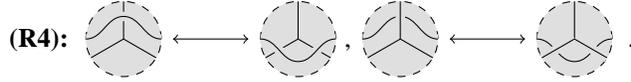
\begin{description}
\item[(R4)]
\tikz[baseline=-.6ex]{
\draw [thin, dashed, fill=lightgray!50] (0,0) circle [radius=.5];
\draw (0:0) -- (90:.5); 
\draw (0:0) -- (210:.5); 
\draw (0:0) -- (-30:.5);
\draw[lightgray!50, double=black, double distance=0.4pt, ultra thick](180:.5) to [out=right, in=left](0,.3) to [out=right, in=left] (-0:.5);
}
\tikz[baseline=-.6ex]{
\draw[<->, xshift=1.5cm] (1,0)--(2,0);
}
\tikz[baseline=-.6ex]{
\draw [thin, dashed, fill=lightgray!50] (0,0) circle [radius=.5];
\draw (0:0) -- (90:.5); 
\draw (0:0) -- (210:.5); 
\draw (0:0) -- (-30:.5);
\draw[lightgray!50, double=black, double distance=0.4pt, ultra thick](180:.5) to [out=right, in=left](0,-.3) to [out=right, in=left] (-0:.5);
}\,, 
\tikz[baseline=-.6ex]{
\draw [thin, dashed, fill=lightgray!50] (0,0) circle [radius=.5];
\draw[lightgray!50, double=black, double distance=0.4pt, ultra thick](180:.5) to [out=right, in=left](0,.3) to [out=right, in=left] (-0:.5);
\draw[lightgray!50, double=black, double distance=0.4pt, ultra thick] (0:0) -- (90:.5); 
\draw (0:0) -- (210:.5); 
\draw (0:0) -- (-30:.5);
}
\tikz[baseline=-.6ex]{
\draw[<->, xshift=1.5cm] (1,0)--(2,0);
}
\tikz[baseline=-.6ex]{
\draw [thin, dashed, fill=lightgray!50] (0,0) circle [radius=.5];
\draw[lightgray!50, double=black, double distance=0.4pt, ultra thick](180:.5) to [out=right, in=left](0,-.3) to [out=right, in=left] (-0:.5);
\draw (0:0) -- (90:.5); 
\draw[lightgray!50, double=black, double distance=0.4pt, ultra thick] (0:0) -- (210:.5); 
\draw[lightgray!50, double=black, double distance=0.4pt, ultra thick] (0:0) -- (-30:.5);
}\, .
\end{description}
{\em Tangled trivalent graphs} in $D_\varepsilon$ are regular isotopy classes of tangled trivalent graph diagrams in $D_\varepsilon$. 
We denote $T_\varepsilon$ the set of tangled trivalent graphs in $D_\varepsilon$.
The diagram below is an example of a tangled trivalent graph diagram in $D_{({+},{-},{+},{+},{+})}$. 
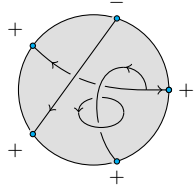
\begin{center}
\begin{tikzpicture}
\draw [thin, fill=lightgray!50] (0,0) circle [radius=1];
\draw[-<-=.5] (0:1) -- (.7,0);
\draw[->-=.8] (.7,0) to[out=west, in=south east](144:1);
\draw[->-=.8] [lightgray!50, double=black, double distance=0.4pt, ultra thick]
(72:1) -- (216:1);
\draw[->, rotate=-90] (.3,.4) to[out=west, in=north] (.1,.1) to[out=south, in=west] (.3,-.2);
\draw[lightgray!50, double=black, double distance=0.4pt, ultra thick] 
(288:1) to[out=north west, in=west] (.4,.3);
\draw[-<-=.2] (.4,.3) to[out=east, in=north] (.7,0);
\draw[rotate=-90, lightgray!50, double=black, double distance=0.4pt, ultra thick] (.3,.4) to[out=east, in=north] (.5,.1) to[out=south, in=east] (.3,-.2);
\foreach \i in {0,1,...,5} 
\draw[fill=cyan] ($(0,0) !1! \i*72:(1,0)$) circle [radius=1pt];
\node at (0:1) [right]{$\scriptstyle{+}$}
node at (72:1) [above]{$\scriptstyle{-}$} 
node at (144:1) [above left]{$\scriptstyle{+}$} 
node at (216:1) [below left]{$\scriptstyle{+}$} 
node at (288:1) [below]{$\scriptstyle{+}$};
\end{tikzpicture}
\end{center}
\begin{DEF}[The $A_2$ bracket~\cite{Kuperberg96}]
We define a $\mathbb{Q}(q^{\frac{1}{6}})$-linear map $\langle\,\cdot\,\rangle_3\colon\mathbb{Q}(q^{\frac{1}{6}})T_\varepsilon\to W_\varepsilon$ by the following.
\begin{itemize}
\item 
$\Big\langle\,\tikz[baseline=-.6ex]{
\draw [thin, dashed, fill=lightgray!50] (0,0) circle [radius=.5];
\draw[->-=.8] (-45:.5) -- (135:.5);
\draw[->-=.8, lightgray!50, double=black, double distance=0.4pt, ultra thick] (-135:.5) -- (45:.5);
}\,\Big\rangle_{\! 3}
=
q^{\frac{1}{3}}\Big\langle\,\tikz[baseline=-.6ex]{
\draw[thin, dashed, fill=lightgray!50] (0,0) circle [radius=.5];
\draw[->-=.5] (-45:.5) to [out=north west, in=south](.2,0) to [out=north, in=south west](45:.5);
\draw[->-=.5] (-135:.5) to [out=north east, in=south](-.2,0) to [out=north, in=south east] (135:.5);
}\,\Big\rangle_{\! 3}
-
q^{-\frac{1}{6}}\Big\langle\,\tikz[baseline=-.6ex]{
\draw[thin, dashed, fill=lightgray!50] (0,0) circle [radius=.5];
\draw[->-=.5] (-45:.5) -- (0,-.2);
\draw[->-=.5] (-135:.5) -- (0,-.2);
\draw[-<-=.5] (0,-.2) -- (0,.2);
\draw[-<-=.5] (45:.5) -- (0,.2);
\draw[-<-=.5] (135:.5) -- (0,.2);
}\,\Big\rangle_{\! 3}
$,
\item 
$\Big\langle\,\tikz[baseline=-.6ex]{
\draw [thin, dashed, fill=lightgray!50] (0,0) circle [radius=.5];
\draw[->-=.8] (-135:.5) -- (45:.5);
\draw[->-=.8, lightgray!50, double=black, double distance=0.4pt, ultra thick] (-45:.5) -- (135:.5);
}\,\Big\rangle_{\! 3}
=
q^{-\frac{1}{3}}\Big\langle\,\tikz[baseline=-.6ex]{
\draw[thin, dashed, fill=lightgray!50] (0,0) circle [radius=.5];
\draw[->-=.5] (-45:.5) to [out=north west, in=south](.2,0) to [out=north, in=south west](45:.5);
\draw[->-=.5] (-135:.5) to [out=north east, in=south](-.2,0) to [out=north, in=south east] (135:.5);
}\,\Big\rangle_{\! 3}
-
q^{\frac{1}{6}}\Big\langle\,\tikz[baseline=-.6ex]{
\draw[thin, dashed, fill=lightgray!50] (0,0) circle [radius=.5];
\draw[->-=.5] (-45:.5) -- (0,-.2);
\draw[->-=.5] (-135:.5) -- (0,-.2);
\draw[-<-=.5] (0,-.2) -- (0,.2);
\draw[-<-=.5] (45:.5) -- (0,.2);
\draw[-<-=.5] (135:.5) -- (0,.2);
}\,\Big\rangle_{\! 3}
$,
\item 
$\Big\langle\,\tikz[baseline=-.6ex, rotate=90]{
\draw[thin, dashed, fill=lightgray!50] (0,0) circle [radius=.5];
\draw[->-=.6] (-45:.5) -- (-45:.3);
\draw[->-=.6] (-135:.5) -- (-135:.3);
\draw[-<-=.6] (45:.5) -- (45:.3);
\draw[->-=.6] (135:.5) -- (135:.3);
\draw[->-=.5] (45:.3) -- (135:.3);
\draw[-<-=.5] (-45:.3) -- (-135:.3);
\draw[->-=.5] (45:.3) -- (-45:.3);
\draw[-<-=.5] (135:.3) -- (-135:.3);
}\,\Big\rangle_{\! 3}
=
\Big\langle\,\tikz[baseline=-.6ex, rotate=90]{
\draw[thin, dashed, fill=lightgray!50] (0,0) circle [radius=.5];
\draw[->-=.5] (-45:.5) to [out=north west, in=south](.2,0) to [out=north, in=south west](45:.5);
\draw[-<-=.5] (-135:.5) to [out=north east, in=south](-.2,0) to [out=north, in=south east] (135:.5);
}\,\Big\rangle_{\! 3}
+
\Big\langle\,\tikz[rotate=90, baseline=-.6ex, rotate=90]{
\draw[thin, dashed, fill=lightgray!50] (0,0) circle [radius=.5];
\draw[-<-=.5] (-45:.5) to [out=north west, in=south](.2,0) to [out=north, in=south west](45:.5);
\draw[->-=.5] (-135:.5) to [out=north east, in=south](-.2,0) to [out=north, in=south east] (135:.5);
}\,\Big\rangle_{\! 3}
$,
\item 
$\Big\langle\,\tikz[baseline=-.6ex, rotate=-90]{
\draw[thin, dashed, fill=lightgray!50] (0,0) circle [radius=.5];
\draw[->-=.5] (0,-.5) -- (0,-.25);
\draw[->-=.5] (0,.25) -- (0,.5);
\draw[-<-=.5] (0,-.25) to [out=60, in=120, relative](0,.25);
\draw[-<-=.5] (0,-.25) to [out=-60, in=-120, relative](0,.25);
}\,\Big\rangle_{\! 3}
=
\left[2\right]\Big\langle\,\tikz[baseline=-.6ex, rotate=-90]{
\draw[thin, dashed, fill=lightgray!50] (0,0) circle [radius=.5];
\draw[->-=.5] (0,-.5) -- (0,.5);
}\,\Big\rangle_{\! 3}
$,
\item 
$
\Big\langle G\sqcup
\,\tikz[baseline=-.6ex]{
\draw[thin, dashed, fill=lightgray!50] (0,0) circle [radius=.5];
\draw[->-=.5] (0,0) circle [radius=.3];
}\,\Big\rangle_{\! 3}
=
\left[3\right] \langle G\rangle_{3}
$.
\end{itemize}
We can confirm that this map is well-defined as with the Kauffman bracket.
\end{DEF}
We next consider $A_2$ web space $W_{n^{+}+n^{-}}=W_{({+},{+},\dots,{+},{-},{-},\dots,{-})}$ whose first $n$ marked points are decorated with ${+}$ and next $n$ marked points are decorated with ${-}$.
We define $A_2$ clasps 
$\tikz[baseline=-.6ex]{
\draw[->-=.8] (-.5,0) -- (.5,0);
\draw[fill=white] (-.1,-.3) rectangle (.1,.3);
\node at (.1,0) [above right]{${\scriptstyle n}$};
}\,\in W_{n^{+}+n^{-}}
$ 
inductively by the following. 
\begin{DEF}(The $A_2$ clasps)
\begin{align}
\tikz[baseline=-.6ex]{
\draw[->-=.8] (-.5,0) -- (.5,0);
\draw[fill=white] (-.1,-.3) rectangle (.1,.3);
\node at (.1,0) [above right]{${\scriptstyle 1}$};
}\,
&= 
\,\tikz[baseline=-.6ex]{
\draw[->-=.5] (-.5,0) -- (.5,0) node at (0,0) [above]{${\scriptstyle 1}$};
}\,\in W_{1^{+}+1^{-}}\notag\\
\tikz[baseline=-.6ex]{
\draw[->-=.8] (-.5,0) -- (.5,0);
\draw[fill=white] (-.1,-.3) rectangle (.1,.3);
\node at (.1,0) [above right]{${\scriptstyle n}$};
}\,
&=
\bigg\langle\,\tikz[baseline=-.6ex]{
\draw[->-=.8] (-.5,.1) -- (.5,.1);
\draw[->-=.5] (-.5,-.4) -- (.5,-.4);
\draw[fill=white] (-.1,-.2) rectangle (.1,.4);
\node at (0,0) [above right]{${\scriptstyle n-1}$};
\node at (.2,-.4) [above right]{${\scriptstyle 1}$};
}\,\bigg\rangle_{\! 3}
-\frac{\left[n-1\right]}{\left[n\right]}
\bigg\langle\,\tikz[baseline=-.6ex]{
\draw[->-=.5] (-.9,.1) -- (-.4,.1);
\draw[->-=.5] (-.3,.2) -- (.3,.2);
\draw[->-=.5] (.4,.1) -- (.9,.1);
\draw[->-=.5] (-.9,-.4) 
to[out=east, in=west] (-.3,-.4) 
to[out=east, in=south] (-.1,-.2);
\draw[-<-=.8] (-.1,-.2) 
to[out=north, in=east] (-.3,0);
\draw[-<-=.5] (.9,-.4) 
to[out=west, in=east] (.3,-.4)
to[out=west, in=south] (.1,-.2);
\draw[->-=.8] (.1,-.2)
to[out=north, in=west] (.3,0);
\draw[fill=white] (-.4,-.2) rectangle (-.3,.4);
\draw[fill=white] (.3,-.2) rectangle (.4,.4);
\draw[-<-=.5] (-.1,-.2) -- (.1,-.2);
\node at (-.3,0) [above left]{${\scriptscriptstyle n-1}$};
\node at (.3,0) [above right]{${\scriptscriptstyle n-1}$};
\node at (0,.1) [above]{${\scriptscriptstyle n-2}$};
\node at (-.6,-.3) {${\scriptscriptstyle 1}$};
\node at (.6,-.3) {${\scriptscriptstyle 1}$};
\node at (0,-.2) [below]{${\scriptscriptstyle 1}$};
\node at (-.3,0) [right]{${\scriptscriptstyle 1}$};
\node at (.3,0) [left]{${\scriptscriptstyle 1}$};
}\,\bigg\rangle_{\! 3} \in W_{n^{+}+n^{-}}
\end{align}
\end{DEF}
$A_2$ clasps have the following properties.
\begin{LEM}[Properties of $A_2$ clasps]
For any positive integer $n$,
\begin{itemize} 
\item $\Big\langle\,\tikz[baseline=-.6ex]{
\draw[->-=.5] (-.6,0) -- (.6,0);
\draw[fill=white] (-.4,-.3) rectangle (-.2,.3);
\draw[fill=white] (.2,-.3) rectangle (.4,.3);
\node at (-.3,0) [above right]{${\scriptstyle n}$};
\node at (.3,0) [above right]{${\scriptstyle n}$};
}\,\Big\rangle_{\! 3}
=
\,\tikz[baseline=-.6ex]{
\draw[->-=.8] (-.5,0) -- (.5,0);
\draw[fill=white] (-.1,-.3) rectangle (.1,.3);
\node at (0,0) [above right] {${\scriptstyle n}$};
}\,
$,
\item $\Big\langle\,\tikz[baseline=-.6ex]{
\draw[->-=.5] (-.5,0) -- (-.1,0);
\draw[->-=.5] (0,.2) -- (.5,.2);
\draw[->-=.5] (0,-.2) -- (.5,-.2);
\draw[->-=.5] (0,.1) 
to[out=east, in=north] (.3,0);
\draw[->-=.5] (0,-.1)
to[out=east, in=south] (.3,0);
\draw[-<-=.5] (.3,0) -- (.5,0);
\draw[fill=white] (-.2,-.3) rectangle (0,.3);
\node at (.4,0) [right] {${\scriptstyle 1}$};
\node at (.4,.2) [right] {${\scriptstyle n-k-2}$};
\node at (.4,-.2) [right] {${\scriptstyle k}$};
}\,\Big\rangle_{\! 3}=0$\quad ($k=0,1,\dots,n-2$).
\end{itemize}
\end{LEM}

We can easily calculate the following.
\begin{LEM}\label{A2clasplem}\ 
For $k=0,1,\dots,n$,
\begin{itemize}
\item $\Big\langle\,\tikz[baseline=-.6ex]{
\draw[->-=.5] (-.5,0) -- (-.1,0);
\draw[->-=.8, white, double=black, double distance=0.4pt, ultra thick] 
(0,-.2) to[out=right, in=left] (.5,.2);
\draw[->-=.8, white, double=black, double distance=0.4pt, ultra thick] 
(0,.2) to[out=right, in=left] (.5,-.2);
\draw[fill=white] (-.2,-.3) rectangle (0,.3);
\node at (-.1,.0) [above left]{${\scriptstyle n}$};
\node at (.4,.2) [right] {${\scriptstyle k}$};
\node at (.4,-.2) [right] {${\scriptstyle n-k}$};
}\,\Big\rangle_{\! 3}
=q^{\frac{k(n-k)}{3}}
\Big\langle\,\tikz[baseline=-.6ex]{
\draw[->-=.2] (-.5,0) -- (.5,0);
\draw[fill=white] (-.1,-.3) rectangle (.1,.3);
\node at (.1,0) [above right]{${\scriptstyle n}$};
}\,\Big\rangle_{\! 3}$,  
$\Big\langle\,\tikz[baseline=-.6ex]{
\draw[->-=.5] (-.5,0) -- (-.1,0);
\draw[->-=.8, white, double=black, double distance=0.4pt, ultra thick] 
(0,.2) to[out=right, in=left] (.5,-.2);
\draw[->-=.8, white, double=black, double distance=0.4pt, ultra thick] 
(0,-.2) to[out=right, in=left] (.5,.2);
\draw[fill=white] (-.2,-.3) rectangle (0,.3);
\node at (-.1,.0) [above left]{${\scriptstyle n}$};
\node at (.4,.2) [right] {${\scriptstyle k}$};
\node at (.4,-.2) [right] {${\scriptstyle n-k}$};
}\,\Big\rangle_{\! 3}
=q^{-\frac{k(n-k)}{3}}
\Big\langle\,\tikz[baseline=-.6ex]{
\draw[->-=.2] (-.5,0) -- (.5,0);
\draw[fill=white] (-.1,-.3) rectangle (.1,.3);
\node at (.1,0) [above right]{${\scriptstyle n}$};
}\,\Big\rangle_{\! 3},$
\item 
$\Big\langle\,\tikz[baseline=-.6ex]{
\draw[->-=.2] (-.4,.2) -- (.4,.2);
\draw[->-=.8, rounded corners=.1cm] (-.2,-.2) rectangle (.2,.1);
\draw[fill=white] (-.05,.0) rectangle (.05,.3);
\node at (.0,.3)[above]{$\scriptstyle{n}$};
\node at (.4,.2)[right]{$\scriptstyle{n-k}$};
\node at (.0,-.2)[below]{$\scriptstyle{k}$};
}\,\Big\rangle_{\! 3}
=\frac{\left[n+1\right]\left[n+2\right]}{\left[n-k+1\right]\left[n-k+2\right]}
\Big\langle\,\tikz[baseline=-.6ex]{
\draw[->-=.2] (-.4,.0) -- (.4,.0);
\draw[fill=white] (-.05,-.2) rectangle (.05,.2);
\node at (.0,.2)[above]{$\scriptstyle{n-k}$};
}\,\Big\rangle_{\! 3},$
\item 
$\Big\langle\,\tikz[baseline=-.6ex]{
\draw[->-=.5] (-.4,.1) -- (.0,.1);
\draw[white, double=black, double distance=0.4pt, ultra thick] 
(.2,-.2) to[out=west, in=west]
(.2,.1) to[out=east, in=west]
(.5,.1);
\draw[->-=.8, white, double=black, double distance=0.4pt, ultra thick] 
(.0,.1) to[out=east, in=west]
(.2,.1) to[out=east, in=east]
(.2,-.2);
\draw[fill=white] (-.1,-.1) rectangle (.0,.3);
\node at (-.1,-.1)[below]{$\scriptstyle{n}$};
}\,\Big\rangle_{\! 3}
=q^{\frac{n^2+3n}{3}}
\Big\langle\,\tikz[baseline=-.6ex]{
\draw[->-=.2] (-.4,.0) -- (.4,.0);
\draw[fill=white] (-.05,-.2) rectangle (.05,.2);
\node at (.0,.2)[above]{$\scriptstyle{n}$};
}\,\Big\rangle_{\! 3}$, 
$\Big\langle\,\tikz[baseline=-.6ex]{
\draw (-.4,.1) -- (.0,.1);
\draw[white, double=black, double distance=0.4pt, ultra thick] 
(.0,.1) to[out=east, in=west]
(.2,.1) to[out=east, in=east]
(.2,-.2);
\draw[->-=.8, white, double=black, double distance=0.4pt, ultra thick] 
(.2,-.2) to[out=west, in=west]
(.2,.1) to[out=east, in=west]
(.5,.1);
\draw[fill=white] (-.1,-.1) rectangle (.0,.3);
\node at (-.1,-.1)[below]{$\scriptstyle{n}$};
}\,\Big\rangle_{\! 3}
=q^{-\frac{n^2+3n}{3}}
\Big\langle\,\tikz[baseline=-.6ex]{
\draw[->-=.2] (-.4,.0) -- (.4,.0);
\draw[fill=white] (-.05,-.2) rectangle (.05,.2);
\node at (.0,.2)[above]{$\scriptstyle{n}$};
}\,\Big\rangle_{\! 3}.$
\end{itemize}
\end{LEM}

Let $n^{+}$ be an $n$-tuple of ${+}$ and $n^{-}$ an $n$-tuple of ${-}$.
For a $k$-tuple of signs $\varepsilon=(\varepsilon_1,\varepsilon_2,\dots,\varepsilon_k)$, 
We define $N^{\varepsilon}=(n_1^{\varepsilon_1},n_2^{\varepsilon_2},\dots,n_k^{\varepsilon_k})$. 
We use a notation $n_1^{\varepsilon_1}+n_2^{\varepsilon_2}+\dots+n_k^{\varepsilon_k}$ to represent $(n_1+n_2+\dots+n_k)$-tuple of signs. 
It is defined in the following manner: the first $n_1$ signs are $\varepsilon_1$, the next $n_2$ signs are $\varepsilon_2$, $\ldots$ , and the last $n_k$ signs are $\varepsilon_k$.
\begin{DEF}[A clasped $A_2$ web space]
We define a subspace $W_{N^{\varepsilon}}$ of $W_{n_1^{\varepsilon_1}+n_2^{\varepsilon_2}+\dots+n_k^{\varepsilon_k}}$ called the {\em clasped $A_2$ web space} as follows:
\[
W_{N^{\varepsilon}}
=\Bigg\{\,
\bigg\langle\tikz[baseline=-.6ex]{
\draw [thin, fill=lightgray!50] (0,0) circle [radius=1];
\draw (0,0) -- (-30:1);
\draw (0,0) -- (0:1);
\draw (0,0) -- (30:1);
\draw (0,0) -- (60:1);
\coordinate (a) at ($(0,0)!.8!-30:(-12:1)$);
\coordinate (b) at ($(0,0)!.9!-30:(12:1)$);
\coordinate (a0) at ($(0,0)!.8!0:(-12:1)$);
\coordinate (b0) at ($(0,0)!.9!0:(12:1)$);
\coordinate (a1) at ($(0,0)!.8!30:(-12:1)$);
\coordinate (b1) at ($(0,0)!.9!30:(12:1)$);
\coordinate (a2) at ($(0,0)!.8!60:(-12:1)$);
\coordinate (b2) at ($(0,0)!.9!60:(12:1)$);
\draw[fill=white,rotate=-30] (a) rectangle (b);
\draw[fill=white,rotate=0] (a0) rectangle (b0);
\draw[fill=white,rotate=30] (a1) rectangle (b1);
\draw[fill=white,rotate=60] (a2) rectangle (b2);
\foreach \i in {0,1,...,11} \draw[fill=cyan] ($(0,0) !1! \i*30:(1,0)$) circle [radius=1pt];
\draw[fill=cyan] (-30:1) circle [radius=1pt] node [right]{${\scriptstyle n_k^{\varepsilon_k}}$};
\draw[fill=cyan] (0:1) circle [radius=1pt] node [right]{${\scriptstyle n_1^{\varepsilon_1}}$};
\draw[fill=cyan] (30:1) circle [radius=1pt] node [right]{${\scriptstyle n_2^{\varepsilon_2}}$};
\draw[fill=cyan] (60:1) circle [radius=1pt] node [right]{${\scriptstyle n_3^{\varepsilon_3}}$};
\node[rotate=10] at (100:.8){$\cdots$} node[rotate=25] at (-65:.8){$\cdots$} node[rotate=-60] at (-150:.8){$\cdots$};
\node[circle, draw, fill=lightgray!30] (0,0) {\quad$w$\quad\quad};
}\bigg\rangle_{\! 2}
\;{\Bigg\vert}\; w\in W_{n_1^{\varepsilon_1}+n_2^{\varepsilon_2}+\dots+n_k^{\varepsilon_k}}\,\Bigg\}
\]
\end{DEF}

\section{Colored skein relations}
In this section,
we introduce colored skein relations for clasped $A_1$ and $A_2$ web spaces. 
Although the formula for a clasped $A_1$ web space is already known by Yamada~\cite{Yamada89,Yamada92} and Hajij~\cite{Hajij14.2}, 
we prove this formula by another method using the theory of integer partitions. 
We use the same method to prove colored skein relations for clasped $A_2$ web spaces.

\subsection{Colored Kauffman bracket skein relations}
We review the colored Kauffman bracket skein relations and give another proof by using theory of integer partitions. 
Let us consider clasped $A_1$ web spaces.
\begin{PROP}\label{coloredA1} Let $n$ be non-negative integers.  
\begin{enumerate}
\item 
$\displaystyle
\Bigg\langle\,\tikz[baseline=-.6ex]{
\draw
(-.4,-.4) -- +(-.2,0)
(.4,.4) -- +(.2,0)
(-.4,.4) -- +(-.2,0)
(.4,-.4) -- +(.2,0);
\draw[white, double=black, double distance=0.4pt, ultra thick] 
(-.4,-.4) to[out=east, in=west] (.4,.4);
\draw[white, double=black, double distance=0.4pt, ultra thick] 
(-.4,.4) to[out=east, in=west] (.4,-.4);
\draw[fill=white] (.4,-.6) rectangle +(.1,.4);
\draw[fill=white] (-.4,-.6) rectangle +(-.1,.4);
\draw[fill=white] (.4,.6) rectangle +(.1,-.4);
\draw[fill=white] (-.4,.6) rectangle +(-.1,-.4);
\node at (.4,-.6)[right]{$\scriptstyle{n}$};
\node at (-.4,-.6)[left]{$\scriptstyle{n}$};
\node at (.4,.6)[right]{$\scriptstyle{n}$};
\node at (-.4,.6)[left]{$\scriptstyle{n}$};
}\,\Bigg\rangle_{\! 2}
=\sum_{k=0}^{n} q^{\frac{-n^2+2k^2}{4}}{n \choose k}_q
\Bigg\langle\,\tikz[baseline=-.6ex]{
\draw 
(-.4,.4) -- +(-.2,0)
(.4,-.4) -- +(.2,0)
(-.4,-.4) -- +(-.2,0)
(.4,.4) -- +(.2,0)
(-.4,.5) -- (.4,.5)
(-.4,-.5) -- (.4,-.5)
(-.4,.3) to[out=east, in=east] (-.4,-.3)
(.4,.3) to[out=west, in=west] (.4,-.3);
\draw[fill=white] (.4,-.6) rectangle +(.1,.4);
\draw[fill=white] (-.4,-.6) rectangle +(-.1,.4);
\draw[fill=white] (.4,.6) rectangle +(.1,-.4);
\draw[fill=white] (-.4,.6) rectangle +(-.1,-.4);
\node at (.4,-.6)[right]{$\scriptstyle{n}$};
\node at (-.4,-.6)[left]{$\scriptstyle{n}$};
\node at (.4,.6)[right]{$\scriptstyle{n}$};
\node at (-.4,.6)[left]{$\scriptstyle{n}$};
\node at (0,.5)[above]{$\scriptstyle{k}$};
\node at (0,-.5)[below]{$\scriptstyle{k}$};
\node at (-.2,0)[left]{$\scriptstyle{n-k}$};
\node at (.2,0)[right]{$\scriptstyle{n-k}$};
}\,\Bigg\rangle_{\! 2}
$
\quad(the {\em colored Kauffman bracket skein relation} by Hajij~\cite{Hajij14.2})
\item
$
\Big\langle\,\tikz[baseline=-.6ex]{
\draw (0,0) circle [radius=.3];
\draw[fill=white] (.1,-.05) rectangle (.5,.05);
\node at (.3,0)[above right]{$\scriptstyle{n}$};
}\,\Big\rangle_{\! 2}
=(-1)^{n}\left[n+1\right] \emptyset
$
\end{enumerate}
\end{PROP}
We prove the colored Kauffman bracket skein relation by using a well-known identity from the theory of integer partitions. 
Let $\lambda=(\lambda_1, \lambda_2, \dots, \lambda_s)$ be a partition of an integer into $s$ parts, 
that is, 
$\lambda$ is the $s$-tuples of positive integers such that $\lambda_i\geq\lambda_{i+1}$ for $i=1,2,\dots,s-1$, and $\left|\lambda\right|$ denotes $\lambda_1+\lambda_2+\dots+\lambda_s$. 
For given non-negative integers $k$ and $l$, 
$\mathcal{P}(k,l)$ denotes the set of partitions $\lambda$ such that $0\leq s\leq k$ and $0\leq \lambda_1\leq l$.
\begin{LEM}[Andrews and Eriksson~\cite{AndrewsEriksson04} etc.]\label{partition}
\[
{k+l \choose k}_q=\sum_{\lambda\in\mathcal{P}(k,l)}q^{\left|\lambda\right|}.
\]
\end{LEM}

\begin{proof}[Proof of Proposition~\ref{coloredA1}~$(1)$]
\begin{align}
\Bigg\langle\,\tikz[baseline=-.6ex]{
\draw
(-.4,-.4) -- +(-.2,0)
(.4,.4) -- +(.2,0)
(-.4,.4) -- +(-.2,0)
(.4,-.4) -- +(.2,0);
\draw[white, double=black, double distance=0.4pt, ultra thick] 
(-.4,-.4) to[out=east, in=west] (.4,.4);
\draw[white, double=black, double distance=0.4pt, ultra thick] 
(-.4,.4) to[out=east, in=west] (.4,-.4);
\draw[fill=white] (.4,-.6) rectangle +(.1,.4);
\draw[fill=white] (-.4,-.6) rectangle +(-.1,.4);
\draw[fill=white] (.4,.6) rectangle +(.1,-.4);
\draw[fill=white] (-.4,.6) rectangle +(-.1,-.4);
\node at (.4,-.6)[right]{$\scriptstyle{n}$};
\node at (-.4,-.6)[left]{$\scriptstyle{n}$};
\node at (.4,.6)[right]{$\scriptstyle{n}$};
\node at (-.4,.6)[left]{$\scriptstyle{n}$};
}\,\Bigg\rangle_{\! 2}
&=
\Bigg\langle\,\tikz[baseline=-.6ex]{
\draw 
(-.6,-.4) -- +(-.2,0)
(.6,.4) -- +(.2,0)
(-.6,.4) -- +(-.2,0)
(.6,-.4) -- +(.2,0);
\draw[triple={[line width=1.6pt, white] in [line width=2.4pt, black] in [line width=5.6pt, white]}]
(-.6,-.4) to[out=east, in=west] (.6,.4);
\draw[triple={[line width=1.6pt, white] in [line width=2.4pt, black] in [line width=5.6pt, white]}]
(-.6,.4) to[out=east, in=west] (.6,-.4);
\draw[fill=white] (.6,-.6) rectangle +(.1,.4);
\draw[fill=white] (-.6,-.6) rectangle +(-.1,.4);
\draw[fill=white] (.6,.6) rectangle +(.1,-.4);
\draw[fill=white] (-.6,.6) rectangle +(-.1,-.4);
\node at (.6,-.6)[right]{$\scriptstyle{n}$};
\node at (-.6,-.6)[left]{$\scriptstyle{n}$};
\node at (.6,.6)[right]{$\scriptstyle{n}$};
\node at (-.6,.6)[left]{$\scriptstyle{n}$};
\node at (-.7,-.6)[right]{$\scriptscriptstyle{n-1}$};
\node at (-.7,.6)[right]{$\scriptscriptstyle{n-1}$};
\node at (-.7,-.2)[right]{$\scriptscriptstyle{1}$};
\node at (-.7,.2)[right]{$\scriptscriptstyle{1}$};
}\,\Bigg\rangle_{\! 2}\notag\\ 
&=q^{\frac{1}{4}}\Bigg\langle\,\tikz[baseline=-.6ex]{
\draw
(-.4,-.4) -- +(-.2,0)
(.4,.4) -- +(.2,0)
(-.4,.4) -- +(-.2,0)
(.4,-.4) -- +(.2,0);
\draw[white, double=black, double distance=0.4pt, ultra thick] 
(-.4,-.5) to[out=east, in=west] (.4,.3);
\draw[white, double=black, double distance=0.4pt, ultra thick] 
(-.4,-.3) to[out=east, in=west] (.4,-.5);
\draw[white, double=black, double distance=0.4pt, ultra thick] 
(-.4,.3) to[out=east, in=west] (.4,.5);
\draw[white, double=black, double distance=0.4pt, ultra thick] 
(-.4,.5) to[out=east, in=west] (.4,-.3);
\draw[fill=white] (.4,-.6) rectangle +(.1,.4);
\draw[fill=white] (-.4,-.6) rectangle +(-.1,.4);
\draw[fill=white] (.4,.6) rectangle +(.1,-.4);
\draw[fill=white] (-.4,.6) rectangle +(-.1,-.4);
\node at (.4,-.6)[right]{$\scriptstyle{n}$};
\node at (-.4,-.6)[left]{$\scriptstyle{n}$};
\node at (.4,.6)[right]{$\scriptstyle{n}$};
\node at (-.4,.6)[left]{$\scriptstyle{n}$};
\node at (-.5,-.6)[right]{$\scriptscriptstyle{n-1}$};
\node at (-.5,.6)[right]{$\scriptscriptstyle{n-1}$};
\node at (-.5,-.2)[right]{$\scriptscriptstyle{1}$};
\node at (-.5,.2)[right]{$\scriptscriptstyle{1}$};
}\,\Bigg\rangle_{\! 2}
+q^{-\frac{1}{4}}\Bigg\langle\,\tikz[baseline=-.6ex]{
\draw
(-.4,-.4) -- +(-.2,0)
(.4,.4) -- +(.2,0)
(-.4,.4) -- +(-.2,0)
(.4,-.4) -- +(.2,0);
\draw[white, double=black, double distance=0.4pt, ultra thick] 
(-.4,-.5) to[out=east, in=west] (.4,.3);
\draw[white, double=black, double distance=0.4pt, ultra thick] 
(-.4,-.3) 
to[out=east, in=south] (-.3,0)
to[out=north, in=east](-.4,.3);
\draw[white, double=black, double distance=0.4pt, ultra thick] 
(.4,-.5)
to[out=west, in=south] (.2,0)
to[out=north, in=west] (.4,.5);
\draw[white, double=black, double distance=0.4pt, ultra thick] 
(-.4,.5) to[out=east, in=west] (.4,-.3);
\draw[fill=white] (.4,-.6) rectangle +(.1,.4);
\draw[fill=white] (-.4,-.6) rectangle +(-.1,.4);
\draw[fill=white] (.4,.6) rectangle +(.1,-.4);
\draw[fill=white] (-.4,.6) rectangle +(-.1,-.4);
\node at (.4,-.6)[right]{$\scriptstyle{n}$};
\node at (-.4,-.6)[left]{$\scriptstyle{n}$};
\node at (.4,.6)[right]{$\scriptstyle{n}$};
\node at (-.4,.6)[left]{$\scriptstyle{n}$};
\node at (-.5,-.6)[right]{$\scriptscriptstyle{n-1}$};
\node at (-.5,.6)[right]{$\scriptscriptstyle{n-1}$};
\node at (-.3,0)[left]{$\scriptscriptstyle{1}$};
\node at (.2,0)[right]{$\scriptscriptstyle{1}$};
}\,\Bigg\rangle_{\! 2}\label{coloredA1eq1}\\
&=q^{\frac{2n-1}{4}}\Bigg\langle\,\tikz[baseline=-.6ex]{
\draw
(-.4,-.4) -- +(-.2,0)
(.4,.4) -- +(.2,0)
(-.4,.4) -- +(-.2,0)
(.4,-.4) -- +(.2,0);
\draw[white, double=black, double distance=0.4pt, ultra thick] 
(-.4,-.3) to[out=east, in=west] (.4,.3);
\draw[white, double=black, double distance=0.4pt, ultra thick] 
(-.4,-.5) to[out=east, in=west] (.4,-.5);
\draw[white, double=black, double distance=0.4pt, ultra thick] 
(-.4,.5) to[out=east, in=west] (.4,.5);
\draw[white, double=black, double distance=0.4pt, ultra thick] 
(-.4,.3) to[out=east, in=west] (.4,-.3);
\draw[fill=white] (.4,-.6) rectangle +(.1,.4);
\draw[fill=white] (-.4,-.6) rectangle +(-.1,.4);
\draw[fill=white] (.4,.6) rectangle +(.1,-.4);
\draw[fill=white] (-.4,.6) rectangle +(-.1,-.4);
\node at (.4,-.6)[right]{$\scriptstyle{n}$};
\node at (-.4,-.6)[left]{$\scriptstyle{n}$};
\node at (.4,.6)[right]{$\scriptstyle{n}$};
\node at (-.4,.6)[left]{$\scriptstyle{n}$};
\node at (0,-.6){$\scriptscriptstyle{1}$};
\node at (0,.6){$\scriptscriptstyle{1}$};
\node at (-.4,.3)[below]{$\scriptscriptstyle{n-1}$};
\node at (.4,.3)[below]{$\scriptscriptstyle{n-1}$};
}\,\Bigg\rangle_{\! 2}
+q^{-\frac{2n-1}{4}}\Bigg\langle\,\tikz[baseline=-.6ex]{
\draw
(-.4,-.4) -- +(-.2,0)
(.4,.4) -- +(.2,0)
(-.4,.4) -- +(-.2,0)
(.4,-.4) -- +(.2,0);
\draw[white, double=black, double distance=0.4pt, ultra thick] 
(-.4,-.5) to[out=east, in=west] (.4,.5);
\draw[white, double=black, double distance=0.4pt, ultra thick] 
(-.4,-.3) 
to[out=east, in=south] (-.3,0)
to[out=north, in=east](-.4,.3);
\draw[white, double=black, double distance=0.4pt, ultra thick] 
(.4,-.3)
to[out=west, in=south] (.3,0)
to[out=north, in=west] (.4,.3);
\draw[white, double=black, double distance=0.4pt, ultra thick] 
(-.4,.5) to[out=east, in=west] (.4,-.5);
\draw[fill=white] (.4,-.6) rectangle +(.1,.4);
\draw[fill=white] (-.4,-.6) rectangle +(-.1,.4);
\draw[fill=white] (.4,.6) rectangle +(.1,-.4);
\draw[fill=white] (-.4,.6) rectangle +(-.1,-.4);
\node at (.4,-.6)[right]{$\scriptstyle{n}$};
\node at (-.4,-.6)[left]{$\scriptstyle{n}$};
\node at (.4,.6)[right]{$\scriptstyle{n}$};
\node at (-.4,.6)[left]{$\scriptstyle{n}$};
\node at (-.5,-.6)[right]{$\scriptscriptstyle{n-1}$};
\node at (-.5,.6)[right]{$\scriptscriptstyle{n-1}$};
\node at (-.3,0)[left]{$\scriptscriptstyle{1}$};
\node at (.3,0)[right]{$\scriptscriptstyle{1}$};
}\,\Bigg\rangle_{\! 2}\label{coloredA1eq2}
\end{align}
We used the Kauffman bracket skein relation in $(\ref{coloredA1eq1})$ and  Lemma~\ref{A1clasplem} in $(\ref{coloredA1eq2})$. 
We define a clasped $A_1$ web $\langle\sigma(k,l;n)\rangle_2$ as follows:
\[
\langle\sigma(k,l;n)\rangle_2
=\Bigg\langle\,\tikz[baseline=-.6ex]{
\draw
(-.4,-.4) -- +(-.2,0)
(.4,.4) -- +(.2,0)
(-.4,.4) -- +(-.2,0)
(.4,-.4) -- +(.2,0);
\draw[white, double=black, double distance=0.4pt, ultra thick] 
(-.4,-.4) to[out=east, in=west] (.4,.4);
\draw[white, double=black, double distance=0.4pt, ultra thick] 
(-.4,-.3) 
to[out=east, in=south] (-.3,0)
to[out=north, in=east](-.4,.3);
\draw[white, double=black, double distance=0.4pt, ultra thick] 
(.4,-.3)
to[out=west, in=south] (.3,0)
to[out=north, in=west] (.4,.3);
\draw[white, double=black, double distance=0.4pt, ultra thick] 
(-.4,.4) to[out=east, in=west] (.4,-.4);
\draw (-.4,.5) -- (.4,.5);
\draw (-.4,-.5) -- (.4,-.5);
\draw[fill=white] (.4,-.6) rectangle +(.1,.4);
\draw[fill=white] (-.4,-.6) rectangle +(-.1,.4);
\draw[fill=white] (.4,.6) rectangle +(.1,-.4);
\draw[fill=white] (-.4,.6) rectangle +(-.1,-.4);
\node at (.4,-.6)[right]{$\scriptstyle{n}$};
\node at (-.4,-.6)[left]{$\scriptstyle{n}$};
\node at (.4,.6)[right]{$\scriptstyle{n}$};
\node at (-.4,.6)[left]{$\scriptstyle{n}$};
\node at (0,-.6){$\scriptscriptstyle{k}$};
\node at (0,.6){$\scriptscriptstyle{k}$};
\node at (-.3,0)[left]{$\scriptscriptstyle{l}$};
\node at (.3,0)[right]{$\scriptscriptstyle{l}$};
}\,\Bigg\rangle_{\! 2}.
\]
By the above calculation,
\begin{equation}\label{A1halfres}
 \langle\sigma(k,l;n)\rangle_2=q^{\frac{2(n-k-l)-1}{4}}\langle\sigma(k+1,l;n)\rangle_2
+q^{-\frac{2(n-k-l)-1}{4}}\langle\sigma(k,l+1;n)\rangle_2.
\end{equation}
We make $\langle\sigma(k,l;n)\rangle_2$ correspond to a lattice point $(k,l)$ in $\mathbb{Z}\times\mathbb{Z}$ for each non-negative integers $k$ and $l$ such that $0\leq k+l\leq n$. 
We decorate vectors $(1,0)$ and $(0,1)$ from $(k,l)$ with coefficients of $\langle\sigma(k+1,l;n)\rangle_2$ and $\langle\sigma(k,l+1;n)\rangle_2$ of (\ref{A1halfres}), respectively. 
The left-hand side of the colored Kauffman bracket skein relation is $\langle\sigma(0,0;n)\rangle_2$ 
and clasped $A_1$ webs appear in the right-hand side are $\langle\sigma(k,l;n)\rangle_2$ such that $k+l=n$.
Fix a pair of $k$ and $l$ such that $k+l=n$,
then the coefficient of $\langle\sigma(k,l;n)\rangle_2$ of the right-hand side of the colored Kauffman bracket skein relation is determined as follows:
\begin{enumerate}
\item Take a lattice path from $(0,0)$ to $(k,l)$ constructed from vectors $(1,0)$ and $(0,1)$. 
\item Product all decorations of vectors appearing in the path.
\item Sum up the product of decorations for all paths from $(0,0)$ to $(k,l)$.
\end{enumerate}
A Young diagram is obtained by cutting out the upper side of a path from $(0,0)$ to $(k,l)$ from a rectangle $\{(x,y)\mid 0\leq x\leq k, 0\leq y\leq l \}$.
The Young diagram corresponds to a partition of an integer (see two examples in Figure~\ref{Young}). 
\begin{figure}
\centering
\begin{tikzpicture}
\foreach \i in {0,...,4}
{
\draw (\i*.4,.0) -- (\i*.4,2-.4*\i);
\draw (.0,\i*.4) -- (2-.4*\i,\i*.4);
}
\draw[->, very thick, cyan]
(.0,.0) -- (.0,.4) -- (.4,.4) -- (.4,.8) -- (.8,.8) -- (.8,1.2);
\fill (.8,1.2) circle (1.2pt);
\fill (.0,.0) circle (1.2pt);
\node at (.8,1.2) [right]{$\scriptstyle{(k,l)}$};
\node at (.0,.0) [left]{$\scriptstyle{(0,0)}$};
\node at (.8,-.2) [below]{path from $(0,0)$ to $(k,l)$};
\end{tikzpicture}
\begin{tikzpicture}
\fill[magenta!50] 
(.0,.4) -- (.0,1.2) -- (.8,1.2) -- (.8,.8) -- (.4,.8) -- (.4,.4) -- cycle;
\foreach \i in {0,...,4}
{
\draw (\i*.4,.0) -- (\i*.4,2-.4*\i);
\draw (.0,\i*.4) -- (2-.4*\i,\i*.4);
}
\draw[->,very thick,cyan]
(.0,.0) -- (.0,.4) -- (.4,.4) -- (.4,.8) -- (.8,.8) -- (.8,1.2);
\fill (.8,1.2) circle (1.2pt);
\fill (.0,.0) circle (1.2pt);
\node at (.8,1.2) [right]{$\scriptstyle{(k,l)}$};
\node at (.0,.0) [left]{$\scriptstyle{(0,0)}$};
\node at (.8,-.2) [below]{Young diagram $\lambda_1$};
\end{tikzpicture}

This Young diagram $\lambda_1$ corresponds to a partition $(2,1)$.

\begin{tikzpicture}
\foreach \i in {0,...,4}
{
\draw (\i*.4,.0) -- (\i*.4,2-.4*\i);
\draw (.0,\i*.4) -- (2-.4*\i,\i*.4);
}
\draw[->,very thick,cyan]
(.0,.0) -- (.4,.0) -- (.4,1.2) -- (.8,1.2);
\fill (.8,1.2) circle (1.2pt);
\fill (.0,.0) circle (1.2pt);
\node at (.8,1.2) [right]{$\scriptstyle{(k,l)}$};
\node at (.0,.0) [left]{$\scriptstyle{(0,0)}$};
\node at (.8,-.2) [below]{path from $(0,0)$ to $(k,l)$};
\end{tikzpicture}
\begin{tikzpicture}
\fill[magenta!50] 
(.0,.0) -- (.4,.0) -- (.4,1.2) -- (.0,1.2) -- cycle;
\foreach \i in {0,...,4}
{
\draw (\i*.4,.0) -- (\i*.4,2-.4*\i);
\draw (.0,\i*.4) -- (2-.4*\i,\i*.4);
}
\draw[->, very thick,cyan]
(.0,.0) -- (.4,.0) -- (.4,1.2) -- (.8,1.2);
\fill (.8,1.2) circle (1.2pt);
\fill (.0,.0) circle (1.2pt);
\node at (.8,1.2) [right]{$\scriptstyle{(k,l)}$};
\node at (.0,.0) [left]{$\scriptstyle{(0,0)}$};
\node at (.8,-.2) [below]{Young diagram $\lambda_2$};
\end{tikzpicture}

This Young diagram $\lambda_2$ corresponds to a partition $(1,1,1)$.

\caption{two examples for $n=5$, $(k,l)=(2,3)$}
\label{Young}
\end{figure}
If a path from $(0,0)$ to $(k,l)$ with $k+l=n$ shifts downward by a box, 
then the product of decorations is multiplied by $q$.
In fact, 
a product of decorations of $(k,l)\to (k,l+1)\to(k+1,l+1)$ is $q^{-\frac{1}{2}}$ 
and $(k,l)\to (k+1,l)\to(k+1,l+1)$ is $q^{\frac{1}{2}}$ (see Figure~\ref{A1halfshift}).
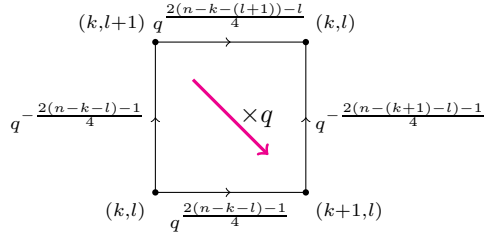
\begin{figure}
\centering
\begin{tikzpicture}
\draw[->-=.5] (0,0) -- (2,0);
\node at (1,0)[below]{$\scriptstyle{q^{\frac{2(n-k-l)-1}{4}}}$};
\draw[->-=.5] (2,0) -- (2,2);
\node at (2,1)[right]{$\scriptstyle{q^{-\frac{2(n-(k+1)-l)-1}{4}}}$};
\draw[->-=.5] (0,0) -- (0,2);
\node at (0,1)[left]{$\scriptstyle{q^{-\frac{2(n-k-l)-1}{4}}}$};
\draw[->-=.5] (0,2) -- (2,2);
\node at (1,2)[above]{$\scriptstyle{q^{\frac{2(n-k-(l+1))-l}{4}}}$};
\draw[->, very thick, magenta](.5,1.5) -- (1.5,.5);
\node at (1,1)[right]{$\times q$};
\node at (0,0)[below left]{$\scriptstyle{(k,l)}$};
\node at (2,0)[below right]{$\scriptstyle{(k+1,l)}$};
\node at (0,2)[above left]{$\scriptstyle{(k,l+1)}$};
\node at (2,2)[above right]{$\scriptstyle{(k,l)}$};
\fill 
(0,0) circle (1.2pt)
(2,0) circle (1.2pt)
(2,2) circle (1.2pt)
(0,2) circle (1.2pt);
\end{tikzpicture}
\caption{shift of a path}
\label{A1halfshift}
\end{figure}
The product of decorations of the top path is $\prod_{i=0}^{l-1}q^{-\frac{2(n-i)-1}{4}}\prod_{j=0}^{k-1}q^{\frac{2(n-l-j)-1}{4}}$.
Therefore,
the coefficient of $\langle \sigma(k,l;n)\rangle_2$ with $k+l=n$ is
\[
\prod_{i=0}^{l-1}q^{-\frac{2(n-i)-1}{4}}\prod_{j=0}^{k-1}q^{\frac{2(n-l-j)-1}{4}}
\sum_{\lambda\in\mathcal{P}(k,l)}q^{\left|\lambda\right|}=q^{\frac{-n^2+2k^2}{4}}{n\choose k}_{q}
\]
by Lemma~\ref{partition}.
\end{proof}

The colored Kauffman bracket skein relation is the expansion of a half twist of two strands colored by $n$.
We also give the expansion of a full twist of two strands colored by $n$.
\begin{PROP}[The full twist formula~{\cite[Lemma~$4.1$]{Masbaum03}}]\label{A1full}
\[
\Bigg\langle\,\tikz[baseline=-.6ex]{
\draw 
(-.5,.4) -- +(-.2,0)
(.5,-.4) -- +(.2,0)
(-.5,-.4) -- +(-.2,0)
(.5,.4) -- +(.2,0);
\draw[white, double=black, double distance=0.4pt, ultra thick] 
(-.5,-.4) to[out=east, in=west] (.0,.4);
\draw[white, double=black, double distance=0.4pt, ultra thick] 
(-.5,.4) to[out=east, in=west] (.0,-.4);
\draw[white, double=black, double distance=0.4pt, ultra thick] 
(0,-.4) to[out=east, in=west] (.5,.4);
\draw[white, double=black, double distance=0.4pt, ultra thick] 
(0,.4) to[out=east, in=west] (.5,-.4);
\draw[fill=white] (.5,-.6) rectangle +(.1,.4);
\draw[fill=white] (-.5,-.6) rectangle +(-.1,.4);
\draw[fill=white] (.5,.6) rectangle +(.1,-.4);
\draw[fill=white] (-.5,.6) rectangle +(-.1,-.4);
\node at (.5,-.6)[right]{$\scriptstyle{n}$};
\node at (-.5,-.6)[left]{$\scriptstyle{n}$};
\node at (.5,.6)[right]{$\scriptstyle{n}$};
\node at (-.5,.6)[left]{$\scriptstyle{n}$};
}\,\Bigg\rangle_{\! 2}
=\sum_{k=0}^{n} (-1)^{n-k}q^{\frac{2k^2-n^2+k-n}{2}}\frac{(q)_n}{(q)_{k}}{n \choose k}_q
\Bigg\langle\,\tikz[baseline=-.6ex]{
\draw 
(-.4,.4) -- +(-.2,0)
(.4,-.4) -- +(.2,0)
(-.4,-.4) -- +(-.2,0)
(.4,.4) -- +(.2,0)
(-.4,.5) -- (.4,.5)
(-.4,-.5) -- (.4,-.5)
(-.4,.3) to[out=east, in=east] (-.4,-.3)
(.4,.3) to[out=west, in=west] (.4,-.3);
\draw[fill=white] (.4,-.6) rectangle +(.1,.4);
\draw[fill=white] (-.4,-.6) rectangle +(-.1,.4);
\draw[fill=white] (.4,.6) rectangle +(.1,-.4);
\draw[fill=white] (-.4,.6) rectangle +(-.1,-.4);
\node at (.4,-.6)[right]{$\scriptstyle{n}$};
\node at (-.4,-.6)[left]{$\scriptstyle{n}$};
\node at (.4,.6)[right]{$\scriptstyle{n}$};
\node at (-.4,.6)[left]{$\scriptstyle{n}$};
\node at (0,.5)[above]{$\scriptstyle{k}$};
\node at (0,-.5)[below]{$\scriptstyle{k}$};
\node at (-.2,0)[left]{$\scriptstyle{n-k}$};
\node at (.2,0)[right]{$\scriptstyle{n-k}$};
}\,\Bigg\rangle_{\! 2}
\]
\end{PROP}
\begin{proof}
We prove this formula by the same argument as the proof of Proposition~\ref{coloredA1}.
First,
the following equation is obtained by using the Kauffman bracket skein relation and Lemma~\ref{A1clasplem}.
\begin{align*}
\Bigg\langle\,\tikz[baseline=-.6ex]{
\draw 
(-.5,.4) -- +(-.2,0)
(.5,-.4) -- +(.2,0)
(-.5,-.4) -- +(-.2,0)
(.5,.4) -- +(.2,0);
\draw[white, double=black, double distance=0.4pt, ultra thick] 
(-.5,-.4) to[out=east, in=west] (.0,.4);
\draw[white, double=black, double distance=0.4pt, ultra thick] 
(-.5,.4) to[out=east, in=west] (.0,-.4);
\draw[white, double=black, double distance=0.4pt, ultra thick] 
(0,-.4) to[out=east, in=west] (.5,.4);
\draw[white, double=black, double distance=0.4pt, ultra thick] 
(0,.4) to[out=east, in=west] (.5,-.4);
\draw[fill=white] (.5,-.6) rectangle +(.1,.4);
\draw[fill=white] (-.5,-.6) rectangle +(-.1,.4);
\draw[fill=white] (.5,.6) rectangle +(.1,-.4);
\draw[fill=white] (-.5,.6) rectangle +(-.1,-.4);
\node at (.5,-.6)[right]{$\scriptstyle{1}$};
\node at (-.5,-.6)[left]{$\scriptstyle{1}$};
\node at (.5,.6)[right]{$\scriptstyle{n}$};
\node at (-.5,.6)[left]{$\scriptstyle{n}$};
}\,\Bigg\rangle_{\! 2}
&=
\Bigg\langle\,\tikz[baseline=-.6ex]{
\draw 
(-.7,.4) -- +(-.2,0)
(.7,-.4) -- +(.2,0)
(-.7,-.4) -- +(-.2,0)
(.7,.4) -- +(.2,0);
\draw[white, double=black, double distance=0.4pt, ultra thick] 
(-.7,-.4) to[out=east, in=west] (.0,.4);
\draw[triple={[line width=1.6pt, white] in [line width=2.4pt, black] in [line width=5.6pt, white]}] 
(-.7,.4) to[out=east, in=west] (.0,-.4);
\draw[triple={[line width=1.6pt, white] in [line width=2.4pt, black] in [line width=5.6pt, white]}] 
(0,-.4) to[out=east, in=west] (.7,.4);
\draw[white, double=black, double distance=0.4pt, ultra thick] 
(0,.4) to[out=east, in=west] (.7,-.4);
\draw[fill=white] (.7,-.6) rectangle +(.1,.4);
\draw[fill=white] (-.7,-.6) rectangle +(-.1,.4);
\draw[fill=white] (.7,.6) rectangle +(.1,-.4);
\draw[fill=white] (-.7,.6) rectangle +(-.1,-.4);
\node at (.7,-.6)[right]{$\scriptstyle{1}$};
\node at (-.7,-.6)[left]{$\scriptstyle{1}$};
\node at (.7,.6)[right]{$\scriptstyle{n}$};
\node at (-.7,.6)[left]{$\scriptstyle{n}$};
\node at (.0,-.4)[above]{$\scriptstyle{1}$};
\node at (.0,-.4)[below]{$\scriptstyle{n-1}$};
}\,\Bigg\rangle_{\! 2}\\
&=q^{\frac{1}{4}}
\Bigg\langle\,\tikz[baseline=-.6ex]{
\draw 
(-.5,.4) -- +(-.2,0)
(.5,-.4) -- +(.2,0)
(-.5,-.4) -- +(-.2,0)
(.5,.4) -- +(.2,0);
\draw[white, double=black, double distance=0.4pt, ultra thick] 
(-.5,-.4) to[out=east, in=west] (.0,-.1);
\draw[white, double=black, double distance=0.4pt, ultra thick] 
(-.5,.3) to[out=east, in=west] (.0,-.4);
\draw[white, double=black, double distance=0.4pt, ultra thick] 
(0,-.4) to[out=east, in=west] (.5,.3);
\draw[white, double=black, double distance=0.4pt, ultra thick] 
(0,-.1) to[out=east, in=west] (.5,.5);
\draw[white, double=black, double distance=0.4pt, ultra thick] 
(-.5,.5) to[out=east, in=west] (.0,.1);
\draw[white, double=black, double distance=0.4pt, ultra thick] 
(0,.1) to[out=east, in=west] (.5,-.4);
\draw[fill=white] (.5,-.6) rectangle +(.1,.4);
\draw[fill=white] (-.5,-.6) rectangle +(-.1,.4);
\draw[fill=white] (.5,.6) rectangle +(.1,-.4);
\draw[fill=white] (-.5,.6) rectangle +(-.1,-.4);
\node at (.5,-.6)[right]{$\scriptstyle{1}$};
\node at (-.5,-.6)[left]{$\scriptstyle{1}$};
\node at (.5,.6)[right]{$\scriptstyle{n}$};
\node at (-.5,.6)[left]{$\scriptstyle{n}$};
\node at (.0,-.4)[below]{$\scriptstyle{n-1}$};
}\,\Bigg\rangle_{\! 2}
+q^{-\frac{1}{4}}
\Bigg\langle\,\tikz[baseline=-.6ex]{
\draw 
(-.5,.4) -- +(-.2,0)
(.5,-.4) -- +(.2,0)
(-.5,-.4) -- +(-.2,0)
(.5,.4) -- +(.2,0);
\draw[white, double=black, double distance=0.4pt, ultra thick] 
(.5,.5) to[out=west, in=east] 
(-.1,.0) to[out=west, in=west] (-.1,.4);
\draw[white, double=black, double distance=0.4pt, ultra thick] 
(-.1,.4) 
to[out=east, in=west] (.5,-.4);
\draw[white, double=black, double distance=0.4pt, ultra thick] 
(-.5,.5) 
to[out=east, in=north] (-.3,.0)
to[out=south, in=east] (-.5,-.4);
\draw[white, double=black, double distance=0.4pt, ultra thick] 
(-.5,.3) to[out=east, in=west] (.0,-.4);
\draw[white, double=black, double distance=0.4pt, ultra thick] 
(0,-.4) to[out=east, in=west] (.5,.3);
\draw[fill=white] (.5,-.6) rectangle +(.1,.4);
\draw[fill=white] (-.5,-.6) rectangle +(-.1,.4);
\draw[fill=white] (.5,.6) rectangle +(.1,-.4);
\draw[fill=white] (-.5,.6) rectangle +(-.1,-.4);
\node at (.5,-.6)[right]{$\scriptstyle{1}$};
\node at (-.5,-.6)[left]{$\scriptstyle{1}$};
\node at (.5,.6)[right]{$\scriptstyle{n}$};
\node at (-.5,.6)[left]{$\scriptstyle{n}$};
\node at (.0,-.4)[below]{$\scriptstyle{n-1}$};
}\,\Bigg\rangle_{\! 2}\\
&=q^{\frac{1}{2}}
\Bigg\langle\,\tikz[baseline=-.6ex]{
\draw 
(-.5,.4) -- +(-.2,0)
(.5,-.4) -- +(.2,0)
(-.5,-.4) -- +(-.2,0)
(.5,.4) -- +(.2,0);
\draw (-.5,.5) -- (.5,.5);
\draw[white, double=black, double distance=0.4pt, ultra thick] 
(-.5,-.4) to[out=east, in=west] (.0,.3);
\draw[white, double=black, double distance=0.4pt, ultra thick] 
(-.5,.3) to[out=east, in=west] (.0,-.3);
\draw[white, double=black, double distance=0.4pt, ultra thick] 
(0,-.3) to[out=east, in=west] (.5,.3);
\draw[white, double=black, double distance=0.4pt, ultra thick] 
(0,.3) to[out=east, in=west] (.5,-.4);
\draw[fill=white] (.5,-.6) rectangle +(.1,.4);
\draw[fill=white] (-.5,-.6) rectangle +(-.1,.4);
\draw[fill=white] (.5,.6) rectangle +(.1,-.4);
\draw[fill=white] (-.5,.6) rectangle +(-.1,-.4);
\node at (.5,-.6)[right]{$\scriptstyle{1}$};
\node at (-.5,-.6)[left]{$\scriptstyle{1}$};
\node at (.5,.6)[right]{$\scriptstyle{n}$};
\node at (-.5,.6)[left]{$\scriptstyle{n}$};
\node at (.0,.5)[above]{$\scriptstyle{1}$};
\node at (.0,-.3)[below]{$\scriptstyle{n-1}$};
}\,\Bigg\rangle_{\! 2}
+(1-q^{-1})q^{-\frac{1}{2}(n-1)}
\Bigg\langle\,\tikz[baseline=-.6ex]{
\draw 
(-.5,.4) -- +(-.2,0)
(.5,-.4) -- +(.2,0)
(-.5,-.4) -- +(-.2,0)
(.5,.4) -- +(.2,0);
\draw (-.5,.5) -- (.5,.5);
\draw[white, double=black, double distance=0.4pt, ultra thick] 
(-.5,-.4) to[out=east, in=south] 
(-.3,.0) to[out=north, in=east] (-.5,.3);
\draw[white, double=black, double distance=0.4pt, ultra thick] 
(.5,.3) to[out=west, in=north] (.3,.0)
to[out=south, in=west] (.5,-.4);
\draw[fill=white] (.5,-.6) rectangle +(.1,.4);
\draw[fill=white] (-.5,-.6) rectangle +(-.1,.4);
\draw[fill=white] (.5,.6) rectangle +(.1,-.4);
\draw[fill=white] (-.5,.6) rectangle +(-.1,-.4);
\node at (.5,-.6)[right]{$\scriptstyle{1}$};
\node at (-.5,-.6)[left]{$\scriptstyle{1}$};
\node at (.5,.6)[right]{$\scriptstyle{n}$};
\node at (-.5,.6)[left]{$\scriptstyle{n}$};
\node at (.0,.5)[above]{$\scriptstyle{n-1}$};
\node at (-.3,.0)[left]{$\scriptstyle{1}$};
\node at (.3,.0)[right]{$\scriptstyle{1}$};
}\,\Bigg\rangle_{\! 2}.
\end{align*}

This equation implies
\begin{equation}\label{A1n1full}
q^{-\frac{1}{2}(n-i-1)}
\Bigg\langle\,\tikz[baseline=-.6ex]{
\draw 
(-.5,.4) -- +(-.2,0)
(.5,-.4) -- +(.2,0)
(-.5,-.4) -- +(-.2,0)
(.5,.4) -- +(.2,0);
\draw (-.5,.5) -- (.5,.5);
\draw[white, double=black, double distance=0.4pt, ultra thick] 
(-.5,-.4) to[out=east, in=west] (.0,.3);
\draw[white, double=black, double distance=0.4pt, ultra thick] 
(-.5,.3) to[out=east, in=west] (.0,-.3);
\draw[white, double=black, double distance=0.4pt, ultra thick] 
(0,-.3) to[out=east, in=west] (.5,.3);
\draw[white, double=black, double distance=0.4pt, ultra thick] 
(0,.3) to[out=east, in=west] (.5,-.4);
\draw[fill=white] (.5,-.6) rectangle +(.1,.4);
\draw[fill=white] (-.5,-.6) rectangle +(-.1,.4);
\draw[fill=white] (.5,.6) rectangle +(.1,-.4);
\draw[fill=white] (-.5,.6) rectangle +(-.1,-.4);
\node at (.5,-.6)[right]{$\scriptstyle{1}$};
\node at (-.5,-.6)[left]{$\scriptstyle{1}$};
\node at (.5,.6)[right]{$\scriptstyle{n}$};
\node at (-.5,.6)[left]{$\scriptstyle{n}$};
\node at (.0,.5)[above]{$\scriptstyle{i+1}$};
\node at (.0,-.3)[below]{$\scriptstyle{n-i-1}$};
}\,\Bigg\rangle_{\! 2}
-q^{-\frac{1}{2}(n-i)}
\Bigg\langle\,\tikz[baseline=-.6ex]{
\draw 
(-.5,.4) -- +(-.2,0)
(.5,-.4) -- +(.2,0)
(-.5,-.4) -- +(-.2,0)
(.5,.4) -- +(.2,0);
\draw (-.5,.5) -- (.5,.5);
\draw[white, double=black, double distance=0.4pt, ultra thick] 
(-.5,-.4) to[out=east, in=west] (.0,.3);
\draw[white, double=black, double distance=0.4pt, ultra thick] 
(-.5,.3) to[out=east, in=west] (.0,-.3);
\draw[white, double=black, double distance=0.4pt, ultra thick] 
(0,-.3) to[out=east, in=west] (.5,.3);
\draw[white, double=black, double distance=0.4pt, ultra thick] 
(0,.3) to[out=east, in=west] (.5,-.4);
\draw[fill=white] (.5,-.6) rectangle +(.1,.4);
\draw[fill=white] (-.5,-.6) rectangle +(-.1,.4);
\draw[fill=white] (.5,.6) rectangle +(.1,-.4);
\draw[fill=white] (-.5,.6) rectangle +(-.1,-.4);
\node at (.5,-.6)[right]{$\scriptstyle{1}$};
\node at (-.5,-.6)[left]{$\scriptstyle{1}$};
\node at (.5,.6)[right]{$\scriptstyle{n}$};
\node at (-.5,.6)[left]{$\scriptstyle{n}$};
\node at (.0,.5)[above]{$\scriptstyle{i}$};
\node at (.0,-.3)[below]{$\scriptstyle{n-i}$};
}\,\Bigg\rangle_{\! 2}
=(1-q)q^{-n-\frac{1}{2}}q^{i}
\Bigg\langle\,\tikz[baseline=-.6ex]{
\draw 
(-.5,.4) -- +(-.2,0)
(.5,-.4) -- +(.2,0)
(-.5,-.4) -- +(-.2,0)
(.5,.4) -- +(.2,0);
\draw (-.5,.5) -- (.5,.5);
\draw[white, double=black, double distance=0.4pt, ultra thick] 
(-.5,-.4) to[out=east, in=south] 
(-.3,.0) to[out=north, in=east] (-.5,.3);
\draw[white, double=black, double distance=0.4pt, ultra thick] 
(.5,.3) to[out=west, in=north] (.3,.0)
to[out=south, in=west] (.5,-.4);
\draw[fill=white] (.5,-.6) rectangle +(.1,.4);
\draw[fill=white] (-.5,-.6) rectangle +(-.1,.4);
\draw[fill=white] (.5,.6) rectangle +(.1,-.4);
\draw[fill=white] (-.5,.6) rectangle +(-.1,-.4);
\node at (.5,-.6)[right]{$\scriptstyle{1}$};
\node at (-.5,-.6)[left]{$\scriptstyle{1}$};
\node at (.5,.6)[right]{$\scriptstyle{n}$};
\node at (-.5,.6)[left]{$\scriptstyle{n}$};
\node at (.0,.5)[above]{$\scriptstyle{n-1}$};
\node at (-.3,.0)[left]{$\scriptstyle{1}$};
\node at (.3,.0)[right]{$\scriptstyle{1}$};
}\,\Bigg\rangle_{\! 2}
\end{equation}
for $i=0,1,\dots,n-1$.
Thus, 
we obtain
\[
\Bigg\langle\,\tikz[baseline=-.6ex]{
\draw 
(-.5,.4) -- +(-.2,0)
(.5,-.4) -- +(.2,0)
(-.5,-.4) -- +(-.2,0)
(.5,.4) -- +(.2,0);
\draw[white, double=black, double distance=0.4pt, ultra thick] 
(-.5,-.4) to[out=east, in=west] (.0,.3);
\draw[white, double=black, double distance=0.4pt, ultra thick] 
(-.5,.4) to[out=east, in=west] (.0,-.3);
\draw[white, double=black, double distance=0.4pt, ultra thick] 
(0,-.3) to[out=east, in=west] (.5,.4);
\draw[white, double=black, double distance=0.4pt, ultra thick] 
(0,.3) to[out=east, in=west] (.5,-.4);
\draw[fill=white] (.5,-.6) rectangle +(.1,.4);
\draw[fill=white] (-.5,-.6) rectangle +(-.1,.4);
\draw[fill=white] (.5,.6) rectangle +(.1,-.4);
\draw[fill=white] (-.5,.6) rectangle +(-.1,-.4);
\node at (.5,-.6)[right]{$\scriptstyle{1}$};
\node at (-.5,-.6)[left]{$\scriptstyle{1}$};
\node at (.5,.6)[right]{$\scriptstyle{n}$};
\node at (-.5,.6)[left]{$\scriptstyle{n}$};
}\,\Bigg\rangle_{\! 2}
=q^{\frac{n}{2}}
\Bigg\langle\,\tikz[baseline=-.6ex]{
\draw 
(-.5,.4) -- +(-.2,0)
(.5,-.4) -- +(.2,0)
(-.5,-.4) -- +(-.2,0)
(.5,.4) -- +(.2,0);
\draw (-.5,.4) -- (.5,.4);
\draw (-.5,-.4) -- (.5,-.4);
\draw[fill=white] (.5,-.6) rectangle +(.1,.4);
\draw[fill=white] (-.5,-.6) rectangle +(-.1,.4);
\draw[fill=white] (.5,.6) rectangle +(.1,-.4);
\draw[fill=white] (-.5,.6) rectangle +(-.1,-.4);
\node at (.5,-.6)[right]{$\scriptstyle{1}$};
\node at (-.5,-.6)[left]{$\scriptstyle{1}$};
\node at (.5,.6)[right]{$\scriptstyle{n}$};
\node at (-.5,.6)[left]{$\scriptstyle{n}$};
}\,\Bigg\rangle_{\! 2}
-(1-q^n)q^{-\frac{n+1}{2}}
\Bigg\langle\,\tikz[baseline=-.6ex]{
\draw 
(-.5,.4) -- +(-.2,0)
(.5,-.4) -- +(.2,0)
(-.5,-.4) -- +(-.2,0)
(.5,.4) -- +(.2,0);
\draw (-.5,.5) -- (.5,.5);
\draw[white, double=black, double distance=0.4pt, ultra thick] 
(-.5,-.4) to[out=east, in=south] 
(-.3,.0) to[out=north, in=east] (-.5,.3);
\draw[white, double=black, double distance=0.4pt, ultra thick] 
(.5,.3) to[out=west, in=north] (.3,.0)
to[out=south, in=west] (.5,-.4);
\draw[fill=white] (.5,-.6) rectangle +(.1,.4);
\draw[fill=white] (-.5,-.6) rectangle +(-.1,.4);
\draw[fill=white] (.5,.6) rectangle +(.1,-.4);
\draw[fill=white] (-.5,.6) rectangle +(-.1,-.4);
\node at (.5,-.6)[right]{$\scriptstyle{1}$};
\node at (-.5,-.6)[left]{$\scriptstyle{1}$};
\node at (.5,.6)[right]{$\scriptstyle{n}$};
\node at (-.5,.6)[left]{$\scriptstyle{n}$};
\node at (.0,.5)[above]{$\scriptstyle{n-1}$};
\node at (-.3,.0)[left]{$\scriptstyle{1}$};
\node at (.3,.0)[right]{$\scriptstyle{1}$};
}\,\Bigg\rangle_{\! 2}
\]
by taking the sum of both sides of (\ref{A1n1full}) for $i=0,1,\dots,n-1$.
Also, 
\begin{align*}
\Bigg\langle\,\tikz[baseline=-.6ex]{
\draw 
(-.5,.4) -- +(-.2,0)
(.5,-.4) -- +(.2,0)
(-.5,-.4) -- +(-.2,0)
(.5,.4) -- +(.2,0);
\draw[triple={[line width=1.6pt, white] in [line width=2.4pt, black] in [line width=5.6pt, white]}] 
(-.5,-.4) to[out=east, in=west] (.0,.4);
\draw[white, double=black, double distance=0.4pt, ultra thick] 
(-.5,.4) to[out=east, in=west] (.0,-.4);
\draw[white, double=black, double distance=0.4pt, ultra thick] 
(0,-.4) to[out=east, in=west] (.5,.4);
\draw[triple={[line width=1.6pt, white] in [line width=2.4pt, black] in [line width=5.6pt, white]}] 
(0,.4) to[out=east, in=west] (.5,-.4);
\draw[fill=white] (.5,-.6) rectangle +(.1,.4);
\draw[fill=white] (-.5,-.6) rectangle +(-.1,.4);
\draw[fill=white] (.5,.6) rectangle +(.1,-.4);
\draw[fill=white] (-.5,.6) rectangle +(-.1,-.4);
\node at (.5,-.6)[right]{$\scriptstyle{i}$};
\node at (-.5,-.6)[left]{$\scriptstyle{i}$};
\node at (.5,.6)[right]{$\scriptstyle{j}$};
\node at (-.5,.6)[left]{$\scriptstyle{j}$};
\node at (.0,.4)[above]{$\scriptstyle{i-1}$};
\node at (.0,.4)[below]{$\scriptstyle{1}$};
}\,\Bigg\rangle_{\! 2}
=
q^{\frac{j}{2}}\Bigg\langle\,\tikz[baseline=-.6ex]{
\draw 
(-.5,.4) -- +(-.2,0)
(.5,-.4) -- +(.2,0)
(-.5,-.4) -- +(-.2,0)
(.5,.4) -- +(.2,0);
\draw (-.5,-.5) -- (.5,-.5);
\draw[white, double=black, double distance=0.4pt, ultra thick] 
(-.5,-.3) to[out=east, in=west] (.0,.4);
\draw[white, double=black, double distance=0.4pt, ultra thick] 
(-.5,.4) to[out=east, in=west] (.0,-.4);
\draw[white, double=black, double distance=0.4pt, ultra thick] 
(0,-.4) to[out=east, in=west] (.5,.4);
\draw[white, double=black, double distance=0.4pt, ultra thick] 
(0,.4) to[out=east, in=west] (.5,-.3);
\draw[fill=white] (.5,-.6) rectangle +(.1,.4);
\draw[fill=white] (-.5,-.6) rectangle +(-.1,.4);
\draw[fill=white] (.5,.6) rectangle +(.1,-.4);
\draw[fill=white] (-.5,.6) rectangle +(-.1,-.4);
\node at (.5,-.6)[right]{$\scriptstyle{i}$};
\node at (-.5,-.6)[left]{$\scriptstyle{i}$};
\node at (.5,.6)[right]{$\scriptstyle{j}$};
\node at (-.5,.6)[left]{$\scriptstyle{j}$};
\node at (.0,.4)[above]{$\scriptstyle{i-1}$};
\node at (.0,-.5)[below]{$\scriptstyle{1}$};
}\,\Bigg\rangle_{\! 2}
-(1-q^j)q^{-\frac{i+j}{2}}
\Bigg\langle\,\tikz[baseline=-.6ex]{
\draw 
(-.6,.4) -- +(-.2,0)
(.6,-.4) -- +(.2,0)
(-.6,-.4) -- +(-.2,0)
(.6,.4) -- +(.2,0);
\draw[white, double=black, double distance=0.4pt, ultra thick] 
(-.6,-.5) to[out=east, in=west] (.0,.4);
\draw[white, double=black, double distance=0.4pt, ultra thick] 
(-.6,.5) to[out=east, in=west] (.0,-.4);
\draw[white, double=black, double distance=0.4pt, ultra thick] 
(0,-.4) to[out=east, in=west] (.6,.5);
\draw[white, double=black, double distance=0.4pt, ultra thick] 
(0,.4) to[out=east, in=west] (.6,-.5);
\draw (-.6,.3) to[out=east, in=north] 
(-.4,.0) to[out=south, in=east] (-.6,-.3);
\draw (.6,.3) to[out=west, in=north] 
(.4,.0) to[out=south, in=west](.6,-.3);
\draw[fill=white] (.6,-.6) rectangle +(.1,.4);
\draw[fill=white] (-.6,-.6) rectangle +(-.1,.4);
\draw[fill=white] (.6,.6) rectangle +(.1,-.4);
\draw[fill=white] (-.6,.6) rectangle +(-.1,-.4);
\node at (.6,-.6)[right]{$\scriptstyle{i}$};
\node at (-.6,-.6)[left]{$\scriptstyle{i}$};
\node at (.6,.6)[right]{$\scriptstyle{j}$};
\node at (-.6,.6)[left]{$\scriptstyle{j}$};
\node at (.0,.4)[above]{$\scriptstyle{i-1}$};
\node at (.0,-.4)[below]{$\scriptstyle{j-1}$};
\node at (-.4,.0)[left]{$\scriptstyle{1}$};
\node at (.4,.0)[right]{$\scriptstyle{1}$};
}\,\Bigg\rangle_{\! 2}
\end{align*}
for any non-negative integers $i$ and $j$.
We define a clasped $A_1$ web $\langle\sigma^2(k,l;n)\rangle_2$ as follows:
\[
\langle\sigma^2(k,l;n)\rangle_2=
\Bigg\langle\,\tikz[baseline=-.6ex]{
\draw 
(-.6,.4) -- +(-.2,0)
(.6,-.4) -- +(.2,0)
(-.6,-.4) -- +(-.2,0)
(.6,.4) -- +(.2,0);
\draw
(-.6,-.5) -- (.6,-.5);
\draw[white, double=black, double distance=0.4pt, ultra thick] 
(-.6,-.4) to[out=east, in=west] (.0,.4);
\draw[white, double=black, double distance=0.4pt, ultra thick] 
(-.6,.5) to[out=east, in=west] (.0,-.3);
\draw[white, double=black, double distance=0.4pt, ultra thick] 
(0,-.3) to[out=east, in=west] (.6,.5);
\draw[white, double=black, double distance=0.4pt, ultra thick] 
(0,.4) to[out=east, in=west] (.6,-.4);
\draw (-.6,.3) to[out=east, in=north] 
(-.5,.0) to[out=south, in=east] (-.6,-.3);
\draw (.6,.3) to[out=west, in=north] 
(.5,.0) to[out=south, in=west](.6,-.3);
\draw[fill=white] (.6,-.6) rectangle +(.1,.4);
\draw[fill=white] (-.6,-.6) rectangle +(-.1,.4);
\draw[fill=white] (.6,.6) rectangle +(.1,-.4);
\draw[fill=white] (-.6,.6) rectangle +(-.1,-.4);
\node at (.6,-.6)[right]{$\scriptstyle{n}$};
\node at (-.6,-.6)[left]{$\scriptstyle{n}$};
\node at (.6,.6)[right]{$\scriptstyle{n}$};
\node at (-.6,.6)[left]{$\scriptstyle{n}$};
\node at (.0,.4)[above]{$\scriptstyle{n-k-l}$};
\node at (.0,-.6)[above]{$\scriptstyle{n-l}$};
\node at (.0,-.5)[below]{$\scriptstyle{k}$};
\node at (-.5,.0)[left]{$\scriptstyle{l}$};
\node at (.5,.0)[right]{$\scriptstyle{l}$};
}\,\Bigg\rangle_{\! 2}.
\]
Then,
we obtain
\begin{equation}\label{A1fullres}
\langle\sigma^2(k,l;n)\rangle_2
=q^{\frac{n-l}{2}}\langle\sigma^2(k+1,l;n)\rangle_2
-(1-q^{n-l})q^{-(n-l)}q^{\frac{k}{2}}\langle\sigma^2(k,l+1;n)\rangle_2
\end{equation}
for non-negative integers $k$ and $l$ such that $k+l\leq n$.
In the same way as the proof of Proposition~\ref{coloredA1}~$(1)$, 
we make $\langle\sigma^2(k,l;n)\rangle_2$ correspond to a lattice point $(k,l)$ and decorate edges with coefficients by using the resolution $(\ref{A1fullres})$.
If a path from $(0,0)$ to $(k,l)$ with $k+l=n$ shifts downward by a box, 
then the product of decorations is multiplied by $q$ (see Figure~\ref{A1fullshift}).
\begin{figure}
\centering
\begin{tikzpicture}
\draw[->-=.5] (0,0) -- (2,0);
\node at (1,0)[below]{$\scriptstyle{q^{\frac{n-l}{2}}}$};
\draw[->-=.5] (2,0) -- (2,2);
\node at (2,1)[right]{$\scriptstyle{-(1-q^{n-l})q^{-(n-l)}q^{\frac{k+1}{2}}}$};
\draw[->-=.5] (0,0) -- (0,2);
\node at (0,1)[left]{$\scriptstyle{-(1-q^{n-l})q^{-(n-l)}q^{\frac{k}{2}}}$};
\draw[->-=.5] (0,2) -- (2,2);
\draw[->, very thick, magenta](.5,1.5) -- (1.5,.5);
\node at (1,1)[right]{$\times q$};
\node at (1,2)[above]{$\scriptstyle{q^{\frac{n-(l+1)}{2}}}$};
\node at (0,0)[below left]{$\scriptstyle{(k,l)}$};
\node at (2,0)[below right]{$\scriptstyle{(k+1,l)}$};
\node at (0,2)[above left]{$\scriptstyle{(k,l+1)}$};
\node at (2,2)[above right]{$\scriptstyle{(k,l)}$};
\fill 
(0,0) circle (1.2pt)
(2,0) circle (1.2pt)
(2,2) circle (1.2pt)
(0,2) circle (1.2pt);
\end{tikzpicture}
\caption{shift of a path}
\label{A1fullshift}
\end{figure}
Therefore,
the coefficient of $\langle\sigma^2(k,l;n)\rangle_2$ with $k+l=n$ for the expansion of a full twist is 
\[
\prod_{i=0}^{l-1}-(1-q^{n-l})q^{-(n-i)}\prod_{j=0}^{k-1}q^{\frac{n-l}{2}}{n\choose k}_q=(-1)^lq^{\frac{2k^2-n^2+k-n}{2}}\frac{(q;q)_n}{(q;q)_{n-l}}{n\choose k}_q.
\]
\end{proof}

We next give expansions of $m$ half twists and $m$ full twists of two strands colored by $n$.
\begin{LEM}[Masbaum~{\cite[Lemma~$4.3$]{Masbaum03}}]\label{A1slide}
For any non-negative integers $k\neq n$,
\[
\Bigg\langle\,\tikz[baseline=-.6ex]{
\draw 
(-.5,.4) -- +(-.2,0)
(.0,-.5) -- +(.4,0)
(-.5,-.4) -- +(-.2,0)
(.0,.5) -- +(.4,0);
\draw[white, double=black, double distance=0.4pt, ultra thick] 
(-.5,-.4) to[out=east, in=west] (.0,.4);
\draw[white, double=black, double distance=0.4pt, ultra thick] 
(-.5,.4) to[out=east, in=west] (.0,-.4);
\draw (.1,.4) to[out=east, in=east] (.1,-.4);
\draw[fill=white] (.0,-.6) rectangle +(.1,.4);
\draw[fill=white] (-.5,-.6) rectangle +(-.1,.4);
\draw[fill=white] (.0,.6) rectangle +(.1,-.4);
\draw[fill=white] (-.5,.6) rectangle +(-.1,-.4);
\node at (.4,-.6)[right]{$\scriptstyle{k}$};
\node at (-.5,-.6)[left]{$\scriptstyle{n}$};
\node at (.4,.6)[right]{$\scriptstyle{k}$};
\node at (-.5,.6)[left]{$\scriptstyle{n}$};
\node at (.3,0)[right]{$\scriptstyle{n-k}$};
\node at (-.1,.4)[above]{$\scriptstyle{n}$};
\node at (-.1,-.4)[below]{$\scriptstyle{n}$};
}\,\Bigg\rangle_{\! 2}
=(-1)^{n-k}q^{-\frac{n^2-k^2+2n-2k}{4}}
\Bigg\langle\,\tikz[baseline=-.6ex]{
\draw[white, double=black, double distance=0.4pt, ultra thick] 
(-.4,-.4) -- +(-.2,0)
(-.4,-.5) to[out=east, in=west] (.4,.4)
(.4,.4) -- +(.2,0);
\draw[white, double=black, double distance=0.4pt, ultra thick] 
(-.4,.4) -- +(-.2,0)
(-.4,.5) to[out=east, in=west] (.4,-.4)
(.4,-.4) -- +(.2,0);
\draw (-.4,.3) to[out=east, in=east] (-.4,-.3);
\draw[fill=white] (.4,-.6) rectangle +(.1,.4);
\draw[fill=white] (-.4,-.6) rectangle +(-.1,.4);
\draw[fill=white] (.4,.6) rectangle +(.1,-.4);
\draw[fill=white] (-.4,.6) rectangle +(-.1,-.4);
\node at (.4,-.6)[right]{$\scriptstyle{k}$};
\node at (-.4,-.6)[left]{$\scriptstyle{n}$};
\node at (.4,.6)[right]{$\scriptstyle{k}$};
\node at (-.4,.6)[left]{$\scriptstyle{n}$};
\node at (-.4,0)[left]{$\scriptstyle{n-k}$};
}\,\Bigg\rangle_{\! 2}
\]
\end{LEM}
\begin{proof}
Slide two right-side $A_1$ clasps to the left and calculate resulting clasped $A_1$ web using Lemma~\ref{A1clasplem}.
\end{proof}

Let $n,m$ be non-negative integers.
\begin{PROP}[$m$ half twists formula]\label{A1mhalf}
\begin{align*}
&\Bigg\langle\,\tikz[baseline=-.6ex]{
\begin{scope}[xshift=-.5cm]
\draw 
(-.5,.4) -- +(-.2,0)
(-.5,-.4) -- +(-.2,0);
\draw[white, double=black, double distance=0.4pt, ultra thick] 
(-.5,-.4) to[out=east, in=west] (.0,.4);
\draw[white, double=black, double distance=0.4pt, ultra thick] 
(-.5,.4) to[out=east, in=west] (.0,-.4);
\draw[fill=white] (-.5,-.6) rectangle +(-.1,.4);
\draw[fill=white] (-.5,.6) rectangle +(-.1,-.4);
\node at (-.5,-.6)[left]{$\scriptstyle{n}$};
\node at (-.5,.6)[left]{$\scriptstyle{n}$};
\end{scope}
\node at (.0,.0){$\cdots$};
\node at (.0,-.4)[below]{$\scriptstyle{m\text{ half twists}}$};
\begin{scope}[xshift=.5cm]
\draw
(.5,-.4) -- +(.2,0)
(.5,.4) -- +(.2,0);
\draw[white, double=black, double distance=0.4pt, ultra thick] 
(0,-.4) to[out=east, in=west] (.5,.4);
\draw[white, double=black, double distance=0.4pt, ultra thick] 
(0,.4) to[out=east, in=west] (.5,-.4);
\draw[fill=white] (.5,-.6) rectangle +(.1,.4);
\draw[fill=white] (.5,.6) rectangle +(.1,-.4);
\node at (.5,-.6)[right]{$\scriptstyle{n}$};
\node at (.5,.6)[right]{$\scriptstyle{n}$};
\end{scope}
}\,\Bigg\rangle_{\! 2}
=(-1)^{mn}q^{-\frac{m}{4}(n^2+2n)}\\
&\times\!\sum_{0\leq k_m\leq \cdots\leq k_1\leq n}
(-1)^{n-k_m}q^{\frac{n-k_m}{2}}
(-1)^{\sum_{i=1}^mk_i}q^{\frac{1}{2}\sum_{i=1}^{m}(k_i^2+k_i)}
{n \choose k_1',k_2',\dots,k_m',k_m}_{q}
\Bigg\langle\,\tikz[baseline=-.6ex]{
\draw 
(-.4,.4) -- +(-.2,0)
(.4,-.4) -- +(.2,0)
(-.4,-.4) -- +(-.2,0)
(.4,.4) -- +(.2,0)
(-.4,.5) -- (.4,.5)
(-.4,-.5) -- (.4,-.5)
(-.4,.3) to[out=east, in=east] (-.4,-.3)
(.4,.3) to[out=west, in=west] (.4,-.3);
\draw[fill=white] (.4,-.6) rectangle +(.1,.4);
\draw[fill=white] (-.4,-.6) rectangle +(-.1,.4);
\draw[fill=white] (.4,.6) rectangle +(.1,-.4);
\draw[fill=white] (-.4,.6) rectangle +(-.1,-.4);
\node at (.4,-.6)[right]{$\scriptstyle{n}$};
\node at (-.4,-.6)[left]{$\scriptstyle{n}$};
\node at (.4,.6)[right]{$\scriptstyle{n}$};
\node at (-.4,.6)[left]{$\scriptstyle{n}$};
\node at (0,.6){$\scriptstyle{k_m}$};
\node at (0,-.6){$\scriptstyle{k_m}$};
\node at (-.2,0)[left]{$\scriptstyle{n-k_m}$};
\node at (.2,0)[right]{$\scriptstyle{n-k_m}$};
}\,\Bigg\rangle_{\! 2},
\end{align*}
where $k_i, k_i'$ are integers such that $k_0=n$, $k_{i+1}'=k_i-k_{i+1}$ for $i=0,1,\dots,m-1$.
\end{PROP}

\begin{proof}
Let $k_i, k_{i+1}$ as above.
We define a clasped $A_1$ web $\langle\sigma^{m-i}(k_i;n)\rangle_2$ for any $k_i=0,1,\dots,n$ as
\[
\langle\sigma^{m-i}(k_i;n)\rangle_2
=
\Bigg\langle\,\tikz[baseline=-.6ex]{
\begin{scope}[xshift=-.7cm]
\draw 
(-.6,.4) -- +(-.2,0)
(-.6,-.4) -- +(-.2,0);
\draw 
(-.6,.3) to[out=east, in=north] 
(-.5,.0) to[out=south, in=east] (-.6,-.3);
\draw[white, double=black, double distance=0.4pt, ultra thick] 
(-.6,-.5) to[out=east, in=west] (.0,.5);
\draw[white, double=black, double distance=0.4pt, ultra thick] 
(-.6,.5) to[out=east, in=west] (.0,-.5);
\draw[fill=white] (-.6,-.6) rectangle +(-.1,.4);
\draw[fill=white] (-.6,.6) rectangle +(-.1,-.4);
\node at (-.6,-.6)[left]{$\scriptstyle{n}$};
\node at (-.6,.6)[left]{$\scriptstyle{n}$};
\node at (-.5,.0)[left]{$\scriptstyle{n-k_i}$};
\end{scope}
\node at (.0,.0){$\cdots$};
\node at (.0,-.5)[below]{$\scriptstyle{m-i\text{ half twists}}$};
\begin{scope}[xshift=.7cm]
\draw
(.6,-.4) -- +(.2,0)
(.6,.4) -- +(.2,0);
\draw 
(.6,.3) to[out=west, in=north] 
(.5,.0) to[out=south, in=west] (.6,-.3);
\draw[white, double=black, double distance=0.4pt, ultra thick] 
(0,-.5) to[out=east, in=west] (.6,.5);
\draw[white, double=black, double distance=0.4pt, ultra thick] 
(0,.5) to[out=east, in=west] (.6,-.5);
\draw[fill=white] (.6,-.6) rectangle +(.1,.4);
\draw[fill=white] (.6,.6) rectangle +(.1,-.4);
\node at (.6,-.6)[right]{$\scriptstyle{n}$};
\node at (.6,.6)[right]{$\scriptstyle{n}$};
\node at (.5,.0)[right]{$\scriptstyle{n-k_i}$};
\end{scope}
}\,\Bigg\rangle_{\! 2}.
\]
We apply Proposition~\ref{coloredA1}~$(1)$ to a rightmost half twist and use Lemma~\ref{A1slide} $m-i-1$ times.
Thus,
we obtain
\begin{align*}
&\langle\sigma^{m-i}(k_i;n)\rangle_2\\
&\quad=
\sum_{k_{i+1}=0}^{k_i}
(-1)^{k_i-k_{i+1}}q^{\frac{1}{4}(-k_i^2+2k_{i+1}^2)}q^{\frac{m-i-1}{4}(k_{i+1}^2-k_i^2+2k_{i+1}-2k_i)}
{k_i \choose k_{i+1}}_q
\langle\sigma^{m-i}(k_i;n)\rangle_2 .\notag
\end{align*}
The right hand side of the $m$ half twist formula is $\langle\sigma^{m}(k_0;n)\rangle_2$.
Therefore,
we can obtain the $m$ half twist formula by using the above equation for $i=0,1,\dots,m-1$ in turn and calculation of the exponent sum of $q$. 
\end{proof}

\begin{PROP}[$m$ full twists formula~\cite{Masbaum03}]\label{A1mfull}
\begin{align*}
\Bigg\langle\,\tikz[baseline=-.6ex]{
\begin{scope}[xshift=-1cm]
\draw 
(-.5,.4) -- +(-.2,0)
(-.5,-.4) -- +(-.2,0);
\draw[white, double=black, double distance=0.4pt, ultra thick] 
(-.5,-.4) to[out=east, in=west] (.0,.4);
\draw[white, double=black, double distance=0.4pt, ultra thick] 
(-.5,.4) to[out=east, in=west] (.0,-.4);
\draw[white, double=black, double distance=0.4pt, ultra thick] 
(0,-.4) to[out=east, in=west] (.5,.4);
\draw[white, double=black, double distance=0.4pt, ultra thick] 
(0,.4) to[out=east, in=west] (.5,-.4);
\draw[fill=white] (-.5,-.6) rectangle +(-.1,.4);
\draw[fill=white] (-.5,.6) rectangle +(-.1,-.4);
\node at (-.5,-.6)[left]{$\scriptstyle{n}$};
\node at (-.5,.6)[left]{$\scriptstyle{n}$};
\end{scope}
\node at (.0,.0){$\cdots$};
\node at (.0,-.4)[below]{$\scriptstyle{m\text{ full twists}}$};
\begin{scope}[xshift=1cm]
\draw
(.5,-.4) -- +(.2,0)
(.5,.4) -- +(.2,0);
\draw[white, double=black, double distance=0.4pt, ultra thick] 
(-.5,-.4) to[out=east, in=west] (.0,.4);
\draw[white, double=black, double distance=0.4pt, ultra thick] 
(-.5,.4) to[out=east, in=west] (.0,-.4);
\draw[white, double=black, double distance=0.4pt, ultra thick] 
(0,-.4) to[out=east, in=west] (.5,.4);
\draw[white, double=black, double distance=0.4pt, ultra thick] 
(0,.4) to[out=east, in=west] (.5,-.4);
\draw[fill=white] (.5,-.6) rectangle +(.1,.4);
\draw[fill=white] (.5,.6) rectangle +(.1,-.4);
\node at (.5,-.6)[right]{$\scriptstyle{n}$};
\node at (.5,.6)[right]{$\scriptstyle{n}$};
\end{scope}
}\,\Bigg\rangle_{\! 2}
=&q^{-\frac{m}{2}(n^2+2n)}
\sum_{0\leq k_m\leq \cdots\leq k_1\leq n}
(-1)^{n-k_m}q^{\frac{n-k_m}{2}}
q^{\sum_{i=1}^m(k_i^2+k_i)}\\
&\quad\times\frac{(q)_n}{(q)_{k_m}}
{n \choose k_1',k_2',\dots,k_m',k_m}_{q}
\Bigg\langle\,\tikz[baseline=-.6ex]{
\draw
(-.4,.4) -- +(-.2,0)
(.4,-.4) -- +(.2,0)
(-.4,-.4) -- +(-.2,0)
(.4,.4) -- +(.2,0)
(-.4,.5) -- (.4,.5)
(-.4,-.5) -- (.4,-.5)
(-.4,.3) to[out=east, in=east] (-.4,-.3)
(.4,.3) to[out=west, in=west] (.4,-.3);
\draw[fill=white] (.4,-.6) rectangle +(.1,.4);
\draw[fill=white] (-.4,-.6) rectangle +(-.1,.4);
\draw[fill=white] (.4,.6) rectangle +(.1,-.4);
\draw[fill=white] (-.4,.6) rectangle +(-.1,-.4);
\node at (.4,-.6)[right]{$\scriptstyle{n}$};
\node at (-.4,-.6)[left]{$\scriptstyle{n}$};
\node at (.4,.6)[right]{$\scriptstyle{n}$};
\node at (-.4,.6)[left]{$\scriptstyle{n}$};
\node at (0,.6){$\scriptstyle{k_m}$};
\node at (0,-.6){$\scriptstyle{k_m}$};
\node at (-.2,0)[left]{$\scriptstyle{n-k_m}$};
\node at (.2,0)[right]{$\scriptstyle{n-k_m}$};
}\,\Bigg\rangle_{\! 2}
\end{align*}
\end{PROP}
\begin{proof}
We can prove this formula by the same way as the proof of Proposition~\ref{A1mhalf}.
We only have to use Proposition~\ref{A1full} instead of Proposition~\ref{coloredA1}~$(1)$.
\end{proof}
\begin{RMK}
\ 
\begin{itemize}
\item Twist formulas in this section treat only right-handed twists.
Left-handed versions of twist formulas can be obtained by substituting $q^{-1}$ for $q$.
\item We can easily calculate twist formulas for an $n$-colored strand and $m$-colored strand by using twist formulas for two $n$-colored strands.
\end{itemize}
\end{RMK}
\subsection{Colored $A_2$ bracket skein relations}
Let us consider clasped $A_2$ web spaces.
We use the following graphical notations to represent certain $A_2$ webs.
\begin{DEF}
For positive integers $n$ and $m$, 
a {\em colored $4$-valent vertex}
\[
\,\tikz[baseline=-.6ex]{
\draw[->-=.1,->-=.9] (-.6,0) -- (.6,0);
\draw[->-=.1,->-=.9] (0,-.6) -- (0,.6);
\draw[fill=white] (-.15,-.3) rectangle (.15,.3);
\draw (-.15,.3) -- (.15,-.3);
\node at (-.4,0)[above]{$\scriptstyle{n}$};
\node at (.4,0)[above]{$\scriptstyle{n}$};
\node at (0,.5)[right]{$\scriptstyle{m}$};
\node at (0,-.5)[right]{$\scriptstyle{m}$};
}\,\in W_{n^{+}+m^{+}+n^{-}+m^{-}}
\]
is defined as follows:
$\,\tikz[baseline=-.6ex]{
\draw[->-=.1,->-=.9] (-.5,0) -- (.5,0);
\draw[->-=.1,->-=.9] (0,-.6) -- (0,.6);
\draw[fill=white] (-.1,-.3) rectangle (.1,.3);
\draw (-.1,.3) -- (.1,-.3);
\node at (-.3,0)[above]{$\scriptstyle{n}$}; 
\node at (.3,0)[above]{$\scriptstyle{n}$};
}\,
=
\,\tikz[baseline=-.6ex]{
\draw[->-=.5] (-.6,-.4) -- (.6,-.4);
\draw[->-=.6] (-.6,-.2) -- (.6,-.2);
\node[rotate=90] at (-.6,.1){$\scriptstyle{\cdots}$};
\node[rotate=90] at (.6,.1){$\scriptstyle{\cdots}$};
\draw[->-=.5] (-.6,.4) -- (.6,.4);
\draw[->-=.5] (-.4,-.6) -- (-.4,-.4);
\draw[->-=.5] (-.3,-.4) -- (-.3,-.2);
\draw[->-=.5] (-.2,-.2) -- (-.2,0);
\node at (0,.1){$\scriptstyle{\cdots}$};
\draw[->-=.5] (.3,.2) -- (.3,.4);
\draw[->-=.5] (.4,.4) -- (.4,.6);
}\,\in W_{n^{+}+1^{+}+n^{-}+1^{+}}$ for $m=1$,
$\,\tikz[baseline=-.6ex]{
\draw[->-=.1,->-=.9] (-.6,0) -- (.6,0);
\draw[->-=.1,->-=.9] (0,-.6) -- (0,.6);
\draw[fill=white] (-.15,-.3) rectangle (.15,.3);
\draw (-.15,.3) -- (.15,-.3);
\node at (-.4,0)[above]{$\scriptstyle{n}$};
\node at (.4,0)[above]{$\scriptstyle{n}$};
\node at (0,.5)[right]{$\scriptstyle{m}$};
\node at (0,-.5)[right]{$\scriptstyle{m}$};
}\,
=
\,\tikz[baseline=-.6ex]{
\draw[->-=.1] (-.6,0) -- (.9,0);
\draw[->-=.1,->-=.9] (0,-.6) -- (0,.6);
\draw[fill=white] (-.15,-.3) rectangle (.15,.3);
\draw (-.15,.3) -- (.15,-.3);
\node at (-.4,0)[above]{$\scriptstyle{n}$};
\node at (0,.5)[left]{$\scriptstyle{m-1}$};
\node at (0,-.5)[left]{$\scriptstyle{m-1}$};
\begin{scope}[xshift=.5cm]
\draw[->-=.1,->-=.9] (-.2,0) -- (.5,0);
\draw[->-=.1,->-=.9] (.1,-.6) -- (.1,.6);
\draw[fill=white] (0,-.3) rectangle (.2,.3);
\draw (0,.3) -- (.2,-.3);
\node at (-.2,0)[above]{$\scriptstyle{n}$}; 
\node at (.4,0)[above]{$\scriptstyle{n}$};
\end{scope}
}\,
$ for $m>1$.
We also define
$
\,\tikz[baseline=-.6ex]{
\draw[->-=.1,->-=.9] (-.6,0) -- (.6,0);
\draw[-<-=.1,-<-=.9] (0,-.6) -- (0,.6);
\draw[fill=white] (-.15,-.3) rectangle (.15,.3);
\draw (-.15,-.3) -- (.15,.3);
\node at (-.4,0)[above]{$\scriptstyle{n}$};
\node at (.4,0)[above]{$\scriptstyle{n}$};
\node at (0,.5)[right]{$\scriptstyle{m}$};
\node at (0,-.5)[right]{$\scriptstyle{m}$};
}\,\in W_{n^{+}+m^{-}+n^{-}+m^{+}}
$
in the same way.
\end{DEF}
\begin{DEF}
For positive integer $n$,
a {\em colored trivalent vertex}
\[
\,\tikz[baseline=-.6ex]{
\draw (30:.5) -- (0,0);
\draw (150:.5) -- (0,0);
\draw[-<-=.2] (270:.5) -- (0,0);
\draw[->-=.2] (30:.5)
to[out=30, in=west] +(.3,.1);
\draw[->-=.2] (150:.5)
to[out=150, in=east] +(-.3,.1);
\draw[fill=white] (-30:.5) -- (90:.5) -- (210:.5) -- cycle;
\node at (30:.5) [above]{$\scriptstyle{n}$};
\node at (150:.5) [above]{$\scriptstyle{n}$};
\node at (270:.5) [left]{$\scriptstyle{n}$};
}\,\in W_{n^{+}+n^{+}+n^{+}}
\]
is defined as follows:
$\,\tikz[baseline=-.6ex]{
\draw[-<-=.5] (30:.4) -- (0,0);
\draw[-<-=.5] (150:.4) -- (0,0);
\draw[-<-=.5] (270:.4) -- (0,0);
}\,$ for $n=1$, 
$
\,\tikz[baseline=-.6ex]{
\draw (30:.4) -- (0,0);
\draw (150:.4) -- (0,0);
\draw[-<-=.2] (270:.4) -- (0,0);
\draw[->-=.2] (30:.4)
to[out=30, in=west] +(.3,.1);
\draw[->-=.2] (150:.4)
to[out=150, in=east] +(-.3,.1);
\draw[fill=white] (-30:.4) -- (90:.4) -- (210:.4) -- cycle;
\node at (30:.4) [above]{$\scriptstyle{n}$};
\node at (150:.4) [above]{$\scriptstyle{n}$};
\node at (270:.4) [left]{$\scriptstyle{n}$};
}\,
=
\,\tikz[baseline=-.6ex, xscale=-1]{
\draw (30:.2) -- (0,0);
\draw[->-=.8] (0,0) to[out=150, in=east] (170:.5);
\draw[-<-=.2] (270:.3) -- (0,0);
\draw[->-=.2, ->-=.9] (30:.2) -- +(.9,0);
\draw[fill=white] (-30:.3) -- (90:.3) -- (210:.3) -- cycle;
\draw[-<-=.2] (.6,-.4) -- (.6,.4);
\draw[-<-=.4, ->-=.9] (-.5,.4) -- (1.0,.4);
\draw[fill=white] (.5,-.2) rectangle (.7,.3);
\draw (.5,-.2) -- (.7,.3);
\node at (1,0) [left]{$\scriptstyle{n-1}$};
\node at (150:.5) [below right]{$\scriptstyle{n-1}$};
\node at (270:.3) [below]{$\scriptstyle{n-1}$};
\node at (150:.5) [above right]{$\scriptstyle{1}$};
\node at (1,.4) [left]{$\scriptstyle{1}$};
}\,$
for $n>1$.
We also define 
$\,\tikz[baseline=-.6ex]{
\draw (30:.4) -- (0,0);
\draw (150:.4) -- (0,0);
\draw[->-=.2] (270:.4) -- (0,0);
\draw[-<-=.2] (30:.4)
to[out=30, in=west] +(.3,.1);
\draw[-<-=.2] (150:.4)
to[out=150, in=east] +(-.3,.1);
\draw[fill=white] (-30:.4) -- (90:.4) -- (210:.4) -- cycle;
\node at (30:.4) [above]{$\scriptstyle{n}$};
\node at (150:.4) [above]{$\scriptstyle{n}$};
\node at (270:.4) [left]{$\scriptstyle{n}$};
}\,\in W_{n^{-}+n^{-}+n^{-}}$ in the same way.
\end{DEF}
We sometimes omit directions of edges of $A_2$ webs.
Then we take compatible directions for colored tri-, $4$-valent vertices and $A_2$ clasps.

\begin{LEM}\label{coloredvertex}\  
\begin{enumerate}
\item
$\,\tikz[baseline=-.6ex]{
\draw (-.6,.0) -- (-.3,.0);
\draw (-.3,.0) -- (.3,.0);
\draw (.3,.0) -- (.6,.0);
\draw (-.3,-.4) -- (-.3,.4);
\draw (.3,-.4) -- (.3,.4);
\draw[fill=white] (-.4,.2) rectangle +(.2,-.4);
\draw (-.4,.2) -- +(.2,-.4);
\draw[fill=white] (.2,.2) rectangle +(.2,-.4);
\draw (.2,.2) -- +(.2,-.4);
\node at (-.6,0)[left]{$\scriptstyle{n}$};
\node at (0,0)[above]{$\scriptstyle{n}$};
\node at (.6,0)[right]{$\scriptstyle{n}$};
\node at (-.3,.4)[above]{$\scriptstyle{m_1}$};
\node at (-.3,-.4)[below]{$\scriptstyle{m_1}$};
\node at (.3,.4)[above]{$\scriptstyle{m_2}$};
\node at (.3,-.4)[below]{$\scriptstyle{m_2}$};
}\,
=
\,\tikz[baseline=-.6ex]{
\draw (-.3,0) -- (.3,0);
\draw (0,-.4) -- (0,.4);
\draw[fill=white] (-.1,-.2) rectangle (.1,.2);
\draw (-.1,.2) -- (.1,-.2);
\node at (-.3,0)[left]{$\scriptstyle{n}$};
\node at (.3,0)[right]{$\scriptstyle{n}$};
\node at (0,.4)[above]{$\scriptstyle{m_1+m_2}$};
\node at (0,-.4)[below]{$\scriptstyle{m_1+m_2}$};
}\,
$
\item
$\,\tikz[baseline=-.6ex]{
\draw (-.4,0) -- (.4,0);
\draw (0,-.4) -- (0,.4);
\draw[fill=white] (-.2,-.2) rectangle +(.4,.4);
\draw (-.2,.2) -- +(.4,-.4);
\node at (-.4,0)[left]{$\scriptstyle{n}$};
\node at (0,-.4)[below]{$\scriptstyle{n}$};
}\,
=
\,\tikz[baseline=-.6ex]{
\draw (-.6,.0) -- (-.3,.0);
\draw (-.3,.0) -- (.3,.0);
\draw (.3,.0) -- (.6,.0);
\draw (-.3,-.4) -- (-.3,.0);
\draw (.3,.0) -- (.3,.4);
\draw[fill=white] (-.5,-.2) -- (-.1,-.2) -- (-.1,.2) -- cycle;
\draw[fill=white, rotate=180] (-.5,-.2) -- (-.1,-.2) -- (-.1,.2) -- cycle;
\node at (-.6,0)[left]{$\scriptstyle{n}$};
\node at (.6,0)[right]{$\scriptstyle{n}$};
\node at (.3,.4)[above]{$\scriptstyle{n}$};
\node at (-.3,-.4)[below]{$\scriptstyle{n}$};
}$
\item
$\,\tikz[baseline=-.6ex]{
\draw (.0,.4) -- (.6,.4);
\draw (.0,.0) -- (.6,.0);
\draw (.3,.3) [rounded corners]-- (.3,-.3) -- (.6,-.3);
\draw[fill=white] (.2,.3) -- (.4,.3) -- (.4,.5) -- cycle;
\draw[fill=white] (.2,.2) rectangle +(.2,-.4);
\draw (.2,.2) -- +(.2,-.4);
\node at (.6,.4)[right]{$\scriptstyle{m}$};
\node at (.6,.0)[right]{$\scriptstyle{n}$};
\node at (.6,-.3)[right]{$\scriptstyle{m}$};
}\,
=
\,\tikz[baseline=-.6ex, yscale=-1]{
\draw (.0,.4) -- (.6,.4);
\draw (.0,.0) -- (.6,.0);
\draw (.3,.3) [rounded corners]-- (.3,-.3) -- (.6,-.3);
\draw[fill=white] (.2,.3) -- (.4,.3) -- (.4,.5) -- cycle;
\draw[fill=white] (.2,.2) rectangle +(.2,-.4);
\draw (.2,.2) -- +(.2,-.4);
\node at (.6,.4)[right]{$\scriptstyle{m}$};
\node at (.6,.0)[right]{$\scriptstyle{n}$};
\node at (.6,-.3)[right]{$\scriptstyle{m}$};
}\,$
\item
$\,\tikz[baseline=-.6ex]{
\draw (.0,.4) -- (.4,.4) -- (.4,-.4);
\draw (.3,.4) -- (.3,-.4);
\draw (.0,.0) -- (.7,0);
\draw[fill=white] (.2,.2) rectangle +(.3,-.4);
\draw (.2,.2) -- +(.3,-.4);
\node at (.0,.0)[left]{$\scriptstyle{n}$};
\node at (.0,.4)[left]{$\scriptstyle{1}$};
}\,
=
\,\tikz[baseline=-.6ex]{
\draw (.0,-.1) -- (.3,-.1) -- (.3,-.3);
\draw (.2,-.1) -- (.2,-.3);
\draw (.0,.1) -- (.5,.1);
\node at (.0,.1)[left]{$\scriptstyle{n}$};
\node at (.0,-.1)[left]{$\scriptstyle{1}$};
}\,
+\sum_{i=0}^{n-1}
\,\tikz[baseline=-.6ex]{
\draw (.0,.4) -- (.7,.4);
\draw (.0,-.05) -- (.7,-.05);
\draw (.0,.3) -- (.3,.3) -- (.3,-.4);
\draw (.0,.2) -- (.3,.2);
\draw (.4,-.4) -- (.4,.2) -- (.7,.2);
\draw[fill=white] (.2,.1) rectangle +(.3,-.3);
\draw (.2,.1) -- +(.3,-.3);
\node at (.7,.4)[right]{$\scriptstyle{n-i-1}$};
\node at (.7,.2)[right]{$\scriptstyle{1}$};
\node at (.7,-.05)[right]{$\scriptstyle{i}$};
}$
\item
$\,\tikz[baseline=-.6ex]{
\draw (.0,.2) -- +(-.2,.0);
\draw (.6,.4) -- +(.2,.0);
\draw (.6,.0) -- +(.2,.0);
\draw (.0,.4) -- (.6,.4);
\draw (.0,.0) -- (.6,.0);
\draw (.3,.3) [rounded corners]-- (.3,-.3) -- (.8,-.3);
\draw[fill=white] (.2,.3) -- (.4,.3) -- (.4,.5) -- cycle;
\draw[fill=white] (.2,.2) rectangle +(.2,-.4);
\draw (.2,.2) -- +(.2,-.4);
\draw[fill=white] (.0,.5) rectangle (-.1,-.2);
\draw[fill=white] (.6,.2) rectangle +(.1,-.4);
\node at (-.2,.2)[left]{$\scriptstyle{n+m}$};
\node at (.8,.4)[right]{$\scriptstyle{m}$};
\node at (.8,.0)[right]{$\scriptstyle{n}$};
\node at (.8,-.3)[right]{$\scriptstyle{m}$};
}\,
=\,\tikz[baseline=-.6ex]{
\draw (.0,.2) -- +(-.2,.0);
\draw (.0,.4) -- (.6,.4);
\draw (.0,.0) -- (.6,.0);
\draw (.3,.3) [rounded corners]-- (.3,-.3) -- (.6,-.3);
\draw[fill=white] (.2,.3) -- (.4,.3) -- (.4,.5) -- cycle;
\draw[fill=white] (.2,.2) rectangle +(.2,-.4);
\draw (.2,.2) -- +(.2,-.4);
\draw[fill=white] (.0,.5) rectangle (-.1,-.2);
\node at (-.2,.2)[left]{$\scriptstyle{n+m}$};
\node at (.6,.4)[right]{$\scriptstyle{m}$};
\node at (.6,.0)[right]{$\scriptstyle{n}$};
\node at (.6,-.3)[right]{$\scriptstyle{m}$};
}\,$
\item
$\,\tikz[baseline=-.6ex]{
\draw (-.4,.0) -- +(-.2,.0);
\draw (.4,.0) -- +(.2,.0);
\draw (.0,-.3) -- +(.0,-.2);
\draw (-.3,.0) -- (.0,.0);
\draw (.3,.0) -- (.0,.0);
\draw (.0,-.3) -- (.0,.0);
\draw[fill=white] (-.4,-.2) rectangle (-.3,.2);
\draw[fill=white] (.4,-.2) rectangle (.3,.2);
\draw[fill=white] (-.2,-.3) rectangle (.2,-.4);
\draw[fill=white] (-.2,-.2) -- (.2,-.2) -- (.2,.2) -- cycle;
\node at (-.6,0)[left]{$\scriptstyle{n}$};
\node at (.6,0)[right]{$\scriptstyle{n}$};
\node at (.0,-.6)[below]{$\scriptstyle{n}$};
}
=\,\tikz[baseline=-.6ex]{
\draw (-.4,.0) -- +(-.2,.0);
\draw (.4,.0) -- +(.2,.0);
\draw (.0,-.3) -- +(.0,-.2);
\draw (-.3,.0) -- (.0,.0);
\draw (.4,.0) -- (.0,.0);
\draw (.0,-.3) -- (.0,.0);
\draw[fill=white] (-.4,-.2) rectangle (-.3,.2);
\draw[fill=white] (-.2,-.3) rectangle (.2,-.4);
\draw[fill=white] (-.2,-.2) -- (.2,-.2) -- (.2,.2) -- cycle;
\node at (-.6,0)[left]{$\scriptstyle{n}$};
\node at (.6,0)[right]{$\scriptstyle{n}$};
\node at (.0,-.6)[below]{$\scriptstyle{n}$};
}$
\end{enumerate}
\end{LEM}
\begin{proof}
$(1)$ and $(2)$ are obtained by definitions of colored tri- and $4$-valent vertices. 
Let us show $(3)$ by the induction on $m$.
We can easily confirm in the case of $m=1$. 
If $m>1$, then
\begin{align}
\,\tikz[baseline=-.6ex]{
\draw (-.1,.5) -- (.7,.5);
\draw (-.1,-.1) -- (.7,-.1);
\draw (.3,.3) [rounded corners]-- (.3,-.5) -- (.7,-.5);
\draw[fill=white] (.1,.3) -- (.5,.3) -- (.5,.7) -- cycle;
\draw[fill=white] (.1,.2) rectangle +(.4,-.6);
\draw (.1,.2) -- +(.4,-.6);
\node at (.7,.5)[right]{$\scriptstyle{m}$};
\node at (.7,-.1)[right]{$\scriptstyle{n}$};
\node at (.7,-.5)[right]{$\scriptstyle{m}$};
}\,
&=
\,\tikz[baseline=-.6ex]{
\draw (-.3,.4) -- (.6,.4);
\draw (-.3,.0) -- (.6,.0);
\draw (-.3,.6) -- (.6,.6);
\draw (.0,.6) [rounded corners] -- (.0,-.4) -- (.6,-.4);
\draw (.3,.3) [rounded corners]-- (.3,-.3) -- (.6,-.3);
\draw[fill=white] (.2,.3) -- (.4,.3) -- (.4,.5) -- cycle;
\draw[fill=white] (.2,.2) rectangle +(.2,-.4);
\draw (.2,.2) -- +(.2,-.4);
\draw[fill=white] (-.1,.5) rectangle +(.2,-.2);
\draw (-.1,.5) -- +(.2,-.2);
\draw[fill=white] (-.1,.2) rectangle +(.2,-.4);
\draw (-.1,.2) -- +(.2,-.4);
\node at (.6,.4)[right]{$\scriptstyle{m-1}$};
\node at (.6,.0)[right]{$\scriptstyle{n}$};
\node at (.6,-.3)[right]{$\scriptstyle{m-1}$};
}\,
=
\,\tikz[baseline=-.6ex]{
\draw (-.3,.4) -- (.6,.4);
\draw (-.3,.0) -- (.6,.0);
\draw (-.3,.6) -- (.6,.6);
\draw (.0,.6) [rounded corners] -- (.0,-.4) -- (.6,-.4);
\draw (.3,.3) [rounded corners]-- (.3,-.3) -- (.6,-.3);
\draw[fill=white] (.2,.3) -- (.4,.3) -- (.4,.5) -- cycle;
\draw[fill=white] (.2,.2) rectangle +(.2,-.4);
\draw (.2,.2) -- +(.2,-.4);
\draw[fill=white] (-.1,.5) rectangle +(.2,-.7);
\draw (-.1,.5) -- +(.2,-.7);
\node at (.6,.4)[right]{$\scriptstyle{m-1}$};
\node at (.6,.0)[right]{$\scriptstyle{n}$};
\node at (.6,-.3)[right]{$\scriptstyle{m-1}$};
}\,\notag\\
&=
\,\tikz[baseline=-.6ex, yscale=-1]{
\draw (-.3,.4) -- (.6,.4);
\draw (-.3,.0) -- (.6,.0);
\draw (-.3,.6) -- (.6,.6);
\draw (.0,.6) [rounded corners] -- (.0,-.4) -- (.6,-.4);
\draw (.3,.3) [rounded corners]-- (.3,-.3) -- (.6,-.3);
\draw[fill=white] (.2,.3) -- (.4,.3) -- (.4,.5) -- cycle;
\draw[fill=white] (.2,.2) rectangle +(.2,-.4);
\draw (.2,.2) -- +(.2,-.4);
\draw[fill=white] (-.1,.5) rectangle +(.2,-.7);
\draw (-.1,.5) -- +(.2,-.7);
\node at (.6,.4)[right]{$\scriptstyle{m-1}$};
\node at (.6,.0)[right]{$\scriptstyle{n}$};
\node at (.6,-.3)[right]{$\scriptstyle{m-1}$};
}\,
=\,\tikz[baseline=-.6ex, yscale=-1]{
\draw (-.1,.5) -- (.7,.5);
\draw (-.1,-.1) -- (.7,-.1);
\draw (.3,.3) [rounded corners]-- (.3,-.5) -- (.7,-.5);
\draw[fill=white] (.1,.3) -- (.5,.3) -- (.5,.7) -- cycle;
\draw[fill=white] (.1,.2) rectangle +(.4,-.6);
\draw (.1,.2) -- +(.4,-.6);
\node at (.7,.5)[right]{$\scriptstyle{m}$};
\node at (.7,-.1)[right]{$\scriptstyle{n}$};
\node at (.7,-.5)[right]{$\scriptstyle{m}$};
}\,.\label{indhypm}
\end{align}
We used the induction hypothesis in $(\ref{indhypm})$.
We can show $(4)$ by an easy calculation for $n=1$.
If $n>1$, then
\begin{align*}
\,\tikz[baseline=-.6ex]{
\draw (.0,.4) -- (.4,.4) -- (.4,-.4);
\draw (.3,.4) -- (.3,-.4);
\draw (.0,.0) -- (.7,0);
\draw[fill=white] (.2,.2) rectangle +(.3,-.4);
\draw (.2,.2) -- +(.3,-.4);
\node at (.0,.0)[left]{$\scriptstyle{n}$};
\node at (.0,.4)[left]{$\scriptstyle{1}$};
}\,
&=
\,\tikz[baseline=-.6ex]{
\draw (.0,.5) -- (.4,.5) -- (.4,-.5);
\draw (.3,.5) -- (.3,-.5);
\draw (.0,.25) -- (.7,.25);
\draw (.0,-.1) -- (.7,-.1);
\draw[fill=white] (.2,.3) rectangle +(.3,-.1);
\draw (.2,.3) -- +(.3,-.1);
\draw[fill=white] (.2,.1) rectangle +(.3,-.4);
\draw (.2,.1) -- +(.3,-.4);
\node at (.0,.25)[left]{$\scriptstyle{1}$};
\node at (.0,-.1)[left]{$\scriptstyle{n-1}$};
\node at (.0,.5)[left]{$\scriptstyle{1}$};
}\,
=
\,\tikz[baseline=-.6ex]{
\draw (.0,.4) -- (.7,.4);
\draw (.0,.3) -- (.4,.3) -- (.4,-.4);
\draw (.3,.3) -- (.3,-.4);
\draw (.0,.0) -- (.7,0);
\draw[fill=white] (.2,.2) rectangle +(.3,-.4);
\draw (.2,.2) -- +(.3,-.4);
\node at (.0,.0)[left]{$\scriptstyle{n-1}$};
\node at (.7,.4)[right]{$\scriptstyle{1}$};
\node at (.0,.3)[left]{$\scriptstyle{1}$};
}\,
+
\,\tikz[baseline=-.6ex]{
\draw (.0,-.05) -- (.7,-.05);
\draw (.0,.3) -- (.3,.3) -- (.3,-.4);
\draw (.0,.2) -- (.3,.2);
\draw (.4,-.4) -- (.4,.2) -- (.7,.2);
\draw[fill=white] (.2,.1) rectangle +(.3,-.3);
\draw (.2,.1) -- +(.3,-.3);
\node at (.7,.2)[right]{$\scriptstyle{1}$};
\node at (.7,-.05)[right]{$\scriptstyle{n-1}$};
}\,\\
&=
\,\tikz[baseline=-.6ex]{
\draw (.0,.3) -- (.5,.3);
\draw (.0,-.1) -- (.3,-.1) -- (.3,-.3);
\draw (.2,-.1) -- (.2,-.3);
\draw (.0,.1) -- (.5,.1);
\node at (.5,.1)[right]{$\scriptstyle{n-1}$};
\node at (.0,-.1)[left]{$\scriptstyle{1}$};
\node at (.5,.3)[right]{$\scriptstyle{1}$};
}\,
+\sum_{i=0}^{n-2}
\,\tikz[baseline=-.6ex]{
\draw (.0,.6) -- (.7,.6);
\draw (.0,.4) -- (.7,.4);
\draw (.0,-.05) -- (.7,-.05);
\draw (.0,.3) -- (.3,.3) -- (.3,-.4);
\draw (.0,.2) -- (.3,.2);
\draw (.4,-.4) -- (.4,.2) -- (.7,.2);
\draw[fill=white] (.2,.1) rectangle +(.3,-.3);
\draw (.2,.1) -- +(.3,-.3);
\node at (.7,.6)[right]{$\scriptstyle{1}$};
\node at (.7,.4)[right]{$\scriptstyle{n-i-2}$};
\node at (.7,.2)[right]{$\scriptstyle{1}$};
\node at (.7,-.05)[right]{$\scriptstyle{i}$};
}
+
\,\tikz[baseline=-.6ex]{
\draw (.0,-.05) -- (.7,-.05);
\draw (.0,.3) -- (.3,.3) -- (.3,-.4);
\draw (.0,.2) -- (.3,.2);
\draw (.4,-.4) -- (.4,.2) -- (.7,.2);
\draw[fill=white] (.2,.1) rectangle +(.3,-.3);
\draw (.2,.1) -- +(.3,-.3);
\node at (.7,.2)[right]{$\scriptstyle{1}$};
\node at (.7,-.05)[right]{$\scriptstyle{n-1}$};
}\,
\end{align*}
by the induction on $n$.
$(5)$ is easily showed by the induction on $m$.
In order to prove $(6)$, 
we only have to show the following:
\[
\,\tikz[baseline=-.6ex]{
\draw (-.4,.0) -- +(-.2,.0);
\draw (.0,-.3) -- +(.0,-.2);
\draw (-.3,.0) -- (.0,.0);
\draw (.0,-.3) -- (.0,.0);
\draw (.2,.05) -- (.4,.05) -- (.4,-.05);
\draw (.2,.-.05) -- (.6,-.05);
\draw (.2,.15) -- (.6,.15);
\draw (.2,-.15) -- (.6,-.15);
\draw[fill=white] (-.4,-.2) rectangle (-.3,.2);
\draw[fill=white] (-.2,-.3) rectangle (.2,-.4);
\draw[fill=white] (-.2,-.2) -- (.2,-.2) -- (.2,.2) -- cycle;
\node at (-.6,0)[left]{$\scriptstyle{n}$};
\node at (.0,-.6)[below]{$\scriptstyle{n}$};
\node at (.6,.15)[above right]{$\scriptstyle{k}$};
\node at (.6,-.05)[right]{$\scriptstyle{1}$};
\node at (.6,-.15)[below right]{$\scriptstyle{n-k-2}$};
}=0
\]
for $k=1,2,\dots, n-2$ through the recursive definition of $A_2$ clasps.
Furthermore, 
\[
\,\tikz[baseline=-.6ex]{
\draw (-.4,.0) -- +(-.2,.0);
\draw (.0,-.3) -- +(.0,-.2);
\draw (-.3,.0) -- (.0,.0);
\draw (.0,-.3) -- (.0,.0);
\draw (.2,.05) -- (.4,.05) -- (.4,-.05);
\draw (.2,.-.05) -- (.6,-.05);
\draw (.2,.15) -- (.6,.15);
\draw (.2,-.15) -- (.6,-.15);
\draw[fill=white] (-.4,-.2) rectangle (-.3,.2);
\draw[fill=white] (-.2,-.3) rectangle (.2,-.4);
\draw[fill=white] (-.2,-.2) -- (.2,-.2) -- (.2,.2) -- cycle;
\node at (-.6,0)[left]{$\scriptstyle{n}$};
\node at (.0,-.6)[below]{$\scriptstyle{n}$};
\node at (.6,.15)[above right]{$\scriptstyle{k}$};
\node at (.6,-.05)[right]{$\scriptstyle{1}$};
\node at (.6,-.15)[below right]{$\scriptstyle{n-k-2}$};
}
=
\,\tikz[baseline=-.6ex]{
\draw (.0,.2) -- +(-.2,.0);
\draw (.0,.4) -- (1.3,.4);
\draw (.0,.0) -- (.9,.0);
\draw (.3,.3) -- (.3,-.6);
\draw (.7,-.2) -- (.7,-.6);
\draw (.9,.1) -- (1.1,.1) -- (1.1,.0);
\draw (.9,.0) -- (1.3,.0);
\draw (.9,-.1) -- (1.3,-.1);
\draw[fill=white] (.2,.3) -- (.4,.3) -- (.4,.5) -- cycle;
\draw[fill=white] (.5,-.2) -- (.9,-.2) -- (.9,.2) -- cycle;
\draw[fill=white] (.2,.2) rectangle +(.2,-.4);
\draw (.2,.2) -- +(.2,-.4);
\draw[fill=white] (.0,.5) rectangle (-.1,-.2);
\draw[fill=white] (.2,-.3) rectangle +(.7,-.1);
\node at (-.2,.2)[left]{$\scriptstyle{n}$};
\node at (1.3,.4)[right]{$\scriptstyle{k}$};
\node at (1.3,.1)[right]{$\scriptstyle{1}$};
\node at (1.3,-.1)[right]{$\scriptstyle{n-k-2}$};
}\,
=
\,\tikz[baseline=-.6ex]{
\draw (.0,.2) -- +(-.2,.0);
\draw (.0,.4) -- (1.5,.4);
\draw (.0,.0) -- (.9,.0);
\draw (.3,.3) -- (.3,-.6);
\draw (.9,-.2) -- (.9,-.6);
\draw (1.1,.1) -- (1.3,.1) -- (1.3,.0);
\draw (1.1,.0) -- (1.5,.0);
\draw (1.1,-.1) -- (1.5,-.1);
\draw[fill=white] (.2,.3) -- (.4,.3) -- (.4,.5) -- cycle;
\draw[fill=white] (.7,-.2) -- (1.1,-.2) -- (1.1,.2) -- cycle;
\draw[fill=white] (.2,.2) rectangle +(.2,-.4);
\draw (.2,.2) -- +(.2,-.4);
\draw[fill=white] (.0,.5) rectangle (-.1,-.2);
\draw[fill=white] (.5,.2) rectangle (.6,-.2);
\draw[fill=white] (.2,-.3) rectangle +(.9,-.1);
\node at (-.2,.2)[left]{$\scriptstyle{n}$};
\node at (1.5,.4)[right]{$\scriptstyle{k}$};
\node at (1.5,.1)[right]{$\scriptstyle{1}$};
\node at (1.5,-.1)[right]{$\scriptstyle{n-k-2}$};
}\,
\]
by Lemma~{\ref{coloredvertex}}~$(5)$.
\begin{align*}
\,\tikz[baseline=-.6ex]{
\draw (-.2,.0) -- (.5,.0);
\draw (.5,-.3) -- (.5,-.7);
\draw (.8,.2) -- (.9,.2) -- (.9,.1);
\draw (.8,.1) -- (1.0,.1);
\draw (.8,-.1) -- (1.0,-.1);
\draw[fill=white] (.0,.3) rectangle +(.1,-.6);
\draw[fill=white] (.2,-.4) rectangle +(.6,-.1);
\draw[fill=white] (.2,-.3) -- (.8,-.3) -- (.8,.3) -- cycle;
\node at (-.2,0)[left]{$\scriptstyle{n}$};
\node at (.5,-.7)[below]{$\scriptstyle{n}$};
\node at (1.0,.1)[right]{$\scriptstyle{1}$};
\node at (1.0,-.1)[right]{$\scriptstyle{n-k-1}$};
}
&=\,\tikz[baseline=-.6ex]{
\draw (.0,.0) -- +(-.2,.0);
\draw (.5,-.5) -- +(.0,-.2);
\draw (.1,-.2) -- (1.0,-.2);
\draw (.25,.2) -- (.25,-.4);
\draw (.45,.0) -- (.45,-.4);
\draw (.7,-.3) -- (.7,-.4);
\draw (.1,.2) -- (.6,.2) -- (.6,.0);
\draw (.3,.0) -- (1.0,.0);
\draw[fill=white] (.2,.1) rectangle +(.1,-.4);
\draw (.2,.1) -- +(.1,-.4);
\draw[fill=white] (.4,-.1) rectangle +(.1,-.2);
\draw (.4,-.1) -- +(.1,-.2);
\draw[fill=white] (.0,.3) rectangle +(.1,-.6);
\draw[fill=white] (.2,-.4) rectangle +(.6,-.1);
\draw[fill=white] (.6,-.3) -- (.8,-.3) -- (.8,-.1) -- cycle;
\node at (-.2,0)[left]{$\scriptstyle{n}$};
\node at (.5,-.7)[below]{$\scriptstyle{n}$};
\node at (1.0,.0)[right]{$\scriptstyle{1}$};
\node at (1.0,-.2)[right]{$\scriptstyle{n-k-1}$};
}
=\,\tikz[baseline=-.6ex]{
\draw (.0,.0) -- +(-.2,.0);
\draw (.5,-.5) -- +(.0,-.2);
\draw (.7,-.3) -- (.7,-.4);
\draw (.1,.1) -- (1.0,.1);
\draw (.3,.2) -- (.3,.1);
\draw (.25,.1) -- (.25,-.4);
\draw (.35,.1) -- (.35,-.4);
\draw (.1,.2) -- (.5,.2) -- (.5,.1);
\draw (.1,-.15) -- (1.0,-.15);
\draw[fill=white] (.2,.0) rectangle +(.2,-.3);
\draw (.2,.0) -- +(.2,-.3);
\draw[fill=white] (.0,.3) rectangle +(.1,-.6);
\draw[fill=white] (.2,-.4) rectangle +(.6,-.1);
\draw[fill=white] (.5,-.3) -- (.8,-.3) -- (.8,.0) -- cycle;
\node at (-.2,0)[left]{$\scriptstyle{n}$};
\node at (.5,-.7)[below]{$\scriptstyle{n}$};
\node at (1.0,.1)[right]{$\scriptstyle{1}$};
\node at (1.0,-.15)[right]{$\scriptstyle{n-k-1}$};
}\\
&=\,\tikz[baseline=-.6ex]{
\draw (.0,.0) -- +(-.2,.0);
\draw (.5,-.5) -- +(.0,-.2);
\draw (.7,-.3) -- (.7,-.4);
\draw (.1,.2) -- (1.0,.2);
\draw (.1,.1) -- (.35,.1);
\draw (.25,.1) -- (.25,-.4);
\draw (.35,.1) -- (.35,-.4);
\draw (.1,-.15) -- (1.0,-.15);
\draw[fill=white] (.2,.0) rectangle +(.2,-.3);
\draw (.2,.0) -- +(.2,-.3);
\draw[fill=white] (.0,.3) rectangle +(.1,-.6);
\draw[fill=white] (.2,-.4) rectangle +(.6,-.1);
\draw[fill=white] (.5,-.3) -- (.8,-.3) -- (.8,.0) -- cycle;
\node at (-.2,0)[left]{$\scriptstyle{n}$};
\node at (.5,-.7)[below]{$\scriptstyle{n}$};
\node at (1.0,.2)[right]{$\scriptstyle{1}$};
\node at (1.0,-.15)[right]{$\scriptstyle{n-k-1}$};
}
=0
\end{align*}
by Lemma~{\ref{coloredvertex}}~$(4)$.
\end{proof}

\begin{THM}\label{coloredA2}
Let $n$ be a positive integer.
\begin{enumerate}
\item 
$\displaystyle
\Bigg\langle\,\tikz[baseline=-.6ex]{
\draw (-.4,.4) -- +(-.2,0);
\draw[->-=.8, white, double=black, double distance=0.4pt, ultra thick] 
(.4,-.4) to[out=west, in=east] (-.4,.4);
\draw (.4,-.4) -- +(.2,0);
\draw (-.4,-.4) -- +(-.2,0);
\draw[->-=.8, white, double=black, double distance=0.4pt, ultra thick] 
(-.4,-.4) to[out=east, in=west] (.4,.4);
\draw (.4,.4) -- +(.2,0);
\draw[fill=white] (.4,-.6) rectangle +(.1,.4);
\draw[fill=white] (-.4,-.6) rectangle +(-.1,.4);
\draw[fill=white] (.4,.6) rectangle +(.1,-.4);
\draw[fill=white] (-.4,.6) rectangle +(-.1,-.4);
\node at (.4,-.6)[right]{$\scriptstyle{n}$};
\node at (-.4,-.6)[left]{$\scriptstyle{n}$};
\node at (.4,.6)[right]{$\scriptstyle{n}$};
\node at (-.4,.6)[left]{$\scriptstyle{n}$};
}\,\Bigg\rangle_{\! 3}
=\sum_{k=0}^{n} (-1)^kq^{\frac{2n^2-6nk+3k^2}{6}}{n\choose k}_q
\Bigg\langle\,\tikz[baseline=-.6ex]{
\draw 
(-.5,.4) -- +(-.2,0)
(.5,-.4) -- +(.2,0)
(-.5,-.4) -- +(-.2,0)
(.5,.4) -- +(.2,0);
\draw[-<-=.5] (-.5,.3) to[out=east, in=east] (-.5,-.3);
\draw[-<-=.5] (.5,.3) to[out=west, in=west] (.5,-.3);
\draw[-<-=.5] (-.5,.5) to[out=east, in=north west] (.0,.0);
\draw[-<-=.5] (.0,.0) to[out=south east, in=west] (.5,-.5);
\draw[-<-=.5] (.5,.5) to[out=west, in=north east] (.0,.0);
\draw[-<-=.5] (.0,.0) to[out=south west, in=east] (-.5,-.5);
\draw[fill=white] (.2,0) -- (0,.2) -- (-.2,0) -- (0,-.2) -- cycle;
\draw (-.2,0) -- (.2,0);
\draw[fill=white] (.5,-.6) rectangle +(.1,.4);
\draw[fill=white] (-.5,-.6) rectangle +(-.1,.4);
\draw[fill=white] (.5,.6) rectangle +(.1,-.4);
\draw[fill=white] (-.5,.6) rectangle +(-.1,-.4);
\node at (.5,-.6)[right]{$\scriptstyle{n}$};
\node at (-.5,-.6)[left]{$\scriptstyle{n}$};
\node at (.5,.6)[right]{$\scriptstyle{n}$};
\node at (-.5,.6)[left]{$\scriptstyle{n}$};
\node at (-.3,0)[left]{$\scriptstyle{n-k}$};
\node at (.3,0)[right]{$\scriptstyle{n-k}$};
\node at (0,.5)[right]{$\scriptstyle{k}$};
\node at (0,.5)[left]{$\scriptstyle{k}$};
\node at (0,-.5)[right]{$\scriptstyle{k}$};
\node at (0,-.5)[left]{$\scriptstyle{k}$};
}\,\Bigg\rangle_{\! 3}
$
\item 
$\displaystyle
\Bigg\langle\,\tikz[baseline=-.6ex]{
\draw (-.4,-.4) -- +(-.2,0);
\draw[->-=.8, white, double=black, double distance=0.4pt, ultra thick] 
(-.4,-.4) to[out=east, in=west] (.4,.4);
\draw (.4,.4) -- +(.2,0);
\draw (-.4,.4) -- +(-.2,0);
\draw[->-=.8, white, double=black, double distance=0.4pt, ultra thick] 
(.4,-.4) to[out=west, in=east] (-.4,.4);
\draw (.4,-.4) -- +(.2,0);
\draw[fill=white] (.4,-.6) rectangle +(.1,.4);
\draw[fill=white] (-.4,-.6) rectangle +(-.1,.4);
\draw[fill=white] (.4,.6) rectangle +(.1,-.4);
\draw[fill=white] (-.4,.6) rectangle +(-.1,-.4);
\node at (.4,-.6)[right]{$\scriptstyle{n}$};
\node at (-.4,-.6)[left]{$\scriptstyle{n}$};
\node at (.4,.6)[right]{$\scriptstyle{n}$};
\node at (-.4,.6)[left]{$\scriptstyle{n}$};
}\,\Bigg\rangle_{\! 3}
=\sum_{k=0}^{n} (-1)^{k}q^{\frac{-2n^2+3k^2}{6}}{n\choose k}_q
\Bigg\langle\,\tikz[baseline=-.6ex]{
\draw 
(-.5,.4) -- +(-.2,0)
(.5,-.4) -- +(.2,0)
(-.5,-.4) -- +(-.2,0)
(.5,.4) -- +(.2,0);
\draw[-<-=.5] (-.5,.3) to[out=east, in=east] (-.5,-.3);
\draw[-<-=.5] (.5,.3) to[out=west, in=west] (.5,-.3);
\draw[-<-=.5] (-.5,.5) to[out=east, in=north west] (.0,.0);
\draw[-<-=.5] (.0,.0) to[out=south east, in=west] (.5,-.5);
\draw[-<-=.5] (.5,.5) to[out=west, in=north east] (.0,.0);
\draw[-<-=.5] (.0,.0) to[out=south west, in=east] (-.5,-.5);
\draw[fill=white] (.2,0) -- (0,.2) -- (-.2,0) -- (0,-.2) -- cycle;
\draw (-.2,0) -- (.2,0);
\draw[fill=white] (.5,-.6) rectangle +(.1,.4);
\draw[fill=white] (-.5,-.6) rectangle +(-.1,.4);
\draw[fill=white] (.5,.6) rectangle +(.1,-.4);
\draw[fill=white] (-.5,.6) rectangle +(-.1,-.4);
\node at (.5,-.6)[right]{$\scriptstyle{n}$};
\node at (-.5,-.6)[left]{$\scriptstyle{n}$};
\node at (.5,.6)[right]{$\scriptstyle{n}$};
\node at (-.5,.6)[left]{$\scriptstyle{n}$};
\node at (-.3,0)[left]{$\scriptstyle{n-k}$};
\node at (.3,0)[right]{$\scriptstyle{n-k}$};
\node at (0,.5)[right]{$\scriptstyle{k}$};
\node at (0,.5)[left]{$\scriptstyle{k}$};
\node at (0,-.5)[right]{$\scriptstyle{k}$};
\node at (0,-.5)[left]{$\scriptstyle{k}$};
}\,\Bigg\rangle_{\! 3}
$
\item
$\displaystyle
\Bigg\langle\,\tikz[baseline=-.6ex]{
\draw 
(-.4,.4) -- +(-.2,0)
(-.4,-.4) -- +(-.2,0);
\draw[-<-=.5] (-.4,.4) to[out=east, in=north west] (0,.3);
\draw[->-=.7] (0,.3) -- (.4,.3);
\draw[->-=.4, ->-=.8] (0,.3) -- (0,-.3);
\draw[-<-=.7] (0,-.3) -- (.4,-.3);
\draw[->-=.5] (-.4,-.4) to[out=east, in=south west] (0,-.3);
\node at (.4,-.4)[below right]{$\scriptstyle{n}$};
\node at (-.4,-.6)[left]{$\scriptstyle{n}$};
\node at (.4,.4)[above right]{$\scriptstyle{n}$};
\node at (-.4,.6)[left]{$\scriptstyle{n}$};
\draw[fill=white] (-.2,.2) -- (.1,.2) -- (.1,.5) -- cycle;
\draw[fill=white] (-.2,-.2) -- (.1,-.2) -- (.1,-.5) -- cycle;
\begin{scope}[xshift=.9cm]
\draw 
(.4,-.4) -- +(.2,0)
(.4,.4) -- +(.2,0);
\draw[->-=.5] (.4,.4) to[out=east, in=north west] (0,.3);
\draw[-<-=.7] (0,.3) -- (-.4,.3);
\draw[-<-=.3, -<-=.7] (0,.3) -- (0,-.3);
\draw[->-=.7] (0,-.3) -- (-.4,-.3);
\draw[-<-=.5] (.4,-.4) to[out=east, in=south west] (0,-.3);
\draw[fill=white] (.4,-.6) rectangle +(.1,.4);
\draw[fill=white] (.4,.6) rectangle +(.1,-.4);
\node at (.4,-.6)[right]{$\scriptstyle{n}$};
\node at (.4,.6)[right]{$\scriptstyle{n}$};
\draw[fill=white] (.2,.2) -- (-.1,.2) -- (-.1,.5) -- cycle;
\draw[fill=white] (.2,-.2) -- (-.1,-.2) -- (-.1,-.5) -- cycle;
\draw[fill=white] (.2,-.05) rectangle (-.1,.05);
\end{scope}
\draw[fill=white] (.4,-.5) rectangle +(.1,.3);
\draw[fill=white] (-.4,-.6) rectangle +(-.1,.4);
\draw[fill=white] (.4,.5) rectangle +(.1,-.3);
\draw[fill=white] (-.4,.6) rectangle +(-.1,-.4);
\draw[fill=white] (-.2,-.05) rectangle (.1,.05);
}\,\Bigg\rangle_{\! 3}
=\sum_{k=0}^{n}
\Bigg\langle\,\tikz[baseline=-.6ex]{
\draw 
(-.4,.4) -- +(-.2,0)
(.4,-.4) -- +(.2,0)
(-.4,-.4) -- +(-.2,0)
(.4,.4) -- +(.2,0);
\draw[-<-=.5] (-.4,.5) -- (.4,.5);
\draw[->-=.5] (-.4,-.5) -- (.4,-.5);
\draw[-<-=.5] (-.4,.3) to[out=east, in=east] (-.4,-.3);
\draw[->-=.5] (.4,.3) to[out=west, in=west] (.4,-.3);
\draw[fill=white] (.4,-.6) rectangle +(.1,.4);
\draw[fill=white] (-.4,-.6) rectangle +(-.1,.4);
\draw[fill=white] (.4,.6) rectangle +(.1,-.4);
\draw[fill=white] (-.4,.6) rectangle +(-.1,-.4);
\node at (.4,-.6)[right]{$\scriptstyle{n}$};
\node at (-.4,-.6)[left]{$\scriptstyle{n}$};
\node at (.4,.6)[right]{$\scriptstyle{n}$};
\node at (-.4,.6)[left]{$\scriptstyle{n}$};
\node at (0,.5)[above]{$\scriptstyle{n-k}$};
\node at (0,-.5)[below]{$\scriptstyle{n-k}$};
\node at (-.2,0)[left]{$\scriptstyle{k}$};
\node at (.2,0)[right]{$\scriptstyle{k}$};
}\,\Bigg\rangle_{\! 3}
$
\item
$\Big\langle\,\tikz[baseline=-.6ex]{
\draw[->-=.7] (-.9,0) -- (-.4,0);
\draw[-<-=.7] (.9,0) -- (.4,0);
\draw[-<-=.7] (-.4,0) to[out=north, in=west] (0,.25);
\draw[->-=.7] (.4,0) to[out=north, in=east] (0,.25);
\draw[-<-=.7] (-.4,0) to[out=south, in=west] (0,-.25);
\draw[->-=.7] (.4,0) to[out=south, in=east] (0,-.25);
\draw[fill=white] (-.05,.1) rectangle (.05,.4);
\draw[fill=white] (-.05,-.1) rectangle (.05,-.4);
\draw[fill=white, xshift=-.4cm] (0:.2) -- (120:.2) -- (240:.2) -- cycle;
\draw[fill=white, xshift=.4cm] (60:.2) -- (180:.2) -- (300:.2) -- cycle;
\draw[fill=white] (-.8,-.2) rectangle (-.7,.2);
\draw[fill=white] (.8,-.2) rectangle (.7,.2);
\node at (0,.2)[above left]{$\scriptstyle{n}$};
\node at (0,-.2)[below left]{$\scriptstyle{n}$};
\node at (-.4,0)[above left]{$\scriptstyle{n}$};
\node at (.4,0)[above right]{$\scriptstyle{n}$};
}\,\Big\rangle_{\! 3}
=
\left[n+1\right]\Big\langle\,\tikz[baseline=-.6ex]{
\draw[->-=.3, ->-=.8] (-.5,0) -- (.5,0);
\draw[fill=white] (-.05,-.2) rectangle (.05,.2);
\node at (0,0)[above right]{$\scriptstyle{n}$};
}\,\Big\rangle_{\! 3}
$
\item
$\displaystyle\Big\langle\,\tikz[baseline=-.6ex]{
\draw[->-=.5] (0,0) circle [radius=.3];
\draw[fill=white] (.1,-.05) rectangle (.5,.05);
\node at (.3,0)[above right]{$\scriptstyle{n}$};
}\,\Big\rangle_{\! 3}
=\frac{\left[n+1\right]\left[n+2\right]}{\left[2\right]} \emptyset$
\end{enumerate}
\end{THM}

We can obtain the following by easy calculation.
\begin{LEM}
For $k=0,1,\dots,n$,
\[
\Big\langle\,\tikz[baseline=-.6ex]{
\draw[-<-=.2] (-.4,.2) -- (.4,.2);
\draw[-<-=.8, rounded corners=.1cm] (-.2,-.2) rectangle (.2,.1);
\draw[fill=white] (-.05,.0) rectangle (.05,.3);
\node at (.0,.3)[above]{$\scriptstyle{n}$};
\node at (.4,.2)[right]{$\scriptstyle{n-k}$};
\node at (.0,-.2)[below]{$\scriptstyle{k}$};
}\,\Big\rangle_{\! 3}
=\frac{\left[n+1\right]\left[n+2\right]}{\left[n-k+1\right]\left[n-k+2\right]}
\Big\langle\,\tikz[baseline=-.6ex]{
\draw[-<-=.2] (-.4,.0) -- (.4,.0);
\draw[fill=white] (-.05,-.2) rectangle (.05,.2);
\node at (.0,.2)[above]{$\scriptstyle{n-k}$};
}\,\Big\rangle_{\! 3}.\]
\end{LEM}
This lemma gives Theorem~\ref{coloredA2}~$(5)$ when $k=n$.
Theorem~\ref{coloredA2}~$(1)$ is obtained from a calculation of Theorem~\ref{coloredA2}~$(2)$ substituting $q^{-1}$ with $q$.
Thus, 
we only need to prove Theorem~\ref{coloredA2}~$(2)$--$(4)$.
To prove Theorem~\ref{coloredA2}~$(2)$--$(4)$, 
we prepare lemmas.
\begin{LEM}\label{A2through}
For $n\geq 2$ and $0\leq k\leq n-1$,
\[
\Bigg\langle\,\tikz[baseline=-.6ex]{
\draw 
(-.4,.4) -- +(-.2,0)
(-.4,-.4) -- +(-.2,0)
(.3,.35) -- +(.2,0)
(.3,-.35) -- +(.2,0);
\draw (-.4,.3) -- (.3,.3);
\draw (-.4,-.3) -- (.3,-.3);
\draw (-.4,.6) -- (.5,.6);
\draw (-.4,-.6) -- (.5,-.6);
\draw (-.1,.6) -- (-.1,-.6);
\draw (.3,.2) to[out=east, in=east] (.3,-.2);
\draw[fill=white] (-.2,.5) -- (.0,.5) -- (.0,.7) -- cycle;
\draw[fill=white, yscale=-1] (-.2,.5) -- (.0,.5) -- (.0,.7) -- cycle;
\draw[fill=white] (-.2,-.4) rectangle +(.2,.3);
\draw[fill=white] (-.2,.4) rectangle +(.2,-.3);
\draw (-.2,-.4) -- +(.2,.3);
\draw (-.2,.4) -- +(.2,-.3);
\draw[fill=white] (-.4,-.7) rectangle +(-.1,.6);
\draw[fill=white] (-.4,.7) rectangle +(-.1,-.6);
\draw[fill=white] (.3,-.5) rectangle +(-.1,.4);
\draw[fill=white] (.3,.5) rectangle +(-.1,-.4);
\node at (-.4,-.4)[below left]{$\scriptstyle{n}$};
\node at (-.4,.4)[above left]{$\scriptstyle{n}$};
\node at (.5,-.6)[right]{$\scriptstyle{k}$};
\node at (.5,.6)[right]{$\scriptstyle{k}$};
\node at (.5,-.3)[right]{$\scriptstyle{n-k-1}$};
\node at (.5,.3)[right]{$\scriptstyle{n-k-1}$};
\node at (.3,0)[right]{$\scriptstyle{1}$};
}\,\Bigg\rangle_{\! 3}
=
\Bigg\langle\,\tikz[baseline=-.6ex]{
\draw 
(-.4,.4) -- +(-.2,0)
(-.4,-.4) -- +(-.2,0)
(.3,.3) -- +(.2,0)
(.3,-.3) -- +(.2,0);
\draw (-.4,.3) -- (.3,.3);
\draw (-.4,-.3) -- (.3,-.3);
\draw (-.4,.6) -- (.5,.6);
\draw (-.4,-.6) -- (.5,-.6);
\draw (.0,.6) -- (.0,-.6);
\draw (-.4,.2) to[out=east, in=east] (-.4,-.2);
\draw[fill=white] (-.1,-.5) -- (.1,-.5) -- (.1,-.7) -- cycle;
\draw[fill=white, yscale=-1] (-.1,-.5) -- (.1,-.5) -- (.1,-.7) -- cycle;
\draw[fill=white] (-.1,-.4) rectangle +(.2,.3);
\draw[fill=white] (-.1,.4) rectangle +(.2,-.3);
\draw (-.1,-.4) -- +(.2,.3);
\draw (-.1,.4) -- +(.2,-.3);
\draw[fill=white] (-.4,-.7) rectangle +(-.1,.6);
\draw[fill=white] (-.4,.7) rectangle +(-.1,-.6);
\draw[fill=white] (.4,-.4) rectangle +(-.1,.3);
\draw[fill=white] (.4,.4) rectangle +(-.1,-.3);
\node at (-.4,-.4)[below left]{$\scriptstyle{n}$};
\node at (-.4,.4)[above left]{$\scriptstyle{n}$};
\node at (.5,-.6)[right]{$\scriptstyle{k}$};
\node at (.5,.6)[right]{$\scriptstyle{k}$};
\node at (.5,-.3)[right]{$\scriptstyle{n-k-1}$};
\node at (.5,.3)[right]{$\scriptstyle{n-k-1}$};
\node at (-.4,0)[right]{$\scriptstyle{1}$};
}\,\Bigg\rangle_{\! 3}.
\]
\end{LEM}
\begin{proof}
Colored trivalent vertices are constructed of colored $4$-valent vertices by its definition. 
This description of colored trivalent vertices implies that we only have to prove in the case of $k=1$.
\begin{align*}
\Bigg\langle\,\tikz[baseline=-.6ex]{
\draw 
(-.4,.4) -- +(-.2,0)
(-.4,-.4) -- +(-.2,0)
(.3,.35) -- +(.2,0)
(.3,-.35) -- +(.2,0);
\draw (-.4,.3) -- (.3,.3);
\draw (-.4,-.3) -- (.3,-.3);
\draw (-.4,.6) -- (.5,.6);
\draw (-.4,-.6) -- (.5,-.6);
\draw (-.1,.6) -- (-.1,-.6);
\draw (.3,.2) to[out=east, in=east] (.3,-.2);
\draw[fill=white] (-.2,-.5) rectangle +(.2,.4);
\draw (-.2,-.5) -- +(.2,.4);
\draw[fill=white, yscale=-1] (-.2,-.5) rectangle +(.2,.4);
\draw[yscale=-1] (-.2,-.5) -- +(.2,.4);
\draw[fill=white] (-.4,-.7) rectangle +(-.1,.6);
\draw[fill=white] (-.4,.7) rectangle +(-.1,-.6);
\draw[fill=white] (.3,-.5) rectangle +(-.1,.4);
\draw[fill=white] (.3,.5) rectangle +(-.1,-.4);
\node at (-.4,-.4)[below left]{$\scriptstyle{n}$};
\node at (-.4,.4)[above left]{$\scriptstyle{n}$};
\node at (.5,-.6)[right]{$\scriptstyle{1}$};
\node at (.5,.6)[right]{$\scriptstyle{1}$};
\node at (.5,-.3)[right]{$\scriptstyle{n-2}$};
\node at (.5,.3)[right]{$\scriptstyle{n-2}$};
\node at (.3,0)[right]{$\scriptstyle{1}$};
}\,\Bigg\rangle_{\! 3}
&=
\Bigg\langle\,\tikz[baseline=-.6ex]{
\draw 
(-.4,.4) -- +(-.2,0)
(-.4,-.4) -- +(-.2,0);
\draw (-.4,.3) -- (.3,.3);
\draw (-.4,-.3) -- (.3,-.3);
\draw (-.4,.6) -- (.3,.6);
\draw (-.4,-.6) -- (.3,-.6);
\draw (-.1,.6) -- (-.1,-.6);
\draw (.0,.2) to[out=east, in=east] (.0,-.2);
\draw[fill=white] (-.2,-.5) rectangle +(.2,.4);
\draw (-.2,-.5) -- +(.2,.4);
\draw[fill=white, yscale=-1] (-.2,-.5) rectangle +(.2,.4);
\draw[yscale=-1] (-.2,-.5) -- +(.2,.4);
\draw[fill=white] (-.4,-.7) rectangle +(-.1,.6);
\draw[fill=white] (-.4,.7) rectangle +(-.1,-.6);
\node at (-.4,-.4)[below left]{$\scriptstyle{n}$};
\node at (-.4,.4)[above left]{$\scriptstyle{n}$};
\node at (.3,-.6)[right]{$\scriptstyle{1}$};
\node at (.3,.6)[right]{$\scriptstyle{1}$};
\node at (.3,-.3)[right]{$\scriptstyle{n-2}$};
\node at (.3,.3)[right]{$\scriptstyle{n-2}$};
\node at (.1,0)[right]{$\scriptstyle{1}$};
}\,\Bigg\rangle_{\! 3}
=
\Bigg\langle\,\tikz[baseline=-.6ex]{
\draw 
(-.4,.4) -- +(-.2,0)
(-.4,-.4) -- +(-.2,0);
\draw (-.4,.4) -- (.3,.4);
\draw (-.4,-.4) -- (.3,-.4);
\draw (-.4,.6) -- (.3,.6);
\draw (-.4,-.6) -- (.3,-.6);
\draw (-.1,.6) -- (-.1,.15);
\draw[yscale=-1] (-.1,.6) -- (-.1,.15);
\draw (-.2,.15) -- (-.2,-.15);
\draw 
(-.4,.15) to[out=east, in=west]
(.0,.15) to[out=east, in=east]
(.0,-.15) to[out=west, in=west] (-.4,-.15);
\draw[fill=white] (-.2,-.5) rectangle +(.2,.25);
\draw (-.2,-.5) -- +(.2,.25);
\draw[fill=white, yscale=-1] (-.2,-.5) rectangle +(.2,.25);
\draw[yscale=-1] (-.2,-.5) -- +(.2,.25);
\draw[fill=white] (-.4,-.7) rectangle +(-.1,.6);
\draw[fill=white] (-.4,.7) rectangle +(-.1,-.6);
\node at (-.4,-.4)[below left]{$\scriptstyle{n}$};
\node at (-.4,.4)[above left]{$\scriptstyle{n}$};
\node at (.3,-.6)[right]{$\scriptstyle{1}$};
\node at (.3,.6)[right]{$\scriptstyle{1}$};
\node at (.3,-.3)[right]{$\scriptstyle{n-2}$};
\node at (.3,.3)[right]{$\scriptstyle{n-2}$};
\node at (.1,0)[right]{$\scriptstyle{1}$};
}\,\Bigg\rangle_{\! 3}\\
&=\Bigg\langle\,\tikz[baseline=-.6ex]{
\draw 
(-.4,.4) -- +(-.2,0)
(-.4,-.4) -- +(-.2,0);
\draw (-.4,.4) -- (.3,.4);
\draw (-.4,-.4) -- (.3,-.4);
\draw (-.4,.6) -- (.3,.6);
\draw (-.4,-.6) -- (.3,-.6);
\draw (-.1,.6) -- (-.1,-.6);
\draw 
(-.4,.15) to[out=east, in=west]
(-.3,.15) to[out=east, in=east]
(-.3,-.15) to[out=west, in=west] (-.4,-.15);
\draw[fill=white] (-.2,-.5) rectangle +(.2,.25);
\draw (-.2,-.5) -- +(.2,.25);
\draw[fill=white, yscale=-1] (-.2,-.5) rectangle +(.2,.25);
\draw[yscale=-1] (-.2,-.5) -- +(.2,.25);
\draw[fill=white] (-.4,-.7) rectangle +(-.1,.6);
\draw[fill=white] (-.4,.7) rectangle +(-.1,-.6);
\node at (-.4,-.4)[below left]{$\scriptstyle{n}$};
\node at (-.4,.4)[above left]{$\scriptstyle{n}$};
\node at (.3,-.6)[right]{$\scriptstyle{1}$};
\node at (.3,.6)[right]{$\scriptstyle{1}$};
\node at (.3,-.3)[right]{$\scriptstyle{n-2}$};
\node at (.3,.3)[right]{$\scriptstyle{n-2}$};
\node at (.1,0)[right]{$\scriptstyle{1}$};
}\,\Bigg\rangle_{\! 3}
+\Bigg\langle\,\tikz[baseline=-.6ex]{
\draw 
(-.4,.4) -- +(-.2,0)
(-.4,-.4) -- +(-.2,0);
\draw (-.4,.4) -- (.3,.4);
\draw (-.4,-.4) -- (.3,-.4);
\draw (-.4,.6) -- (.3,.6);
\draw (-.4,-.6) -- (.3,-.6);
\draw (-.1,.6) -- (-.1,.15);
\draw[yscale=-1] (-.1,.6) -- (-.1,.15);
\draw (-.4,.15) -- (-.1,.15);
\draw[yscale=-1] (-.4,.15) -- (-.1,.15);
\draw[fill=white] (-.2,-.5) rectangle +(.2,.25);
\draw (-.2,-.5) -- +(.2,.25);
\draw[fill=white, yscale=-1] (-.2,-.5) rectangle +(.2,.25);
\draw[yscale=-1] (-.2,-.5) -- +(.2,.25);
\draw[fill=white] (-.4,-.7) rectangle +(-.1,.6);
\draw[fill=white] (-.4,.7) rectangle +(-.1,-.6);
\node at (-.4,-.4)[below left]{$\scriptstyle{n}$};
\node at (-.4,.4)[above left]{$\scriptstyle{n}$};
\node at (.3,-.6)[right]{$\scriptstyle{1}$};
\node at (.3,.6)[right]{$\scriptstyle{1}$};
\node at (.3,-.3)[right]{$\scriptstyle{n-2}$};
\node at (.3,.3)[right]{$\scriptstyle{n-2}$};
\node at (.1,0)[right]{$\scriptstyle{1}$};
}\,\Bigg\rangle_{\! 3},
\end{align*}
and the second term vanishes by a property of $A_2$ clasp.
\end{proof}

\begin{LEM}
For any positive integer $n$,
\[
\Bigg\langle\,\tikz[baseline=-.6ex]{
\draw 
(.0,.3) -- +(-.2,.0)
(.9,.3) -- +(.2,.0)
(.5,-.2) -- +(.0,-.2);
\draw (.0,.3) -- (.8,.3);
\draw (.8,.3) -- (.9,.3);
\draw (.5,.0) -- (.5,-.2);
\draw[fill=white] (.2,.0) -- (.8,.0) -- (.8,.6) -- cycle;
\draw[fill=white] (.0,.0) rectangle (.1,.6);
\draw[fill=white] (.9,.0) rectangle (1.0,.6);
\draw[fill=white] (.2,-.1) rectangle (.8,-.2);
\node at (-.2,.3)[left]{$\scriptstyle{n}$};
\node at (1.1,.3)[right]{$\scriptstyle{n}$};
\node at (.5,-.4)[below]{$\scriptstyle{n}$};
}\,\Bigg\rangle_{\! 3}
=
\Bigg\langle\,\tikz[baseline=-.6ex]{
\draw 
(.0,.3) -- +(-.2,.0)
(1.2,.2) -- +(.2,.0)
(.8,-.2) -- +(.0,-.2);
\draw (.0,.5) -- (1.4,.5);
\draw (.0,.2) -- (1.0,.2);
\draw (1.0,.2) -- (1.2,.2);
\draw (.8,.0) -- (.8,-.2);
\draw (.25,.5) -- (.25,-.4);
\draw[fill=white] (.6,.0) -- (1.0,.0) -- (1.0,.4) -- cycle;
\draw[fill=white] (.2,.0) rectangle (.3,.4);
\draw (.2,.4) -- (.3,.0);
\draw[fill=white] (.0,.0) rectangle (.1,.6);
\draw[fill=white] (.4,.0) rectangle (.5,.4);
\draw[fill=white] (1.1,.0) rectangle (1.2,.4);
\draw[fill=white] (.6,-.1) rectangle (1.0,-.2);
\node at (-.2,.3)[left]{$\scriptstyle{n}$};
\node at (1.4,.2)[right]{$\scriptstyle{n-1}$};
\node at (1.4,.5)[right]{$\scriptstyle{1}$};
\node at (.8,-.4)[below]{$\scriptstyle{n-1}$};
\node at (.25,-.4)[below]{$\scriptstyle{1}$};
}\,\Bigg\rangle_{\! 3}
-\frac{\left[n-1\right]}{\left[n\right]}
\Bigg\langle\,\tikz[baseline=-.6ex]{
\draw 
(.0,.3) -- +(-.2,.0)
(1.0,.2) -- +(.2,.0)
(.5,-.3) -- +(.0,-.2);
\draw (.0,.5) -- (1.2,.5);
\draw (.0,.3) [rounded corners]-- (.4,.3) -- (.7,.0);
\draw (1.0,.2) -- (1.2,.2);
\draw (.55,.0) -- (.55,-.2);
\draw (.7,.15) -- (.9,.15);
\draw (.1,.1) -- (.2,.1) -- (.2,-.5);
\draw (.2,-.1) -- (.4,-.1) -- (.4,-.2);
\draw (.8,.5) -- (.8,.3) -- (.9,.3);
\draw[fill=white] (.4,.0) -- (.7,.0) -- (.7,.3) -- cycle;
\draw[fill=white] (.0,.0) rectangle (.1,.6);
\draw[fill=white] (.9,.0) rectangle (1.0,.4);
\draw[fill=white] (.3,-.2) rectangle (.7,-.3);
\node at (-.2,.3)[left]{$\scriptstyle{n}$};
\node at (1.2,.2)[right]{$\scriptstyle{n-1}$};
\node at (1.2,.5)[right]{$\scriptstyle{1}$};
\node at (.6,-.5)[below]{$\scriptstyle{n-1}$};
\node at (.2,-.5)[below]{$\scriptstyle{1}$};
}\,\Bigg\rangle_{\! 3}.
\]
\end{LEM}
\begin{proof}
By using the definition of an $A_2$ clasp and Lemma~\ref{coloredvertex}~$(6)$, 
\begin{align*}
\Bigg\langle\,\tikz[baseline=-.6ex]{
\draw 
(.0,.3) -- +(-.2,.0)
(.9,.3) -- +(.2,.0)
(.5,-.2) -- +(.0,-.2);
\draw (.0,.3) -- (.8,.3);
\draw (.8,.3) -- (.9,.3);
\draw (.5,.0) -- (.5,-.2);
\draw[fill=white] (.2,.0) -- (.8,.0) -- (.8,.6) -- cycle;
\draw[fill=white] (.0,.0) rectangle (.1,.6);
\draw[fill=white] (.9,.0) rectangle (1.0,.6);
\draw[fill=white] (.2,-.1) rectangle (.8,-.2);
\node at (-.2,.3)[left]{$\scriptstyle{n}$};
\node at (1.1,.3)[right]{$\scriptstyle{n}$};
\node at (.5,-.4)[below]{$\scriptstyle{n}$};
}\,\Bigg\rangle_{\! 3}
&=
\Bigg\langle\,\tikz[baseline=-.6ex]{
\draw 
(.0,.3) -- +(-.2,.0)
(1.2,.2) -- +(.2,.0)
(.8,-.2) -- +(.0,-.2);
\draw (.0,.5) -- (1.4,.5);
\draw (.0,.2) -- (1.0,.2);
\draw (1.0,.2) -- (1.2,.2);
\draw (.8,.0) -- (.8,-.2);
\draw (.25,.5) -- (.25,-.4);
\draw[fill=white] (.6,.0) -- (1.0,.0) -- (1.0,.4) -- cycle;
\draw[fill=white] (.2,.0) rectangle (.3,.4);
\draw (.2,.4) -- (.3,.0);
\draw[fill=white] (.0,.0) rectangle (.1,.6);
\draw[fill=white] (1.1,.0) rectangle (1.2,.6);
\node at (-.2,.3)[left]{$\scriptstyle{n}$};
\node at (1.4,.2)[right]{$\scriptstyle{n-1}$};
\node at (1.4,.5)[right]{$\scriptstyle{1}$};
\node at (.8,-.4)[below]{$\scriptstyle{n-1}$};
\node at (.25,-.4)[below]{$\scriptstyle{1}$};
}\,\Bigg\rangle_{\! 3}\\
&=\Bigg\langle\,\tikz[baseline=-.6ex]{
\draw 
(.0,.3) -- +(-.2,.0)
(1.2,.2) -- +(.2,.0)
(.8,-.2) -- +(.0,-.2);
\draw (.0,.5) -- (1.4,.5);
\draw (.0,.2) -- (1.0,.2);
\draw (1.0,.2) -- (1.2,.2);
\draw (.8,.0) -- (.8,-.2);
\draw (.25,.5) -- (.25,-.4);
\draw[fill=white] (.6,.0) -- (1.0,.0) -- (1.0,.4) -- cycle;
\draw[fill=white] (.2,.0) rectangle (.3,.4);
\draw (.2,.4) -- (.3,.0);
\draw[fill=white] (.0,.0) rectangle (.1,.6);
\draw[fill=white] (1.1,.0) rectangle (1.2,.4);
\node at (-.2,.3)[left]{$\scriptstyle{n}$};
\node at (1.4,.2)[right]{$\scriptstyle{n-1}$};
\node at (1.4,.5)[right]{$\scriptstyle{1}$};
\node at (.8,-.4)[below]{$\scriptstyle{n-1}$};
\node at (.25,-.4)[below]{$\scriptstyle{1}$};
}\,\Bigg\rangle_{\! 3}
-\frac{\left[n-1\right]}{\left[n\right]}
\Bigg\langle\,\tikz[baseline=-.6ex]{
\draw 
(.0,.3) -- +(-.2,.0)
(1.2,.3) -- +(.3,.0)
(1.2,.1) -- +(.3,.0)
(1.5,.2) -- +(.2,.0)
(.8,-.2) -- +(.0,-.2);
\draw (.0,.5) -- (1.3,.5) -- (1.3,.3);
\draw (1.4,.3) -- (1.4,.5) -- (1.7,.5);
\draw (.0,.2) -- (1.0,.2);
\draw (1.0,.2) -- (1.2,.2);
\draw (.8,.0) -- (.8,-.2);
\draw (.25,.5) -- (.25,-.4);
\draw[fill=white] (.6,.0) -- (1.0,.0) -- (1.0,.4) -- cycle;
\draw[fill=white] (.2,.0) rectangle (.3,.4);
\draw (.2,.4) -- (.3,.0);
\draw[fill=white] (.0,.0) rectangle (.1,.6);
\draw[fill=white] (1.1,.0) rectangle (1.2,.4);
\draw[fill=white] (1.5,.0) rectangle (1.6,.4);
\node at (-.2,.3)[left]{$\scriptstyle{n}$};
\node at (1.7,.2)[right]{$\scriptstyle{n-1}$};
\node at (1.7,.5)[right]{$\scriptstyle{1}$};
\node at (.8,-.4)[below]{$\scriptstyle{n-1}$};
\node at (.25,-.4)[below]{$\scriptstyle{1}$};
}\,\Bigg\rangle_{\! 3}.
\end{align*}
The $A_2$ web of the second term is computed as follows:
\begin{align*}
\Bigg\langle\,\tikz[baseline=-.6ex]{
\draw 
(.0,.3) -- +(-.2,.0)
(1.2,.3) -- +(.3,.0)
(1.2,.1) -- +(.3,.0)
(1.5,.2) -- +(.2,.0)
(.8,-.2) -- +(.0,-.2);
\draw (.0,.5) -- (1.3,.5) -- (1.3,.3);
\draw (1.4,.3) -- (1.4,.5) -- (1.7,.5);
\draw (.0,.2) -- (1.0,.2);
\draw (1.0,.2) -- (1.2,.2);
\draw (.8,.0) -- (.8,-.2);
\draw (.25,.5) -- (.25,-.4);
\draw[fill=white] (.6,.0) -- (1.0,.0) -- (1.0,.4) -- cycle;
\draw[fill=white] (.2,.0) rectangle (.3,.4);
\draw (.2,.4) -- (.3,.0);
\draw[fill=white] (.0,.0) rectangle (.1,.6);
\draw[fill=white] (1.1,.0) rectangle (1.2,.4);
\draw[fill=white] (1.5,.0) rectangle (1.6,.4);
\node at (-.2,.3)[left]{$\scriptstyle{n}$};
\node at (1.7,.2)[right]{$\scriptstyle{n-1}$};
\node at (1.7,.5)[right]{$\scriptstyle{1}$};
\node at (.8,-.4)[below]{$\scriptstyle{n-1}$};
\node at (.25,-.4)[below]{$\scriptstyle{1}$};
}\,\Bigg\rangle_{\! 3}
&=\Bigg\langle\,\tikz[baseline=-.6ex]{
\draw 
(.0,.3) -- +(-.2,.0)
(1.2,.3) -- +(.3,.0)
(1.2,.1) -- +(.3,.0)
(1.5,.2) -- +(.2,.0)
(.8,-.2) -- +(.0,-.2);
\draw (.0,.5) -- (1.3,.5) -- (1.3,.3);
\draw (1.4,.3) -- (1.4,.5) -- (1.7,.5);
\draw (.0,.2) -- (1.0,.2);
\draw (1.0,.2) -- (1.2,.2);
\draw (.8,.0) -- (.8,-.2);
\draw (.25,.5) -- (.25,-.4);
\draw[fill=white] (.6,.0) -- (1.0,.0) -- (1.0,.4) -- cycle;
\draw[fill=white] (.2,.0) rectangle (.3,.4);
\draw (.2,.4) -- (.3,.0);
\draw[fill=white] (.4,.0) rectangle +(.1,.4);
\draw[fill=white] (.0,.0) rectangle (.1,.6);
\draw[fill=white] (1.1,.0) rectangle (1.2,.4);
\draw[fill=white] (1.5,.0) rectangle (1.6,.4);
\node at (-.2,.3)[left]{$\scriptstyle{n}$};
\node at (1.7,.2)[right]{$\scriptstyle{n-1}$};
\node at (1.7,.5)[right]{$\scriptstyle{1}$};
\node at (.8,-.4)[below]{$\scriptstyle{n-1}$};
\node at (.25,-.4)[below]{$\scriptstyle{1}$};
}\,\Bigg\rangle_{\! 3}
=\Bigg\langle\,\tikz[baseline=-.6ex]{
\draw 
(.0,.3) -- +(-.2,.0)
(1.0,.3) -- +(.5,.0)
(1.0,.1) -- +(.5,.0)
(1.5,.2) -- +(.2,.0)
(.8,-.2) -- +(.0,-.2);
\draw (.0,.5) -- (1.3,.5) -- (1.3,.3);
\draw (1.4,.3) -- (1.4,.5) -- (1.7,.5);
\draw (.0,.2) -- (1.0,.2);
\draw (.8,.0) -- (.8,-.2);
\draw (.25,.5) -- (.25,-.4);
\draw[fill=white] (.6,.0) -- (1.0,.0) -- (1.0,.4) -- cycle;
\draw[fill=white] (.2,.0) rectangle (.3,.4);
\draw (.2,.4) -- (.3,.0);
\draw[fill=white] (.4,.0) rectangle +(.1,.4);
\draw[fill=white] (.0,.0) rectangle (.1,.6);
\draw[fill=white] (1.5,.0) rectangle (1.6,.4);
\draw[fill=white] (.6,-.2) rectangle +(.4,.1);
\node at (-.2,.3)[left]{$\scriptstyle{n}$};
\node at (1.7,.2)[right]{$\scriptstyle{n-1}$};
\node at (1.7,.5)[right]{$\scriptstyle{1}$};
\node at (.8,-.4)[below]{$\scriptstyle{n-1}$};
\node at (.25,-.4)[below]{$\scriptstyle{1}$};
}\,\Bigg\rangle_{\! 3}\\
&=\Bigg\langle\,\tikz[baseline=-.6ex]{
\draw 
(.0,.3) -- +(-.2,.0)
(1.0,.3) -- +(.5,.0)
(1.0,.1) -- +(.5,.0)
(1.5,.2) -- +(.2,.0)
(.8,-.2) -- +(.0,-.2);
\draw (.0,.5) -- (1.3,.5) -- (1.3,.3);
\draw (1.4,.3) -- (1.4,.5) -- (1.7,.5);
\draw (.0,.2) -- (1.0,.2);
\draw (.8,.0) -- (.8,-.2);
\draw (.25,.5) -- (.25,-.4);
\draw[fill=white] (.6,.0) -- (1.0,.0) -- (1.0,.4) -- cycle;
\draw[fill=white] (.2,.0) rectangle (.3,.4);
\draw (.2,.4) -- (.3,.0);
\draw[fill=white] (.0,.0) rectangle (.1,.6);
\draw[fill=white] (1.5,.0) rectangle (1.6,.4);
\draw[fill=white] (.6,-.2) rectangle +(.4,.1);
\node at (-.2,.3)[left]{$\scriptstyle{n}$};
\node at (1.7,.2)[right]{$\scriptstyle{n-1}$};
\node at (1.7,.5)[right]{$\scriptstyle{1}$};
\node at (.8,-.4)[below]{$\scriptstyle{n-1}$};
\node at (.25,-.4)[below]{$\scriptstyle{1}$};
}\,\Bigg\rangle_{\! 3}
=\Bigg\langle\,\tikz[baseline=-.6ex]{
\draw 
(.0,.3) -- +(-.2,.0)
(.3,.3) -- +(1.2,.0)
(1.0,.1) -- +(.5,.0)
(1.5,.2) -- +(.2,.0)
(.8,-.2) -- +(.0,-.2);
\draw (.0,.5) -- (1.3,.5) -- (1.3,.3);
\draw (1.4,.3) -- (1.4,.5) -- (1.7,.5);
\draw (.0,.2) -- +(.2,.0);
\draw (.3,.1) -- +(.6,.0);
\draw (.9,.0) -- (.9,-.2);
\draw (.25,.5) -- (.25,-.4);
\draw (.65,.3) -- (.65,-.2);
\draw[fill=white] (.8,.0) -- (1.0,.0) -- (1.0,.2) -- cycle;
\draw[fill=white] (.2,.0) rectangle (.3,.4);
\draw (.2,.4) -- (.3,.0);
\draw[fill=white] (.6,.0) rectangle (.7,.2);
\draw (.6,.2) -- (.7,.0);
\draw[fill=white] (.0,.0) rectangle (.1,.6);
\draw[fill=white] (1.5,.0) rectangle (1.6,.4);
\draw[fill=white] (.6,-.2) rectangle +(.4,.1);
\node at (-.2,.3)[left]{$\scriptstyle{n}$};
\node at (1.7,.2)[right]{$\scriptstyle{n-1}$};
\node at (1.7,.5)[right]{$\scriptstyle{1}$};
\node at (.8,-.4)[below]{$\scriptstyle{n-1}$};
\node at (.25,-.4)[below]{$\scriptstyle{1}$};
}\,\Bigg\rangle_{\! 3}\\
&=\Bigg\langle\,\tikz[baseline=-.6ex]{
\draw 
(.0,.3) -- +(-.2,.0)
(.0,.3) -- (.65,.3)
(1.0,.1) -- +(.5,.0)
(1.5,.2) -- +(.2,.0)
(.8,-.2) -- +(.0,-.2);
\draw (.0,.5) -- (1.7,.5);
\draw (1.3,.5) -- (1.3,.3) -- (1.5,.3);
\draw (.0,.1) -- +(.2,.0);
\draw (.3,.1) -- +(.6,.0);
\draw (.9,.0) -- (.9,-.2);
\draw (.25,.3) -- (.25,-.4);
\draw (.65,.3) -- (.65,-.2);
\draw[fill=white] (.8,.0) -- (1.0,.0) -- (1.0,.2) -- cycle;
\draw[fill=white] (.2,.0) rectangle (.3,.2);
\draw (.2,.2) -- (.3,.0);
\draw[fill=white] (.6,.0) rectangle (.7,.2);
\draw (.6,.2) -- (.7,.0);
\draw[fill=white] (.0,.0) rectangle (.1,.6);
\draw[fill=white] (1.5,.0) rectangle (1.6,.4);
\draw[fill=white] (.6,-.2) rectangle +(.4,.1);
\node at (-.2,.3)[left]{$\scriptstyle{n}$};
\node at (1.7,.2)[right]{$\scriptstyle{n-1}$};
\node at (1.7,.5)[right]{$\scriptstyle{1}$};
\node at (.8,-.4)[below]{$\scriptstyle{n-1}$};
\node at (.25,-.4)[below]{$\scriptstyle{1}$};
}\,\Bigg\rangle_{\! 3}
=\Bigg\langle\,\tikz[baseline=-.6ex]{
\draw 
(.0,.3) -- +(-.2,.0)
(.0,.2) -- (.9,.2)
(1.0,.2) -- +(.5,.0)
(1.5,.2) -- +(.2,.0)
(.8,-.2) -- +(.0,-.2);
\draw (.0,.5) -- (1.7,.5);
\draw (1.3,.5) -- (1.3,.3) -- (1.5,.3);
\draw (.9,.0) -- (.9,-.2);
\draw (.1,.1) -- (.6,.1) -- (.6,-.2);
\draw (.3,.1) -- (.3,-.4);
\draw[fill=white] (.7,.0) -- (1.1,.0) -- (1.1,.4) -- cycle;
\draw[fill=white] (.0,.0) rectangle (.1,.6);
\draw[fill=white] (1.5,.0) rectangle (1.6,.4);
\draw[fill=white] (.5,-.2) rectangle +(.6,.1);
\node at (-.2,.3)[left]{$\scriptstyle{n}$};
\node at (1.7,.2)[right]{$\scriptstyle{n-1}$};
\node at (1.7,.5)[right]{$\scriptstyle{1}$};
\node at (.8,-.4)[below]{$\scriptstyle{n-1}$};
\node at (.25,-.4)[below]{$\scriptstyle{1}$};
}\,\Bigg\rangle_{\! 3}.
\end{align*}
We used Lemma~\ref{coloredvertex}~$(5)$ and $(6)$ for the first three equalities, 
Lemma~\ref{coloredvertex}~$(4)$ for the last equality.
\end{proof}

\begin{proof}[Proof of Theorem~\ref{coloredA2}~$(2)$]
We prove Theorem~\ref{coloredA2}~$(2)$ in a similar way to the proof of the colored Kauffman bracket skein relation.
By using skein relations the $A_2$ bracket and properties of $A_2$ clasps, 
we can easily calculate as follows:
\begin{align*}
\Bigg\langle\,\tikz[baseline=-.6ex]{
\draw (-.4,-.4) -- +(-.2,0);
\draw[->-=.8, white, double=black, double distance=0.4pt, ultra thick] 
(-.4,-.4) to[out=east, in=west] (.4,.4);
\draw (.4,.4) -- +(.2,0);
\draw (-.4,.4) -- +(-.2,0);
\draw[->-=.8, white, double=black, double distance=0.4pt, ultra thick] 
(.4,-.4) to[out=west, in=east] (-.4,.4);
\draw (.4,-.4) -- +(.2,0);
\draw[fill=white] (.4,-.6) rectangle +(.1,.4);
\draw[fill=white] (-.4,-.6) rectangle +(-.1,.4);
\draw[fill=white] (.4,.6) rectangle +(.1,-.4);
\draw[fill=white] (-.4,.6) rectangle +(-.1,-.4);
\node at (.4,-.6)[right]{$\scriptstyle{n}$};
\node at (-.4,-.6)[left]{$\scriptstyle{n}$};
\node at (.4,.6)[right]{$\scriptstyle{n}$};
\node at (-.4,.6)[left]{$\scriptstyle{n}$};
}\,\Bigg\rangle_{\! 3}
&=
\Bigg\langle\,\tikz[baseline=-.6ex]{
\draw 
(-.6,-.4) -- +(-.2,0)
(.6,.4) -- +(.2,0)
(-.6,.4) -- +(-.2,0)
(.6,-.4) -- +(.2,0);
\draw[triple={[line width=1.6pt, white] in [line width=2.4pt, black] in [line width=5.6pt, white]}]
(-.6,-.4) to[out=east, in=west] (.6,.4);
\draw[triple={[line width=1.6pt, white] in [line width=2.4pt, black] in [line width=5.6pt, white]}]
(-.6,.4) to[out=east, in=west] (.6,-.4);
\draw[fill=white] (.6,-.6) rectangle +(.1,.4);
\draw[fill=white] (-.6,-.6) rectangle +(-.1,.4);
\draw[fill=white] (.6,.6) rectangle +(.1,-.4);
\draw[fill=white] (-.6,.6) rectangle +(-.1,-.4);
\node at (.6,-.6)[right]{$\scriptstyle{n}$};
\node at (-.6,-.6)[left]{$\scriptstyle{n}$};
\node at (.6,.6)[right]{$\scriptstyle{n}$};
\node at (-.6,.6)[left]{$\scriptstyle{n}$};
\node at (-.7,-.6)[right]{$\scriptscriptstyle{n-1}$};
\node at (-.7,.6)[right]{$\scriptscriptstyle{n-1}$};
\node at (-.7,-.2)[right]{$\scriptscriptstyle{1}$};
\node at (-.7,.2)[right]{$\scriptscriptstyle{1}$};
}\,\Bigg\rangle_{\! 3}\\ 
&=
q^{-\frac{1}{3}}\Bigg\langle\,\tikz[baseline=-.6ex]{
\draw
(-.4,-.4) -- +(-.2,0)
(.4,.4) -- +(.2,0)
(-.4,.4) -- +(-.2,0)
(.4,-.4) -- +(.2,0);
\draw[white, double=black, double distance=0.4pt, ultra thick] 
(-.4,-.5) to[out=east, in=west] (.4,.3);
\draw[white, double=black, double distance=0.4pt, ultra thick] 
(-.4,-.3) 
to[out=east, in=south] (-.3,0)
to[out=north, in=east](-.4,.3);
\draw[white, double=black, double distance=0.4pt, ultra thick] 
(.4,-.5)
to[out=west, in=south] (.2,0)
to[out=north, in=west] (.4,.5);
\draw[white, double=black, double distance=0.4pt, ultra thick] 
(-.4,.5) to[out=east, in=west] (.4,-.3);
\draw[fill=white] (.4,-.6) rectangle +(.1,.4);
\draw[fill=white] (-.4,-.6) rectangle +(-.1,.4);
\draw[fill=white] (.4,.6) rectangle +(.1,-.4);
\draw[fill=white] (-.4,.6) rectangle +(-.1,-.4);
\node at (.4,-.6)[right]{$\scriptstyle{n}$};
\node at (-.4,-.6)[left]{$\scriptstyle{n}$};
\node at (.4,.6)[right]{$\scriptstyle{n}$};
\node at (-.4,.6)[left]{$\scriptstyle{n}$};
\node at (-.5,-.6)[right]{$\scriptscriptstyle{n-1}$};
\node at (-.5,.6)[right]{$\scriptscriptstyle{n-1}$};
\node at (-.3,0)[left]{$\scriptscriptstyle{1}$};
\node at (.2,0)[right]{$\scriptscriptstyle{1}$};
}\,\Bigg\rangle_{\! 3}
-q^{\frac{1}{6}}\Bigg\langle\,\tikz[baseline=-.6ex]{
\draw
(-.4,-.4) -- +(-.2,0)
(.4,.4) -- +(.2,0)
(-.4,.4) -- +(-.2,0)
(.4,-.4) -- +(.2,0);
\draw[white, double=black, double distance=0.4pt, ultra thick] 
(-.4,-.5) 
to[out=east, in=west] (-.3,-.5)
to[out=east, in=west] (.4,.3);
\draw[white, double=black, double distance=0.4pt, ultra thick] 
(-.4,-.3) 
to[out=east, in=west] (-.3,-.3)
to[out=east, in=west] (.4,-.5);
\draw[white, double=black, double distance=0.4pt, ultra thick] 
(-.4,.3) 
to[out=east, in=west] (-.3,.3)
to[out=east, in=west] (.4,.5);
\draw[white, double=black, double distance=0.4pt, ultra thick] 
(-.4,.5) 
to[out=east, in=west] (-.3,.5)
to[out=east, in=west] (.4,-.3);
\draw (-.3,.3) -- (-.3,-.3);
\draw[fill=white] (.4,-.6) rectangle +(.1,.4);
\draw[fill=white] (-.4,-.6) rectangle +(-.1,.4);
\draw[fill=white] (.4,.6) rectangle +(.1,-.4);
\draw[fill=white] (-.4,.6) rectangle +(-.1,-.4);
\node at (.4,-.6)[right]{$\scriptstyle{n}$};
\node at (-.4,-.6)[left]{$\scriptstyle{n}$};
\node at (.4,.6)[right]{$\scriptstyle{n}$};
\node at (-.4,.6)[left]{$\scriptstyle{n}$};
\node at (-.5,-.6)[right]{$\scriptscriptstyle{n-1}$};
\node at (-.5,.6)[right]{$\scriptscriptstyle{n-1}$};
\node at (.5,.5)[above left]{$\scriptscriptstyle{1}$};
\node at (.5,-.5)[below left]{$\scriptscriptstyle{1}$};
}\,\Bigg\rangle_{\! 3}\\
&=q^{-\frac{2n-1}{3}}
\Bigg\langle\,\tikz[baseline=-.6ex]{
\draw
(-.4,-.4) -- +(-.2,0)
(.4,.4) -- +(.2,0)
(-.4,.4) -- +(-.2,0)
(.4,-.4) -- +(.2,0);
\draw[white, double=black, double distance=0.4pt, ultra thick] 
(-.4,-.5) to[out=east, in=west] (.4,.5);
\draw[white, double=black, double distance=0.4pt, ultra thick] 
(-.4,-.3) 
to[out=east, in=south] (-.3,0)
to[out=north, in=east](-.4,.3);
\draw[white, double=black, double distance=0.4pt, ultra thick] 
(.4,-.3)
to[out=west, in=south] (.3,0)
to[out=north, in=west] (.4,.3);
\draw[white, double=black, double distance=0.4pt, ultra thick] 
(-.4,.5) to[out=east, in=west] (.4,-.5);
\draw[fill=white] (.4,-.6) rectangle +(.1,.4);
\draw[fill=white] (-.4,-.6) rectangle +(-.1,.4);
\draw[fill=white] (.4,.6) rectangle +(.1,-.4);
\draw[fill=white] (-.4,.6) rectangle +(-.1,-.4);
\node at (.4,-.6)[right]{$\scriptstyle{n}$};
\node at (-.4,-.6)[left]{$\scriptstyle{n}$};
\node at (.4,.6)[right]{$\scriptstyle{n}$};
\node at (-.4,.6)[left]{$\scriptstyle{n}$};
\node at (-.5,-.6)[right]{$\scriptscriptstyle{n-1}$};
\node at (-.5,.6)[right]{$\scriptscriptstyle{n-1}$};
\node at (-.3,0)[left]{$\scriptscriptstyle{1}$};
\node at (.3,0)[right]{$\scriptscriptstyle{1}$};
}\,\Bigg\rangle_{\! 3}
-q^{\frac{2n-1}{6}}
\Bigg\langle\,\tikz[baseline=-.6ex]{
\draw
(-.4,-.4) -- +(-.2,0)
(.4,.4) -- +(.2,0)
(-.4,.4) -- +(-.2,0)
(.4,-.4) -- +(.2,0);
\draw (-.4,.5) -- (.4,.5);
\draw (-.4,-.5) -- (.4,-.5);
\draw[white, double=black, double distance=0.4pt, ultra thick] 
(-.4,-.3) 
to[out=east, in=west] (-.2,-.3)
to[out=east, in=west] (.4,.3);
\draw[white, double=black, double distance=0.4pt, ultra thick] 
(-.4,.3) 
to[out=east, in=west] (-.2,.3)
to[out=east, in=west] (.4,-.3);
\draw (-.25,.5) -- (-.25,-.5);
\draw[fill=white] (-.3,.2) rectangle +(.1,.2);
\draw (-.3,.4) -- +(.1,-.2);
\draw[fill=white] (-.3,-.2) rectangle +(.1,-.2);
\draw (-.3,-.4) -- +(.1,.2);
\draw[fill=white] (.4,-.6) rectangle +(.1,.4);
\draw[fill=white] (-.4,-.6) rectangle +(-.1,.4);
\draw[fill=white] (.4,.6) rectangle +(.1,-.4);
\draw[fill=white] (-.4,.6) rectangle +(-.1,-.4);
\node at (.4,-.6)[right]{$\scriptstyle{n}$};
\node at (-.4,-.6)[left]{$\scriptstyle{n}$};
\node at (.4,.6)[right]{$\scriptstyle{n}$};
\node at (-.4,.6)[left]{$\scriptstyle{n}$};
\node at (0,-.6){$\scriptscriptstyle{1}$};
\node at (0,.6){$\scriptscriptstyle{1}$};
\node at (.4,.3)[below]{$\scriptscriptstyle{n-1}$};
\node at (.4,-.3)[above]{$\scriptscriptstyle{n-1}$};
}\,\Bigg\rangle_{\! 3}.
\end{align*}
We define clasped $A_2$ webs $\langle\sigma(k,l;n)\rangle_3$ as follows:
\[
\langle\sigma(k,l;n)\rangle_3
=
\Bigg\langle\,\tikz[baseline=-.6ex]{
\draw
(-.6,-.6) -- +(-.2,0)
(.6,.6) -- +(.2,0)
(-.6,.6) -- +(-.2,0)
(.6,-.6) -- +(.2,0);
\draw (-.6,.7) -- (.6,.7);
\draw (-.6,-.7) -- (.6,-.7);
\draw[white, double=black, double distance=0.4pt, ultra thick] 
(-.6,-.4)
to[out=east, in=west] (-.1,-.4)
to[out=east, in=west] (.6,.4);
\draw[white, double=black, double distance=0.4pt, ultra thick] 
(-.6,.4) 
to[out=east, in=west] (-.1,.4)
to[out=east, in=west] (.6,-.4);
\draw
(-.6,.3) to[out=east, in=north]
(-.5,.0) to[out=south, in=east] (-.6,-.3);
\draw
(.6,.3) to[out=west, in=north]
(.5,.0) to[out=south, in=west] (.6,-.3);
\draw (-.3,.7) -- (-.3,-.7);
\draw[fill=white] (-.1,.2) rectangle +(.1,.3);
\draw[fill=white] (-.1,-.2) rectangle +(.1,-.3);
\draw[fill=white] (.6,-.8) rectangle +(.1,.6);
\draw[fill=white] (-.6,-.8) rectangle +(-.1,.6);
\draw[fill=white] (.6,.8) rectangle +(.1,-.6);
\draw[fill=white] (-.6,.8) rectangle +(-.1,-.6);
\draw[fill=white] (-.4,.2) rectangle +(.2,.3);
\draw (-.4,.5) -- +(.2,-.3);
\draw[fill=white] (-.4,-.2) rectangle +(.2,-.3);
\draw (-.4,-.5) -- +(.2,.3);
\draw[fill=white] (-.4,.6) -- (-.2,.6) -- (-.2,.8) -- cycle;
\draw[fill=white] (-.4,-.6) -- (-.2,-.6) -- (-.2,-.8) -- cycle;
\node at (.6,-.8)[right]{$\scriptstyle{n}$};
\node at (-.6,-.8)[left]{$\scriptstyle{n}$};
\node at (.6,.8)[right]{$\scriptstyle{n}$};
\node at (-.6,.8)[left]{$\scriptstyle{n}$};
\node at (0,.6)[above]{$\scriptscriptstyle{k}$};
\node at (0,-.6)[below]{$\scriptscriptstyle{k}$};
\node at (-.4,.6)[above]{$\scriptscriptstyle{k}$};
\node at (-.4,-.6)[below]{$\scriptscriptstyle{k}$};
\node at (-.3,.0)[right]{$\scriptscriptstyle{k}$};
\node at (.5,.0)[right]{$\scriptscriptstyle{l}$};
\node at (-.5,.0)[left]{$\scriptscriptstyle{l}$};
}\,\Bigg\rangle_{\! 3}.
\]
The above calculation and Lemma~\ref{A2through} imply
\[
\langle\sigma(k,l;n)\rangle_3
=q^{-\frac{1}{3}(2(n-k-l)-1)}\langle\sigma(k,l+1;n)\rangle_3
-q^{\frac{1}{6}(2(n-k-l)-1)}\langle\sigma(k+1,l;n)\rangle_3.
\]
We can calculate the coefficient of $\langle\sigma(k,l;n)\rangle_3$ with $k+l=n$:
\[
\prod_{i=0}^{l-1}q^{-\frac{1}{3}(2(n-i)-1)}
\prod_{j=0}^{k-1}(-q^{\frac{1}{6}(2(n-j-l)-1)})
(\sum_{\lambda\in\mathcal{P}(k,l)}q^{\left|\lambda\right|})
=(-1)^kq^{\frac{-2n^2+3k^2}{6}}{n \choose k}_q,
\]
in a similar way to the proof of Proposition~\ref{coloredA1}~$(1)$ (see Figure~\ref{A2halfshift}).
\begin{figure}
\centering
\begin{tikzpicture}
\draw[->-=.5] (0,0) -- (2,0);
\node at (1,0)[below]{$\scriptstyle{-q^{\frac{2(n-k-l)-1}{6}}}$};
\draw[->-=.5] (2,0) -- (2,2);
\node at (2,1)[right]{$\scriptstyle{q^{-\frac{2(n-(k+1)-l)-1}{3}}}$};
\draw[->-=.5] (0,0) -- (0,2);
\node at (0,1)[left]{$\scriptstyle{q^{-\frac{2(n-k-l)-1}{3}}}$};
\draw[->-=.5] (0,2) -- (2,2);
\node at (1,2.2)[above]{$\scriptstyle{-q^{\frac{2(n-k-(l+1)-1)}{6}}}$};
\draw[->, very thick, magenta](.5,1.5) -- (1.5,.5);
\node at (1,1)[right]{$\times q$};
\node at (0,0)[below left]{$\scriptstyle{(k,l)}$};
\node at (2,0)[below right]{$\scriptstyle{(k+1,l)}$};
\node at (0,2)[above left]{$\scriptstyle{(k,l+1)}$};
\node at (2,2)[above right]{$\scriptstyle{(k,l)}$};
\fill 
(0,0) circle (1.2pt)
(2,0) circle (1.2pt)
(2,2) circle (1.2pt)
(0,2) circle (1.2pt);
\end{tikzpicture}
\caption{shift of a path}
\label{A2halfshift}
\end{figure}
\end{proof}

We also give a full twist and $m$ full twists formula for the $A_2$ bracket.
\begin{PROP}[a full twist formula]\label{A2full}
\begin{align*}
\Bigg\langle\,\tikz[baseline=-.6ex]{
\draw 
(-.5,.4) -- +(-.2,0)
(-.5,-.4) -- +(-.2,0);
\draw[->-=1, white, double=black, double distance=0.4pt, ultra thick] 
(-.5,-.4) to[out=east, in=west] (.0,.4);
\draw[-<-=1, white, double=black, double distance=0.4pt, ultra thick] 
(-.5,.4) to[out=east, in=west] (.0,-.4);
\draw[white, double=black, double distance=0.4pt, ultra thick] 
(0,-.4) to[out=east, in=west] (.5,.4);
\draw[white, double=black, double distance=0.4pt, ultra thick] 
(0,.4) to[out=east, in=west] (.5,-.4);
\draw[fill=white] (-.5,-.6) rectangle +(-.1,.4);
\draw[fill=white] (-.5,.6) rectangle +(-.1,-.4);
\node at (-.5,-.6)[left]{$\scriptstyle{n}$};
\node at (-.5,.6)[left]{$\scriptstyle{n}$};
\draw
(.5,-.4) -- +(.2,0)
(.5,.4) -- +(.2,0);
\draw[fill=white] (.5,-.6) rectangle +(.1,.4);
\draw[fill=white] (.5,.6) rectangle +(.1,-.4);
\node at (.5,-.6)[right]{$\scriptstyle{n}$};
\node at (.5,.6)[right]{$\scriptstyle{n}$};
}\,\Bigg\rangle_{\! 3}
=q^{\frac{n^2}{3}}
\sum_{k=0}^n
q^{k^2-n^2+k-n}
\frac{(q)_n}{(q)_k}{n \choose k}_{q}
\Bigg\langle\,\tikz[baseline=-.6ex]{
\draw
(-.4,.4) -- +(-.2,0)
(.4,-.4) -- +(.2,0)
(-.4,-.4) -- +(-.2,0)
(.4,.4) -- +(.2,0);
\draw[-<-=.5] (-.4,.5) -- (.4,.5);
\draw[->-=.5] (-.4,-.5) -- (.4,-.5);
\draw[-<-=.5] (-.4,.3) to[out=east, in=east] (-.4,-.3);
\draw[->-=.5] (.4,.3) to[out=west, in=west] (.4,-.3);
\draw[fill=white] (.4,-.6) rectangle +(.1,.4);
\draw[fill=white] (-.4,-.6) rectangle +(-.1,.4);
\draw[fill=white] (.4,.6) rectangle +(.1,-.4);
\draw[fill=white] (-.4,.6) rectangle +(-.1,-.4);
\node at (.4,-.6)[right]{$\scriptstyle{n}$};
\node at (-.4,-.6)[left]{$\scriptstyle{n}$};
\node at (.4,.6)[right]{$\scriptstyle{n}$};
\node at (-.4,.6)[left]{$\scriptstyle{n}$};
\node at (0,.5)[above]{$\scriptstyle{k}$};
\node at (0,-.5)[below]{$\scriptstyle{k}$};
\node at (-.2,0)[left]{$\scriptstyle{n-k}$};
\node at (.2,0)[right]{$\scriptstyle{n-k}$};
}\,\Bigg\rangle_{\! 3}.
\end{align*}
\end{PROP}

\begin{proof}
We prove the above formula in a similar way to the proof of Proposition~\ref{A1full}. 
\begin{align*}
\Bigg\langle\,\tikz[baseline=-.6ex]{
\draw 
(-.5,.4) -- +(-.2,0)
(.5,-.4) -- +(.2,0)
(-.5,-.4) -- +(-.2,0)
(.5,.4) -- +(.2,0);
\draw[->-=1, white, double=black, double distance=0.4pt, ultra thick] 
(-.5,-.4) to[out=east, in=west] (.0,.4);
\draw[-<-=1, white, double=black, double distance=0.4pt, ultra thick] 
(-.5,.4) to[out=east, in=west] (.0,-.4);
\draw[white, double=black, double distance=0.4pt, ultra thick] 
(0,-.4) to[out=east, in=west] (.5,.4);
\draw[white, double=black, double distance=0.4pt, ultra thick] 
(0,.4) to[out=east, in=west] (.5,-.4);
\draw[fill=white] (.5,-.6) rectangle +(.1,.4);
\draw[fill=white] (-.5,-.6) rectangle +(-.1,.4);
\draw[fill=white] (.5,.6) rectangle +(.1,-.4);
\draw[fill=white] (-.5,.6) rectangle +(-.1,-.4);
\node at (.5,-.6)[right]{$\scriptstyle{1}$};
\node at (-.5,-.6)[left]{$\scriptstyle{1}$};
\node at (.5,.6)[right]{$\scriptstyle{n}$};
\node at (-.5,.6)[left]{$\scriptstyle{n}$};
}\,\Bigg\rangle_{\! 3}
&=
\Bigg\langle\,\tikz[baseline=-.6ex]{
\draw 
(-.7,.4) -- +(-.2,0)
(.7,-.4) -- +(.2,0)
(-.7,-.4) -- +(-.2,0)
(.7,.4) -- +(.2,0);
\draw[white, double=black, double distance=0.4pt, ultra thick] 
(-.7,-.4) to[out=east, in=west] (.0,.4);
\draw[triple={[line width=1.6pt, white] in [line width=2.4pt, black] in [line width=5.6pt, white]}] 
(-.7,.4) to[out=east, in=west] (.0,-.4);
\draw[triple={[line width=1.6pt, white] in [line width=2.4pt, black] in [line width=5.6pt, white]}] 
(0,-.4) to[out=east, in=west] (.7,.4);
\draw[white, double=black, double distance=0.4pt, ultra thick] 
(0,.4) to[out=east, in=west] (.7,-.4);
\draw[fill=white] (.7,-.6) rectangle +(.1,.4);
\draw[fill=white] (-.7,-.6) rectangle +(-.1,.4);
\draw[fill=white] (.7,.6) rectangle +(.1,-.4);
\draw[fill=white] (-.7,.6) rectangle +(-.1,-.4);
\node at (.7,-.6)[right]{$\scriptstyle{1}$};
\node at (-.7,-.6)[left]{$\scriptstyle{1}$};
\node at (.7,.6)[right]{$\scriptstyle{n}$};
\node at (-.7,.6)[left]{$\scriptstyle{n}$};
\node at (.0,-.4)[above]{$\scriptstyle{1}$};
\node at (.0,-.4)[below]{$\scriptstyle{n-1}$};
}\,\Bigg\rangle_{\! 3}\\
&=-q^{\frac{1}{6}}
\Bigg\langle\,\tikz[baseline=-.6ex]{
\draw 
(-.5,.4) -- +(-.2,0)
(.5,-.4) -- +(.2,0)
(-.5,-.4) -- +(-.2,0)
(.5,.4) -- +(.2,0);
\draw[white, double=black, double distance=0.4pt, ultra thick] 
(-.5,-.4) to[out=east, in=west] (-.1,.0);
\draw[white, double=black, double distance=0.4pt, ultra thick] 
(-.5,.3) to[out=east, in=west] (.0,-.4);
\draw[white, double=black, double distance=0.4pt, ultra thick] 
(0,-.4) to[out=east, in=west] (.5,.3);
\draw[white, double=black, double distance=0.4pt, ultra thick] 
(-.1,.0) to[out=east, in=west] (.5,.5);
\draw[white, double=black, double distance=0.4pt, ultra thick] 
(-.5,.5) to[out=east, in=west] (-.1,.3);
\draw[white, double=black, double distance=0.4pt, ultra thick] 
(-.1,.3) to[out=east, in=west] (.5,-.4);
\draw (-.1,.0) -- (-.1,.3);
\draw[fill=white] (.5,-.6) rectangle +(.1,.4);
\draw[fill=white] (-.5,-.6) rectangle +(-.1,.4);
\draw[fill=white] (.5,.6) rectangle +(.1,-.4);
\draw[fill=white] (-.5,.6) rectangle +(-.1,-.4);
\node at (.5,-.6)[right]{$\scriptstyle{1}$};
\node at (-.5,-.6)[left]{$\scriptstyle{1}$};
\node at (.5,.6)[right]{$\scriptstyle{n}$};
\node at (-.5,.6)[left]{$\scriptstyle{n}$};
\node at (.0,-.4)[below]{$\scriptstyle{n-1}$};
}\,\Bigg\rangle_{\! 3}
+q^{-\frac{1}{3}}
\Bigg\langle\,\tikz[baseline=-.6ex]{
\draw 
(-.5,.4) -- +(-.2,0)
(.5,-.4) -- +(.2,0)
(-.5,-.4) -- +(-.2,0)
(.5,.4) -- +(.2,0);
\draw[white, double=black, double distance=0.4pt, ultra thick] 
(.5,.5) to[out=west, in=east] 
(-.1,.0) to[out=west, in=west] (-.1,.4);
\draw[white, double=black, double distance=0.4pt, ultra thick] 
(-.1,.4) 
to[out=east, in=west] (.5,-.4);
\draw[white, double=black, double distance=0.4pt, ultra thick] 
(-.5,.5) 
to[out=east, in=north] (-.3,.0)
to[out=south, in=east] (-.5,-.4);
\draw[white, double=black, double distance=0.4pt, ultra thick] 
(-.5,.3) to[out=east, in=west] (.0,-.4);
\draw[white, double=black, double distance=0.4pt, ultra thick] 
(0,-.4) to[out=east, in=west] (.5,.3);
\draw[fill=white] (.5,-.6) rectangle +(.1,.4);
\draw[fill=white] (-.5,-.6) rectangle +(-.1,.4);
\draw[fill=white] (.5,.6) rectangle +(.1,-.4);
\draw[fill=white] (-.5,.6) rectangle +(-.1,-.4);
\node at (.5,-.6)[right]{$\scriptstyle{1}$};
\node at (-.5,-.6)[left]{$\scriptstyle{1}$};
\node at (.5,.6)[right]{$\scriptstyle{n}$};
\node at (-.5,.6)[left]{$\scriptstyle{n}$};
\node at (.0,-.4)[below]{$\scriptstyle{n-1}$};
}\,\Bigg\rangle_{\! 3}\\
&=q^{\frac{1}{3}}
\Bigg\langle\,\tikz[baseline=-.6ex]{
\draw 
(-.5,.4) -- +(-.2,0)
(.5,-.4) -- +(.2,0)
(-.5,-.4) -- +(-.2,0)
(.5,.4) -- +(.2,0);
\draw[white, double=black, double distance=0.4pt, ultra thick] 
(-.5,-.4) to[out=east, in=west] 
(-.1,.0) to[out=east, in=west] (.1,.0);
\draw[white, double=black, double distance=0.4pt, ultra thick] 
(-.5,.3) to[out=east, in=west] (.0,-.4);
\draw[white, double=black, double distance=0.4pt, ultra thick] 
(0,-.4) to[out=east, in=west] (.5,.3);
\draw[white, double=black, double distance=0.4pt, ultra thick] 
(.1,.0) to[out=east, in=west] (.5,-.4);
\draw[white, double=black, double distance=0.4pt, ultra thick] 
(-.5,.5) to[out=east, in=west] (-.1,.3);
\draw[white, double=black, double distance=0.4pt, ultra thick] 
(-.1,.3) to[out=east, in=west]
(.1,.3) to[out=east, in=west] (.5,.5);
\draw (-.1,.0) -- (-.1,.3);
\draw (.1,.0) -- (.1,.3);
\draw[fill=white] (.5,-.6) rectangle +(.1,.4);
\draw[fill=white] (-.5,-.6) rectangle +(-.1,.4);
\draw[fill=white] (.5,.6) rectangle +(.1,-.4);
\draw[fill=white] (-.5,.6) rectangle +(-.1,-.4);
\node at (.5,-.6)[right]{$\scriptstyle{1}$};
\node at (-.5,-.6)[left]{$\scriptstyle{1}$};
\node at (.5,.6)[right]{$\scriptstyle{n}$};
\node at (-.5,.6)[left]{$\scriptstyle{n}$};
\node at (.0,-.4)[below]{$\scriptstyle{n-1}$};
}\,\Bigg\rangle_{\! 3}
-q^{-\frac{1}{6}}
\Bigg\langle\,\tikz[baseline=-.6ex]{
\draw 
(-.5,.4) -- +(-.2,0)
(.5,-.4) -- +(.2,0)
(-.5,-.4) -- +(-.2,0)
(.5,.4) -- +(.2,0);
\draw[white, double=black, double distance=0.4pt, ultra thick] 
(-.5,-.4) to[out=east, in=west] (-.1,.0);
\draw[white, double=black, double distance=0.4pt, ultra thick] 
(-.5,.3) to[out=east, in=west] (.0,-.4);
\draw[white, double=black, double distance=0.4pt, ultra thick] 
(0,-.4) to[out=east, in=west] (.5,.3);
\draw[white, double=black, double distance=0.4pt, ultra thick] 
(-.5,.5) to[out=east, in=west] (-.1,.3);
\draw[white, double=black, double distance=0.4pt, ultra thick] 
(.5,-.4) to[out=west, in=south]
(.3,.0) to[out=north, in=west] (.5,.5);
\draw (-.1,.0) -- (-.1,.3);
\draw (-.1,.0) 
to[out=east, in=south] (.1,.15)
to[out=north, in=east] (-.1,.3);
\draw[fill=white] (.5,-.6) rectangle +(.1,.4);
\draw[fill=white] (-.5,-.6) rectangle +(-.1,.4);
\draw[fill=white] (.5,.6) rectangle +(.1,-.4);
\draw[fill=white] (-.5,.6) rectangle +(-.1,-.4);
\node at (.5,-.6)[right]{$\scriptstyle{1}$};
\node at (-.5,-.6)[left]{$\scriptstyle{1}$};
\node at (.5,.6)[right]{$\scriptstyle{n}$};
\node at (-.5,.6)[left]{$\scriptstyle{n}$};
\node at (.0,-.4)[below]{$\scriptstyle{n-1}$};
}\,\Bigg\rangle_{\! 3}
+q^{-\frac{5}{3}}q^{-\frac{2}{3}(n-1)}
\Bigg\langle\,\tikz[baseline=-.6ex]{
\draw 
(-.5,.4) -- +(-.2,0)
(.5,-.4) -- +(.2,0)
(-.5,-.4) -- +(-.2,0)
(.5,.4) -- +(.2,0);
\draw (-.5,.5) -- (.5,.5);
\draw[white, double=black, double distance=0.4pt, ultra thick] 
(-.5,-.4) to[out=east, in=south] 
(-.3,.0) to[out=north, in=east] (-.5,.3);
\draw[white, double=black, double distance=0.4pt, ultra thick] 
(.5,.3) to[out=west, in=north] (.3,.0)
to[out=south, in=west] (.5,-.4);
\draw[fill=white] (.5,-.6) rectangle +(.1,.4);
\draw[fill=white] (-.5,-.6) rectangle +(-.1,.4);
\draw[fill=white] (.5,.6) rectangle +(.1,-.4);
\draw[fill=white] (-.5,.6) rectangle +(-.1,-.4);
\node at (.5,-.6)[right]{$\scriptstyle{1}$};
\node at (-.5,-.6)[left]{$\scriptstyle{1}$};
\node at (.5,.6)[right]{$\scriptstyle{n}$};
\node at (-.5,.6)[left]{$\scriptstyle{n}$};
\node at (.0,.5)[above]{$\scriptstyle{n-1}$};
\node at (-.3,.0)[left]{$\scriptstyle{1}$};
\node at (.3,.0)[right]{$\scriptstyle{1}$};
}\,\Bigg\rangle_{\! 3}\\
&=q^{\frac{1}{3}}
\Bigg\langle\,\tikz[baseline=-.6ex]{
\draw 
(-.5,.4) -- +(-.2,0)
(.5,-.4) -- +(.2,0)
(-.5,-.4) -- +(-.2,0)
(.5,.4) -- +(.2,0);
\draw (-.5,.5) -- (.5,.5);
\draw[white, double=black, double distance=0.4pt, ultra thick] 
(-.5,-.4) to[out=east, in=west] (.0,.3);
\draw[white, double=black, double distance=0.4pt, ultra thick] 
(-.5,.3) to[out=east, in=west] (.0,-.3);
\draw[white, double=black, double distance=0.4pt, ultra thick] 
(0,-.3) to[out=east, in=west] (.5,.3);
\draw[white, double=black, double distance=0.4pt, ultra thick] 
(0,.3) to[out=east, in=west] (.5,-.4);
\draw[fill=white] (.5,-.6) rectangle +(.1,.4);
\draw[fill=white] (-.5,-.6) rectangle +(-.1,.4);
\draw[fill=white] (.5,.6) rectangle +(.1,-.4);
\draw[fill=white] (-.5,.6) rectangle +(-.1,-.4);
\node at (.5,-.6)[right]{$\scriptstyle{1}$};
\node at (-.5,-.6)[left]{$\scriptstyle{1}$};
\node at (.5,.6)[right]{$\scriptstyle{n}$};
\node at (-.5,.6)[left]{$\scriptstyle{n}$};
\node at (.0,.5)[above]{$\scriptstyle{1}$};
\node at (.0,-.3)[below]{$\scriptstyle{n-1}$};
}\,\Bigg\rangle_{\! 3}
+q^{-\frac{2}{3}n}(q^{-1}-1)
\Bigg\langle\,\tikz[baseline=-.6ex]{
\draw 
(-.5,.4) -- +(-.2,0)
(.5,-.4) -- +(.2,0)
(-.5,-.4) -- +(-.2,0)
(.5,.4) -- +(.2,0);
\draw (-.5,.5) -- (.5,.5);
\draw[white, double=black, double distance=0.4pt, ultra thick] 
(-.5,-.4) to[out=east, in=south] 
(-.3,.0) to[out=north, in=east] (-.5,.3);
\draw[white, double=black, double distance=0.4pt, ultra thick] 
(.5,.3) to[out=west, in=north] (.3,.0)
to[out=south, in=west] (.5,-.4);
\draw[fill=white] (.5,-.6) rectangle +(.1,.4);
\draw[fill=white] (-.5,-.6) rectangle +(-.1,.4);
\draw[fill=white] (.5,.6) rectangle +(.1,-.4);
\draw[fill=white] (-.5,.6) rectangle +(-.1,-.4);
\node at (.5,-.6)[right]{$\scriptstyle{1}$};
\node at (-.5,-.6)[left]{$\scriptstyle{1}$};
\node at (.5,.6)[right]{$\scriptstyle{n}$};
\node at (-.5,.6)[left]{$\scriptstyle{n}$};
\node at (.0,.5)[above]{$\scriptstyle{n-1}$};
\node at (-.3,.0)[left]{$\scriptstyle{1}$};
\node at (.3,.0)[right]{$\scriptstyle{1}$};
}\,\Bigg\rangle_{\! 3}.
\end{align*}
Then, 
we obtain
\begin{equation}\label{A2n1full}
q^{-\frac{1}{3}(n-i)}
\Bigg\langle\,\tikz[baseline=-.6ex]{
\draw 
(-.5,.4) -- +(-.2,0)
(.5,-.4) -- +(.2,0)
(-.5,-.4) -- +(-.2,0)
(.5,.4) -- +(.2,0);
\draw (-.5,.5) -- (.5,.5);
\draw[white, double=black, double distance=0.4pt, ultra thick] 
(-.5,-.4) to[out=east, in=west] (.0,.3);
\draw[white, double=black, double distance=0.4pt, ultra thick] 
(-.5,.3) to[out=east, in=west] (.0,-.3);
\draw[white, double=black, double distance=0.4pt, ultra thick] 
(0,-.3) to[out=east, in=west] (.5,.3);
\draw[white, double=black, double distance=0.4pt, ultra thick] 
(0,.3) to[out=east, in=west] (.5,-.4);
\draw[fill=white] (.5,-.6) rectangle +(.1,.4);
\draw[fill=white] (-.5,-.6) rectangle +(-.1,.4);
\draw[fill=white] (.5,.6) rectangle +(.1,-.4);
\draw[fill=white] (-.5,.6) rectangle +(-.1,-.4);
\node at (.5,-.6)[right]{$\scriptstyle{1}$};
\node at (-.5,-.6)[left]{$\scriptstyle{1}$};
\node at (.5,.6)[right]{$\scriptstyle{n}$};
\node at (-.5,.6)[left]{$\scriptstyle{n}$};
\node at (.0,.5)[above]{$\scriptstyle{i}$};
\node at (.0,-.3)[below]{$\scriptstyle{n-i}$};
}\,\Bigg\rangle_{\! 3}
-q^{-\frac{1}{3}(n-i-1)}
\Bigg\langle\,\tikz[baseline=-.6ex]{
\draw 
(-.5,.4) -- +(-.2,0)
(.5,-.4) -- +(.2,0)
(-.5,-.4) -- +(-.2,0)
(.5,.4) -- +(.2,0);
\draw (-.5,.5) -- (.5,.5);
\draw[white, double=black, double distance=0.4pt, ultra thick] 
(-.5,-.4) to[out=east, in=west] (.0,.3);
\draw[white, double=black, double distance=0.4pt, ultra thick] 
(-.5,.3) to[out=east, in=west] (.0,-.3);
\draw[white, double=black, double distance=0.4pt, ultra thick] 
(0,-.3) to[out=east, in=west] (.5,.3);
\draw[white, double=black, double distance=0.4pt, ultra thick] 
(0,.3) to[out=east, in=west] (.5,-.4);
\draw[fill=white] (.5,-.6) rectangle +(.1,.4);
\draw[fill=white] (-.5,-.6) rectangle +(-.1,.4);
\draw[fill=white] (.5,.6) rectangle +(.1,-.4);
\draw[fill=white] (-.5,.6) rectangle +(-.1,-.4);
\node at (.5,-.6)[right]{$\scriptstyle{1}$};
\node at (-.5,-.6)[left]{$\scriptstyle{1}$};
\node at (.5,.6)[right]{$\scriptstyle{n}$};
\node at (-.5,.6)[left]{$\scriptstyle{n}$};
\node at (.0,.5)[above]{$\scriptstyle{i}$};
\node at (.0,-.3)[below]{$\scriptstyle{n-i}$};
}\,\Bigg\rangle_{\! 3}
=(1-q)q^{-n-1}q^{i}
\Bigg\langle\,\tikz[baseline=-.6ex]{
\draw 
(-.5,.4) -- +(-.2,0)
(.5,-.4) -- +(.2,0)
(-.5,-.4) -- +(-.2,0)
(.5,.4) -- +(.2,0);
\draw (-.5,.5) -- (.5,.5);
\draw[white, double=black, double distance=0.4pt, ultra thick] 
(-.5,-.4) to[out=east, in=south] 
(-.3,.0) to[out=north, in=east] (-.5,.3);
\draw[white, double=black, double distance=0.4pt, ultra thick] 
(.5,.3) to[out=west, in=north] (.3,.0)
to[out=south, in=west] (.5,-.4);
\draw[fill=white] (.5,-.6) rectangle +(.1,.4);
\draw[fill=white] (-.5,-.6) rectangle +(-.1,.4);
\draw[fill=white] (.5,.6) rectangle +(.1,-.4);
\draw[fill=white] (-.5,.6) rectangle +(-.1,-.4);
\node at (.5,-.6)[right]{$\scriptstyle{1}$};
\node at (-.5,-.6)[left]{$\scriptstyle{1}$};
\node at (.5,.6)[right]{$\scriptstyle{n}$};
\node at (-.5,.6)[left]{$\scriptstyle{n}$};
\node at (.0,.5)[above]{$\scriptstyle{n-1}$};
\node at (-.3,.0)[left]{$\scriptstyle{1}$};
\node at (.3,.0)[right]{$\scriptstyle{1}$};
}\,\Bigg\rangle_{\! 3}
\end{equation}
for $i=0,1,\dots,n-1$.
By taking the sum of both sides of (\ref{A2n1full}) for $i=0,1,\dots,n-1$,
\[
\Bigg\langle\,\tikz[baseline=-.6ex]{
\draw 
(-.5,.4) -- +(-.2,0)
(.5,-.4) -- +(.2,0)
(-.5,-.4) -- +(-.2,0)
(.5,.4) -- +(.2,0);
\draw[->-=1, white, double=black, double distance=0.4pt, ultra thick] 
(-.5,-.4) to[out=east, in=west] (.0,.3);
\draw[-<-=1, white, double=black, double distance=0.4pt, ultra thick] 
(-.5,.4) to[out=east, in=west] (.0,-.3);
\draw[white, double=black, double distance=0.4pt, ultra thick] 
(0,-.3) to[out=east, in=west] (.5,.4);
\draw[white, double=black, double distance=0.4pt, ultra thick] 
(0,.3) to[out=east, in=west] (.5,-.4);
\draw[fill=white] (.5,-.6) rectangle +(.1,.4);
\draw[fill=white] (-.5,-.6) rectangle +(-.1,.4);
\draw[fill=white] (.5,.6) rectangle +(.1,-.4);
\draw[fill=white] (-.5,.6) rectangle +(-.1,-.4);
\node at (.5,-.6)[right]{$\scriptstyle{1}$};
\node at (-.5,-.6)[left]{$\scriptstyle{1}$};
\node at (.5,.6)[right]{$\scriptstyle{n}$};
\node at (-.5,.6)[left]{$\scriptstyle{n}$};
}\,\Bigg\rangle_{\! 3}
=q^{\frac{n}{3}}
\Bigg\langle\,\tikz[baseline=-.6ex]{
\draw 
(-.5,.4) -- +(-.2,0)
(.5,-.4) -- +(.2,0)
(-.5,-.4) -- +(-.2,0)
(.5,.4) -- +(.2,0);
\draw (-.5,.4) -- (.5,.4);
\draw (-.5,-.4) -- (.5,-.4);
\draw[fill=white] (.5,-.6) rectangle +(.1,.4);
\draw[fill=white] (-.5,-.6) rectangle +(-.1,.4);
\draw[fill=white] (.5,.6) rectangle +(.1,-.4);
\draw[fill=white] (-.5,.6) rectangle +(-.1,-.4);
\node at (.5,-.6)[right]{$\scriptstyle{1}$};
\node at (-.5,-.6)[left]{$\scriptstyle{1}$};
\node at (.5,.6)[right]{$\scriptstyle{n}$};
\node at (-.5,.6)[left]{$\scriptstyle{n}$};
}\,\Bigg\rangle_{\! 3}
+(1-q^n)q^{-\frac{2n}{3}-1}
\Bigg\langle\,\tikz[baseline=-.6ex]{
\draw 
(-.5,.4) -- +(-.2,0)
(.5,-.4) -- +(.2,0)
(-.5,-.4) -- +(-.2,0)
(.5,.4) -- +(.2,0);
\draw (-.5,.5) -- (.5,.5);
\draw[white, double=black, double distance=0.4pt, ultra thick] 
(-.5,-.4) to[out=east, in=south] 
(-.3,.0) to[out=north, in=east] (-.5,.3);
\draw[white, double=black, double distance=0.4pt, ultra thick] 
(.5,.3) to[out=west, in=north] (.3,.0)
to[out=south, in=west] (.5,-.4);
\draw[fill=white] (.5,-.6) rectangle +(.1,.4);
\draw[fill=white] (-.5,-.6) rectangle +(-.1,.4);
\draw[fill=white] (.5,.6) rectangle +(.1,-.4);
\draw[fill=white] (-.5,.6) rectangle +(-.1,-.4);
\node at (.5,-.6)[right]{$\scriptstyle{1}$};
\node at (-.5,-.6)[left]{$\scriptstyle{1}$};
\node at (.5,.6)[right]{$\scriptstyle{n}$};
\node at (-.5,.6)[left]{$\scriptstyle{n}$};
\node at (.0,.5)[above]{$\scriptstyle{n-1}$};
\node at (-.3,.0)[left]{$\scriptstyle{1}$};
\node at (.3,.0)[right]{$\scriptstyle{1}$};
}\,\Bigg\rangle_{\! 3}.
\]
From the above equation,
\[
\Bigg\langle\,\tikz[baseline=-.6ex]{
\draw 
(-.5,.4) -- +(-.2,0)
(.5,-.4) -- +(.2,0)
(-.5,-.4) -- +(-.2,0)
(.5,.4) -- +(.2,0);
\draw[->-=1, white, double=black, double distance=0.4pt, ultra thick] 
(-.5,-.4) to[out=east, in=west] (.0,.3);
\draw[-<-=1, white, double=black, double distance=0.4pt, ultra thick] 
(-.5,.4) to[out=east, in=west] (.0,-.3);
\draw[white, double=black, double distance=0.4pt, ultra thick] 
(0,-.3) to[out=east, in=west] (.5,.4);
\draw[white, double=black, double distance=0.4pt, ultra thick] 
(0,.3) to[out=east, in=west] (.5,-.4);
\draw[fill=white] (.5,-.6) rectangle +(.1,.4);
\draw[fill=white] (-.5,-.6) rectangle +(-.1,.4);
\draw[fill=white] (.5,.6) rectangle +(.1,-.4);
\draw[fill=white] (-.5,.6) rectangle +(-.1,-.4);
\node at (.5,-.6)[right]{$\scriptstyle{1}$};
\node at (-.5,-.6)[left]{$\scriptstyle{1}$};
\node at (.5,.6)[right]{$\scriptstyle{n}$};
\node at (-.5,.6)[left]{$\scriptstyle{n}$};
}\,\Bigg\rangle_{\! 3}
=q^{\frac{j}{3}}\Bigg\langle\,\tikz[baseline=-.6ex]{
\draw 
(-.5,.4) -- +(-.2,0)
(.5,-.4) -- +(.2,0)
(-.5,-.4) -- +(-.2,0)
(.5,.4) -- +(.2,0);
\draw (-.5,-.5) -- (.5,-.5);
\draw[white, double=black, double distance=0.4pt, ultra thick] 
(-.5,-.3) to[out=east, in=west] (.0,.4);
\draw[white, double=black, double distance=0.4pt, ultra thick] 
(-.5,.4) to[out=east, in=west] (.0,-.4);
\draw[white, double=black, double distance=0.4pt, ultra thick] 
(0,-.4) to[out=east, in=west] (.5,.4);
\draw[white, double=black, double distance=0.4pt, ultra thick] 
(0,.4) to[out=east, in=west] (.5,-.3);
\draw[fill=white] (.5,-.6) rectangle +(.1,.4);
\draw[fill=white] (-.5,-.6) rectangle +(-.1,.4);
\draw[fill=white] (.5,.6) rectangle +(.1,-.4);
\draw[fill=white] (-.5,.6) rectangle +(-.1,-.4);
\node at (.5,-.6)[right]{$\scriptstyle{i}$};
\node at (-.5,-.6)[left]{$\scriptstyle{i}$};
\node at (.5,.6)[right]{$\scriptstyle{j}$};
\node at (-.5,.6)[left]{$\scriptstyle{j}$};
\node at (.0,.4)[above]{$\scriptstyle{i-1}$};
\node at (.0,-.5)[below]{$\scriptstyle{1}$};
}\,\Bigg\rangle_{\! 3}
+(1-q^j)q^{-\frac{2}{3}(i+j)-\frac{1}{3}}
\Bigg\langle\,\tikz[baseline=-.6ex]{
\draw 
(-.6,.4) -- +(-.2,0)
(.6,-.4) -- +(.2,0)
(-.6,-.4) -- +(-.2,0)
(.6,.4) -- +(.2,0);
\draw[white, double=black, double distance=0.4pt, ultra thick] 
(-.6,-.5) to[out=east, in=west] (.0,.4);
\draw[white, double=black, double distance=0.4pt, ultra thick] 
(-.6,.5) to[out=east, in=west] (.0,-.4);
\draw[white, double=black, double distance=0.4pt, ultra thick] 
(0,-.4) to[out=east, in=west] (.6,.5);
\draw[white, double=black, double distance=0.4pt, ultra thick] 
(0,.4) to[out=east, in=west] (.6,-.5);
\draw (-.6,.3) to[out=east, in=north] 
(-.4,.0) to[out=south, in=east] (-.6,-.3);
\draw (.6,.3) to[out=west, in=north] 
(.4,.0) to[out=south, in=west](.6,-.3);
\draw[fill=white] (.6,-.6) rectangle +(.1,.4);
\draw[fill=white] (-.6,-.6) rectangle +(-.1,.4);
\draw[fill=white] (.6,.6) rectangle +(.1,-.4);
\draw[fill=white] (-.6,.6) rectangle +(-.1,-.4);
\node at (.6,-.6)[right]{$\scriptstyle{i}$};
\node at (-.6,-.6)[left]{$\scriptstyle{i}$};
\node at (.6,.6)[right]{$\scriptstyle{j}$};
\node at (-.6,.6)[left]{$\scriptstyle{j}$};
\node at (.0,.4)[above]{$\scriptstyle{i-1}$};
\node at (.0,-.4)[below]{$\scriptstyle{j-1}$};
\node at (-.4,.0)[left]{$\scriptstyle{1}$};
\node at (.4,.0)[right]{$\scriptstyle{1}$};
}\,\Bigg\rangle_{\! 2}
\]
for any non-negative integers $i$ and $j$.
We set a clasped $A_2$ web $\langle\sigma^2(k,l;n)\rangle_3$ as
\[
\langle\sigma^2(k,l;n)\rangle_3=
\Bigg\langle\,\tikz[baseline=-.6ex]{
\draw 
(-.6,.4) -- +(-.2,0)
(.6,-.4) -- +(.2,0)
(-.6,-.4) -- +(-.2,0)
(.6,.4) -- +(.2,0);
\draw[->-=.5]
(-.6,-.5) -- (.6,-.5);
\draw[->-=1, white, double=black, double distance=0.4pt, ultra thick] 
(-.6,-.4) to[out=east, in=west] (.0,.4);
\draw[-<-=1, white, double=black, double distance=0.4pt, ultra thick] 
(-.6,.5) to[out=east, in=west] (.0,-.3);
\draw[white, double=black, double distance=0.4pt, ultra thick] 
(0,-.3) to[out=east, in=west] (.6,.5);
\draw[white, double=black, double distance=0.4pt, ultra thick] 
(0,.4) to[out=east, in=west] (.6,-.4);
\draw[-<-=.5] (-.6,.3) to[out=east, in=north] 
(-.5,.0) to[out=south, in=east] (-.6,-.3);
\draw[->-=.5] (.6,.3) to[out=west, in=north] 
(.5,.0) to[out=south, in=west](.6,-.3);
\draw[fill=white] (.6,-.6) rectangle +(.1,.4);
\draw[fill=white] (-.6,-.6) rectangle +(-.1,.4);
\draw[fill=white] (.6,.6) rectangle +(.1,-.4);
\draw[fill=white] (-.6,.6) rectangle +(-.1,-.4);
\node at (.6,-.6)[right]{$\scriptstyle{n}$};
\node at (-.6,-.6)[left]{$\scriptstyle{n}$};
\node at (.6,.6)[right]{$\scriptstyle{n}$};
\node at (-.6,.6)[left]{$\scriptstyle{n}$};
\node at (.0,.4)[above]{$\scriptstyle{n-k-l}$};
\node at (.0,-.6)[above]{$\scriptstyle{n-l}$};
\node at (.0,-.5)[below]{$\scriptstyle{k}$};
\node at (-.5,.0)[left]{$\scriptstyle{l}$};
\node at (.5,.0)[right]{$\scriptstyle{l}$};
}\,\Bigg\rangle_{\! 3},
\]
and obtain
\begin{equation}\label{A2fullres}
\langle\sigma^2(k,l;n)\rangle_3
=q^{\frac{n-l}{3}}\langle\sigma^2(k+1,l;n)\rangle_3
+(1-q^{n-l})q^{-\frac{4}{3}(n-l)}q^{\frac{2k-1}{3}}\langle\sigma^2(k,l+1;n)\rangle_3
\end{equation}
for non-negative integers $k$ and $l$ such that $k+l\leq n$.
We make $\langle\sigma^2(k,l;n)\rangle_3$ correspond to lattice point $(k,l)$ in a similar way to the proof of Proposition~\ref{A1full}.
The coefficient of $\langle\sigma^2(k,l;n)\rangle_3$ with $k+l=n$ can be calculated by using the resolution (\ref{A2fullres}) and Figure~\ref{A2fullshift}:
\[
\prod_{i=1}^{l-1}(1-q^{n-i})q^{-\frac{4}{3}(n-i)}q^{-\frac{1}{3}}\prod_{j=0}^{k-1}q^{\frac{n-l}{3}}
\sum_{\lambda\in\mathcal{P}(k,l)}q^{\left|\lambda\right|}
=q^{\frac{n^2}{3}}q^{k^2-n^2+k-n}\frac{(q;q)_n}{(q;q)_k}{n \choose k}_q.
\]
\begin{figure}
\centering
\begin{tikzpicture}
\draw[->-=.5] (0,0) -- (2,0);
\node at (1,0)[below]{$\scriptstyle{q^{\frac{n-l}{3}}}$};
\draw[->-=.5] (2,0) -- (2,2);
\node at (2,1)[right]{$\scriptstyle{(1-q^{n-l})q^{-\frac{4}{3}(n-l)}q^{\frac{2(k+1)-1}{3}}}$};
\draw[->-=.5] (0,0) -- (0,2);
\node at (0,1)[left]{$\scriptstyle{(1-q^{n-l})q^{-\frac{4}{3}(n-l)}q^{\frac{2k-1}{3}}}$};
\draw[->-=.5] (0,2) -- (2,2);
\draw[->, very thick, magenta](.5,1.5) -- (1.5,.5);
\node at (1,1)[right]{$\times q$};
\node at (1,2)[above]{$\scriptstyle{q^{\frac{n-(l+1)}{3}}}$};
\node at (0,0)[below left]{$\scriptstyle{(k,l)}$};
\node at (2,0)[below right]{$\scriptstyle{(k+1,l)}$};
\node at (0,2)[above left]{$\scriptstyle{(k,l+1)}$};
\node at (2,2)[above right]{$\scriptstyle{(k,l)}$};
\fill 
(0,0) circle (1.2pt)
(2,0) circle (1.2pt)
(2,2) circle (1.2pt)
(0,2) circle (1.2pt);
\end{tikzpicture}
\caption{shift of a path}
\label{A2fullshift}
\end{figure}
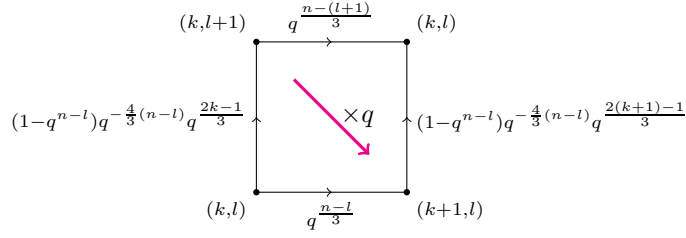
\end{proof}

\begin{LEM}\label{A2slide}
\[
\Bigg\langle\,\tikz[baseline=-.6ex]{
\draw 
(-.5,.4) -- +(-.2,0)
(-.5,-.4) -- +(-.2,0);
\draw[-<-=.7] (.0,-.5) -- +(.4,0);
\draw[->-=.7] (.0,.5) -- +(.4,0);
\draw[->-=.2, white, double=black, double distance=0.4pt, ultra thick] 
(-.5,-.4) to[out=east, in=west] (.0,.4);
\draw[-<-=.2, white, double=black, double distance=0.4pt, ultra thick] 
(-.5,.4) to[out=east, in=west] (.0,-.4);
\draw[->-=.5] (.1,.4) to[out=east, in=east] (.1,-.4);
\draw[fill=white] (.0,-.6) rectangle +(.1,.4);
\draw[fill=white] (-.5,-.6) rectangle +(-.1,.4);
\draw[fill=white] (.0,.6) rectangle +(.1,-.4);
\draw[fill=white] (-.5,.6) rectangle +(-.1,-.4);
\node at (.4,-.6)[right]{$\scriptstyle{k}$};
\node at (-.5,-.6)[left]{$\scriptstyle{n}$};
\node at (.4,.6)[right]{$\scriptstyle{k}$};
\node at (-.5,.6)[left]{$\scriptstyle{n}$};
\node at (.3,0)[right]{$\scriptstyle{n-k}$};
\node at (-.1,.4)[above]{$\scriptstyle{n}$};
\node at (-.1,-.4)[below]{$\scriptstyle{n}$};
}\,\Bigg\rangle_{\! 3}
=q^{-\frac{n^2-k^2+3n-3k}{3}}\Bigg\langle\,\tikz[baseline=-.6ex]{
\draw (-.4,.4) -- +(-.2,0);
\draw (.4,-.4) -- +(.2,0);
\draw (-.4,-.4) -- +(-.2,0);
\draw (.4,.4) -- +(.2,0);
\draw[->-=.3, white, double=black, double distance=0.4pt, ultra thick] 
(-.4,-.5) to[out=east, in=west] (.4,.4);
\draw[-<-=.3, white, double=black, double distance=0.4pt, ultra thick] 
(-.4,.5) to[out=east, in=west] (.4,-.4);
\draw[-<-=.5] (-.4,.3) to[out=east, in=east] (-.4,-.3);
\draw[fill=white] (.4,-.6) rectangle +(.1,.4);
\draw[fill=white] (-.4,-.6) rectangle +(-.1,.4);
\draw[fill=white] (.4,.6) rectangle +(.1,-.4);
\draw[fill=white] (-.4,.6) rectangle +(-.1,-.4);
\node at (.4,-.6)[right]{$\scriptstyle{k}$};
\node at (-.4,-.6)[left]{$\scriptstyle{n}$};
\node at (.4,.6)[right]{$\scriptstyle{k}$};
\node at (-.4,.6)[left]{$\scriptstyle{n}$};
\node at (-.4,0)[left]{$\scriptstyle{n-k}$};
}\,\Bigg\rangle_{\! 3}
\]
\end{LEM}
\begin{proof}
We can prove it in a similar way to the proof of Lemma~\ref{A1slide}.
\end{proof}

\begin{THM}[$m$ full twists formula for the $A_2$ bracket]\label{A2mfull}
\begin{align*}
\Bigg\langle\,\tikz[baseline=-.6ex]{
\begin{scope}[xshift=-1cm]
\draw 
(-.5,.4) -- +(-.2,0)
(-.5,-.4) -- +(-.2,0);
\draw[->-=1, white, double=black, double distance=0.4pt, ultra thick] 
(-.5,-.4) to[out=east, in=west] (.0,.4);
\draw[-<-=1, white, double=black, double distance=0.4pt, ultra thick] 
(-.5,.4) to[out=east, in=west] (.0,-.4);
\draw[white, double=black, double distance=0.4pt, ultra thick] 
(0,-.4) to[out=east, in=west] (.5,.4);
\draw[white, double=black, double distance=0.4pt, ultra thick] 
(0,.4) to[out=east, in=west] (.5,-.4);
\draw[fill=white] (-.5,-.6) rectangle +(-.1,.4);
\draw[fill=white] (-.5,.6) rectangle +(-.1,-.4);
\node at (-.5,-.6)[left]{$\scriptstyle{n}$};
\node at (-.5,.6)[left]{$\scriptstyle{n}$};
\end{scope}
\node at (.0,.0){$\cdots$};
\node at (.0,-.4)[below]{$\scriptstyle{m\text{ full twists}}$};
\begin{scope}[xshift=1cm]
\draw
(.5,-.4) -- +(.2,0)
(.5,.4) -- +(.2,0);
\draw[->-=1, white, double=black, double distance=0.4pt, ultra thick] 
(-.5,-.4) to[out=east, in=west] (.0,.4);
\draw[-<-=1, white, double=black, double distance=0.4pt, ultra thick] 
(-.5,.4) to[out=east, in=west] (.0,-.4);
\draw[white, double=black, double distance=0.4pt, ultra thick] 
(0,-.4) to[out=east, in=west] (.5,.4);
\draw[white, double=black, double distance=0.4pt, ultra thick] 
(0,.4) to[out=east, in=west] (.5,-.4);
\draw[fill=white] (.5,-.6) rectangle +(.1,.4);
\draw[fill=white] (.5,.6) rectangle +(.1,-.4);
\node at (.5,-.6)[right]{$\scriptstyle{n}$};
\node at (.5,.6)[right]{$\scriptstyle{n}$};
\end{scope}
}\,\Bigg\rangle_{\! 3}
&=q^{-\frac{2m}{3}(n^2+3n)}
\sum_{0\leq k_m\leq \cdots\leq k_1\leq n}
q^{n-k_m}
q^{\sum_{i=1}^{m}(k_i^2+2k_i)}\\
&\qquad\times\frac{(q)_n}{(q)_{k_m}}
{n \choose k_1',k_2',\dots,k_m',k_m}_{q}
\Bigg\langle\,\tikz[baseline=-.6ex]{
\draw
(-.4,.4) -- +(-.2,0)
(.4,-.4) -- +(.2,0)
(-.4,-.4) -- +(-.2,0)
(.4,.4) -- +(.2,0);
\draw[-<-=.5] (-.4,.5) -- (.4,.5);
\draw[->-=.5] (-.4,-.5) -- (.4,-.5);
\draw[-<-=.5] (-.4,.3) to[out=east, in=east] (-.4,-.3);
\draw[->-=.5] (.4,.3) to[out=west, in=west] (.4,-.3);
\draw[fill=white] (.4,-.6) rectangle +(.1,.4);
\draw[fill=white] (-.4,-.6) rectangle +(-.1,.4);
\draw[fill=white] (.4,.6) rectangle +(.1,-.4);
\draw[fill=white] (-.4,.6) rectangle +(-.1,-.4);
\node at (.4,-.6)[right]{$\scriptstyle{n}$};
\node at (-.4,-.6)[left]{$\scriptstyle{n}$};
\node at (.4,.6)[right]{$\scriptstyle{n}$};
\node at (-.4,.6)[left]{$\scriptstyle{n}$};
\node at (0,.5)[above]{$\scriptstyle{k_m}$};
\node at (0,-.5)[below]{$\scriptstyle{k_m}$};
\node at (-.2,0)[left]{$\scriptstyle{n-k_m}$};
\node at (.2,0)[right]{$\scriptstyle{n-k_m}$};
}\,\Bigg\rangle_{\! 3},
\end{align*}
where $k_i, k_i'$ are integers such that $k_0=n$, $k_{i+1}'=k_i-k_{i+1}$ for $i=0,1,\dots,m-1$.
\end{THM}
\begin{proof}
We can prove it in the same way as proof of Proposition~\ref{A1mhalf} by use of Proposition~\ref{A2full} and Lemma~\ref{A2slide}.
\end{proof}

\section{Bubble skein expansion formulas}
In this section,
we consider the bubble skein expansion formula.
In the case of clasped $A_1$ web spaces,
Hajij~\cite{Hajij14} proved the formula.
First, 
we rewrite coefficients of the formula by using quantum binomial coefficients.
Next, 
we give the bubble skein expansion formula for clasped $A_2$ web spaces.
\subsection{The Kauffman bracket bubble skein expansion formula}
Let $\Delta_n$ denote the coefficients of 
$
\Big\langle\,\tikz[baseline=-.6ex]{
\draw (0,0) circle [radius=.3];
\draw[fill=white] (.1,-.05) rectangle (.5,.05);
\node at (.3,0)[above right]{$\scriptstyle{n}$};
}\,\Big\rangle_{\! 2}
$, that is, 
$\Delta_n=(-1)^n\left[n+1\right]$.
\begin{THM}[The Kauffman bracket bubble skein expansion formula by Hajij~\cite{Hajij14}]\label{A1bubble}
Let $m,n\geq k,l$ be positive integers.
\[
\Bigg\langle\,\tikz[baseline=-.6ex, scale=0.8]{
\draw (-.4,.5) -- +(-.2,0);
\draw (.4,-.5) -- +(.2,0);
\draw (-.4,-.5) -- +(-.2,0);
\draw (.4,.5) -- +(.2,0);
\draw (-.4,.5) -- (0,.5);
\draw (0,.5) -- (.4,.5);
\draw (-.4,-.5) -- (0,-.5);
\draw (0,-.5) -- (.4,-.5);
\draw (.05,.3) to[out=east, in=east] (.05,.-.3);
\draw (-.05,.3) to[out=west, in=west] (-.05,-.3);
\draw[fill=white] (-.4,.3) rectangle +(-.1,.3);
\draw[fill=white] (.4,.3) rectangle +(.1,.3);
\draw[fill=white] (-.4,-.3) rectangle +(-.1,-.3);
\draw[fill=white] (.4,-.3) rectangle +(.1,-.3);
\draw[fill=white] (-.05,.2) rectangle +(.1,.4);
\draw[fill=white] (-.05,-.2) rectangle +(.1,-.4);
\node at (.4,-.5)[below right]{$\scriptstyle{m-l}$};
\node at (-.4,-.5)[below left]{$\scriptstyle{m-k}$};
\node at (.4,.5)[above right]{$\scriptstyle{n-l}$};
\node at (-.4,.5)[above left]{$\scriptstyle{n-k}$};
\node at (-.2,0)[left]{$\scriptstyle{k}$};
\node at (.2,0)[right]{$\scriptstyle{l}$};
\node at (0,-.5)[below]{$\scriptstyle{m}$};
\node at (0,.5)[above]{$\scriptstyle{n}$};
}\,\Bigg\rangle_{\! 2}
=
\sum_{t=\max\{k, l\}}^{\min\{k+l, n, m\}}
(-1)^{t-k-l}
\frac{{n\brack t}{m\brack t}{t\brack k}{t\brack l}{n+m-t+1\brack n+m-k-l+1}}{{n\brack k}{m\brack k}{n\brack l}{m\brack l}}
\Bigg\langle\,\tikz[baseline=-.6ex, scale=0.8]{
\draw (-.4,.4) -- +(-.2,0);
\draw (.4,-.4) -- +(.2,0);
\draw (-.4,-.4) -- +(-.2,0);
\draw (.4,.4) -- +(.2,0);
\draw (-.4,.5) -- (0,.5);
\draw (0,.5) -- (.4,.5);
\draw (-.4,-.5) -- (0,-.5);
\draw (0,-.5) -- (.4,-.5);
\draw (-.4,.3) to[out=east, in=east] (-.4,.-.3);
\draw (.4,.3) to[out=west, in=west] (.4,-.3);
\draw[fill=white] (-.4,.2) rectangle +(-.1,.4);
\draw[fill=white] (.4,.2) rectangle +(.1,.4);
\draw[fill=white] (-.4,-.2) rectangle +(-.1,-.4);
\draw[fill=white] (.4,-.2) rectangle +(.1,-.4);
\node at (.4,-.6)[right]{$\scriptstyle{m-l}$};
\node at (-.4,-.6)[left]{$\scriptstyle{m-k}$};
\node at (.4,.6)[right]{$\scriptstyle{n-l}$};
\node at (-.4,.6)[left]{$\scriptstyle{n-k}$};
\node at (-.2,0)[left]{$\scriptstyle{t-k}$};
\node at (.2,0)[right]{$\scriptstyle{t-l}$};
\node at (0,-.5)[below]{$\scriptstyle{m-t}$};
\node at (0,.5)[above]{$\scriptstyle{n-t}$};
}\,\Bigg\rangle_{\! 2}
\]
\end{THM}
\begin{proof}
Hajij gave this formula in~\cite{Hajij14} as follows. 
Let $M,N,M',N'$ be non-negative integers and $k\geq l \geq 1$ positive integers.
Then, 
\begin{align*}
&\Bigg\langle\,\tikz[baseline=-.6ex, scale=0.8]{
\draw (-.4,.5) -- +(-.2,0);
\draw (.4,-.5) -- +(.2,0);
\draw (-.4,-.5) -- +(-.2,0);
\draw (.4,.5) -- +(.2,0);
\draw (-.4,.5) -- (0,.5);
\draw (0,.5) -- (.4,.5);
\draw (-.4,-.5) -- (0,-.5);
\draw (0,-.5) -- (.4,-.5);
\draw (.05,.3) to[out=east, in=east] (.05,.-.3);
\draw (-.05,.3) to[out=west, in=west] (-.05,-.3);
\draw[fill=white] (-.4,.3) rectangle +(-.1,.3);
\draw[fill=white] (.4,.3) rectangle +(.1,.3);
\draw[fill=white] (-.4,-.3) rectangle +(-.1,-.3);
\draw[fill=white] (.4,-.3) rectangle +(.1,-.3);
\draw[fill=white] (-.05,.2) rectangle +(.1,.4);
\draw[fill=white] (-.05,-.2) rectangle +(.1,-.4);
\node at (.4,-.5)[below right]{$\scriptstyle{M'}$};
\node at (-.4,-.5)[below left]{$\scriptstyle{M}$};
\node at (.4,.5)[above right]{$\scriptstyle{N'}$};
\node at (-.4,.5)[above left]{$\scriptstyle{N}$};
\node at (-.2,0)[left]{$\scriptstyle{k}$};
\node at (.2,0)[right]{$\scriptstyle{l}$};
}\,\Bigg\rangle_{\! 2}\\
&=
\sum_{i=0}^{\min\{M,N,l\}}
(-1)^{i(i-l)}
\frac{(\prod_{j=0}^{l-i-1}\Delta_{k-j-1}\Delta_{M+N+k-i-j})(\prod_{j=0}^{i-1}\Delta_{N-j-1}\Delta_{M-j-1})}{\prod_{j=0}^{l-1}\Delta_{N+k-t-1}\Delta_{M+k-t-1}}{l\brack i}
\Bigg\langle\,\tikz[baseline=-.6ex, scale=0.8]{
\draw (-.4,.4) -- +(-.2,0);
\draw (.4,-.4) -- +(.2,0);
\draw (-.4,-.4) -- +(-.2,0);
\draw (.4,.4) -- +(.2,0);
\draw (-.4,.5) -- (0,.5);
\draw (0,.5) -- (.4,.5);
\draw (-.4,-.5) -- (0,-.5);
\draw (0,-.5) -- (.4,-.5);
\draw (-.4,.3) to[out=east, in=east] (-.4,.-.3);
\draw (.4,.3) to[out=west, in=west] (.4,-.3);
\draw[fill=white] (-.4,.2) rectangle +(-.1,.4);
\draw[fill=white] (.4,.2) rectangle +(.1,.4);
\draw[fill=white] (-.4,-.2) rectangle +(-.1,-.4);
\draw[fill=white] (.4,-.2) rectangle +(.1,-.4);
\node at (.4,-.6)[right]{$\scriptstyle{M'}$};
\node at (-.4,-.6)[left]{$\scriptstyle{M}$};
\node at (.4,.6)[right]{$\scriptstyle{N'}$};
\node at (-.4,.6)[left]{$\scriptstyle{N}$};
\node at (-.2,0)[left]{$\scriptstyle{i}$};
\node at (.2,0)[right]{$\scriptstyle{k-l+i}$};
}\,\Bigg\rangle_{\! 2}.
\end{align*}
We can easily rewrite coefficients of the above formula and obtain our formula.
\end{proof}
\subsection{The $A_2$ bracket bubble skein expansion formula}
\begin{THM}[The $A_2$ bracket bubble skein expansion formula]\label{A2bubble}
\[
\Bigg\langle\,\tikz[baseline=-.6ex, scale=0.8]{
\draw (-.4,.5) -- +(-.2,0);
\draw (.4,-.5) -- +(.2,0);
\draw (-.4,-.5) -- +(-.2,0);
\draw (.4,.5) -- +(.2,0);
\draw[-<-=.5] (-.4,.5) -- (0,.5);
\draw[-<-=.5] (0,.5) -- (.4,.5);
\draw[->-=.5] (-.4,-.5) -- (0,-.5);
\draw[->-=.5] (0,-.5) -- (.4,-.5);
\draw[-<-=.5] (.05,.3) to[out=east, in=east] (.05,.-.3);
\draw[->-=.5] (-.05,.3) to[out=west, in=west] (-.05,-.3);
\draw[fill=white] (-.4,.3) rectangle +(-.1,.3);
\draw[fill=white] (.4,.3) rectangle +(.1,.3);
\draw[fill=white] (-.4,-.3) rectangle +(-.1,-.3);
\draw[fill=white] (.4,-.3) rectangle +(.1,-.3);
\draw[fill=white] (-.05,.2) rectangle +(.1,.4);
\draw[fill=white] (-.05,-.2) rectangle +(.1,-.4);
\node at (.4,-.5)[below right]{$\scriptstyle{m-l}$};
\node at (-.4,-.5)[below left]{$\scriptstyle{m-k}$};
\node at (.4,.5)[above right]{$\scriptstyle{n-l}$};
\node at (-.4,.5)[above left]{$\scriptstyle{n-k}$};
\node at (-.2,0)[left]{$\scriptstyle{k}$};
\node at (.2,0)[right]{$\scriptstyle{l}$};
\node at (0,-.6)[below]{$\scriptstyle{m}$};
\node at (0,.6)[above]{$\scriptstyle{n}$};
}\,\Bigg\rangle_{\! 3}
=
\sum_{t=\max\{k, l\}}^{\min\{k+l, n, m\}}
\frac{{n\brack t}{m\brack t}{t\brack k}{t\brack l}{n+m-t+2\brack n+m-k-l+2}}{{n\brack k}{m\brack k}{n\brack l}{m\brack l}}
\Bigg\langle\,\tikz[baseline=-.6ex, scale=0.8]{
\draw (-.4,.4) -- +(-.2,0);
\draw (.4,-.4) -- +(.2,0);
\draw (-.4,-.4) -- +(-.2,0);
\draw (.4,.4) -- +(.2,0);
\draw[-<-=.5] (-.4,.5) -- (.4,.5);
\draw[->-=.5] (-.4,-.5)  -- (.4,-.5);
\draw[-<-=.5] (-.4,.3) to[out=east, in=east] (-.4,.-.3);
\draw[->-=.5] (.4,.3) to[out=west, in=west] (.4,-.3);
\draw[fill=white] (-.4,.2) rectangle +(-.1,.4);
\draw[fill=white] (.4,.2) rectangle +(.1,.4);
\draw[fill=white] (-.4,-.2) rectangle +(-.1,-.4);
\draw[fill=white] (.4,-.2) rectangle +(.1,-.4);
\node at (.4,-.6)[right]{$\scriptstyle{m-l}$};
\node at (-.4,-.6)[left]{$\scriptstyle{m-k}$};
\node at (.4,.6)[right]{$\scriptstyle{n-l}$};
\node at (-.4,.6)[left]{$\scriptstyle{n-k}$};
\node at (-.2,0)[left]{$\scriptstyle{t-k}$};
\node at (.2,0)[right]{$\scriptstyle{t-l}$};
\node at (0,-.5)[below]{$\scriptstyle{m-t}$};
\node at (0,.5)[above]{$\scriptstyle{n-t}$};
}\,\Bigg\rangle_{\! 3}
\]
\end{THM}
Firstly, 
we calculate the $A_2$ web appearing in the right-hand side of Theorem~\ref{A2bubble}, we call it an $A_2$ bubble skein element, when $k=1$ or $l=1$.
\begin{LEM}\label{bubblelem}\ 
\begin{enumerate}
\item 
$
\Bigg\langle\,\tikz[baseline=-.6ex, scale=0.8]{
\draw (-.4,.5) -- +(-.2,0);
\draw (.4,-.5) -- +(.2,0);
\draw (-.4,-.5) -- +(-.2,0);
\draw (.4,.5) -- +(.2,0);
\draw[-<-=.5] (-.4,.5) -- (0,.5);
\draw[-<-=.5] (0,.5) -- (.4,.5);
\draw[->-=.5] (-.4,-.5) -- (0,-.5);
\draw[->-=.5] (0,-.5) -- (.4,-.5);
\draw[-<-=.5] (.05,.3) to[out=east, in=east] (.05,.-.3);
\draw[->-=.5] (-.05,.3) to[out=west, in=west] (-.05,-.3);
\draw[fill=white] (-.4,.3) rectangle +(-.1,.3);
\draw[fill=white] (.4,.3) rectangle +(.1,.3);
\draw[fill=white] (-.4,-.3) rectangle +(-.1,-.3);
\draw[fill=white] (.4,-.3) rectangle +(.1,-.3);
\draw[fill=white] (-.05,.2) rectangle +(.1,.4);
\draw[fill=white] (-.05,-.2) rectangle +(.1,-.4);
\node at (.4,-.5)[below right]{$\scriptstyle{m-l}$};
\node at (-.4,-.5)[below left]{$\scriptstyle{m-1}$};
\node at (.4,.5)[above right]{$\scriptstyle{n-l}$};
\node at (-.4,.5)[above left]{$\scriptstyle{n-1}$};
\node at (-.2,0)[left]{$\scriptstyle{1}$};
\node at (.2,0)[right]{$\scriptstyle{l}$};
\node at (0,-.6)[below]{$\scriptstyle{m}$};
\node at (0,.6)[above]{$\scriptstyle{n}$};
}\,\Bigg\rangle_{\! 3}
=
\frac{\left[n+m-l+2\right]\left[l\right]}{\left[n\right]\left[m\right]}
\Bigg\langle\,\tikz[baseline=-.6ex, scale=0.8]{
\draw (-.4,.4) -- +(-.2,0);
\draw (-.4,-.4) -- +(-.2,0);
\draw[-<-=.5] (-.4,.5) -- (.0,.5);
\draw[->-=.5] (-.4,-.5)  -- (.0,-.5);
\draw[-<-=.5] (-.4,.3) to[out=east, in=east] (-.4,.-.3);
\draw[fill=white] (-.4,.2) rectangle +(-.1,.4);
\draw[fill=white] (-.4,-.2) rectangle +(-.1,-.4);
\node at (-.4,-.6)[left]{$\scriptstyle{m-1}$};
\node at (-.4,.6)[left]{$\scriptstyle{n-1}$};
\node at (-.2,0)[left]{$\scriptstyle{l-1}$};
\node at (0,-.5)[below]{$\scriptstyle{m-l}$};
\node at (0,.5)[above]{$\scriptstyle{n-l}$};
}\,\Bigg\rangle_{\! 3}
+
\frac{\left[n-l\right]\left[m-l\right]}{\left[n\right]\left[m\right]}
\Bigg\langle\,\tikz[baseline=-.6ex, scale=0.8]{
\draw (-.4,.4) -- +(-.2,0);
\draw (.4,-.4) -- +(.2,0);
\draw (-.4,-.4) -- +(-.2,0);
\draw (.4,.4) -- +(.2,0);
\draw[-<-=.5] (-.4,.5) -- (.4,.5);
\draw[->-=.5] (-.4,-.5)  -- (.4,-.5);
\draw[-<-=.5] (-.4,.3) to[out=east, in=east] (-.4,.-.3);
\draw[->-=.5] (.4,.3) to[out=west, in=west] (.4,-.3);
\draw[fill=white] (-.4,.2) rectangle +(-.1,.4);
\draw[fill=white] (.4,.2) rectangle +(.1,.4);
\draw[fill=white] (-.4,-.2) rectangle +(-.1,-.4);
\draw[fill=white] (.4,-.2) rectangle +(.1,-.4);
\node at (.4,-.6)[right]{$\scriptstyle{m-l}$};
\node at (-.4,-.6)[left]{$\scriptstyle{m-1}$};
\node at (.4,.6)[right]{$\scriptstyle{n-l}$};
\node at (-.4,.6)[left]{$\scriptstyle{n-1}$};
\node at (-.2,0)[left]{$\scriptstyle{l}$};
\node at (.2,0)[right]{$\scriptstyle{1}$};
\node at (0,-.5)[below]{$\scriptstyle{m-l-1}$};
\node at (0,.5)[above]{$\scriptstyle{n-l-1}$};
}\,\Bigg\rangle_{\! 3}
$
\item
$
\Bigg\langle\,\tikz[baseline=-.6ex, scale=0.8]{
\draw (-.4,.5) -- +(-.2,0);
\draw (.4,-.5) -- +(.2,0);
\draw (-.4,-.5) -- +(-.2,0);
\draw (.4,.5) -- +(.2,0);
\draw[-<-=.5] (-.4,.5) -- (0,.5);
\draw[-<-=.5] (0,.5) -- (.4,.5);
\draw[->-=.5] (-.4,-.5) -- (0,-.5);
\draw[->-=.5] (0,-.5) -- (.4,-.5);
\draw[-<-=.5] (.05,.3) to[out=east, in=east] (.05,.-.3);
\draw[->-=.5] (-.05,.3) to[out=west, in=west] (-.05,-.3);
\draw[fill=white] (-.4,.3) rectangle +(-.1,.3);
\draw[fill=white] (.4,.3) rectangle +(.1,.3);
\draw[fill=white] (-.4,-.3) rectangle +(-.1,-.3);
\draw[fill=white] (.4,-.3) rectangle +(.1,-.3);
\draw[fill=white] (-.05,.2) rectangle +(.1,.4);
\draw[fill=white] (-.05,-.2) rectangle +(.1,-.4);
\node at (.4,-.5)[below right]{$\scriptstyle{m-l}$};
\node at (-.4,-.5)[below left]{$\scriptstyle{m-1}$};
\node at (.4,.5)[above right]{$\scriptstyle{n-l}$};
\node at (-.4,.5)[above left]{$\scriptstyle{n-1}$};
\node at (-.2,0)[left]{$\scriptstyle{k}$};
\node at (.2,0)[right]{$\scriptstyle{1}$};
\node at (0,-.6)[below]{$\scriptstyle{m}$};
\node at (0,.6)[above]{$\scriptstyle{n}$};
}\,\Bigg\rangle_{\! 3}
=
\frac{\left[n+m-l+2\right]\left[l\right]}{\left[n\right]\left[m\right]}
\Bigg\langle\,\tikz[baseline=-.6ex, scale=0.8, xscale=-1]{
\draw (-.4,.4) -- +(-.2,0);
\draw (-.4,-.4) -- +(-.2,0);
\draw[-<-=.5] (-.4,.5) -- (.0,.5);
\draw[->-=.5] (-.4,-.5)  -- (.0,-.5);
\draw[-<-=.5] (-.4,.3) to[out=east, in=east] (-.4,.-.3);
\draw[fill=white] (-.4,.2) rectangle +(-.1,.4);
\draw[fill=white] (-.4,-.2) rectangle +(-.1,-.4);
\node at (-.4,-.6)[right]{$\scriptstyle{m-1}$};
\node at (-.4,.6)[right]{$\scriptstyle{n-1}$};
\node at (-.2,0)[right]{$\scriptstyle{k-1}$};
\node at (0,-.5)[below]{$\scriptstyle{m-k}$};
\node at (0,.5)[above]{$\scriptstyle{n-k}$};
}\,\Bigg\rangle_{\! 3}
+
\frac{\left[n-l\right]\left[m-l\right]}{\left[n\right]\left[m\right]}
\Bigg\langle\,\tikz[baseline=-.6ex, scale=0.8]{
\draw (-.4,.4) -- +(-.2,0);
\draw (.4,-.4) -- +(.2,0);
\draw (-.4,-.4) -- +(-.2,0);
\draw (.4,.4) -- +(.2,0);
\draw[-<-=.5] (-.4,.5) -- (.4,.5);
\draw[->-=.5] (-.4,-.5)  -- (.4,-.5);
\draw[-<-=.5] (-.4,.3) to[out=east, in=east] (-.4,.-.3);
\draw[->-=.5] (.4,.3) to[out=west, in=west] (.4,-.3);
\draw[fill=white] (-.4,.2) rectangle +(-.1,.4);
\draw[fill=white] (.4,.2) rectangle +(.1,.4);
\draw[fill=white] (-.4,-.2) rectangle +(-.1,-.4);
\draw[fill=white] (.4,-.2) rectangle +(.1,-.4);
\node at (.4,-.6)[right]{$\scriptstyle{m-1}$};
\node at (-.4,-.6)[left]{$\scriptstyle{m-k}$};
\node at (.4,.6)[right]{$\scriptstyle{n-1}$};
\node at (-.4,.6)[left]{$\scriptstyle{n-k}$};
\node at (-.2,0)[left]{$\scriptstyle{1}$};
\node at (.2,0)[right]{$\scriptstyle{k}$};
\node at (0,-.5)[below]{$\scriptstyle{m-k-1}$};
\node at (0,.5)[above]{$\scriptstyle{n-k-1}$};
}\,\Bigg\rangle_{\! 3}
$
\end{enumerate}
\end{LEM}
\begin{proof}
We only prove $(1)$.
For any integers $k$ and $l$ such that $1\leq k,l\leq \min\{k,l\}$,
\begin{align*}
\Bigg\langle\,\tikz[baseline=-.6ex, scale=0.8]{
\draw (-.4,.5) -- +(-.2,0);
\draw (.4,-.5) -- +(.2,0);
\draw (-.4,-.5) -- +(-.2,0);
\draw (.4,.5) -- +(.2,0);
\draw[-<-=.5] (-.4,.5) -- (0,.5);
\draw[-<-=.5] (0,.5) -- (.4,.5);
\draw[->-=.5] (-.4,-.5) -- (0,-.5);
\draw[->-=.5] (0,-.5) -- (.4,-.5);
\draw[-<-=.5] (.05,.3) to[out=east, in=east] (.05,.-.3);
\draw[->-=.5] (-.05,.3) to[out=west, in=west] (-.05,-.3);
\draw[fill=white] (-.4,.3) rectangle +(-.1,.3);
\draw[fill=white] (.4,.3) rectangle +(.1,.3);
\draw[fill=white] (-.4,-.3) rectangle +(-.1,-.3);
\draw[fill=white] (.4,-.3) rectangle +(.1,-.3);
\draw[fill=white] (-.05,.2) rectangle +(.1,.4);
\draw[fill=white] (-.05,-.2) rectangle +(.1,-.4);
\node at (.4,-.5)[below right]{$\scriptstyle{m-l}$};
\node at (-.4,-.5)[below left]{$\scriptstyle{m-k}$};
\node at (.4,.5)[above right]{$\scriptstyle{n-l}$};
\node at (-.4,.5)[above left]{$\scriptstyle{n-k}$};
\node at (-.2,0)[left]{$\scriptstyle{k}$};
\node at (.2,0)[right]{$\scriptstyle{l}$};
\node at (0,-.6)[below]{$\scriptstyle{m}$};
\node at (0,.6)[above]{$\scriptstyle{n}$};
}\,\Bigg\rangle_{\! 3}
&=\frac{\left[m+2\right]}{\left[m\right]}
\Bigg\langle\,\tikz[baseline=-.6ex, scale=0.8]{
\draw (-.4,.5) -- +(-.2,0);
\draw (.4,-.5) -- +(.2,0);
\draw (-.4,-.5) -- +(-.2,0);
\draw (.4,.5) -- +(.2,0);
\draw[-<-=.5] (-.4,.5) -- (0,.5);
\draw[-<-=.5] (0,.5) -- (.4,.5);
\draw[->-=.5] (-.4,-.5) -- (0,-.5);
\draw[->-=.5] (0,-.5) -- (.4,-.5);
\draw[-<-=.5] (.05,.3) to[out=east, in=east] (.05,.-.3);
\draw[->-=.5] (-.05,.3) to[out=west, in=west] (-.05,-.3);
\draw[fill=white] (-.4,.3) rectangle +(-.1,.3);
\draw[fill=white] (.4,.3) rectangle +(.1,.3);
\draw[fill=white] (-.4,-.3) rectangle +(-.1,-.3);
\draw[fill=white] (.4,-.3) rectangle +(.1,-.3);
\draw[fill=white] (-.05,.2) rectangle +(.1,.4);
\draw[fill=white] (-.05,-.2) rectangle +(.1,-.4);
\node at (-.2,0)[left]{$\scriptstyle{k-1}$};
\node at (.2,0)[right]{$\scriptstyle{l-1}$};
\node at (0,-.6)[below]{$\scriptstyle{m}$};
\node at (0,.6)[above]{$\scriptstyle{n}$};
}\,\Bigg\rangle_{\! 3}
-\frac{\left[n-1\right]}{\left[n\right]}
\Bigg\langle\,\tikz[baseline=-.6ex, scale=0.8]{
\draw (-.4,.5) -- +(-.2,0);
\draw (.4,-.5) -- +(.2,0);
\draw (-.4,-.5) -- +(-.2,0);
\draw (.4,.5) -- +(.2,0);
\draw[-<-=.5] (-.4,.5) -- (0,.5);
\draw (0,.5) -- (.4,.5);
\draw[->-=.5] (-.4,-.5) -- (0,-.5);
\draw (0,-.5) -- (.4,-.5);
\draw[->-=.5] (-.05,.3) to[out=west, in=west] (-.05,-.4);
\draw[fill=white] (-.4,.3) rectangle +(-.1,.3);
\draw[fill=white] (-.4,-.3) rectangle +(-.1,-.3);
\draw[fill=white] (-.05,.2) rectangle +(.1,.4);
\node at (-.4,-.5)[below]{$\scriptstyle{m-k}$};
\node at (-.4,.5)[above]{$\scriptstyle{n-k}$};
\node at (-.2,.0)[left]{$\scriptstyle{k-1}$};
\node at (.3,-.1){$\scriptstyle{1}$};
\begin{scope}[xshift=1.2cm]
\draw (-.4,.5) -- +(-.2,0);
\draw (.4,-.5) -- +(.2,0);
\draw (-.4,-.5) -- +(-.2,0);
\draw (.4,.5) -- +(.2,0);
\draw (-.4,.5) -- (0,.5);
\draw[-<-=.5] (0,.5) -- (.4,.5);
\draw (-.4,-.5) -- (0,-.5);
\draw[->-=.5] (0,-.5) -- (.4,-.5);
\draw[-<-=.5] (.05,.3) to[out=east, in=east] (.05,-.4);
\draw[fill=white] (.4,.3) rectangle +(.1,.3);
\draw[fill=white] (.4,-.3) rectangle +(.1,-.3);
\draw[fill=white] (-.05,.2) rectangle +(.1,.4);
\node at (.4,-.5)[below]{$\scriptstyle{m-l}$};
\node at (.4,.5)[above]{$\scriptstyle{n-l}$};
\node at (-.3,-.1){$\scriptstyle{1}$};
\node at (.2,0)[right]{$\scriptstyle{l-1}$};
\end{scope}
\draw (.05,.3) -- (1.15,.3);
\draw (-.05,-.4) -- (1.25,-.4);
\draw (.75,.3) -- (.75,-.3) -- (.65,-.3);
\draw (.45,.3) -- (.45,-.3) -- (.55,-.3);
\draw[fill=white] (.65,-.2) rectangle +(-.1,-.4);
\node at (.6,.5)[above]{$\scriptstyle{n-1}$};
\node at (.6,-.5)[below]{$\scriptstyle{m}$};
}\,\Bigg\rangle_{\! 3}\\
&=\left(\frac{\left[m+2\right]}{\left[m\right]}-\left[2\right]\frac{\left[n-1\right]}{\left[n\right]}\right)
\Bigg\langle\,\tikz[baseline=-.6ex, scale=0.8]{
\draw (-.4,.5) -- +(-.2,0);
\draw (.4,-.5) -- +(.2,0);
\draw (-.4,-.5) -- +(-.2,0);
\draw (.4,.5) -- +(.2,0);
\draw[-<-=.5] (-.4,.5) -- (0,.5);
\draw[-<-=.5] (0,.5) -- (.4,.5);
\draw[->-=.5] (-.4,-.5) -- (0,-.5);
\draw[->-=.5] (0,-.5) -- (.4,-.5);
\draw[-<-=.5] (.05,.3) to[out=east, in=east] (.05,.-.3);
\draw[->-=.5] (-.05,.3) to[out=west, in=west] (-.05,-.3);
\draw[fill=white] (-.4,.3) rectangle +(-.1,.3);
\draw[fill=white] (.4,.3) rectangle +(.1,.3);
\draw[fill=white] (-.4,-.3) rectangle +(-.1,-.3);
\draw[fill=white] (.4,-.3) rectangle +(.1,-.3);
\draw[fill=white] (-.05,.2) rectangle +(.1,.4);
\draw[fill=white] (-.05,-.2) rectangle +(.1,-.4);
\node at (-.2,0)[left]{$\scriptstyle{k-1}$};
\node at (.2,0)[right]{$\scriptstyle{l-1}$};
\node at (0,-.6)[below]{$\scriptstyle{m}$};
\node at (0,.6)[above]{$\scriptstyle{n}$};
}\,\Bigg\rangle_{\! 3}\\
&\qquad+\frac{\left[n-1\right]\left[m-1\right]}{\left[n\right]\left[m\right]}
\Bigg\langle\,\tikz[baseline=-.6ex, scale=0.8]{
\draw (-.4,.5) -- +(-.2,0);
\draw (.4,-.5) -- +(.2,0);
\draw (-.4,-.5) -- +(-.2,0);
\draw (.4,.5) -- +(.2,0);
\draw[-<-=.5] (-.4,.5) -- (0,.5);
\draw (0,.5) -- (.4,.5);
\draw[->-=.5] (-.4,-.5) -- (0,-.5);
\draw (0,-.5) -- (.4,-.5);
\draw[->-=.5] (-.05,.3) to[out=west, in=west] (-.05,-.3);
\draw[fill=white] (-.4,.3) rectangle +(-.1,.3);
\draw[fill=white] (-.4,-.3) rectangle +(-.1,-.3);
\draw[fill=white] (-.05,.2) rectangle +(.1,.4);
\draw[fill=white] (-.05,-.2) rectangle +(.1,-.4);
\node at (-.4,-.5)[below]{$\scriptstyle{m-k}$};
\node at (-.4,.5)[above]{$\scriptstyle{n-k}$};
\node at (-.2,.0)[left]{$\scriptstyle{k-1}$};
\node at (.3,.0){$\scriptstyle{1}$};
\begin{scope}[xshift=1.2cm]
\draw (-.4,.5) -- +(-.2,0);
\draw (.4,-.5) -- +(.2,0);
\draw (-.4,-.5) -- +(-.2,0);
\draw (.4,.5) -- +(.2,0);
\draw (-.4,.5) -- (0,.5);
\draw[-<-=.5] (0,.5) -- (.4,.5);
\draw (-.4,-.5) -- (0,-.5);
\draw[->-=.5] (0,-.5) -- (.4,-.5);
\draw[-<-=.5] (.05,.3) to[out=east, in=east] (.05,-.3);
\draw[fill=white] (.4,.3) rectangle +(.1,.3);
\draw[fill=white] (.4,-.3) rectangle +(.1,-.3);
\draw[fill=white] (-.05,.2) rectangle +(.1,.4);
\draw[fill=white] (-.05,-.2) rectangle +(.1,-.4);
\node at (.4,-.5)[below]{$\scriptstyle{m-l}$};
\node at (.4,.5)[above]{$\scriptstyle{n-l}$};
\node at (-.3,.0){$\scriptstyle{1}$};
\node at (.2,0)[right]{$\scriptstyle{l-1}$};
\end{scope}
\draw (.05,.3) -- (1.15,.3);
\draw (.05,-.3) -- (1.15,-.3);
\draw (.75,.3) -- (.75,-.3);
\draw (.45,.3) -- (.45,-.3);
\node at (.6,.5)[above]{$\scriptstyle{n-1}$};
\node at (.6,-.5)[below]{$\scriptstyle{m}$};
}\,\Bigg\rangle_{\! 3}\\
&=\left(\frac{\left[n\right]\left[m+2\right]-\left[2\right]\left[n-1\right]\left[m\right]+\left[n-1\right]\left[m-1\right]}{\left[n\right]\left[m\right]}\right)
\Bigg\langle\,\tikz[baseline=-.6ex, scale=0.8]{
\draw (-.4,.5) -- +(-.2,0);
\draw (.4,-.5) -- +(.2,0);
\draw (-.4,-.5) -- +(-.2,0);
\draw (.4,.5) -- +(.2,0);
\draw[-<-=.5] (-.4,.5) -- (0,.5);
\draw[-<-=.5] (0,.5) -- (.4,.5);
\draw[->-=.5] (-.4,-.5) -- (0,-.5);
\draw[->-=.5] (0,-.5) -- (.4,-.5);
\draw[-<-=.5] (.05,.3) to[out=east, in=east] (.05,.-.3);
\draw[->-=.5] (-.05,.3) to[out=west, in=west] (-.05,-.3);
\draw[fill=white] (-.4,.3) rectangle +(-.1,.3);
\draw[fill=white] (.4,.3) rectangle +(.1,.3);
\draw[fill=white] (-.4,-.3) rectangle +(-.1,-.3);
\draw[fill=white] (.4,-.3) rectangle +(.1,-.3);
\draw[fill=white] (-.05,.2) rectangle +(.1,.4);
\draw[fill=white] (-.05,-.2) rectangle +(.1,-.4);
\node at (.4,-.5)[below right]{$\scriptstyle{m-l}$};
\node at (-.4,-.5)[below left]{$\scriptstyle{m-k}$};
\node at (.4,.5)[above right]{$\scriptstyle{n-l}$};
\node at (-.4,.5)[above left]{$\scriptstyle{n-k}$};
\node at (-.2,0)[left]{$\scriptstyle{k-1}$};
\node at (.2,0)[right]{$\scriptstyle{l-1}$};
\node at (0,-.6)[below]{$\scriptstyle{m-1}$};
\node at (0,.6)[above]{$\scriptstyle{n-1}$};
}\,\Bigg\rangle_{\! 3}\\
&\qquad+\frac{\left[n-1\right]\left[m-1\right]}{\left[n\right]\left[m\right]}
\Bigg\langle\,\tikz[baseline=-.6ex, scale=0.8]{
\draw (-.4,.5) -- +(-.2,0);
\draw (.4,-.5) -- +(.2,0);
\draw (-.4,-.5) -- +(-.2,0);
\draw (.4,.5) -- +(.2,0);
\draw[-<-=.5] (-.4,.5) -- (0,.5);
\draw[-<-=.5] (0,.5) -- (.4,.5);
\draw[->-=.5] (-.4,-.5) -- (0,-.5);
\draw[->-=.5] (0,-.5) -- (.4,-.5);
\draw[-<-=.5] (.05,.3) to[out=east, in=east] (.05,.-.3);
\draw[->-=.5] (-.05,.3) to[out=west, in=west] (-.05,-.3);
\draw[fill=white] (-.4,.3) rectangle +(-.1,.3);
\draw[fill=white] (-.4,-.3) rectangle +(-.1,-.3);
\draw[fill=white] (-.05,.2) rectangle +(.1,.4);
\draw[fill=white] (-.05,-.2) rectangle +(.1,-.4);
\node at (-.4,-.5)[below]{$\scriptstyle{m-k}$};
\node at (-.4,.5)[above]{$\scriptstyle{n-k}$};
\node at (-.2,0)[left]{$\scriptstyle{k-1}$};
\node at (.2,0)[right]{$\scriptstyle{1}$};
\begin{scope}[xshift=1.2cm]
\draw (-.4,.5) -- +(-.2,0);
\draw (.4,-.5) -- +(.2,0);
\draw (-.4,-.5) -- +(-.2,0);
\draw (.4,.5) -- +(.2,0);
\draw[-<-=.5] (-.4,.5) -- (0,.5);
\draw[-<-=.5] (0,.5) -- (.4,.5);
\draw[->-=.5] (-.4,-.5) -- (0,-.5);
\draw[->-=.5] (0,-.5) -- (.4,-.5);
\draw[-<-=.5] (.05,.3) to[out=east, in=east] (.05,.-.3);
\draw[->-=.5] (-.05,.3) to[out=west, in=west] (-.05,-.3);
\draw[fill=white] (.4,.3) rectangle +(.1,.3);
\draw[fill=white] (.4,-.3) rectangle +(.1,-.3);
\draw[fill=white] (-.05,.2) rectangle +(.1,.4);
\draw[fill=white] (-.05,-.2) rectangle +(.1,-.4);
\node at (.4,-.5)[below]{$\scriptstyle{m-l}$};
\node at (.4,.5)[above]{$\scriptstyle{n-l}$};
\node at (-.2,0)[left]{$\scriptstyle{1}$};
\node at (.2,0)[right]{$\scriptstyle{l-1}$};
\end{scope}
\node at (.6,.5)[above]{$\scriptstyle{n-2}$};
\node at (.6,-.5)[below]{$\scriptstyle{m-2}$};
}\,\Bigg\rangle_{\! 3}.
\end{align*}
We use the definition of $A_2$ clasp and Lemma~\ref{A2clasplem} in the first equality.
By Lemma~\ref{qinteger}~(1) and (2), 
we obtain $\left[2\right]\left[m\right]=\left[m+1\right]+\left[m-1\right]$ and $\left[n\right]\left[m+2\right]-\left[n-1\right]\left[m+1\right]=\left[n+m+1\right]$.
Thus,
\begin{equation}\label{bubbleeq}
\Bigg\langle\,\tikz[baseline=-.6ex, scale=0.8]{
\draw (-.4,.5) -- +(-.2,0);
\draw (.4,-.5) -- +(.2,0);
\draw (-.4,-.5) -- +(-.2,0);
\draw (.4,.5) -- +(.2,0);
\draw[-<-=.5] (-.4,.5) -- (0,.5);
\draw[-<-=.5] (0,.5) -- (.4,.5);
\draw[->-=.5] (-.4,-.5) -- (0,-.5);
\draw[->-=.5] (0,-.5) -- (.4,-.5);
\draw[-<-=.5] (.05,.3) to[out=east, in=east] (.05,.-.3);
\draw[->-=.5] (-.05,.3) to[out=west, in=west] (-.05,-.3);
\draw[fill=white] (-.4,.3) rectangle +(-.1,.3);
\draw[fill=white] (.4,.3) rectangle +(.1,.3);
\draw[fill=white] (-.4,-.3) rectangle +(-.1,-.3);
\draw[fill=white] (.4,-.3) rectangle +(.1,-.3);
\draw[fill=white] (-.05,.2) rectangle +(.1,.4);
\draw[fill=white] (-.05,-.2) rectangle +(.1,-.4);
\node at (.4,-.5)[below right]{$\scriptstyle{m-l}$};
\node at (-.4,-.5)[below left]{$\scriptstyle{m-k}$};
\node at (.4,.5)[above right]{$\scriptstyle{n-l}$};
\node at (-.4,.5)[above left]{$\scriptstyle{n-k}$};
\node at (-.2,0)[left]{$\scriptstyle{k}$};
\node at (.2,0)[right]{$\scriptstyle{l}$};
\node at (0,-.6)[below]{$\scriptstyle{m}$};
\node at (0,.6)[above]{$\scriptstyle{n}$};
}\,\Bigg\rangle_{\! 3}
=\frac{\left[n+m+1\right]}{\left[n\right]\left[m\right]}
\Bigg\langle\,\tikz[baseline=-.6ex, scale=0.8]{
\draw (-.4,.5) -- +(-.2,0);
\draw (.4,-.5) -- +(.2,0);
\draw (-.4,-.5) -- +(-.2,0);
\draw (.4,.5) -- +(.2,0);
\draw[-<-=.5] (-.4,.5) -- (0,.5);
\draw[-<-=.5] (0,.5) -- (.4,.5);
\draw[->-=.5] (-.4,-.5) -- (0,-.5);
\draw[->-=.5] (0,-.5) -- (.4,-.5);
\draw[-<-=.5] (.05,.3) to[out=east, in=east] (.05,.-.3);
\draw[->-=.5] (-.05,.3) to[out=west, in=west] (-.05,-.3);
\draw[fill=white] (-.4,.3) rectangle +(-.1,.3);
\draw[fill=white] (.4,.3) rectangle +(.1,.3);
\draw[fill=white] (-.4,-.3) rectangle +(-.1,-.3);
\draw[fill=white] (.4,-.3) rectangle +(.1,-.3);
\draw[fill=white] (-.05,.2) rectangle +(.1,.4);
\draw[fill=white] (-.05,-.2) rectangle +(.1,-.4);
\node at (.4,-.5)[below right]{$\scriptstyle{m-l}$};
\node at (-.4,-.5)[below left]{$\scriptstyle{m-k}$};
\node at (.4,.5)[above right]{$\scriptstyle{n-l}$};
\node at (-.4,.5)[above left]{$\scriptstyle{n-k}$};
\node at (-.2,0)[left]{$\scriptstyle{k-1}$};
\node at (.2,0)[right]{$\scriptstyle{l-1}$};
\node at (0,-.6)[below]{$\scriptstyle{m-1}$};
\node at (0,.6)[above]{$\scriptstyle{n-1}$};
}\,\Bigg\rangle_{\! 3}
+\frac{\left[n-1\right]\left[m-1\right]}{\left[n\right]\left[m\right]}
\Bigg\langle\,\tikz[baseline=-.6ex, scale=0.8]{
\draw (-.4,.5) -- +(-.2,0);
\draw (.4,-.5) -- +(.2,0);
\draw (-.4,-.5) -- +(-.2,0);
\draw (.4,.5) -- +(.2,0);
\draw[-<-=.5] (-.4,.5) -- (0,.5);
\draw[-<-=.5] (0,.5) -- (.4,.5);
\draw[->-=.5] (-.4,-.5) -- (0,-.5);
\draw[->-=.5] (0,-.5) -- (.4,-.5);
\draw[-<-=.5] (.05,.3) to[out=east, in=east] (.05,.-.3);
\draw[->-=.5] (-.05,.3) to[out=west, in=west] (-.05,-.3);
\draw[fill=white] (-.4,.3) rectangle +(-.1,.3);
\draw[fill=white] (-.4,-.3) rectangle +(-.1,-.3);
\draw[fill=white] (-.05,.2) rectangle +(.1,.4);
\draw[fill=white] (-.05,-.2) rectangle +(.1,-.4);
\node at (-.4,-.5)[below]{$\scriptstyle{m-k}$};
\node at (-.4,.5)[above]{$\scriptstyle{n-k}$};
\node at (-.2,0)[left]{$\scriptstyle{k-1}$};
\node at (.2,0)[right]{$\scriptstyle{1}$};
\begin{scope}[xshift=1.2cm]
\draw (-.4,.5) -- +(-.2,0);
\draw (.4,-.5) -- +(.2,0);
\draw (-.4,-.5) -- +(-.2,0);
\draw (.4,.5) -- +(.2,0);
\draw[-<-=.5] (-.4,.5) -- (0,.5);
\draw[-<-=.5] (0,.5) -- (.4,.5);
\draw[->-=.5] (-.4,-.5) -- (0,-.5);
\draw[->-=.5] (0,-.5) -- (.4,-.5);
\draw[-<-=.5] (.05,.3) to[out=east, in=east] (.05,.-.3);
\draw[->-=.5] (-.05,.3) to[out=west, in=west] (-.05,-.3);
\draw[fill=white] (.4,.3) rectangle +(.1,.3);
\draw[fill=white] (.4,-.3) rectangle +(.1,-.3);
\draw[fill=white] (-.05,.2) rectangle +(.1,.4);
\draw[fill=white] (-.05,-.2) rectangle +(.1,-.4);
\node at (.4,-.5)[below]{$\scriptstyle{m-l}$};
\node at (.4,.5)[above]{$\scriptstyle{n-l}$};
\node at (-.2,0)[left]{$\scriptstyle{1}$};
\node at (.2,0)[right]{$\scriptstyle{l-1}$};
\end{scope}
\node at (.6,.5)[above]{$\scriptstyle{n-2}$};
\node at (.6,-.5)[below]{$\scriptstyle{m-2}$};
}\,\Bigg\rangle_{\! 3}.
\end{equation}
We substitute $k=1$ in (\ref{bubbleeq}). 
The $A_2$ web appearing in the second term of the right-hand side has only an $A_2$ bubble skein element decorated with $1$ and $l-1$. 
By using (\ref{bubbleeq}) of $k=1$ repeatedly,
we can obtain
\begin{align*}
\Bigg\langle\,\tikz[baseline=-.6ex, scale=0.8]{
\draw (-.4,.5) -- +(-.2,0);
\draw (.4,-.5) -- +(.2,0);
\draw (-.4,-.5) -- +(-.2,0);
\draw (.4,.5) -- +(.2,0);
\draw[-<-=.5] (-.4,.5) -- (0,.5);
\draw[-<-=.5] (0,.5) -- (.4,.5);
\draw[->-=.5] (-.4,-.5) -- (0,-.5);
\draw[->-=.5] (0,-.5) -- (.4,-.5);
\draw[-<-=.5] (.05,.3) to[out=east, in=east] (.05,.-.3);
\draw[->-=.5] (-.05,.3) to[out=west, in=west] (-.05,-.3);
\draw[fill=white] (-.4,.3) rectangle +(-.1,.3);
\draw[fill=white] (.4,.3) rectangle +(.1,.3);
\draw[fill=white] (-.4,-.3) rectangle +(-.1,-.3);
\draw[fill=white] (.4,-.3) rectangle +(.1,-.3);
\draw[fill=white] (-.05,.2) rectangle +(.1,.4);
\draw[fill=white] (-.05,-.2) rectangle +(.1,-.4);
\node at (.4,-.5)[below right]{$\scriptstyle{m-l}$};
\node at (-.4,-.5)[below left]{$\scriptstyle{m-1}$};
\node at (.4,.5)[above right]{$\scriptstyle{n-l}$};
\node at (-.4,.5)[above left]{$\scriptstyle{n-1}$};
\node at (-.2,0)[left]{$\scriptstyle{1}$};
\node at (.2,0)[right]{$\scriptstyle{l}$};
\node at (0,-.6)[below]{$\scriptstyle{m}$};
\node at (0,.6)[above]{$\scriptstyle{n}$};
}\,\Bigg\rangle_{\! 3}
&=
\frac{\left[n+m+1\right]+\left[n+m-1\right]+\dots+\left[n+m-2l+3\right]}{\left[n\right]\left[m\right]}
\Bigg\langle\,\tikz[baseline=-.6ex, scale=0.8]{
\draw (-.4,.4) -- +(-.2,0);
\draw (-.4,-.4) -- +(-.2,0);
\draw[-<-=.5] (-.4,.5) -- (.0,.5);
\draw[->-=.5] (-.4,-.5)  -- (.0,-.5);
\draw[-<-=.5] (-.4,.3) to[out=east, in=east] (-.4,.-.3);
\draw[fill=white] (-.4,.2) rectangle +(-.1,.4);
\draw[fill=white] (-.4,-.2) rectangle +(-.1,-.4);
\node at (-.4,-.6)[left]{$\scriptstyle{m-1}$};
\node at (-.4,.6)[left]{$\scriptstyle{n-1}$};
\node at (-.2,0)[left]{$\scriptstyle{l-1}$};
\node at (0,-.5)[below]{$\scriptstyle{m-l}$};
\node at (0,.5)[above]{$\scriptstyle{n-l}$};
}\,\Bigg\rangle_{\! 3}\\
&\qquad+
\frac{\left[n-l\right]\left[m-l\right]}{\left[n\right]\left[m\right]}
\Bigg\langle\,\tikz[baseline=-.6ex, scale=0.8]{
\draw (-.4,.4) -- +(-.2,0);
\draw (.4,-.4) -- +(.2,0);
\draw (-.4,-.4) -- +(-.2,0);
\draw (.4,.4) -- +(.2,0);
\draw[-<-=.5] (-.4,.5) -- (.4,.5);
\draw[->-=.5] (-.4,-.5)  -- (.4,-.5);
\draw[-<-=.5] (-.4,.3) to[out=east, in=east] (-.4,.-.3);
\draw[->-=.5] (.4,.3) to[out=west, in=west] (.4,-.3);
\draw[fill=white] (-.4,.2) rectangle +(-.1,.4);
\draw[fill=white] (.4,.2) rectangle +(.1,.4);
\draw[fill=white] (-.4,-.2) rectangle +(-.1,-.4);
\draw[fill=white] (.4,-.2) rectangle +(.1,-.4);
\node at (.4,-.6)[right]{$\scriptstyle{m-l}$};
\node at (-.4,-.6)[left]{$\scriptstyle{m-1}$};
\node at (.4,.6)[right]{$\scriptstyle{n-l}$};
\node at (-.4,.6)[left]{$\scriptstyle{n-1}$};
\node at (-.2,0)[left]{$\scriptstyle{l}$};
\node at (.2,0)[right]{$\scriptstyle{1}$};
\node at (0,-.5)[below]{$\scriptstyle{m-l-1}$};
\node at (0,.5)[above]{$\scriptstyle{n-l-1}$};
}\,\Bigg\rangle_{\! 3},
\end{align*}
and confirm that $\left[n+m+1\right]+\left[n+m-1\right]+\dots+\left[n+m-2l+3\right]=\left[n+m-l+2\right]\left[l\right]$ (see Lemma~\ref{qinteger}~(1)).
\end{proof}

\begin{proof}[Proof of Theorem~\ref{A2bubble}]
We assume that $0\leq k\leq l\leq\min\{n,m\}$.
By substituting Lemma~\ref{bubblelem} (1) for (\ref{bubbleeq}),
\begin{align}
\Bigg\langle\,\tikz[baseline=-.6ex, scale=0.8]{
\draw (-.4,.5) -- +(-.2,0);
\draw (.4,-.5) -- +(.2,0);
\draw (-.4,-.5) -- +(-.2,0);
\draw (.4,.5) -- +(.2,0);
\draw[-<-=.5] (-.4,.5) -- (0,.5);
\draw[-<-=.5] (0,.5) -- (.4,.5);
\draw[->-=.5] (-.4,-.5) -- (0,-.5);
\draw[->-=.5] (0,-.5) -- (.4,-.5);
\draw[-<-=.5] (.05,.3) to[out=east, in=east] (.05,.-.3);
\draw[->-=.5] (-.05,.3) to[out=west, in=west] (-.05,-.3);
\draw[fill=white] (-.4,.3) rectangle +(-.1,.3);
\draw[fill=white] (.4,.3) rectangle +(.1,.3);
\draw[fill=white] (-.4,-.3) rectangle +(-.1,-.3);
\draw[fill=white] (.4,-.3) rectangle +(.1,-.3);
\draw[fill=white] (-.05,.2) rectangle +(.1,.4);
\draw[fill=white] (-.05,-.2) rectangle +(.1,-.4);
\node at (.4,-.5)[below right]{$\scriptstyle{m-l}$};
\node at (-.4,-.5)[below left]{$\scriptstyle{m-k}$};
\node at (.4,.5)[above right]{$\scriptstyle{n-l}$};
\node at (-.4,.5)[above left]{$\scriptstyle{n-k}$};
\node at (-.2,0)[left]{$\scriptstyle{k}$};
\node at (.2,0)[right]{$\scriptstyle{l}$};
\node at (0,-.6)[below]{$\scriptstyle{m}$};
\node at (0,.6)[above]{$\scriptstyle{n}$};
}\,\Bigg\rangle_{\! 3}
&=\frac{\left[n+m-l+2\right]\left[l\right]}{\left[n\right]\left[m\right]}
\Bigg\langle\,\tikz[baseline=-.6ex, scale=0.8]{
\draw (-.4,.5) -- +(-.2,0);
\draw (.4,-.5) -- +(.2,0);
\draw (-.4,-.5) -- +(-.2,0);
\draw (.4,.5) -- +(.2,0);
\draw[-<-=.5] (-.4,.5) -- (0,.5);
\draw[-<-=.5] (0,.5) -- (.4,.5);
\draw[->-=.5] (-.4,-.5) -- (0,-.5);
\draw[->-=.5] (0,-.5) -- (.4,-.5);
\draw[-<-=.5] (.05,.3) to[out=east, in=east] (.05,.-.3);
\draw[->-=.5] (-.05,.3) to[out=west, in=west] (-.05,-.3);
\draw[fill=white] (-.4,.3) rectangle +(-.1,.3);
\draw[fill=white] (.4,.3) rectangle +(.1,.3);
\draw[fill=white] (-.4,-.3) rectangle +(-.1,-.3);
\draw[fill=white] (.4,-.3) rectangle +(.1,-.3);
\draw[fill=white] (-.05,.2) rectangle +(.1,.4);
\draw[fill=white] (-.05,-.2) rectangle +(.1,-.4);
\node at (.4,-.5)[below right]{$\scriptstyle{m-l}$};
\node at (-.4,-.5)[below left]{$\scriptstyle{m-k}$};
\node at (.4,.5)[above right]{$\scriptstyle{n-l}$};
\node at (-.4,.5)[above left]{$\scriptstyle{n-k}$};
\node at (-.2,0)[left]{$\scriptstyle{k-1}$};
\node at (.2,0)[right]{$\scriptstyle{l-1}$};
\node at (0,-.6)[below]{$\scriptstyle{m-1}$};
\node at (0,.6)[above]{$\scriptstyle{n-1}$};
}\,\Bigg\rangle_{\! 3}\label{bubbleeq2}\\
&\qquad+\frac{\left[n-l\right]\left[m-l\right]}{\left[n\right]\left[m\right]}
\Bigg\langle\,\tikz[baseline=-.6ex, scale=0.8]{
\draw (-.4,.5) -- +(-.2,0);
\draw (.4,-.5) -- +(.2,0);
\draw (-.4,-.5) -- +(-.2,0);
\draw (.4,.5) -- +(.2,0);
\draw[-<-=.5] (-.4,.5) -- (0,.5);
\draw[-<-=.5] (0,.5) -- (.4,.5);
\draw[->-=.5] (-.4,-.5) -- (0,-.5);
\draw[->-=.5] (0,-.5) -- (.4,-.5);
\draw[-<-=.5] (.05,.3) to[out=east, in=east] (.05,.-.3);
\draw[->-=.5] (-.05,.3) to[out=west, in=west] (-.05,-.3);
\draw[fill=white] (-.4,.3) rectangle +(-.1,.3);
\draw[fill=white] (-.4,-.3) rectangle +(-.1,-.3);
\draw[fill=white] (-.05,.2) rectangle +(.1,.4);
\draw[fill=white] (-.05,-.2) rectangle +(.1,-.4);
\node at (-.4,-.5)[below]{$\scriptstyle{m-k}$};
\node at (-.4,.5)[above]{$\scriptstyle{n-k}$};
\node at (-.2,0)[left]{$\scriptstyle{k-1}$};
\node at (.2,0)[right]{$\scriptstyle{l}$};
\begin{scope}[xshift=1.2cm]
\draw (-.4,.5) -- +(-.2,0);
\draw (-.4,-.5) -- +(-.2,0);
\draw[-<-=.5] (-.4,.5) -- (0,.5);
\draw[-<-=.5] (0,.5) -- (.4,.5);
\draw[->-=.5] (-.4,-.5) -- (0,-.5);
\draw[->-=.5] (0,-.5) -- (.4,-.5);
\draw[->-=.5] (-.05,.3) to[out=west, in=west] (-.05,-.3);
\draw[fill=white] (-.05,.2) rectangle +(.1,.4);
\draw[fill=white] (-.05,-.2) rectangle +(.1,-.4);
\node at (.0,-.5)[below right]{$\scriptstyle{m-l}$};
\node at (.0,.5)[above right]{$\scriptstyle{n-l}$};
\node at (-.2,0)[left]{$\scriptstyle{1}$};
\end{scope}
\node at (.6,.5)[above]{$\scriptstyle{n-l-1}$};
\node at (.6,-.5)[below]{$\scriptstyle{m-l-1}$};
}\,\Bigg\rangle_{\! 3}.\notag
\end{align}
We used 
\begin{align*}
\left[n+m+1\right]+\left[n+m-l+1\right]\left[l-1\right]
&=\left[n+m+1\right]\left[1\right]-\left[n+m+1-l\right]\left[1-l\right]\\
&=\left[n+m-l+2\right]\left[l\right]
\end{align*} 
in the above equation (see Lemma~\ref{qinteger}~(2)).
We prove this theorem by the induction on $\max\{m,n\}$.
From (\ref{bubbleeq2}) and the induction hypothesis,
\begin{align*}
&\Bigg\langle\,\tikz[baseline=-.6ex, scale=0.8]{
\draw (-.4,.5) -- +(-.2,0);
\draw (.4,-.5) -- +(.2,0);
\draw (-.4,-.5) -- +(-.2,0);
\draw (.4,.5) -- +(.2,0);
\draw[-<-=.5] (-.4,.5) -- (0,.5);
\draw[-<-=.5] (0,.5) -- (.4,.5);
\draw[->-=.5] (-.4,-.5) -- (0,-.5);
\draw[->-=.5] (0,-.5) -- (.4,-.5);
\draw[-<-=.5] (.05,.3) to[out=east, in=east] (.05,.-.3);
\draw[->-=.5] (-.05,.3) to[out=west, in=west] (-.05,-.3);
\draw[fill=white] (-.4,.3) rectangle +(-.1,.3);
\draw[fill=white] (.4,.3) rectangle +(.1,.3);
\draw[fill=white] (-.4,-.3) rectangle +(-.1,-.3);
\draw[fill=white] (.4,-.3) rectangle +(.1,-.3);
\draw[fill=white] (-.05,.2) rectangle +(.1,.4);
\draw[fill=white] (-.05,-.2) rectangle +(.1,-.4);
\node at (.4,-.5)[below right]{$\scriptstyle{m-l}$};
\node at (-.4,-.5)[below left]{$\scriptstyle{m-k}$};
\node at (.4,.5)[above right]{$\scriptstyle{n-l}$};
\node at (-.4,.5)[above left]{$\scriptstyle{n-k}$};
\node at (-.2,0)[left]{$\scriptstyle{k}$};
\node at (.2,0)[right]{$\scriptstyle{l}$};
\node at (0,-.6)[below]{$\scriptstyle{m}$};
\node at (0,.6)[above]{$\scriptstyle{n}$};
}\,\Bigg\rangle_{\! 3}\\
&=\frac{\left[n+m-l+2\right]\left[l\right]}{\left[n\right]\left[m\right]}
\sum_{t=l-1}^{\min\{k+l-2, n-1, m-1\}}
\frac{{n-1\brack t}{m-1\brack t}{t\brack k-1}{t\brack l-1}{n+m-t\brack n+m-k-l+2}}{{n-1\brack k-1}{m-1\brack k-1}{n-1\brack l-1}{m-1\brack l-1}}
\Bigg\langle\,\tikz[baseline=-.6ex, scale=0.8]{
\draw (-.4,.4) -- +(-.2,0);
\draw (.4,-.4) -- +(.2,0);
\draw (-.4,-.4) -- +(-.2,0);
\draw (.4,.4) -- +(.2,0);
\draw[-<-=.5] (-.4,.5) -- (.4,.5);
\draw[->-=.5] (-.4,-.5)  -- (.4,-.5);
\draw[-<-=.5] (-.4,.3) to[out=east, in=east] (-.4,.-.3);
\draw[->-=.5] (.4,.3) to[out=west, in=west] (.4,-.3);
\draw[fill=white] (-.4,.2) rectangle +(-.1,.4);
\draw[fill=white] (.4,.2) rectangle +(.1,.4);
\draw[fill=white] (-.4,-.2) rectangle +(-.1,-.4);
\draw[fill=white] (.4,-.2) rectangle +(.1,-.4);
\node at (.4,-.6)[right]{$\scriptstyle{m-l}$};
\node at (-.4,-.6)[left]{$\scriptstyle{m-k}$};
\node at (.4,.6)[right]{$\scriptstyle{n-l}$};
\node at (-.4,.6)[left]{$\scriptstyle{n-k}$};
\node at (-.2,0)[left]{$\scriptstyle{t-k+1}$};
\node at (.2,0)[right]{$\scriptstyle{t-l+1}$};
\node at (0,-.5)[below]{$\scriptstyle{m-t-1}$};
\node at (0,.5)[above]{$\scriptstyle{n-t-1}$};
}\,\Bigg\rangle_{\! 3}\\
&+\frac{\left[n-l\right]\left[m-l\right]}{\left[n\right]\left[m\right]}
\sum_{t=l}^{\min\{k+l-1, n-1, m-1\}}
\frac{{n-1\brack t}{m-1\brack t}{t\brack k-1}{t\brack l}{n+m-t\brack n+m-k-l+1}}{{n-1\brack k-1}{m-1\brack k-1}{n-1\brack l}{m-1\brack l}}
\Bigg\langle\,\tikz[baseline=-.6ex, scale=0.8]{
\draw (-.4,.4) -- +(-.2,0);
\draw (.4,-.4) -- +(.2,0);
\draw (-.4,-.4) -- +(-.2,0);
\draw (.4,.4) -- +(.2,0);
\draw[-<-=.5] (-.4,.5) -- (.4,.5);
\draw[->-=.5] (-.4,-.5)  -- (.4,-.5);
\draw[-<-=.5] (-.4,.3) to[out=east, in=east] (-.4,.-.3);
\draw[->-=.5] (.4,.3) to[out=west, in=west] (.4,-.3);
\draw[fill=white] (-.4,.2) rectangle +(-.1,.4);
\draw[fill=white] (.4,.2) rectangle +(.1,.4);
\draw[fill=white] (-.4,-.2) rectangle +(-.1,-.4);
\draw[fill=white] (.4,-.2) rectangle +(.1,-.4);
\node at (.4,-.6)[right]{$\scriptstyle{m-l}$};
\node at (-.4,-.6)[left]{$\scriptstyle{m-k}$};
\node at (.4,.6)[right]{$\scriptstyle{n-l}$};
\node at (-.4,.6)[left]{$\scriptstyle{n-k}$};
\node at (-.2,0)[left]{$\scriptstyle{t-k+1}$};
\node at (.2,0)[right]{$\scriptstyle{t-l+1}$};
\node at (0,-.5)[below]{$\scriptstyle{m-t-1}$};
\node at (0,.5)[above]{$\scriptstyle{n-t-1}$};
}\,\Bigg\rangle_{\! 3}\\
\end{align*}
For $t>l$ and $\min\{k+l,n,m\}\neq k+l$, 
\begin{align*}
&\text{the coefficient of}\ 
\Bigg\langle\,\tikz[baseline=-.6ex, scale=0.8]{
\draw (-.4,.4) -- +(-.2,0);
\draw (.4,-.4) -- +(.2,0);
\draw (-.4,-.4) -- +(-.2,0);
\draw (.4,.4) -- +(.2,0);
\draw[-<-=.5] (-.4,.5) -- (.4,.5);
\draw[->-=.5] (-.4,-.5)  -- (.4,-.5);
\draw[-<-=.5] (-.4,.3) to[out=east, in=east] (-.4,.-.3);
\draw[->-=.5] (.4,.3) to[out=west, in=west] (.4,-.3);
\draw[fill=white] (-.4,.2) rectangle +(-.1,.4);
\draw[fill=white] (.4,.2) rectangle +(.1,.4);
\draw[fill=white] (-.4,-.2) rectangle +(-.1,-.4);
\draw[fill=white] (.4,-.2) rectangle +(.1,-.4);
\node at (.4,-.6)[right]{$\scriptstyle{m-l}$};
\node at (-.4,-.6)[left]{$\scriptstyle{m-k}$};
\node at (.4,.6)[right]{$\scriptstyle{n-l}$};
\node at (-.4,.6)[left]{$\scriptstyle{n-k}$};
\node at (-.2,0)[left]{$\scriptstyle{t-k}$};
\node at (.2,0)[right]{$\scriptstyle{t-l}$};
\node at (0,-.5)[below]{$\scriptstyle{m-t}$};
\node at (0,.5)[above]{$\scriptstyle{n-t}$};
}\,\Bigg\rangle_{\! 3}\\
&\quad=\left(\frac{\left[n+m-l+2\right]\left[l\right]}{\left[n\right]\left[m\right]}\right)
\left(\frac{{n-1\brack t-1}{m-1\brack t-1}{t-1\brack k-1}{t-1\brack l-1}{n+m-t+1\brack n+m-k-l+2}}{{n-1\brack k-1}{m-1\brack k-1}{n-1\brack l-1}{m-1\brack l-1}}\right)\\
&\qquad+\left(\frac{\left[n-l\right]\left[m-l\right]}{\left[n\right]\left[m\right]}\right)
\left(\frac{{n-1\brack t-1}{m-1\brack t-1}{t-1\brack k-1}{t-1\brack l}{n+m-t+1\brack n+m-k-l+1}}{{n-1\brack k-1}{m-1\brack k-1}{n-1\brack l}{m-1\brack l}}\right)\\
&\quad=\left(\frac{\left[n+m-l+2\right]\left[k+l-t\right]+\left[t-l\right]\left[n+m-k-l+2\right]}{\left[n+m-k-l+2\right]}\right)\\
&\qquad\times\left(\frac{\left[l\right]}{\left[n\right]\left[m\right]}\right)
\left(\frac{{n-1\brack t-1}{m-1\brack t-1}{t-1\brack k-1}{t-1\brack l}{n+m-t+1\brack n+m-k-l+1}}{{n-1\brack k-1}{m-1\brack k-1}{n-1\brack l}{m-1\brack l}}\right)\\
&\quad=\left(\frac{\left[n+m+2-t\right]\left[k\right]}{\left[n+m-k-l+2\right]}\right)
\left(\frac{\left[l\right]}{\left[n\right]\left[m\right]}\right)
\left(\frac{\left[n\right]\left[m\right]}{\left[k\right]\left[l\right]}\right)
\left(\frac{{n\brack t}{m\brack t}{t\brack k}{t\brack l}{n+m-t+1\brack n+m-k-l+1}}{{n\brack k}{m\brack k}{n\brack l}{m\brack l}}\right)\\
&\quad=\frac{{n\brack t}{m\brack t}{t\brack k}{t\brack l}{n+m-t+2\brack n+m-k-l+2}}{{n\brack k}{m\brack k}{n\brack l}{m\brack l}}.
\end{align*}
We used 
\[
\left[n+m-l+2\right]\left[k-(t-l)\right]+\left[t-l\right]\left[n+m-l+2-k\right]=\left[k\right]\left[n+m+2-t\right]
\]
in the third equality (see Lemma~\ref{qinteger}~(3)).
If $t=l$, 
then the coefficient of 
$\Bigg\langle\,\tikz[baseline=-.6ex, scale=0.8]{
\draw (-.4,.4) -- +(-.2,0);
\draw (-.4,-.4) -- +(-.2,0);
\draw[-<-=.5] (-.4,.5) -- (.0,.5);
\draw[->-=.5] (-.4,-.5)  -- (.0,-.5);
\draw[-<-=.5] (-.4,.3) to[out=east, in=east] (-.4,.-.3);
\draw[fill=white] (-.4,.2) rectangle +(-.1,.4);
\draw[fill=white] (-.4,-.2) rectangle +(-.1,-.4);
\node at (-.4,-.6)[left]{$\scriptstyle{m-k}$};
\node at (-.4,.6)[left]{$\scriptstyle{n-k}$};
\node at (-.2,0)[left]{$\scriptstyle{l-k}$};
\node at (0,-.5)[below]{$\scriptstyle{m-l}$};
\node at (0,.5)[above]{$\scriptstyle{n-l}$};
}\,\Bigg\rangle_{\! 3}$ is
\[
\left(\frac{\left[n+m-l+2\right]\left[l\right]}{\left[n\right]\left[m\right]}\right)
\left(\frac{{n-1\brack l-1}{m-1\brack l-1}{l-1\brack k-1}{n+m-l+1\brack n+m-k-l+2}}{{n-1\brack k-1}{m-1\brack k-1}{n-1\brack l-1}{m-1\brack l-1}}\right)
=\frac{{n\brack l}{m\brack l}{l\brack k}{n+m-l+2\brack n+m-k-l+2}}{{n\brack k}{m\brack k}{n\brack l}{m\brack l}}.
\]
When $\min\{k+l,n,m\}=k+l$, 
we need to confirm the coefficient of 
$\Bigg\langle\,\tikz[baseline=-.6ex, scale=0.8]{
\draw (-.4,.4) -- +(-.2,0);
\draw (.4,-.4) -- +(.2,0);
\draw (-.4,-.4) -- +(-.2,0);
\draw (.4,.4) -- +(.2,0);
\draw[-<-=.5] (-.4,.5) -- (.4,.5);
\draw[->-=.5] (-.4,-.5)  -- (.4,-.5);
\draw[-<-=.5] (-.4,.3) to[out=east, in=east] (-.4,.-.3);
\draw[->-=.5] (.4,.3) to[out=west, in=west] (.4,-.3);
\draw[fill=white] (-.4,.2) rectangle +(-.1,.4);
\draw[fill=white] (.4,.2) rectangle +(.1,.4);
\draw[fill=white] (-.4,-.2) rectangle +(-.1,-.4);
\draw[fill=white] (.4,-.2) rectangle +(.1,-.4);
\node at (.4,-.6)[right]{$\scriptstyle{m-l}$};
\node at (-.4,-.6)[left]{$\scriptstyle{m-k}$};
\node at (.4,.6)[right]{$\scriptstyle{n-l}$};
\node at (-.4,.6)[left]{$\scriptstyle{n-k}$};
\node at (-.2,0)[left]{$\scriptstyle{l}$};
\node at (.2,0)[right]{$\scriptstyle{k}$};
\node at (0,-.5)[below]{$\scriptstyle{m-k-l}$};
\node at (0,.5)[above]{$\scriptstyle{n-k-l}$};
}\,\Bigg\rangle_{\! 3}$:
\[
\left(\frac{\left[n-l\right]\left[m-l\right]}{\left[n\right]\left[m\right]}\right)
\left(\frac{{n-1\brack k+l-1}{m-1\brack k+l-1}{k+l-1\brack k-1}{k+l-1\brack l}}{{n-1\brack k-1}{m-1\brack k-1}{n-1\brack l}{m-1\brack l}}\right)
=\frac{{n\brack k+l}{m\brack k+l}{k+l\brack k}{k+l\brack l}}{{n\brack k}{m\brack k}{n\brack l}{m\brack l}}.
\]
We remark that we can prove when $0\leq l\leq k\leq\min\{n,m\}$ through the use of Lemma~\ref{bubblelem} (2) and (\ref{bubbleeq}).
\end{proof}
\begin{RMK}
We can prove directly Theorem~\ref{A1bubble} in a similar way to the proof of Theorem~\ref{A2bubble}.
\end{RMK}

\section{Colored Jones polynomials of $2$-bridge links}
In this section,
we compute $\mathfrak{sl}_2$ and $\mathfrak{sl}_3$ colored Jones polynomials for $2$-bridge links by using twist formulas. 
An explicit formula for the $n+1$ dimensional $\mathfrak{sl}_2$ Jones polynomial of twist knots was given by Masbaum~\cite{Masbaum03} and, more generally, for $2$-bridge knots was given by Takata~\cite{Takata08}.
Both of these formulas were derived from the linear skein theory for the Kauffman bracket.
We will give a more explicit formula of the $\mathfrak{sl}_2$ Jones polynomial for $2$-bridge links in a similar way to \cite{Masbaum03} .
\subsection{2-bridge knot and link diagrams}
We briefly recall link diagrams of the $2$-bridge knots and links. 
You can find details on definitions of the $2$-bridge knots and links and the classification by Schubert~\cite{Schubert56} in \cite{BurdeZieschang85} and \cite{Murasugi93}.
Let $m$ be an integer. 
A boxed $m$ implies $m$ half twists (see Figure~\ref{boxedm}).
\begin{figure}
\centering
\begin{align*}
\tikz[baseline=-.6ex]{
\draw (-.5,.2) -- (.5,.2);
\draw (-.5,-.2) -- (.5,-.2);
\draw[fill=white] (-.2,-.3) rectangle (.2,.3);
\node at (.0,.0) {${\scriptstyle m}$};
}\,
=
\begin{cases}
\,\tikz[baseline=-.6ex]{
\begin{scope}[xshift=-.5cm]
\draw (-.5,.2) -- +(-.2,0);
\draw (-.5,-.2) -- +(-.2,0);
\draw[white, double=black, double distance=0.4pt, ultra thick] 
(-.5,-.2) to[out=east, in=west] (.0,.2);
\draw[white, double=black, double distance=0.4pt, ultra thick] 
(-.5,.2) to[out=east, in=west] (.0,-.2);
\end{scope}
\node at (.0,.0){$\cdots$};
\node at (.0,-.2)[below]{$\scriptstyle{\text{right-handed }m\text{ half twists}}$};
\begin{scope}[xshift=.5cm]
\draw
(.5,-.2) -- +(.2,0)
(.5,.2) -- +(.2,0);
\draw[white, double=black, double distance=0.4pt, ultra thick] 
(0,-.2) to[out=east, in=west] (.5,.2);
\draw[white, double=black, double distance=0.4pt, ultra thick] 
(0,.2) to[out=east, in=west] (.5,-.2);
\end{scope}
}\,  & \text{if $m>$0,}\\
\,\tikz[baseline=-.6ex]{
\begin{scope}[xshift=-.5cm]
\draw (-.5,.2) -- +(-.2,0);
\draw (-.5,-.2) -- +(-.2,0);
\draw[white, double=black, double distance=0.4pt, ultra thick] 
(-.5,.2) to[out=east, in=west] (.0,-.2);
\draw[white, double=black, double distance=0.4pt, ultra thick] 
(-.5,-.2) to[out=east, in=west] (.0,.2);
\end{scope}
\node at (.0,.0){$\cdots$};
\node at (.0,-.2)[below]{$\scriptstyle{\text{left-handed } m\text{ half twists}}$};
\begin{scope}[xshift=.5cm]
\draw
(.5,-.2) -- +(.2,0)
(.5,.2) -- +(.2,0);
\draw[white, double=black, double distance=0.4pt, ultra thick] 
(0,.2) to[out=east, in=west] (.5,-.2);
\draw[white, double=black, double distance=0.4pt, ultra thick] 
(0,-.2) to[out=east, in=west] (.5,.2);
\end{scope}
}\, & \text{if $m<0$,}\\
\quad\tikz[baseline=-.6ex]{
\draw (-1.2,.1) -- (1.2,.1);
\draw (-1.2,-.2) -- (1.2,-.2);
}\, & \text{if $m=0$.}
\end{cases}
\end{align*}
\caption{$m$ half twists}
\label{boxedm}
\end{figure}

\begin{LEM}\label{2bridgepre}
The $2$-bridge knots and links have the following standard presentation.
\begin{align*}
l\ \text{is odd :}&
\ \tikz[baseline=-.6ex, scale=1.2]{
\begin{scope}
\draw (-.3,.3) to[out=west, in=west] (-.3,.1);
\draw (-.3,-.1) to[out=west, in=west] (-.3,-.3);
\draw[-<-=.5] (-.3,.3) -- (.3,.3);
\draw (-.3,.1) -- (.3,.1);
\draw (-.3,-.1) -- (.3,-.1);
\draw[->-=.5] (-.3,-.3) -- (.3,-.3);
\draw[fill=white] (-.2,-.2) rectangle (.2,.2);
\node at (.0,.0) {$\scriptstyle{2a_1}$};
\end{scope}
\begin{scope}[xshift=.6cm]
\draw (-.3,.3) -- (.3,.3);
\draw (-.3,.1) -- (.3,.1);
\draw (-.3,-.1) -- (.3,-.1);
\draw (-.3,-.3) -- (.3,-.3);
\draw[fill=white] (-.2,-.4) rectangle (.2,.0);
\node at (.0,-.2) {$\scriptstyle{2a_2}$};
\end{scope}
\begin{scope}[xshift=1.2cm]
\draw (-.3,.3) -- (.3,.3);
\draw (-.3,.1) -- (.3,.1);
\draw (-.3,-.1) -- (.3,-.1);
\draw (-.3,-.3) -- (.3,-.3);
\draw[fill=white] (-.2,-.2) rectangle (.2,.2);
\node at (.0,.0) {$\scriptstyle{2a_3}$};
\end{scope}
\begin{scope}[xshift=1.8cm]
\draw (-.3,.3) -- (.3,.3);
\draw (-.3,.1) -- (.3,.1);
\draw (-.3,-.1) -- (.3,-.1);
\draw (-.3,-.3) -- (.3,-.3);
\draw[fill=white] (-.2,-.4) rectangle (.2,.0);
\node at (.0,-.2) {$\scriptstyle{2a_4}$};
\end{scope}
\begin{scope}[xshift=2.4cm]
\draw (-.3,.3) -- (.3,.3);
\node at (.0,.0) {$\cdots$};
\node at (.0,-.2) {$\cdots$};
\end{scope}
\begin{scope}[xshift=3cm]
\draw (.3,.3) to[out=east, in=east] (.3,.1);
\draw (.3,-.1) to[out=east, in=east] (.3,-.3);
\draw (-.3,.3) -- (.3,.3);
\draw (-.3,.1) -- (.3,.1);
\draw (-.3,-.1) -- (.3,-.1);
\draw (-.3,-.3) -- (.3,-.3);
\draw[fill=white] (-.2,-.2) rectangle (.2,.2);
\node at (.0,.0) {$\scriptstyle{2a_l}$};
\end{scope}
}\\
l\ \text{is even :}&
\ \tikz[baseline=-.6ex, scale=1.2]{
\begin{scope}
\draw (-.3,.3) to[out=west, in=west] (-.3,.1);
\draw (-.3,-.1) to[out=west, in=west] (-.3,-.3);
\draw[-<-=.5] (-.3,.3) -- (.3,.3);
\draw (-.3,.1) -- (.3,.1);
\draw (-.3,-.1) -- (.3,-.1);
\draw[->-=.5] (-.3,-.3) -- (.3,-.3);
\draw[fill=white] (-.2,-.2) rectangle (.2,.2);
\node at (.0,.0) {$\scriptstyle{2a_1}$};
\end{scope}
\begin{scope}[xshift=.6cm]
\draw (-.3,.3) -- (.3,.3);
\draw (-.3,.1) -- (.3,.1);
\draw (-.3,-.1) -- (.3,-.1);
\draw (-.3,-.3) -- (.3,-.3);
\draw[fill=white] (-.2,-.4) rectangle (.2,.0);
\node at (.0,-.2) {$\scriptstyle{2a_2}$};
\end{scope}
\begin{scope}[xshift=1.2cm]
\draw (-.3,.3) -- (.3,.3);
\draw (-.3,.1) -- (.3,.1);
\draw (-.3,-.1) -- (.3,-.1);
\draw (-.3,-.3) -- (.3,-.3);
\draw[fill=white] (-.2,-.2) rectangle (.2,.2);
\node at (.0,.0) {$\scriptstyle{2a_3}$};
\end{scope}
\begin{scope}[xshift=1.8cm]
\draw (-.3,.3) -- (.3,.3);
\draw (-.3,.1) -- (.3,.1);
\draw (-.3,-.1) -- (.3,-.1);
\draw (-.3,-.3) -- (.3,-.3);
\draw[fill=white] (-.2,-.4) rectangle (.2,.0);
\node at (.0,-.2) {$\scriptstyle{2a_4}$};
\end{scope}
\begin{scope}[xshift=2.4cm]
\draw (-.3,.3) -- (.3,.3);
\node at (.0,.0) {$\cdots$};
\node at (.0,-.2) {$\cdots$};
\end{scope}
\begin{scope}[xshift=3cm]
\draw (.3,.3) to[out=east, in=east] (.3,-.3);
\draw (.3,.1) to[out=east, in=east] (.3,-.1);
\draw (-.3,.3) -- (.3,.3);
\draw (-.3,.1) -- (.3,.1);
\draw (-.3,-.1) -- (.3,-.1);
\draw (-.3,-.3) -- (.3,-.3);
\draw[fill=white] (-.2,-.4) rectangle (.2,.0);
\node at (.0,-.2) {$\scriptstyle{2a_l}$};
\end{scope}
}
\end{align*}
The above $a_1,a_2,\dots,a_l$ are non-zero integers. 
We denote this link diagram by $[2a_1,2a_2,\dots,2a_l]$.
\end{LEM}
\begin{proof}
See, for example, \cite[Chapter~$2$]{Kawauchi96}.
\end{proof}

\subsection{The $\mathfrak{sl}_2$ colored Jones polynomial}
We introduce the $\mathfrak{sl}_2$ colored Jones polynomial.
Let $L$ be an oriented link diagram in $D_0$ with ordered link components $(L_1,L_2,\dots,L_r)$. 
We denote by $\bar{L}=(\bar{L}_1,\bar{L}_2,\dots,\bar{L}_r)$ an unoriented link diagram obtained by forgetting the orientation of $L$.
For each link component $\bar{L}_i$ ($i=1,2,\dots,r$) of $\bar{L}$,
we cut away a short strand from $\bar{L}_i$ and substitute
\tikz[baseline=-.6ex]{
\draw (-.5,0) -- (.5,0);
\draw[fill=white] (-.1,-.3) rectangle (.1,.3);
\node at (.1,0) [above right]{${\scriptstyle n}$};
}\,
for it.
Thus,
we obtain the unoriented link diagram colored by $n$. 
We denote it by $\bar{L}(n)$. (See Figure~\ref{L(n)}.)
\begin{figure}
\centering
\begin{tikzpicture}
\draw[rotate=120] (90:.7) to[out=east, in=east] (90:1.3);
\draw[->-=.5] (-30:.3) to[out=-120, in=-60] (210:1);
\draw[rotate=120] (-30:.3) to[out=-120, in=-60] (210:1);
\draw[rotate=240] (-30:.3) to[out=-120, in=-60] (210:1);
\draw[white, double=black, double distance=0.4pt, ultra thick] 
(90:1) to[out=0, in=60] (-30:.3);
\draw[white, double=black, double distance=0.4pt, ultra thick, rotate=120] 
(90:1) to[out=0, in=60] (-30:.3);
\draw[white, double=black, double distance=0.4pt, ultra thick, rotate=240] 
(90:1) to[out=0, in=60] (-30:.3);
\draw[-<-=1, white, double=black, double distance=0.4pt, ultra thick, rotate=120] 
(90:.7) to[out=west, in=west] (90:1.3);
\node at (-90:1.2){$L$};
\end{tikzpicture}
\begin{tikzpicture}
\draw[rotate=120] (90:.7) to[out=east, in=east] (90:1.3);
\draw (-30:.3) to[out=-120, in=-60] (210:1);
\draw[rotate=120] (-30:.3) to[out=-120, in=-60] (210:1);
\draw[rotate=240] (-30:.3) to[out=-120, in=-60] (210:1);
\draw[white, double=black, double distance=0.4pt, ultra thick] 
(90:1) to[out=0, in=60] (-30:.3);
\draw[white, double=black, double distance=0.4pt, ultra thick, rotate=120] 
(90:1) to[out=0, in=60] (-30:.3);
\draw[white, double=black, double distance=0.4pt, ultra thick, rotate=240] 
(90:1) to[out=0, in=60] (-30:.3);
\draw[white, double=black, double distance=0.4pt, ultra thick, rotate=120] 
(90:.7) to[out=west, in=west] (90:1.3);
\draw[cyan, dashed] (90:1) circle (.2);
\draw[cyan, dashed] (210:1.3) circle (.2);
\node at (-90:1.2){$\bar{L}$};
\end{tikzpicture}
\begin{tikzpicture}
\draw[rotate=120] (90:.7) to[out=east, in=east] (90:1.3);
\draw (-30:.3) to[out=-120, in=-60] (210:1);
\draw[rotate=120] (-30:.3) to[out=-120, in=-60] (210:1);
\draw[rotate=240] (-30:.3) to[out=-120, in=-60] (210:1);
\draw[white, double=black, double distance=0.4pt, ultra thick] 
(90:1) to[out=0, in=60] (-30:.3);
\draw[white, double=black, double distance=0.4pt, ultra thick, rotate=120] 
(90:1) to[out=0, in=60] (-30:.3);
\draw[white, double=black, double distance=0.4pt, ultra thick, rotate=240] 
(90:1) to[out=0, in=60] (-30:.3);
\draw[white, double=black, double distance=0.4pt, ultra thick, rotate=120] 
(90:.7) to[out=west, in=west] (90:1.3);
\draw[fill=white] (-.05,.8) rectangle (.05,1.2);
\draw[fill=white, rotate=120] (-.05,1.1) rectangle (.05,1.5);
\node at (90:1) [above right]{${\scriptstyle n}$};
\node at (210:1.3) [below right]{${\scriptstyle n}$};
\node at (-90:1.2) {$\bar{L}(n)$};
\end{tikzpicture}\\
Replace 
\,\tikz[baseline=-.6ex]{
\draw (-.3,.0) -- (.3,.0);
\draw[cyan, dashed] (.0,.0) circle (.3);
}\,
with
\,\tikz[baseline=-.6ex]{
\draw (-.3,.0) -- (.3,.0);
\draw[fill=white] (-.05,-.2) rectangle (.05,.2);
\draw[cyan, dashed] (.0,.0) circle (.3);
}\,
\caption{A Link diagram colored by $n$}
\label{L(n)}
\end{figure}

\begin{DEF}
The $\mathfrak{sl}_2$ colored Jones polynomial $J_{n+1}^{\mathfrak{sl}_2}(L;q)$ of a link represented by $L$ is defined by
\[
J_{n+1}^{\mathfrak{sl}_2}(L;q)
=((-1)^nq^{\frac{n^2+2n}{4}})^{-\mathop{w}(L)}
\langle \bar{L}(n)\rangle_{2}/\langle\,\tikz[baseline=-.6ex, scale=.5]{
\draw (0,0) circle [radius=.3];
\draw[fill=white] (.1,-.05) rectangle (.5,.05);
\node at (.3,0)[above right]{$\scriptstyle{n}$};
}\,\rangle_{2},
\]
where $\mathop{w}(L)$ is the writhe of $L$.
\end{DEF}

\begin{LEM}\label{A1closure}
\[
\Bigg\langle\,\tikz[baseline=-1.8ex]{
\draw[rounded corners=.1cm] (-.4,.3) -- ++(-.4,0) -- ++(.0,-1.2) -- ++(1.6,.0) -- ++(.0,1.2) -- (.4,.3);
\draw[rounded corners=.1cm] (-.4,-.3) -- ++(-.2,0) -- ++(.0,-.4) -- ++(1.2,.0) -- ++(.0,.4) -- (.4,-.3);
\draw (-.4,.4) -- (.4,.4);
\draw (-.4,-.4) -- (.4,-.4);
\draw (-.4,.2) to[out=east, in=east] (-.4,-.2);
\draw (.4,.2) to[out=west, in=west] (.4,-.2);
\draw[fill=white] (.4,-.5) rectangle +(.1,.4);
\draw[fill=white] (-.4,-.5) rectangle +(-.1,.4);
\draw[fill=white] (.4,.5) rectangle +(.1,-.4);
\draw[fill=white] (-.4,.5) rectangle +(-.1,-.4);
\node at (.4,-.1)[right]{$\scriptstyle{n}$};
\node at (-.4,-.1)[left]{$\scriptstyle{n}$};
\node at (.4,.5)[right]{$\scriptstyle{n}$};
\node at (-.4,.5)[left]{$\scriptstyle{n}$};
\node at (0,.4)[above]{$\scriptstyle{n-k}$};
\node at (0,-.4)[above]{$\scriptstyle{n-k}$};
}\,\Bigg\rangle_{\! 2}
=(-1)^{n-k}q^{\frac{k-n}{2}}\frac{(1-q^{n+1})}{(1-q^{k+1})}
\Big\langle\,\tikz[baseline=-.6ex]{
\draw (0,0) circle [radius=.3];
\draw[fill=white] (.1,-.05) rectangle (.5,.05);
\node at (.3,0)[above right]{$\scriptstyle{n}$};
}\,\Big\rangle_{\! 2}
\]
\end{LEM}

\begin{THM}\label{A12bridge}
Let $a_1, a_2, \dots, a_l$ be non-zero integers.
\begin{align*}
&J_{n+1}^{\mathfrak{sl}_2}(\left[2a_1,2a_2,\dots,2a_l\right];q)\\
&\quad=\prod_{j=0}^{l-1}\sum_{0\leq k_{\left|a_{j+1}\right|}^{(j+1)}\leq\dots\leq k_1^{(j+1)}\leq K_j}
(-1)^{K_j-k_{\left|a_{j+1}\right|}^{(j+1)}}
q^{a_{j+1}(n^2+2n)}
q^{\frac{\varepsilon_{j+1}}{2}(K_j-k_{\left|a_{j+1}\right|}^{(j+1)})}\\
&\qquad\times q^{\varepsilon_{j+1}\sum_{i=1}^{\left|a_{j+1}\right|}({k_{i}^{(j+1)}}^2+k_{i}^{(j+1)})}
\frac{(q^{\varepsilon_{j+1}})_{K_j}}{(q^{\varepsilon_{j+1}})_{k_{\left|a_{j+1}\right|}^{(j+1)}}}
{n\choose {k_{1}^{(j+1)}}',{k_{2}^{(j+1)}}',\dots,{k_{\left|a_{j+1}\right|}^{(j+1)}}',k_{\left|a_{j+1}\right|}^{(j+1)}}_{q^{\varepsilon_{j+1}}}\\
&\qquad\times(-1)^{n-K_l}q^{\frac{K_l-n}{2}}\frac{1-q^{n+1}}{1-q^{K_l+1}},
\end{align*}
where $\varepsilon_{j+1}=\frac{a_{j+1}}{\left|a_{j+1}\right|}$, $K_0=n, K_j=n-k_{\left|a_{j}\right|}^{(j)}$ and $k_0^{(j)}=K_j, {k_{i+1}^{(j+1)}}'=k_{i}^{(j)}-k_{i+1}^{(j)}$.
\end{THM}
\begin{proof}
We calculate the following clasped $A_1$ web using Lemma~\ref{A1slide} and Proposition~\ref{A1mfull}.
\begin{align*}
&\Bigg\langle\,\tikz[baseline=-.6ex]{
\draw[rounded corners] 
(-.5,.65) -- (-1.5,.65) -- (-1.5,-.65) -- (-.5,-.65);
\draw (-.5,.55) to[out=west, in=west] (-.5,.25);
\draw (-.5,.15) to[out=west, in=west] (-.5,-.15);
\draw (-.5,-.25) to[out=west, in=west] (-.5,-.55);
\draw (-.5,.6) -- (.8,.6);
\draw (-.5,.2) -- (.8,.2);
\draw (-.5,-.2) -- (.8,-.2);
\draw (-.5,-.6) -- (.8,-.6);
\draw[fill=white] (-.5,.5) rectangle +(.1,.2);
\draw[fill=white] (-.5,.1) rectangle +(.1,.2);
\draw[fill=white] (-.5,.-.3) rectangle +(.1,.2);
\draw[fill=white] (-.5,.-.7) rectangle +(.1,.2);
\draw[fill=white] (.8,.5) rectangle +(.1,.2);
\draw[fill=white] (.8,.1) rectangle +(.1,.2);
\draw[fill=white] (.8,.-.3) rectangle +(.1,.2);
\draw[fill=white] (.8,.-.7) rectangle +(.1,.2);
\draw[fill=white] (-.2,-.3) rectangle (.6,.3);
\node at (.2,.0) {$\scriptstyle{2a_{j+1}}$};
\node at (-.5,.4) [left]{$\scriptstyle{K_j}$};
\node at (-.5,.0) [left]{$\scriptstyle{n-K_j}$};
\node at (-.5,-.4) [left]{$\scriptstyle{K_j}$};
\node at (-1.5,.0) [left]{$\scriptstyle{n-K_j}$};
}\,\Bigg\rangle_{\! 2}\\
&\quad=\delta_n(K_j;q^{\varepsilon_{j+1}})_2^{2\left|a_{j+1}\right|}
\Bigg\langle\,\tikz[baseline=-.6ex]{
\draw[rounded corners] 
(-.5,.65) -- (-1.5,.65) -- (-1.5,-.65) -- (-.5,-.65);
\draw (-.5,.55) to[out=west, in=west] (-.5,.25);
\draw (.8,.15) to[out=west, in=west] (.8,-.15);
\draw (-.5,-.25) to[out=west, in=west] (-.5,-.55);
\draw (-.5,.6) -- (.8,.6);
\draw (-.5,.25) -- (.8,.25);
\draw (-.5,-.25) -- (.8,-.25);
\draw (-.5,-.6) -- (.8,-.6);
\draw[fill=white] (-.5,.5) rectangle +(.1,.2);
\draw[fill=white] (-.5,.-.7) rectangle +(.1,.2);
\draw[fill=white] (.8,.5) rectangle +(.1,.2);
\draw[fill=white] (.8,.1) rectangle +(.1,.2);
\draw[fill=white] (.8,.-.3) rectangle +(.1,.2);
\draw[fill=white] (.8,.-.7) rectangle +(.1,.2);
\draw[fill=white] (-.2,-.3) rectangle (.6,.3);
\node at (.2,.0) {$\scriptstyle{2a_{j+1}}$};
\node at (-.5,.4) [left]{$\scriptstyle{K_j}$};
\node at (.8,.0) [right]{$\scriptstyle{n-K_j}$};
\node at (-.5,-.4) [left]{$\scriptstyle{K_j}$};
\node at (-1.5,.0) [left]{$\scriptstyle{n-K_j}$};
}\,\Bigg\rangle_{\! 2}\\
&\quad=\delta_n(K_j;q^{\varepsilon_{j+1}})_2^{2\left|a_{j+1}\right|}
\sum_{0\leq k_{\left|a_{j+1}\right|}^{(j+1)}\leq\dots\leq k_1^{(j+1)}\leq K_j}
\gamma_{K_j}(k_1^{(j+1)},k_2^{(j+1)},\dots,k_{\left|a_{j+1}\right|}^{(j+1)};q^{\varepsilon_{j+1}})_2\\
&\qquad\times\Bigg\langle\,\tikz[baseline=-.6ex]{
\draw[rounded corners] 
(-.5,.65) -- (-1.8,.65) -- (-1.8,-.65) -- (-.5,-.65);
\draw (-.5,.55) to[out=west, in=west] (-.5,.25);
\draw (.8,.15) to[out=west, in=west] (.8,-.15);
\draw (-.5,-.25) to[out=west, in=west] (-.5,-.55);
\draw (-.5,.6) -- (.8,.6);
\draw (-.5,.25) -- (.8,.25);
\draw (-.5,-.25) -- (.8,-.25);
\draw (-.5,-.6) -- (.8,-.6);
\draw[fill=white] (-.5,.5) rectangle +(.1,.2);
\draw[fill=white] (-.5,.-.7) rectangle +(.1,.2);
\draw[fill=white] (.8,.5) rectangle +(.1,.2);
\draw[fill=white] (.8,.1) rectangle +(.1,.2);
\draw[fill=white] (.8,.-.3) rectangle +(.1,.2);
\draw[fill=white] (.8,.-.7) rectangle +(.1,.2);
\node at (-.5,.35) [left]{$\scriptstyle{k_{\left|a_{j+1}\right|}^{(j+1)}}$};
\node at (.8,.0) [right]{$\scriptstyle{n-k_{\left|a_{j+1}\right|}^{(j+1)}}$};
\node at (-.5,-.35) [left]{$\scriptstyle{k_{\left|a_{j+1}\right|}^{(j+1)}}$};
\node at (-1.8,.0) [left]{$\scriptstyle{n-k_{\left|a_{j+1}\right|}^{(j+1)}}$};
}\,\Bigg\rangle_{\! 2}\\
&\quad=\delta_n(K_j;q^{\varepsilon_{j+1}})_2^{\left|a_{j+1}\right|}
\sum_{0\leq k_{\left|a_{j+1}\right|}^{(j+1)}\leq\dots\leq k_1^{(j+1)}\leq K_j}
\gamma_{K_j}(k_1^{(j+1)},k_2^{(j+1)},\dots,k_{\left|a_{j+1}\right|}^{(j+1)};q^{\varepsilon_{j+1}})_2\\
&\qquad\times\Bigg\langle\,\tikz[baseline=-.6ex]{
\draw[rounded corners] 
(-.5,.65) -- (-1.8,.65) -- (-1.8,-.65) -- (-.5,-.65);
\draw (-.5,.55) to[out=west, in=west] (-.5,.25);
\draw (-.5,.15) to[out=west, in=west] (-.5,-.15);
\draw (-.5,-.25) to[out=west, in=west] (-.5,-.55);
\draw[fill=white] (-.5,.5) rectangle +(.1,.2);
\draw[fill=white] (-.5,.1) rectangle +(.1,.2);
\draw[fill=white] (-.5,.-.3) rectangle +(.1,.2);
\draw[fill=white] (-.5,.-.7) rectangle +(.1,.2);
\node at (-.5,.4) [left]{$\scriptstyle{n-K_{j+1}}$};
\node at (-.5,.0) [left]{$\scriptstyle{K_{j+1}}$};
\node at (-.5,-.4) [left]{$\scriptstyle{n-K_{j+1}}$};
\node at (-1.8,.0) [left]{$\scriptstyle{K_{j+1}}$};
}\,\Bigg\rangle_{\! 2},
\end{align*}
where
\begin{align*}
\delta_n(K_j;q)_2&=(-1)^{n-K_j}q^{-\frac{n^2+2n-K_j^2-2K_j}{4}}\quad\text{and}\\
\gamma_{K_j}(k_1^{(j+1)},k_2^{(j+1)},\dots,k_{\left|a_{j+1}\right|}^{(j+1)};q)_2
&=(-1)^{K_j-k_{\left|a_{j+1}\right|}^{(j+1)}}
q^{-\frac{\left|a_{j+1}\right|}{2}({k_{\left|a_{j+1}\right|}^{(j+1)}}^2+2{k_{\left|a_{j+1}\right|}^{(j+1)}})}\\
&\qquad\times q^{\frac{1}{2}(K_j-k_{\left|a_{j+1}\right|}^{(j+1)})}
q^{\sum_{i=1}^{\left|a_{j+1}\right|}({k_{i}^{(j+1)}}^2+k_{i}^{(j+1)})}\\
&\qquad\times \frac{(q)_{K_j}}{(q)_{k_{\left|a_{j+1}\right|}^{(j+1)}}}
{n\choose {k_{1}^{(j+1)}}',{k_{2}^{(j+1)}}',\dots,{k_{\left|a_{j+1}\right|}^{(j+1)}}',k_{\left|a_{j+1}\right|}^{(j+1)}}_{q}.
\end{align*}
We also obtain the following in the same way.
\begin{align*}
&\Bigg\langle\,\tikz[baseline=-.6ex]{
\draw[rounded corners] 
(-.5,.65) -- (-1.5,.65) -- (-1.5,-.65) -- (-.5,-.65);
\draw (-.5,.55) to[out=west, in=west] (-.5,.25);
\draw (-.5,.15) to[out=west, in=west] (-.5,-.15);
\draw (-.5,-.25) to[out=west, in=west] (-.5,-.55);
\draw (-.5,.6) -- (.8,.6);
\draw (-.5,.2) -- (.8,.2);
\draw (-.5,-.2) -- (.8,-.2);
\draw (-.5,-.6) -- (.8,-.6);
\draw[fill=white] (-.5,.5) rectangle +(.1,.2);
\draw[fill=white] (-.5,.1) rectangle +(.1,.2);
\draw[fill=white] (-.5,.-.3) rectangle +(.1,.2);
\draw[fill=white] (-.5,.-.7) rectangle +(.1,.2);
\draw[fill=white] (.8,.5) rectangle +(.1,.2);
\draw[fill=white] (.8,.1) rectangle +(.1,.2);
\draw[fill=white] (.8,.-.3) rectangle +(.1,.2);
\draw[fill=white] (.8,.-.7) rectangle +(.1,.2);
\draw[fill=white] (-.2,-.7) rectangle (.6,-.1);
\node at (.2,-.4) {$\scriptstyle{2a_{j+1}}$};
\node at (-.5,.4) [left]{$\scriptstyle{n-K_j}$};
\node at (-.5,.0) [left]{$\scriptstyle{K_j}$};
\node at (-.5,-.4) [left]{$\scriptstyle{n-K_j}$};
\node at (-1.5,.0) [left]{$\scriptstyle{K_j}$};
}\,\Bigg\rangle_{\! 2}\\
&\quad=\delta_n(K_j;q^{\varepsilon_{j+1}})_2^{2\left|a_{j+1}\right|}
\sum_{0\leq k_{\left|a_{j+1}\right|}^{(j+1)}\leq\dots\leq k_1^{(j+1)}\leq K_j}
\gamma_{K_j}(k_1^{(j+1)},k_2^{(j+1)},\dots,k_{\left|a_{j+1}\right|}^{(j+1)};q^{\varepsilon_{j+1}})_2\\
&\qquad\times\Bigg\langle\,\tikz[baseline=-.6ex]{
\draw[rounded corners] 
(-.5,.65) -- (-1.8,.65) -- (-1.8,-.65) -- (-.5,-.65);
\draw (-.5,.55) to[out=west, in=west] (-.5,.25);
\draw (-.5,.15) to[out=west, in=west] (-.5,-.15);
\draw (-.5,-.25) to[out=west, in=west] (-.5,-.55);
\draw[fill=white] (-.5,.5) rectangle +(.1,.2);
\draw[fill=white] (-.5,.1) rectangle +(.1,.2);
\draw[fill=white] (-.5,.-.3) rectangle +(.1,.2);
\draw[fill=white] (-.5,.-.7) rectangle +(.1,.2);
\node at (-.5,.4) [left]{$\scriptstyle{K_{j+1}}$};
\node at (-.5,.0) [left]{$\scriptstyle{n-K_{j+1}}$};
\node at (-.5,-.4) [left]{$\scriptstyle{K_{j+1}}$};
\node at (-1.8,.0) [left]{$\scriptstyle{n-K_{j+1}}$};
}\,\Bigg\rangle_{\! 2}.
\end{align*}
Therefore,
\begin{align*}
&\langle \bar{L}(n)\rangle_2\\
&\quad=\prod_{j=0}^{l-1}\sum_{0\leq k_{\left|a_{j+1}\right|}^{(j+1)}\leq\dots\leq k_1^{(j+1)}\leq K_j}
\delta_n(K_j;q^{\varepsilon_{j+1}})_2^{2\left|a_{j+1}\right|}
\gamma_{K_j}(k_1^{(j+1)},k_2^{(j+1)},\dots,k_{\left|a_{j+1}\right|}^{(j+1)};q^{\varepsilon_{j+1}})_2\\
&\qquad\times
\Bigg\langle\,\tikz[baseline=-1.8ex, scale=1.2]{
\draw[rounded corners=.1cm] (-.4,.3) -- ++(-.4,0) -- ++(.0,-1.2) -- ++(1.6,.0) -- ++(.0,1.2) -- (.4,.3);
\draw[rounded corners=.1cm] (-.4,-.3) -- ++(-.2,0) -- ++(.0,-.4) -- ++(1.2,.0) -- ++(.0,.4) -- (.4,-.3);
\draw (-.4,.4) -- (.4,.4);
\draw (-.4,-.4) -- (.4,-.4);
\draw (-.4,.2) to[out=east, in=east] (-.4,-.2);
\draw (.4,.2) to[out=west, in=west] (.4,-.2);
\draw[fill=white] (.4,-.5) rectangle +(.1,.4);
\draw[fill=white] (-.4,-.5) rectangle +(-.1,.4);
\draw[fill=white] (.4,.5) rectangle +(.1,-.4);
\draw[fill=white] (-.4,.5) rectangle +(-.1,-.4);
\node at (.4,-.1)[right]{$\scriptstyle{n}$};
\node at (-.4,-.1)[left]{$\scriptstyle{n}$};
\node at (.4,.5)[right]{$\scriptstyle{n}$};
\node at (-.4,.5)[left]{$\scriptstyle{n}$};
\node at (0,.4)[above]{$\scriptstyle{n-K_l}$};
\node at (0,-.4)[above]{$\scriptstyle{n-K_l}$};
}\,\Bigg\rangle_{\! 2}
\end{align*}
for $L=\left[2a_1,2a_2,\dots,2a_l\right]$.
The writhe of $\left[2a_1,2a_2,\dots,2a_l\right]$ is $-2(a_1+a_2+\dots+a_l)$. 
Lemma~\ref{A1closure} and explicit calculation of the coefficient imply the formula in this theorem.
\end{proof}

\subsection{The $\mathfrak{sl}_3$ colored Jones polynomials}
We introduce the $\mathfrak{sl}_3$ colored Jones polynomial of type $(n,0)$.
Let $L$ be an oriented link diagram in $D_0$ with ordered link components $(L_1,L_2,\dots,L_r)$. 
We replace a part of $L_i$ with 
\tikz[baseline=-.6ex]{
\draw[->-=.8] (-.5,0) -- (.5,0);
\draw[fill=white] (-.1,-.3) rectangle (.1,.3);
\node at (.1,0) [above right]{${\scriptstyle n}$};
}\, 
for $i=1,2,\dots,r$ (see Figure~\ref{L(n,0)}).
We denote this oriented link diagram decorated with white boxes by $L(n,0)$.
\begin{figure}
\centering
\begin{tikzpicture}
\draw[rotate=120] (90:.7) to[out=east, in=east] (90:1.3);
\draw (-30:.3) to[out=-120, in=-60] (210:1);
\draw[rotate=120] (-30:.3) to[out=-120, in=-60] (210:1);
\draw[rotate=240] (-30:.3) to[out=-120, in=-60] (210:1);
\draw[->-=0, white, double=black, double distance=0.4pt, ultra thick] 
(90:1) to[out=0, in=60] (-30:.3);
\draw[white, double=black, double distance=0.4pt, ultra thick, rotate=120] 
(90:1) to[out=0, in=60] (-30:.3);
\draw[white, double=black, double distance=0.4pt, ultra thick, rotate=240] 
(90:1) to[out=0, in=60] (-30:.3);
\draw[-<-=1, white, double=black, double distance=0.4pt, ultra thick, rotate=120] 
(90:.7) to[out=west, in=west] (90:1.3);
\draw[cyan, dashed] (90:1) circle (.2);
\draw[cyan, dashed] (210:1.3) circle (.2);
\node at (-90:1.2){$L$};
\end{tikzpicture}
\begin{tikzpicture}
\draw[rotate=120] (90:.7) to[out=east, in=east] (90:1.3);
\draw (-30:.3) to[out=-120, in=-60] (210:1);
\draw[rotate=120] (-30:.3) to[out=-120, in=-60] (210:1);
\draw[rotate=240] (-30:.3) to[out=-120, in=-60] (210:1);
\draw[->-=.15, white, double=black, double distance=0.4pt, ultra thick] 
(90:1) to[out=0, in=60] (-30:.3);
\draw[white, double=black, double distance=0.4pt, ultra thick, rotate=120] 
(90:1) to[out=0, in=60] (-30:.3);
\draw[white, double=black, double distance=0.4pt, ultra thick, rotate=240] 
(90:1) to[out=0, in=60] (-30:.3);
\draw[-<-=.85, white, double=black, double distance=0.4pt, ultra thick, rotate=120] 
(90:.7) to[out=west, in=west] (90:1.3);
\draw[fill=white] (-.05,.8) rectangle (.05,1.2);
\draw[fill=white, rotate=120] (-.05,1.1) rectangle (.05,1.5);
\node at (90:1) [above right]{${\scriptstyle n}$};
\node at (210:1.3) [below right]{${\scriptstyle n}$};
\node at (-90:1.2) {$L(n,0)$};
\end{tikzpicture}\\
Replace 
\,\tikz[baseline=-.6ex]{
\draw[->-=.8] (-.3,.0) -- (.3,.0);
\draw[cyan, dashed] (.0,.0) circle (.3);
}\,
with
\,\tikz[baseline=-.6ex]{
\draw[->-=.8] (-.3,.0) -- (.3,.0);
\draw[fill=white] (-.05,-.2) rectangle (.05,.2);
\draw[cyan, dashed] (.0,.0) circle (.3);
}\,
\caption{A Link diagram colored by $(n,0)$}
\label{L(n,0)}
\end{figure}
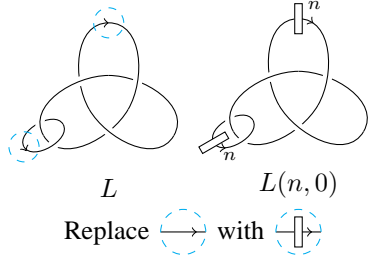
\begin{DEF}
The colored $\mathfrak{sl}_3$ Jones polynomial $J_{(n,0)}^{\mathfrak{sl}_2}(L;q)$ of a link represented by a link diagram $L$ is defined by
\[
J_{(n,0)}^{\mathfrak{sl}_3}(L;q)
=(q^{\frac{n^2+3n}{3}})^{-\mathop{w}(L)}
\langle L(n,0)\rangle_{3}/\langle\,\tikz[baseline=-.6ex, scale=.5]{
\draw[->-=.5] (0,0) circle [radius=.3];
\draw[fill=white] (.1,-.05) rectangle (.5,.05);
\node at (.3,0)[above right]{$\scriptstyle{n}$};
}\,\rangle_{3},
\]
where $\mathop{w}(L)$ is the writhe of $L$.
\end{DEF}
\begin{LEM}\label{A2closure}
\[
\Bigg\langle\,\tikz[baseline=-1.8ex, scale=1.2]{
\draw[rounded corners=.1cm] (-.4,.3) -- ++(-.4,0) -- ++(.0,-1.2) -- ++(1.6,.0) -- ++(.0,1.2) -- (.4,.3);
\draw[rounded corners=.1cm] (-.4,-.3) -- ++(-.2,0) -- ++(.0,-.4) -- ++(1.2,.0) -- ++(.0,.4) -- (.4,-.3);
\draw[-<-=.5] (-.4,.4) -- (.4,.4);
\draw[->-=.5] (-.4,-.4) -- (.4,-.4);
\draw[-<-=.5] (-.4,.2) to[out=east, in=east] (-.4,-.2);
\draw[->-=.5] (.4,.2) to[out=west, in=west] (.4,-.2);
\draw[fill=white] (.4,-.5) rectangle +(.1,.4);
\draw[fill=white] (-.4,-.5) rectangle +(-.1,.4);
\draw[fill=white] (.4,.5) rectangle +(.1,-.4);
\draw[fill=white] (-.4,.5) rectangle +(-.1,-.4);
\node at (.4,-.1)[right]{$\scriptstyle{n}$};
\node at (-.4,-.1)[left]{$\scriptstyle{n}$};
\node at (.4,.5)[right]{$\scriptstyle{n}$};
\node at (-.4,.5)[left]{$\scriptstyle{n}$};
\node at (0,.4)[above]{$\scriptstyle{n-k}$};
\node at (0,-.4)[above]{$\scriptstyle{n-k}$};
}\,\Bigg\rangle_{\! 3}
=q^{-(n-k)}\frac{(1-q^{n+1})(1-q^{n+2})}{(1-q^{k+1})(1-q^{k+2})}
\Big\langle\,\tikz[baseline=-.6ex]{
\draw[->-=.5] (0,0) circle [radius=.3];
\draw[fill=white] (.1,-.05) rectangle (.5,.05);
\node at (.3,0)[above right]{$\scriptstyle{n}$};
}\,\Big\rangle_{\! 3}
\]
\end{LEM}
\begin{proof}
Use Lemma~\ref{A2clasplem}.
\end{proof}
\begin{THM}\label{A22bridge}
\begin{align*}
&J_{(n,0)}^{\mathfrak{sl}_3}(\left[2a_1,2a_2,\dots,2a_l\right];q)\\
&\quad=\prod_{j=0}^{l-1}\sum_{0\leq k_{\left|a_{j+1}\right|}^{(j+1)}\leq\dots\leq k_1^{(j+1)}\leq K_j}
q^{\varepsilon_{j+1}(K_j-k_{\left|a_{j+1}\right|}^{(j+1)})}
q^{\varepsilon_{j+1}\sum_{i=1}^{\left|a_{j+1}\right|}({k_{i}^{(j+1)}}^2+2k_{i}^{(j+1)})}\\
&\qquad\times\frac{(q^{\varepsilon_{j+1}})_{K_j}}{(q^{\varepsilon_{j+1}})_{k_{\left|a_{j+1}\right|}^{(j+1)}}}
{n\choose {k_{1}^{(j+1)}}',{k_{2}^{(j+1)}}',\dots,{k_{\left|a_{j+1}\right|}^{(j+1)}}',k_{\left|a_{j+1}\right|}^{(j+1)}}_{q^{\varepsilon_{j+1}}}\\
&\qquad\times q^{-(n-K_l)}\frac{(1-q^{n+1})(1-q^{n+2})}{(1-q^{K_l+1})(1-q^{K_l+2})},
\end{align*}
where $\varepsilon_{j+1}=\frac{a_{j+1}}{\left|a_{j+1}\right|}$, $K_0=n, K_j=n-k_{\left|a_{j}\right|}^{(j)}$ and $k_0^{(j)}=K_j, {k_{\left|a_{i+1}\right|}^{(j+1)}}'=k_{i}^{(j)}-k_{i+1}^{(j)}$.
\end{THM}
\begin{proof}
It is proved in the same way as the proof of Theorem~\ref{A12bridge} using Proposition~\ref{A2mfull} and Lemma~\ref{A2closure}. 
Instead of $\delta_n(K_j;q)_2$ and $\gamma_{K_j}(k_1^{(j+1)},k_2^{(j+1)},\dots,k_{\left|a_{j+1}\right|}^{(j+1)};q)_2
$, 
we use 
\begin{align*}
\delta_n(K_j;q)_3&=q^{-\frac{n^2+3n-K_j^2-3K_j}{3}}\quad\text{and}\\
\gamma_{K_j}(k_1^{(j+1)},k_2^{(j+1)},\dots,k_{\left|a_{j+1}\right|}^{(j+1)};q)_3
&=q^{-\frac{2\left|a_{j+1}\right|}{3}({k_{\left|a_{j+1}\right|}^{(j+1)}}^2+3{k_{\left|a_{j+1}\right|}^{(j+1)}})}\\
&\quad\times q^{K_j-k_{\left|a_{j+1}\right|}^{(j+1)}}
q^{\sum_{i=1}^{\left|a_{j+1}\right|}({k_{i}^{(j+1)}}^2+2k_{i}^{(j+1)})}\\
&\quad\times\frac{(q)_{K_j}}{(q)_{k_{\left|a_{j+1}\right|}^{(j+1)}}}
{n\choose {k_{1}^{(j+1)}}',{k_{2}^{(j+1)}}',\dots,{k_{\left|a_{j+1}\right|}^{(j+1)}}',k_{\left|a_{j+1}\right|}^{(j+1)}}_{q},
\end{align*}
respectively.
\end{proof}
\begin{RMK}
Wedrich gave a formula for the HOMFLY polynomials of $2$-bridge links colored with one-column Young diagrams in \cite{Wedrich14}. 
We consider that our formula should conicide with specialized his formula. 
He also suggest us the relationship between our calculus and the calculus of MOY graphs~\cite{MurakamiOhtsukiYamada98} thorough a work of Tubbenhauer, Vaz and Wedrich~\cite{TubbenhauerVazWedrich15}. 
However, 
we have not confirmedk the specific relationship.
\end{RMK}

\subsection*{Acknowledgment}
The author would like to express his gratitude to his adviser, Professor Hisaaki Endo, for his encouragement. He also would like to thank Paul Wedrich for helpful comments. 

\bibliographystyle{amsplain}
\bibliography{yuasa3}
\end{document}